\documentclass[10pt,reqno]{amsart}
\usepackage{amsmath,amsthm,amssymb,mathrsfs,stmaryrd,color}
\usepackage[all]{xy}
\usepackage{url}
\usepackage{geometry}
\usepackage{yhmath}
\usepackage[utf8]{inputenc}
\usepackage[T1]{fontenc}

\usepackage{relsize}
\usepackage[bbgreekl]{mathbbol}
\usepackage{amsfonts}

\DeclareSymbolFontAlphabet{\mathbb}{AMSb}
\DeclareSymbolFontAlphabet{\mathbbl}{bbold}
\newcommand{\Prism}{\mathbbl{\Delta}}

\usepackage{enumitem}
\setlist[enumerate]{itemsep=2pt,parsep=2pt,before={\parskip=2pt}}

\usepackage[colorlinks=true,hyperindex, linkcolor=magenta, pagebackref=false, citecolor=cyan,pdfpagelabels]{hyperref}

\newcommand{\cosimp}[3]{\xymatrix@1{#1 \ar@<.4ex>[r] \ar@<-.4ex>[r] & {\ }#2 \ar@<0.8ex>[r] \ar[r] \ar@<-.8ex>[r] & {\ } #3 \ar@<1.2ex>[r] \ar@<.4ex>[r] \ar@<-.4ex>[r] \ar@<-1.2ex>[r] & \cdots }}

\newcommand{\adjunction}[4]{\xymatrix@1{#1{\ } \ar@<0.3ex>[r]^{ {\scriptstyle #2}} & {\ } #3 \ar@<0.3ex>[l]^{ {\scriptstyle #4}}}}

\usepackage{amsfonts}
\usepackage{amsmath}
\usepackage{amssymb}
\usepackage{amsthm}
\usepackage{mathrsfs}
\usepackage{graphicx}
\usepackage{tikz}
\usepackage{tikz-cd}
\usepackage{tabu}
\usepackage{hyperref}
\hypersetup
{
	colorlinks=true,
    linkcolor=red,
    filecolor=magenta,
    urlcolor=cyan,
    linkbordercolor=0 0 1,
    bookmarks=true,
}

\newcommand{\Ainf}[0]{\mathbb{A}_{\operatorname{inf}}}

\newcommand{\ad}[0]{\operatorname{ad}}

\newcommand{\aff}[0]{\mathbb{A}}

\newcommand{\arc}[0]{\operatorname{arc}}

\newcommand{\Betti}[0]{\operatorname{Betti}}

\newcommand{\bra}[0]{\langle}
\newcommand{\C}[0]{\mathcal{C}}

\newcommand{\cn}[0]{\mathbb{C}}

\newcommand{\cofib}[0]{\operatorname{cofib}}
\newcommand{\coker}[0]{\operatorname{coker}}

\newcommand{\D}[0]{\mathcal{D}}

\newcommand{\Div}[0]{\operatorname{Div}}

\newcommand{\dR}[0]{\operatorname{dR}}

\newcommand{\disk}[0]{\mathbb{D}}

\newcommand{\E}[0]{\mathcal{E}}

\newcommand{\End}[0]{\operatorname{End}}

\newcommand{\Ext}[0]{\operatorname{Ext}}
\newcommand{\Exp}[0]{\operatorname{Exp}}
\newcommand{\e}[0]{\exists\,}

\newcommand{\F}[0]{\mathcal{F}}
\newcommand{\FF}[0]{\operatorname{FF}}

\newcommand{\f}[0]{\forall\,}
\newcommand{\ff}[0]{\mathbb{F}}

\newcommand{\GL}[0]{\mathrm{GL}}

\newcommand{\Ga}[0]{\mathbb{G}_{a}}
\newcommand{\gff}[0]{\mathbb{G}_{m}}

\newcommand{\Hom}[0]{\operatorname{Hom}}

\newcommand{\I}[0]{\mathcal{I}}

\newcommand{\id}[0]{\operatorname{id}}

\newcommand{\im}[0]{\operatorname{im}}

\newcommand{\itg}[0]{\mathbb{Z}}

\newcommand{\K}[0]{\mathcal{K}}
\newcommand{\ket}[0]{\rangle}

\renewcommand{\L}[0]{\mathcal{L}}

\newcommand{\M}[0]{\mathcal{M}}

\newcommand{\N}[0]{\mathcal{N}}

\renewcommand{\O}[0]{\mathcal{O}}

\newcommand{\op}[0]{\operatorname{op}}

\newcommand{\perf}[0]{\operatorname{perf}}

\newcommand{\proj}[0]{\mathbb{P}}

\newcommand{\R}[0]{\mathcal{R}}

\newcommand{\rig}[0]{\operatorname{rig}}

\newcommand{\rn}[0]{\mathbb{R}}

\newcommand{\rt}[0]{\mathbb{Q}}

\newcommand{\Spa}[0]{\operatorname{Spa}}
\newcommand{\Spd}[0]{\operatorname{Spd}}
\newcommand{\Spf}[0]{\operatorname{Spf}}
\newcommand{\Spec}[0]{\operatorname{Spec}}

\newcommand{\st}[0]{\operatorname{st}}

\newcommand{\X}[0]{\mathfrak{X}_{\eta}}
\newcommand{\Xf}[0]{\mathfrak{X}}
\newcommand{\Y}[0]{\mathcal{Y}}

\newcommand{\cat}[1]{\textsf{\upshape{#1}}}

	\newtheorem{theorem}{Theorem}[subsection]
	\newtheorem*{theorem*}{Theorem}
	\newtheorem*{definition*}{Definition}
	\newtheorem{proposition}[theorem]{Proposition}
	\newtheorem{lemma}[theorem]{Lemma}
	\newtheorem{corollary}[theorem]{Corollary}
	\newtheorem{definition-theorem}[theorem]{Definition-Theorem}

	\theoremstyle{definition}
	\newtheorem{definition}[theorem]{Definition}
	
	\newtheorem{remark}[theorem]{Remark}
	
	\newtheorem{example}[theorem]{Example}

	\newtheorem{construction}[theorem]{Construction}

  \newtheorem{discussion}[theorem]{}

\newcommand{\red}[0]{\operatorname{red}}
\newcommand{\AnSpec}[0]{\operatorname{AnSpec}}
\newcommand{\an}[0]{\operatorname{an}}

\newcommand{\Nil}[0]{\operatorname{Nil}}

\newcommand{\qfd}[0]{\operatorname{qfd}}
\newcommand{\RIG}[0]{\widetilde{\operatorname{HK}}}
\newcommand{\arith}[0]{\operatorname{arith}}
\newcommand{\torus}[0]{\mathbb{T}}
\newcommand{\Be}[0]{\operatorname{Be}}

\newcommand{\HK}[0]{\operatorname{HK}}
\makeatletter
\newcommand*{\doublerightarrow}[2]{\mathrel{
  \settowidth{\@tempdima}{$\scriptstyle#1$}
  \settowidth{\@tempdimb}{$\scriptstyle#2$}
  \ifdim\@tempdimb>\@tempdima \@tempdima=\@tempdimb\fi
  \mathop{\vcenter{
    \offinterlineskip\ialign{\hbox to\dimexpr\@tempdima+1em{##}\cr
    \rightarrowfill\cr\noalign{\kern.5ex}
    \rightarrowfill\cr}}}\limits^{\!#1}_{\!#2}}}
\makeatother
\begin{document}
\title{The arithmetic de Rham stack}
\author{Junhui Qin}
\address{
  Institut de Recherche Mathématique Avancée, 7 rue René Descartes,
  67084 Strasbourg, France
}
\email{junhui.qin@unistra.fr}

\begin{abstract}
  We define and study the arithmetic de Rham stack $X^{\arith}$ of a scheme $X$ over a field of characteristic $p$,
  and analyze its relation with related stacks such as the Hyodo--Kato stack.
  We show that $X^{\arith}$ gives a stack-theoretic approach to rigid cohomology and its coefficients,
  known as overconvergent isocrystals and arithmetic $D$-modules,
  which avoids the long-standing problem of frame-choosing in earlier approaches.
  We finally give some arithmetic applications of our formalism,
  such as a proof of Berthelot's conjecture on the preservation of overconvergent isocrystals by smooth proper pushforward.
\end{abstract}
\maketitle
\begingroup
\setlength{\baselineskip}{0.9\baselineskip}
\tableofcontents
\endgroup
\section{Introduction}

Let $p$ be a prime number.

\subsection{Stacky approach and transmutation}

A general way to study a cohomology theory is through the stacky approach.
Viewing a cohomology theory as a general mechanism 
\[\begin{array}{ccc}
  \{\text{geometric objects e.g. schemes}\} &  \longrightarrow &\{\text{linear-algebraic objects}\}\\
  X & \longmapsto & R\Gamma_{?}(X),
\end{array}\]
then a stacky approach to the cohomology theory,
in one sentence,
is to factorize this mechanism through an intermediate category of stacks whose coherent cohomology recovers the given cohomology theory:
\[\begin{array}{ccccc}
  \{\text{geometric objects}\} &  \longrightarrow &\{\text{stacks}\} & \longrightarrow &\{\text{linear-algebraic objects}\}\\
  X & \longmapsto & X^? & \longmapsto & R\Gamma(X^?,\O_{X^?})\cong R\Gamma_{?}(X).
\end{array}\]
The advantage of having a stacky approach to a given cohomology theory is that
 it provides for it a theory of the coefficients together with six operations: indeed, it suffices to consider the category of quasi-coherent sheaves on the associated stack.
For example,
Simpson's algebraic de Rham stack in \cite{Sim96} is a stacky approach to algebraic de Rham cohomology in characteristic $0$,
where the cohomology of its structure sheaf will recover algebraic de Rham cohomology,
but also the category of quasi-coherent sheaves on it will recover the category of algebraic $D$-modules,
cf. \cite{GR14} and \cite{RC25a}.

Under mild conditions,
all stacky approaches come from the method of \emph{transmutation}.
Pioneered by Bhatt--Lurie \cite{BL22} and Drinfeld \cite{Dri20},
the transmutation means that the stack $X^?$ is determined by the ring stack $\aff^{1,?}$:
\[
X^?(-)=X(\aff^{1,?}(-)).
\]
Thus,
to implement such a stacky approach,
it suffices to construct $\aff^{1,?}$ and then do transmutation,
and many properties about $X^?$ reduce to properties about $\aff^{1,?}$.

Finally,
let us mention that not all cohomology theories come from stacky approaches.
For example,
the theory of $\ell$-adic étale cohomology does not come from \textit{any} stack in the classical sense due to the failure of the categorical Künneth formula.
A better way is to work at higher categorical levels, replacing the notion of (analytic) stacks by Gestalten,
cf. \cite{Sch25}. This generality will not be needed in this work.

\subsection{Rigid cohomology and its coefficients}\label{padiccoefficientintroduction}

The study of the geometry of schemes over a finite field of characteristic $p$
lies at the core of arithmetic geometry,
for example the Weil conjecture and the geometric Langlands program. One would like a $p$-adic companion to $\ell$-adic étale cohomology for such schemes.
It should have good properties such as being a Weil cohomology theory,
producing the right Betti numbers and so on.
Some naïve candidates,
such as de Rham cohomology or $p$-adic étale cohomology,
fail to be finite-dimensional and $\aff^1$-invariant in general.
For example,
for affine $X$,
the algebra $\D_X$ has non-trivial center due to the existence of the $p$-curvature morphism.
Since de Rham cohomology is an algebra over the center of $\D_X$,
this implies it is generally infinite-dimensional over $k$.

The correct $p$-adic companion is \emph{rigid cohomology}.
It was introduced by Berthelot \cite{Ber86},
and shown to be a Weil cohomology theory.
The idea of its construction can be summarized in one sentence:
let $X$ be a separated scheme of finite type over a perfect field $k$ of characteristic $p$.
$W(k)$ be the ring of Witt vectors of $k$,
and $K=W(k)[p^{-1}]$,
then the rigid cohomology $R\Gamma_{\rig}(X/K)$ of $X$ is the overconvergent de Rham cohomology of the generic fiber of a smooth formal lift of $X$ (when it exists) to $W(k)$.

There exists a theory of coefficients for rigid cohomology,
the study of which was pioneered by Berthelot \cite{Ber86, Ber96a, Ber96b}.
The category $\cat{Isoc}^{\dagger}(X/K)$ of overconvergent isocrystals,
serves as the analogue for rigid cohomology of de Rham local systems (i.e. vector bundles with flat connections) for algebraic de Rham cohomology.
Several categories of arithmetic $D$-modules are also defined,
and serve, informally speaking, as an analogue of algebraic $D$-modules.

However,
the magic and also the complexity of the whole theory is that the definition of rigid cohomology involves the choice of a formal lift and an embedding,
or equivalently saying, a \emph{frame};
while finally it is proved to be independent of these choices and thus well defined, this makes it sometimes difficult to check that various constructions are well defined or functorial.
Even though several approaches independent of the choice of the frame have been developped,
e.g.,
by using the \emph{overconvergent site} in \cite{LS11},
an obstruction of these approaches appears when handling with six operations.
For example,
to define and compute the pushforward in overconvergent site,
one needs to collect the information about all of the frames,
which is hard to control.

Another subtlety is that the sheaf of rings of arithmetic differential operators is naturally equipped with a topology,
which makes the category of arithmetic $D$-modules hard to define.
Contrary to crystalline cohomology with which it shares some relation, the theory of rigid cohomology is inherently analytic in nature.

\subsection{Main results of the paper}
The main goal of this paper is to construct stacky approaches to rigid cohomology and its coefficients, study their properties in this abstract framework and derive some applications.
Thanks to the nature of the stacky approach,
one will obtain freely a frame-independent definition of rigid cohomology,
a category of arithmetic $D$-modules,
and a six-functor formalism (hence some control on the six operations).

Since rigid cohomology is not defined in an algebro-geometric way,
a stacky approach in the setting of algebraic stacks does not seem appropriate.
The theory of analytic stacks \cite{CS23} provides a good framework for dealing with rigid geometry,
hence rigid cohomology.
One may work in the category of analytic stacks over $\rt_{p,\blacksquare}$,
but the test category is still too big for our theory. We want a category of analytic stacks closer to the world of rigid geometry. 
We will work in the setting of Gelfand stacks,
cf. \cite{ABLBRCS25}.
They are defined analogously to analytic stacks, but replacing general analytic rings by Gelfand rings\footnote{Strictly speaking, it is most convenient to restrict to separable and quasi-finite dimensional Gelfand rings, see the main body of the text, but we will gloss over this technicality in this introduction.}, a class of analytic rings roughly formed by overconvergent thickenings of Banach $\rt_p$-algebras. This leads to the test category of (separable qfd) Gelfand rings.
It has as a basis (for the descendable topology) a subcategory consisting of uniformly strictly totally disconnected perfectoid rings, that is the Gelfand rings whose uniform completion is strictly totally disconnected perfectoid.
Let $\cat{uPerfd}_{\omega_1}^{\operatorname{std},\qfd}\subset \cat{GelfRing}_{\omega_1}^{\qfd}$ be this fully faithful embedding.
We fix terminology on Gelfand rings.
Let $A$ be a Gelfand ring.
We denote $\operatorname{GSpec}A$ the Gelfand stack represented by $A$.
We let $A^{\le 1}$ be the subring of ``elements of norm $\le 1$'',
$A^{<1}\subset A^{\le 1}$ be the ideal of ``elements of norm $< 1$'',
and $A^u$ be the uniform completion of $A$.
If $A$ is further uniformly perfectoid,
then we denote $A^{u,\flat}$ as the tilting of $A^u$.

Let $k$ be a field of characteristic $p$ (which is allowed to be non-perfect!),
and $X$ be a qcqs separated scheme of finite type over $k$.
Let $K$ be a complete discrete valuation field with residue field $k$.
We define the (absolute) arithmetic de Rham stack as follows.

\begin{definition}[Definition \ref{arithdeRhamdefinition}]
  Define the \emph{arithmetic de Rham stack} $X^{\arith}$ of $X$,
  as a (qfd) Gelfand stack,
  to be the sheafification of the following pre-stack
  \[
  X^{\arith,\operatorname{pre}}(A):=X(A^{\le 1}/A^{<1}),\,A\in\cat{uPerfd}^{\qfd}_{\omega_1}.
  \]
\end{definition}

\begin{remark}
  As a toy model that we will use throughout this section,
  let $X=\aff^1_{\ff_p}$,
  then $\aff^{1,\arith}_{\ff_p}=\disk_{\rt_p}^{\dagger}/\disk^{\circ}_{\rt_p}$,
  which is already mentioned in \cite[Lecture XII]{Sch23}.
  Here $\disk^{\dagger}_{\rt_p}$ (resp. $\disk^{\circ}_{\rt_p}$) is the overconvergent (resp. open) unit disk over $\rt_p$.
\end{remark}

Intuitively,
if $\Xf$ is a formal lift of $X$ to $\itg_p$,
then there is a map from the \emph{overconvergent}\footnote{With respect to some embedding into a partially proper rigid-analytic space.} generic fiber $\X^{\dagger}$ to $X^{\arith}$,
and $X^{\arith}$ is universal among such lifts.
In particular,
if $X=\Spec k$,
then we have a morphism $\operatorname{GSpec}K\rightarrow\Spec k$.
We can define a relative version of arithmetic de Rham stack as the fiber product below:
\[\xymatrix{
    X^{\arith/K}\ar[r]\ar[d] & \operatorname{GSpec}K\ar[d]\\
    X^{\arith}\ar[r] & (\Spec k)^{\arith}
  }\]

The main six-functor properties of arithmetic de Rham stack are as follows.

\begin{theorem}[Theorem \ref{arithsixfunctor}]
  The arithmetic de Rham stack $X^{\arith}$ has the following properties.
  \begin{enumerate}
    \item Let $f\colon X\rightarrow Y$ be an étale (resp. proper; smooth) morphism.
    Then $f^{\arith}\colon X^{\arith}\rightarrow Y^{\arith}$ is a cohomologically proper (resp. cohomologically étale; prim with invertible prim dual) morphism.
    \item Strong $\aff^1$-invariant property holds,
    i.e. $\F\rightarrow f^{\arith}_*f^{\arith,*}\F$ is an isomorphism for all $X$ and $\F\in D(X^{\arith})$,
    where $f\colon \aff^1_X\rightarrow X$ is the projection.
    \item If we have a closed-open decomposition $i\colon Z\rightarrow X\leftarrow U\colon j$,
    then $U^{\arith}$ is a closed subspace of $X^{\arith}$ with complement open $Z^{\arith}$ in the sense of Gelfand stacks.
    \item Assume that $f\colon X\rightarrow Y$ is smooth of relative dimension $n$.
    Then the prim dual of $f^{\arith}\colon X^{\arith}\rightarrow Y^{\arith}$ is identified as $\mathbb{P}_f(1_X)=1_X[-2n]$.
  \end{enumerate}
\end{theorem}

\begin{remark}
  There is also a similar statement for the relative arithmetic de Rham stacks;
  see Theorem \ref{relativearithsixfunctor}.
\end{remark}

\begin{remark}
  As the previous statement shows, the cohomological properties of arithmetic de Rham stacks are opposite to our geometric intuition.
  It may sound suprising at first,
  but let us mention that there is a similar phenomenon happening in the theory of algebraic de Rham stacks,
  cf. \cite[Appendix to Lecture VIII]{Sch23};
   still,
  some important results, e.g. Poincaré duality,
  can be obtained as consequences of these ``opposite'' cohomological properties;
  see Corollary \ref{Poincaréarith}.
\end{remark}

The advantage of the arithmetic de Rham stack is that
its category of quasi-coherent sheaves in the sense of analytic stacks \cite{CS23}
recovers classical objects that have been studied previously.
Recall that in Section \ref{padiccoefficientintroduction},
we introduced the category $\cat{Isoc}^{\dagger}(X/K)$ of overconvergent isocrystals,
which is generalized to the notion of overconvergent crystals in \cite[Definition 3.3.7]{LS11},
and several categories of arithmetic $D$-modules but only for a realizable\footnote{Here, following \cite[Definition 1.1.3]{AC13},
a scheme is called \emph{realizable} if it can be embedded into a proper smooth scheme $Y$ over $k$ with a proper smooth lift $\mathfrak{Y}$ over $K^{\circ}$.
Quasi-projective schemes are realizable, hence every $k$-scheme of finite type is Zariski locally a realizable scheme.} scheme $X$,
over which some cohomological operations are defined.
We can form similarly the category $\cat{Crys}^{\dagger}(X/K)$ of \emph{solid} overconvergent crystals and $D_{\blacksquare}(\D_{X/K}^{\dagger})$ of \textit{$\rt_p$-solid} arithmetic $D$-modules over $X$.
The following theorem is a combination of Theorem \ref{arithrigidcoh2} and Theorem \ref{Xarithcoefficients}.

\begin{theorem}[Theorem \ref{arithrigidcoh2} and Theorem \ref{Xarithcoefficients}] \label{theorem-comparison-arithmetic-d-modules}
  Let $X$ be a qcqs separated scheme of finite type over $k$.
  The following statements hold.
  \begin{enumerate}
    \item $D(X^{\arith/K})$ is equivalent to the category $\cat{Crys}^{\dagger}(X/K)$ of solid overconvergent crystals over $X$,
    in which we also have a natural $t$-structure on $\cat{Perf}(X^{\arith/K})$ and an equivalence of subcategories $\cat{Perf}^{\heartsuit}(X^{\arith/K})\cong \cat{Isoc}^{\dagger}(X/K)$.
    Under this equivalence,
    $f^{\arith/K}_*$ identifies with the pushforward of overconvergent isocrystals there.
    In particular,
    $R\Gamma(X^{\arith/K},\O_{X^{\arith/K}})\cong R\Gamma_{\rig}(X/K)$ computes the rigid cohomology.
    \item Let $X$ be a realizable scheme with a frame $(X,\overline{X},P)$,
    then $D_{\blacksquare}(\D_{X/K}^{\dagger})\cong \cat{LMod}_{\D^{\dagger}_{]X[^{\dagger}_P}}(D(]X[^{\dagger}_P))$ is described in the form of $D$-modules.
    Under this equivalence,
    in the case of liftable\footnote{It means that the schemes and morphisms can be lifted to morphisms between proper smooth formal schemes over $K^{\circ}$.} proper smooth schemes,
    $f^{\arith/K}_*$ identifies with the naïve pullback functor $f^{\triangle}$ and $f_!^{\arith/K}$ with a shift of the extraordinary pushforward functor $f_+[\dim X-\dim Y]$.
  \end{enumerate}
\end{theorem}

To understand this equivalence,
one can think about a toy model which is $X=\aff^1_k$.
Then $X^{\arith/K}=\disk^{\dagger}_K/\disk^{\circ}_K$ where $\disk^{\dagger}_K$ (resp. $\disk^{\circ}_K$) is the overconvergent (resp. open) unit disk over $K$ (Proposition \ref{arithadditivegroup}),
and an element $M\in D(X^{\arith/K})=D(\disk^{\dagger}_K/\disk^{\circ}_K)$ is equipped with a $T$-action that converges overconvergently (coming from $\disk^{\dagger}_K$),
and a $\partial$-action that $\left\{\frac{\partial^n}{n!}\right\}$ converges overconvergently (coming from $\disk^{\circ}_K$ via a ``twisted version'' of Cartier duality),
where they satisfy the Weyl relation $[\partial,T]=1$ (coming from the non-trivial action of $\disk^{\circ}_K$ on $\disk^{\dagger}_K$).
This is exactly the algebra of arithmetic differential operators over $\aff^1_k$.

\begin{remark}
Let us highlight the contrast between the approach of this paper to the problem of defining a category of coefficients for rigid cohomology and earlier approaches following the seminal work of Berthelot.
Once the definition of the arithmetic de Rham stack is made, we can use the full power of the theory of analytic stacks. First, the category of quasi-coherent sheaves on $X^{\arith}$ provides for any $X$ a suitable category of ``arithmetic $D$-modules''.
The descent of $X\mapsto X^{\arith}$ ensures that this construction glues well and locally on $X$ (to make it realizable), after picking a frame, this category can be explicitly described in terms of modules over a ring of differential operators defined by the choice of a frame, as shown by Theorem \ref{theorem-comparison-arithmetic-d-modules};
basically, this comes from the fact that these data supply an explicit presentation of the arithmetic de Rham stack.
In the more classical approach, one rather starts from these local descriptions and tries to glue them, running into the difficulty that these a priori depend on extra, non-canonical, choices.
Second,  the category of quasi-coherent sheaves on $X^{\arith}$ comes equipped (for free) with a $6$-functor formalism. This critically makes use of the solid analytic ring structure:
in earlier approaches, constructing the six operations has been a major difficulty.
\end{remark}

The equivalence of Theorem \ref{theorem-comparison-arithmetic-d-modules} is very useful,
because it provides a bridge that allows us to deduce classical results on overconvergent isocrystals and arithmetic $D$-modules from the abstract theory.
One application is to prove Berthelot's conjecture:
\begin{theorem}[Theorem \ref{Berthelotconj}]
  Let $f\colon X\rightarrow Y$ be a proper smooth morphism between qcqs separated schemes of finite type over $k$.
  Then the derived pushforward in each degree along $f$ sends overconvergent isocrystals over $X$ to overconvergent isocrystals over $Y$.
\end{theorem}

Let us say one sentence about the proof.
We can identify overconvergent isocrystals as perfect complexes over the arithmetic de Rham stack,
and the pushforward of overconvergent isocrystals as the pushforward functor in $6$-functor formalisms.
Then the conclusion is translated to a property in abstract $6$-functor formalism,
which follows from the cohomological properties of arithmetic de Rham stacks.

Another application is to study the Fourier transform of arithmetic $D$-modules,
recovering some results in \cite{Huy04}.
Using the language of \cite{RC25b},
the following theorem is an abstract way of stating the Fourier transform of arithmetic $D$-modules.
Let $K_D$ be the category of kernels for the $6$-functor formalism of Gelfand stacks.
and $[-]^*$, $[-]_!$ the two constructions in \cite[Construction 2.4.1]{RC25b}.
Let $\pi$ be a $(p-1)$-th root of $-p$\footnote{The choice of $-p$ is not crucial,
which we only use to match with the convention of the Cartier duality of $[\aff^1]_{\operatorname{SH}}$ in \cite[Lecture X]{Sch25}.}.
Denote $L=\rt_p(\pi)$,
over which we can define a Fourier--Mukai kernel $\L_{\pi}$.

\begin{theorem}[Theorem \ref{CartierarithmeticdR} and Remark \ref{FMkernel}]
  The Fourier--Mukai kernel $\L_{\pi}$ defines a Fourier--Mukai functor $\operatorname{FM}_{\pi}\colon [\aff^{1,\arith/L}]_!\rightarrow [\aff^{1,\arith/L}]^*$,
  and it is an equivalence in $\cat{bCAlg}(K_D)$.
\end{theorem}

In fact,
from this abstract theorem,
one can deduce the following theorem about the Fourier transform of arithmetic $D$-modules.
Let $D^{\dagger}(\infty)_L$ be the following solid $\rt_p$-algebra:
\[D^{\dagger}(\infty)_L=\left\{\sum_{m,n=0}^{\infty} a_{m,n}T^m\frac{\partial^n}{n!}\mid a_{m,n}\in L,\e c>0,\eta<1,\f n,m,|a_{m,n}|<c\eta^{m+n}\right\}.\]

\begin{theorem}[Theorem \ref{FouriertransformarithmeticDmodule}]
  The Fourier--Mukai functor defines a Fourier transform functor,
  which is an equivalence: 
  \[
  \operatorname{Four}_{\pi}\colon \cat{LMod}_{D^{\dagger}(\infty)_L}(D_{\blacksquare}(\rt_p))\overset{\cong}{\longrightarrow}\cat{LMod}_{D^{\dagger}(\infty)_L}(D_{\blacksquare}(\rt_p)).
  \]
  Moreover,
  given a left $D^{\dagger}(\infty)_L$-module $M$,
  $\operatorname{Four}_{\pi}(M)=D^{\dagger}(\infty)_L\otimes_{D^{\dagger}(\infty)_L,S_{\pi}}M[1]$,
  where $S_{\pi}\colon D^{\dagger}(\infty)_L\cong D^{\dagger}(\infty)_L$ is the isomorphism sending $T$ to $\pi^{-1}\partial$ and $\partial$ to $-\pi T$.
\end{theorem}

Till now,
we have systematically studied the theory of arithmetic de Rham stacks,
and showed some powerful applications from the theory.
We turn to the study of another type of objects,
which are Hyodo--Kato stacks (introduced in \cite{ABLBRCS25}) associated to $k$-schemes.
It has strong connections with arithmetic de Rham stacks,
hence also (partly) recovers $p$-adic coefficients.
Its interest beyond the arithmetic de Rham stack is twofold:
\begin{enumerate}
  \item Its $6$-functor formalism is compatible with geometric intuition.
  \item It serves not only as a part of the Hyodo--Kato stack,
  but also a part of the analytic prismatization developed by Anschütz--Le Bras--Rodríguez Camargo--Scholze \cite{ALBRCS};
  see Proposition \ref{analyticprismatization}.
\end{enumerate} 

We will use the theory of Hyodo--Kato stacks in \cite{ABLBRCS25} to define and study these stacks.
However,
in the introductory section,
we prefer to define these stacks directly via functor-of-points description,
to imitate the definition of the arithmetic de Rham stack.

\begin{definition}[Remark \ref{bigrigiddefinition}]
  Let $X$ be a scheme over $k$. Define the stack $X^{\RIG}$ as the (sheafification of) the following qfd Gelfand (pre-)stack:
  \[
  X^{\RIG}(A):=X(A^{u,\flat}),\,A\in\cat{uPerfd}^{\qfd}_{\omega_1}.
  \]
  It is equipped with a Frobenius morphism $\varphi_X\colon X^{\RIG}\rightarrow X^{\RIG}$,
  by acting as $x\mapsto x^p$ on functor of points.
  Then define
  \[
  X^{\HK}:=X^{\RIG}/\varphi^{\itg}_X.
  \]
\end{definition}

Intuitively,
the way one can visualize $X^{\HK}$ is to choose a universal $\delta$-lift $\mathcal{X}$ of $X$ to $\rt_p$ if $X$ admits one,
then $X^{\HK}$ is the perfect analytic de Rham stack of $\mathcal{X}$ modulo Frobenius $\varprojlim_{\varphi}\mathcal{X}^{\dR}/\varphi^{\itg}$.
Let $k$ be a \emph{perfect} field of characteristic $p$,
and $K$ admit a Frobenius lift $\varphi_K$ and also be perfect.
Then we have a morphism $\operatorname{GSpec}K/\varphi_K^{\itg}\rightarrow (\Spec k)^{\HK}$ and we can define a relative version of the Hyodo--Kato stack;
see Definition \ref{relativeHK}.

The main six-functor properties of $X^{\HK}$ are as follows.

\begin{theorem}[Theorem \ref{PoincaredualityrigidFrobenius}]
  The $6$-functor formalism $X\rightarrow D(X^{\HK})$ satisfies the following properties.
  \begin{enumerate}
    \item Let $f\colon X\rightarrow Y$ be an étale (resp. proper; smooth) morphism,
    then $f^{\HK}\colon X^{\HK}\rightarrow Y^{\HK}$ is cohomologically étale (resp. cohomologically proper; cohomologically smooth).
    \item Strong $\aff^1$-invariant property holds,
    i.e. $\F\rightarrow f_*^{\HK}f^{\HK,*}\F$ is an isomorphism for all $X$ and $\F\in D(X^{\HK})$,
    where $f\colon \aff^1_X\rightarrow X$ is the projection;
    \item If we have a closed-open decomposition $i\colon Z\rightarrow X\leftarrow U\colon j$,
    then $D(Z^{\HK})$ is a closed subspace of $D(X^{\HK})$ with complement open $D(Y^{\HK})$.
    \item For a smooth map $f\colon X\rightarrow Y$ of relative dimension $n$,
    we have $f^{\HK,!}1\cong 1(n)[2n]$\footnote{Here $(-)(n)$ denotes the $n$-th product of the Tate twist in Definition \ref{Tatetwist}.}.
  \end{enumerate}
\end{theorem}
\begin{remark}
  A similar statement holds for the relative Hyodo--Kato stack;
  see Theorem \ref{PoincaredualityrigidFrobeniusrelative}.
\end{remark}

One can ask about the classification of coefficients $D(X^{\HK})$ associated to $X$.
However,
we don't expect this category to correspond to any classical objects that people studied previously.
Nevertheless,
we can still classify the category of perfect complexes on it in the following sense.
First let us introduce an object connecting $X^{\arith}$ and $X^{\HK}$.

\begin{definition}[Remark \ref{smallrigiddefinition}]
  Define the stack $X^{\RIG}_{\le 1}$ as the (sheafification of) the following qfd Gelfand (pre-)stack:
  \[
    X^{\RIG}_{\le 1}(A):=X(A^{u,\circ,\flat}),\,A\in\cat{uPerfd}^{\qfd}_{\omega_1}.
  \]
  It is equipped with a Frobenius morphism $\varphi_X\colon X^{\RIG}_{\le 1}\rightarrow X^{\RIG}_{\le 1}$,
  by acting as $x\mapsto x^p$ on functor of points.
  Then define
  \[
  X^{\HK}_{\le 1}:=X^{\RIG}_{\le 1}/\varphi^{\itg}_X.
  \]
  Let $K$ admit a Frobenius lift $\varphi_K$ and be perfect.
  Then we also have a relative version of these stacks given in Definition \ref{relativeHK}.
\end{definition}

\begin{theorem}[Theorem \ref{equivalenceperfectcomplex}]
  Let $X$ be a scheme of finite type over a perfect field $k$.
  Let $K$ admit a Frobenius lift and be perfect.
  Then there exists a natural diagram 
  \[
  \xymatrix{
    & X^{\HK/K}_{\le 1}\ar[dr]^{\alpha}\ar[dl]_{\beta} & \\
   X^{\arith/K}/\varphi^{\itg}_{X} & & X^{\HK/K}
  }
  \]
  It induces an equivalence of categories of perfect complexes,
  i.e.
  we have 
  $\alpha^*\colon\cat{Perf}(X^{\HK/K})\cong \cat{Perf}(X^{\HK/K}_{\le 1})$
  and $\beta^*\colon\cat{Perf}(X^{\arith/K}/\varphi^{\itg}_{X})\cong \cat{Perf}(X^{\HK/K}_{\le 1})$.
\end{theorem}

The main input of the proof is the ``spreading out'' property of perfect complexes,
cf. \cite[Section 2.3]{Hau26}.
Let us explain it in the toy model $X=\aff^1$ over $k$.
Let $\mathbb{G}^{\dagger}_{a,K}=\operatorname{GSpec}K\left\{T\right\}^{\dagger}$ be the overconvergent neighborhood of $0$ in $\Ga$.
In this case,
the picture becomes 
\[
  \xymatrix{
    & (\varprojlim_{\varphi}(\disk^{\dagger}_K/\mathbb{G}^{\dagger}_{a,K}))/\varphi^{\itg}\ar[dr]^{\alpha}\ar[dl]_{\beta} & \\
   (\disk^{\dagger}_K/\disk^{\circ}_K)/\varphi^{\itg}=(\varprojlim_{\varphi}(\disk^{\dagger}_K/\disk^{\circ}_K))/\varphi^{\itg} & & (\varprojlim_{\varphi}(\aff^{1,\an}_K/\mathbb{G}^{\dagger}_{a,K}))/\varphi^{\itg}
  }
\]
The result that $\alpha$ (resp. $\beta$) is an equivalence essentially comes from two steps:
\begin{enumerate}
  \item Spreading out of perfect complexes over $\disk^{\dagger}_K$ (resp. $\mathbb{G}^{\dagger}_{a,K}$) to a small neighborhood of it.
  \item Use Frobenius-equivariance to extend the perfect complex over the small neighborhood to a perfect complex over $\aff^{1,\an}_K$ (resp. $\disk^{\circ}_K$).
\end{enumerate} 

The equivalence of perfect complexes actually shows that $X^{\HK/K}$ is also a stacky approach to rigid cohomology and overconvergent $F$-isocrystals.

\begin{theorem}[Theorem \ref{rigidstackrigidcohomology}]
  Let $X$ be a scheme of finite type over a perfect field $k$.
  Let $K$ admit a Frobenius lift $\varphi_K$ and be perfect\footnote{A similar statement for non-perfect field also holds;
  see Remark \ref{passingtononperfect}.}.
  Then there exists a natural $t$-structure on $\cat{Perf}(X^{\HK/K})$ and we have an equivalence of categories
  $\cat{Perf}^{\heartsuit}(X^{\HK/K})\cong \cat{F-Isoc}^{\dagger}(X/K)$,
  and a natural isomorphism
  \[
  \pi^*f_*1\cong R\Gamma_{\rig}(X/K)\in  D_{\blacksquare}(K)^{\varphi_K\textnormal{-equiv}}.
  \]
  Here $f\colon X^{\HK/K}\rightarrow \operatorname{GSpec}K/\varphi^{\itg}_K$,
  and $\pi\colon \operatorname{GSpec}K\rightarrow \operatorname{GSpec}K/\varphi^{\itg}_K$.
\end{theorem}

\begin{remark}
  The full category of sheaves $D(X^{\HK})$ on the Hyodo--Kato stack $X^{\HK}$ will not give the theory of arithmetic $D$-module.
  Literally,
  it gives the category of \emph{analytic $D$-modules}\footnote{In the sense of the forthcoming paper \cite{RCRJ}.} over the Fargues--Fontaine curve $\FF_X$ of $X$.
  Nevertheless,
  the category of solid arithmetic $D$-modules with Frobenius structure over $X$ embeds fully faithfully into this category;
  see Theorem \ref{equivalenceperfectcomplex},
  and this category has relations with Frobenius-equivariant \emph{regular} $D$-modules on a $\delta$-lift $\mathcal{X}$ over $\rt_p$,
  at least in the case of perfect complexes;
  see Remark \ref{regularDmodule} for the toy model $X=\aff^1_k$.
\end{remark}

The equivalence of perfect complexes is quite useful.
An application is to reprove finite dimensionality of rigid cohomology à la Kedlaya,
cf. the main theorem of \cite{Ked06}.
The new proof is quite direct,
instead of using a lot of geometry as the original proof.

\begin{theorem}[Theorem \ref{Kedlayafinitedimensionality}]
  Let $X$ be a smooth variety over $k$,
  and $\E$ be an overconvergent $F$-isocrystal on $X$ over $K$.
  Then $H^i_{\rig}(X/K,\E)$ is finite-dimensional over $K$.
\end{theorem}

Let us say a few words about the proof.
First we can reduce to the case that $k$ is perfect.
Then using the equivalence of categories of perfect complexes,
one can view an overconvergent $F$-isocrystal as an object in both $\cat{Perf}(X^{\HK/K})$ and $\cat{Perf}(X^{\arith/K}/\varphi^{\itg}_{X})$.
Working separately in these two categories,
we can prove that $R\Gamma_{\rig}(X/K,\E)$ is both basic nuclear and is the dual of a basic nuclear object in $D_{\blacksquare}(K)$.
Solid functional analysis tells us it is also dualizable,
hence its cohomology $H^i_{\rig}(X/K,\E)$ in each degree is finite-dimensional.

Another application is to recover the overconvergence of rigid cohomology for varieties over $\mathbb{F}_q(\!(t)\!)$ in Theorem \ref{LazdaPal},
recovering part of the work of Kedlaya \cite{Ked00},
and Lazda--P\'{a}l \cite{LP14a, LP14b, LP15}.

\begin{theorem}[Theorem \ref{LazdaPal}]
  Let $X$ be a variety over $\Spec\mathbb{F}_q(\!(t)\!)$.
  Let $f^{\HK}\colon X^{\HK}\rightarrow (\Spec\mathbb{F}_q(\!(t)\!))^{\HK}$ be the induced morphisms on Hyodo--Kato stacks.
  Then $f^{\HK}_*1$ gives the information of a finite-dimensional Frobenius-equivariant $(\varphi,\nabla)$-module over the bounded Robba ring $\E^{\dagger}_{\rt_q}$.
\end{theorem}

\subsection{Relation to other work}
A $6$-functor formalism for (overholonomic) arithmetic $D$-modules has been considered by Caro, Abe and etc.,
for example in \cite{AC13};
while we provide a full six-functor formalism on the big category of all (solid) arithmetic $D$-modules,
in the sense of \cite{Man22}.
We do not attempt a precise comparison of their results and ours.
An approach to overcome the frame-dependency inherent to the original theory of rigid cohomology has been proposed by Le Stum in \cite{LS11},
using the overconvergent site.
We briefly discuss the relation with the arithmetic de Rham stack in Section \ref{arithmeticcrystals}.

In \cite[Lecture XII]{Sch23},
Scholze introduced an object called the ``tempered disks'',
and then $X^{\operatorname{temp}}$ by transmutation for any scheme $X$ over $k$,
which is related to the tempered disk cohomology studied in \cite{BCV25}.
The relation of our results to their work is also briefly mentioned in \cite[Lecture XII]{Sch23},
that $X^{\operatorname{temp}}$ is a modification of $X^{\arith}$ which yields a realization of motives and gives a theory of arithmetic $D$-modules.
Recently, Marco D'Addezio has attached to any scheme in characteristic $p$ its convergent stack in \cite{D'A26},
geometrizing convergent cohomology.
We do not investigate here the relation between the constructions of \cite{D'A26} and the ones of this paper.

During the writing of this paper,
we learnt from Marco D'Addezio that he has independently obtained a proof of Berthelot's conjecture by different methods,
using his theory of edged crystalline cohomology in \cite{D'A24}.

\subsection{Organization of the paper}
We recall some results about Gelfand stacks in Section \ref{recollectionofGelfand}.
The theory of analytic de Rham stacks and analytic de Rham stacks of Fargues--Fontaine curve are also useful,
which we will review briefly in Sections \ref{recollectionofGelfand} and \ref{recollectiondR}.

After these preliminaries,
we define and study the arithmetic de Rham stack $X^{\arith}$ in Section \ref{arithmeticstackfirstproperties},
where a main result is the $h$-hyperdescent property of the arithmetic de Rham stack.
Then we study the relative arithmetic de Rham stack $X^{\arith/K}$ in Section \ref{sectionrelativearith},
where we focus on the geometry and give a chart-description of it.
We next study the $6$-functor formalism of arithmetic de Rham stacks in Section \ref{Sectionsixfunctorarith}.
ly,
in Section \ref{Cartierduality},
we move on to establishing the Cartier duality for $\aff^{1,\arith}$.

Using the result of Section \ref{SectionarithdR},
we are able to develop the theory of solid sheaves on arithmetic de Rham stacks.
There are usually two ways to describe these sheaves,
namely crystals and $D$-modules.
In Section \ref{arithmeticcrystals},
we study sheaves on the relative arithmetic de Rham stack by arithmetic crystals and solid overconvergent crystals,
and relate them with classical objects such as rigid cohomology and overconvergent isocrystals.
In Section \ref{reviewpadiccoefficient},
we relate sheaves on the relative arithmetic de Rham stack to solid arithmetic $D$-modules,
with an emphasis on the case of realizable schemes.

We then turn to the study of the Hyodo--Kato stacks associated to $X$.
In Section \ref{rigidexplicitproperty},
we define the stacks $X^{\HK}$ and $X^{\HK}_{\le 1}$,
show $h$-hyperdescent properties,
give some basic examples,
and end up by studying the $6$-functor formalism for $X^{\HK}$.
After a discussion about the relation with analytic prismatization in Section \ref{Sectionanalyticprismatization},
we turn to study the relation between the categories of sheaves on the arithmetic de Rham stacks and the Hyodo--Kato stacks in Section \ref{Sectionidentificationperfectcomplex}.
We can finally establish the stacky approach by the Hyodo--Kato stack $X^{\HK}$ to rigid cohomology and overconvergent $F$-isocrystals of $X$ in Section \ref{Sectionrigidstackrigidcoh}.

Finally,
we study applications to arithmetic by stacky approaches to $p$-adic coefficients.
There are two applications using the arithmetic de Rham stacks,
namely Berthelot's conjecture in Section \ref{SectionBerthelotconj},
and Fourier--Huyghe transform in Section \ref{FourierHuyghe}.
We also give two further applications involving the Hyodo--Kato stacks too,
namely Kedlaya's finite dimensionality in Section \ref{SectionKedfinite},
and Kedlaya--Lazda--P\'{a}l's overconvergence of rigid cohomology in Section \ref{sectionoverconvergent}.

\subsection*{Notations and conventions}\label{sectionnotation}
Throughout the paper,
we work in the context of $\infty$-categories following \cite{Lur09}.
We write $\cat{Cat}_{\infty}$ as the $\infty$-category of $\infty$-categories,
$\cat{Pr}^L$ as the $\infty$-category of presentable $\infty$-categories,
and $\cat{Pr}^{L,\operatorname{st}}$ as the $\infty$-category of stable presentable $\infty$-categories.

Throughout the paper,
we will freely use the language of condensed mathematics,
especially the theory of solid modules over $\rt_p$ and solid $p$-adic functional analysis.
Standard notations and definitions can be found in \cite{RCRJ22}, \cite{RC24b} and \cite[Appendix A]{ABLBRCS25}.
We will also freely use the language of $6$-functor formalisms,
where the reader can find all of the terminologies in \cite{Sch23} and \cite{HM24}.

We also fix a (possibly non-perfect) field $k$ of characteristic $p$.
Let $K^{\circ}$ be a (possibly ramified) complete discrete valuation ring of mixed characteristic with residue field $k$,
and $K$ be the fraction field of $K^{\circ}$.
If $k$ is perfect and $K^{\circ}$ is unramified,
then $K^{\circ}$ can be uniquely chosen as $W(k)$.
All $k$-schemes $X$ will be assumed qcqs and separated of finite type over $k$.

We will also freely use the language of frames.
A frame of $X$ is a triple $(X,\overline{X},P)$ such that $P$ is a formal scheme over $K^{\circ}$ and $\overline{X}$ is a $k$-scheme,
equipped with an open immersion $X\hookrightarrow \overline{X}$,
and a closed immersion $\overline{X}\rightarrow P$.
We say that a frame $(X,\overline{X},P)$ is proper (resp. smooth) if $\overline{X}$ is proper (resp. if $P$ is smooth in a neighbourhood of $X$).

\subsection*{Acknowledgements}
I would like to thank my advisors Arthur-César Le Bras and Juan Esteban Rodríguez Camargo sincerely,
for suggesting the subject to me and guiding me during the project.
I also thank them warmly for many helpful conversations
and for carefully reading preliminary versions of this paper.
I would also like to thank Peter Scholze for helpful suggestions and discussions,
as well as for his broad influence on the related materials.
I am also grateful to Zhenghui Li for pointing out a subtlety about $p$-adic functional analysis in Section \ref{SectionKedfinite},
and Zhouhang Mao for pointing out a mistake about the derived category of overconvergent isocrystals. 
Finally,
I would also like to thank Ko Aoki,
Marco D'Addezio,
Maximilian Hauck,
Siyan Daniel Li-Huerta,
and Daming Zhou for helpful discussions, comments and exchanges.
Special thanks go to Alberto Vezzani,
for his interest and suggestions.
This work was written when I was a PhD student at the Université de Strasbourg and Institut de Recherche Mathématique Avancée,
and I thank them for hospitality.
I also want to thank Max Planck Institute for Mathematics in Bonn and Tsinghua University for giving the opportunity to present some results of the paper,
and for their hospitality during a short-term visit.
This work is financially supported by Contrat Doctoral Spécifique Normalien of École Normale Supérieure de Paris,
and the Agence Nationale de la Recherche (project ID: ANR-25-CE40-7869-01).

\subsection*{AI disclosure}
No AI was used for the mathematical content and the writing of this paper.

\section{Recollections on analytic de Rham stacks of Fargues--Fontaine curves}

Since this paper is based on and closely related to the theory of Gelfand stacks and analytic de Rham stacks of Fargues--Fontaine curves,
developed in \cite{ABLBRCS25},
it is worth giving a detailed recollection of results we need.

\subsection{Gelfand rings and Gelfand stacks}\label{recollectionofGelfand}

To realize the stacky approach to rigid cohomology,
we need to work in the situation of Gelfand rings and Gelfand stacks,
which first appeared in \cite{ABLBRCS25}.
We recall some basic notions.

Fix terminology as follows.
We freely use the notions in condensed mathematics and analytic rings for which a good written reference is \cite{RC24b}.
Let $A$ be a bounded ring over $\rt_p$ with induced solid structure,
defined in \cite[Definition 2.6.10]{RC24a}.
Let $A^{\le r}\subset A$ be the subring of elements of norm $\le r$ for a nonnegative real number $r$,
defined in \cite[Definition 2.2.8]{ABLBRCS25}.
Briefly,
$f\in\pi_0(A)(*)=\Hom(\rt_p[T],A)$ lies in $\pi_0(A^{\le r})(*)$ if it extends to a morphism $\rt_p\bra T\ket_{\le r}\rightarrow A$ where $\rt_p\bra T\ket_{\le r}$ is the ring of overconvergent functions on the disk of radius $r$.
Denote $\Nil^{\dagger}(A)=A^{\le 0}$ and $A^{\dagger-\red}=\cofib(\Nil^{\dagger}(A)\rightarrow A)$.

\begin{definition}
  The uniform completion of a bounded ring $A$ is defined as 
  \[
  A^u:=\left(\varprojlim_{r\rightarrow 0} A^{\le 1}/A^{\le r}\right)[p^{-1}],
  \]
  with the induced solid structure.
  A bounded ring is called uniform if the canonical morphism $A\rightarrow A^u$ is an isomorphism.
\end{definition}

\begin{remark}
  There is a canonical morphism $A^{\dagger-\red}\rightarrow A^u$.
  It is injective
  but not surjective in general.
  Hence if $A$ is $\dagger$-reduced,
  then the map $A\rightarrow A^u$ is injective and $A^u$ can be seen as the completion of $A$.
\end{remark}

\begin{definition}[{\cite[Definition 3.1.1, Proposition 3.1.4]{ABLBRCS25}}]
  A bounded ring $A$ over $\rt_p$ is called a Gelfand ring if its uniform completion is the condensification of a Banach $\rt_p$-algebra.
  It is called separable if the uniform completion is separable as a Banach algebra.
  Denote the category of (resp. separable) Gelfand rings as $\cat{GelfRing}$ (resp. $\cat{GelfRing}_{\omega_1}$).
\end{definition}

\begin{definition}[{\cite[Definition 3.1.5]{ABLBRCS25}}]
  The Berkovich spectrum $\M(A)$ of a Gelfand ring $A$ is defined as $\M(A):=\M(A^u)$,
  the Berkovich spectrum of its uniform completion,
  which is a Banach $\rt_p$-algebra on which the notion of Berkovich spectrum is classical.
\end{definition}

Recall that for any light condensed set,
we can associate to it its Betti stack,
which is an analytic stack over $\AnSpec\itg$.

\begin{proposition}\cite[Proposition 3.2.10]{ABLBRCS25}
  Let $A$ be a Gelfand ring over $\rt_p$ such that $\M(A)$ is metrizable and cohomologically finite-dimensional.
  Then there is a natural map 
  \[
  \AnSpec A\longrightarrow \M(A)_{\Betti}\times_{\AnSpec\itg}\AnSpec \rt_{p,\blacksquare},
  \]
  satisfying universal $!$-descent.
\end{proposition}

\begin{definition}
  An analytic ring is called Fredholm if all dualizable $A$-modules come from perfect $\underline{A}(*)$-modules via the fully faithful embedding $D(\underline{A}(*))\rightarrow D(A)$.
\end{definition}

\begin{proposition}[{\cite[Proposition 3.3.5]{ABLBRCS25}}]\label{Fredholmproperty1}
  Gelfand rings are Fredholm.
\end{proposition}

Now we introduce a subclass of Gelfand rings,
which are called quasi-finite dimensional (or for simplicity qfd).

\begin{definition}[{\cite[Definition 4.1.11]{ABLBRCS25}}]
  A morphism $f\colon A\rightarrow B$ between Banach algebras over $\rt_p$ is called quasi-finite-dimensional
  if the induced morphism on the associated arc-stacks $\M_{\arc}(B)\rightarrow\M_{\arc}(A)$ is quasi-pro-étale over $\aff^n_{\M_{\arc}(A)}$ for some $n$.
  A Gelfand ring $A$ is called quasi-finite-dimensional if $\rt_p\rightarrow A^u$ is quasi-finite-dimensional.
  Denote the category of qfd separable Gelfand rings as $\cat{GelfRing}_{\omega_1}^{\qfd}$.
\end{definition}

\begin{definition}[{\cite[Definition 4.1.11]{ABLBRCS25}}]
  Define the category of qfd Gelfand stacks
  \[
  \cat{GelfStk}^{\qfd}:=\cat{AnStk}(\cat{GelfRing}_{\omega_1}^{\op,\qfd})
  \]
  as the category of ``analytic stacks'' (cf. \cite[Definition 4.2.1]{ABLBRCS25}) on $\cat{GelfRing}_{\omega_1}^{\op,\qfd}$.

  For any qfd separable Gelfand ring $A$,
  we denote the qfd Gelfand stack represented by $A$ by $\operatorname{GSpec}A$.
\end{definition}

\begin{example}[Disks and tori]\label{disk}
  We define an important class of qfd Gelfand stacks.
  Let $\rt_p\langle T\rangle_{\le r}$ be the Gelfand ring of overconvergent functions on the disk of radius $r$.
  Let $\rt_p\langle T\rangle^{\sharp}_{\le 1}=\{\sum a_n\frac{T^n}{n!}\mid \e \eta>1, |a_n|\eta^n\rightarrow 0\}$ be the Gelfand ring of overconvergent functions on the PD-disk (PD for divided power) of radius $1$,
  and similarly define $\rt_p\langle T\rangle^{\sharp}_{\le r}=\{\sum a_n\frac{T^n}{n!}\mid \e \eta>r, |a_n|\eta^n\rightarrow 0\}$ to be the Gelfand ring of overconvergent functions on the PD-disk of radius $r$.
  \begin{enumerate}
    \item The closed unit disk $\disk^{\dagger}:=\operatorname{GSpec}\rt_p\langle T\rangle_{\le 1}$.
    \item The open unit disk $\disk^{\circ}:=\varinjlim_{r<1}\operatorname{GSpec}\rt_p\langle T\rangle_{\le r}$.
    \item The analytic affine line $\aff^{1,\an}:=\varinjlim_{r\rightarrow\infty}\operatorname{GSpec}\rt_p\langle T\rangle_{\le r}$.
    \item The closed unit PD-disk $\disk^{\dagger,\sharp}:=\operatorname{GSpec}\rt_p\langle T\rangle^{\sharp}_{\le 1}$.
    \item The open unit PD-disk $\disk^{\circ,\sharp}:=\varinjlim_{r<1}\operatorname{GSpec}\rt_p\langle T\rangle^{\sharp}_{\le r}$.
    \item The analytic torus $\gff:= \aff^{1,\an}\backslash\{0\}$.
    \item The overconvergent torus $\mathbb{T}^{\dagger}:=\disk^{\dagger}\backslash\disk^{\circ}=\operatorname{GSpec}\rt_p\langle T^{\pm 1}\rangle_{\le 1}$.
  \end{enumerate}
  Working with Gelfand stacks eliminates some pathologies.
  Let $(-)_{\operatorname{Gelf}}$ the Gelfandification be the functor taking an analytic stack to its restriction to Gelfand rings on the test categories.
  Then the algebraic affine line and the pathological open disk become analytic affine line and analytic open disk after Gelfandification:
  \[
  (\AnSpec\rt_p[T])_{\operatorname{Gelf}}=\aff^{1,\an},\,(\AnSpec\itg_p[\![T]\!][p^{-1}])_{\operatorname{Gelf}}=\disk^{\circ}.
  \]
\end{example}

\begin{remark}\label{Fredholmproperty}
  The Fredholm property of Gelfand rings (\ref{Fredholmproperty1}) tells us that we can define the category of perfect complexes for any qfd Gelfand stack.
  Indeed,
  this implies that sending $A\in \cat{GelfRing}_{\omega_1}^{\qfd}$ to $\cat{Perf}(A)$ satisfies $!$-descent,
  because perfect complexes over $A$ are the same as dualizable objects in $D(A)$ and being dualizable is local in the $!$-topology.
  Then the left Kan extension of this functor to the presheaf category $\cat{Psh}(\cat{GelfRing}^{\op,\qfd}_{\omega_1})$ inverts $!$-equivalences,
  hence induces a functor $\cat{GelfStk}^{\qfd}\rightarrow \cat{Cat}_{\infty}$,
  sending $X$ to $\cat{Perf}(X)$.
\end{remark}

Within the category of qfd separable Gelfand rings,
we have a subclass consisting of those that behave like perfectoid algebras,
which we will define as follows.

\begin{definition}
  A Gelfand ring $A$ over $\rt_p$ is called uniformly perfectoid (resp. nilperfectoid) if $A^u$ (resp. $A^{\dagger-\red}$) is perfectoid.
  It is called uniformly strictly totally disconnected if
  the Berkovich spectrum $\M(A)$ is profinite with all Berkovich residue fields being algebraically closed.

  Denote $\cat{uPerfd}_{\omega_1}$ (resp. $\cat{uPerfd}_{\omega_1}^{\qfd}$; $\cat{uPerfd}_{\omega_1}^{\operatorname{std}}$; $\cat{uPerfd}_{\omega_1}^{\operatorname{std},\qfd}$) as the category of separable uniformly
  (resp. qfd; strictly totally disconnected; qfd strictly totally disconnected) perfectoid rings.
  Denote similarly $\cat{NilPerfd}^{(\operatorname{std}),(\qfd)}_{\omega_1}$ the category of separable (qfd) (strictly totally disconnected) nilperfectoid rings.
\end{definition} 

Here comes the key property of qfd separable Gelfand rings.

\begin{proposition}[{\cite[Proposition 4.6.4]{ABLBRCS25}}]
  Let $A$ be a qfd separable Gelfand ring.
  Then there is a descendable cover $A\rightarrow B$ where $B$ is qfd separable strictly totally disconnected nilperfectoid.
\end{proposition}

\begin{remark}
  The proposition above implies that $\cat{NilPerfd}_{\omega_1}^{\operatorname{std},\qfd}$ form a basis for $\cat{GelfRing}_{\omega_1}^{\qfd}$ under $!$-covers,
  i.e. to determine a qfd Gelfand stack it suffices to test on $\cat{NilPerfd}_{\omega_1}^{\operatorname{std},\qfd}$\footnote{Since nilperfectoid rings are uniformly perfectoid,
  using $\cat{NilPerfd}_{\omega_1}^{\operatorname{std},\qfd}$ is equivalent to using $\cat{uPerfd}_{\omega_1}^{\operatorname{std},\qfd}$.
  We will use both of them throughout the remaining of the paper.}.
  The category $\cat{uPerfd}_{\omega_1}^{\qfd}$ is stable under pushouts,
  and this also implies that the category $\cat{AnStk}(\cat{uPerfd}_{\omega_1}^{\op,\qfd})$ is equivalent to $\cat{GelfStk}^{\qfd}$.
\end{remark}

Let $\cat{Perfd}^{\qfd}_{\rt_p,\omega_1}$ be the category of separable qfd perfectoid rings over $\rt_p$.
The arc-topology on it is studied in \cite{Sch24b},
and one can define the category $\cat{ArcStk}^{\qfd}_{\rt_p}$ of qfd arc-stacks,
which is the category of arc-hypersheaves on $\cat{Perfd}^{\qfd}_{\rt_p,\omega_1}$,
cf. \cite[Section 4.1]{ABLBRCS25}.
Finally, in the section,
we study the relation between Gelfand stacks and arc-stacks.
One should think of the analogue over $\cn$ as the discussion about the relation between analytic de Rham stacks,
Betti stacks and stacks on totally disconnecteds in \cite[Chapter 2; Section V.3]{Sch24a}.

\begin{proposition}\label{adjointnessui}
  There is an adjoint pair of functors $(u\dashv \iota)$ between categories $u\colon \cat{uPerfd}_{\omega_1}^{\operatorname{std},\qfd}\leftrightarrows \cat{Perfd}_{\omega_1}^{\operatorname{std},\qfd}\colon \iota$,
  where $u$ takes the uniform completion and $\iota$ is the embedding.
  Moreover,
  both functors preserve $!$-equivalences on $\cat{uPerfd}_{\omega_1}^{\operatorname{std},\qfd}$ and $\infty$-connective maps in arc-topos on $\cat{Perfd}_{\omega_1}^{\operatorname{std},\qfd}$.
\end{proposition}
\begin{proof}
  The adjointness of the pair is literally from the definition.
  To show $u$ sends $!$-equivalences to $\infty$-connective maps in arc-topos,
  we refer to the proof of \cite[Lemma 4.5.1]{ABLBRCS25}.
  To show $\iota$ sends $\infty$-connective maps in arc-topos to $!$-equivalences,
  we refer to \cite[Lemma 4.4.3]{ABLBRCS25} for the proof.
\end{proof}

\begin{proposition}[{\cite[Section 4.5]{ABLBRCS25}}]\label{ABLBRCS25Section4.5}
  We have the following diagram
  \[
    \begin{tikzcd}
    \cat{GelfStk}^{\qfd}\arrow[r, shift left=0ex, "\iota_*=u^*"]& \arrow[l, shift left=-3ex, "u_*"']\arrow[l, shift left=3ex, "\widehat{(-)}:=\iota^{*}", "\top"']\cat{ArcStk}^{\qfd}_{\rt_p}
    \end{tikzcd}
  \]
  Equivalently,
  the perfectoidization functor $(-)^{\diamond}:=u^*=\iota_*$ admits both a left adjoint $\widehat{(-)}:=\iota^{*}$ and a right adjoint,
  called the analytic de Rham stack functor $(-)^{\dR}:=u_*$.
\end{proposition}

\begin{remark}\label{noneedsheafify}
  It follows directly from Proposition \ref{adjointnessui} that $X^{\dR}(A)=X(A^u)$ for $A$ qfd uniformly strictly totally disconnected perfectoid,
  and $\widehat{X}=\varinjlim_{\M_{\arc}(A)\rightarrow X}\operatorname{GSpec}A$ where $A$ is qfd strictly totally disconnected perfectoid.
  The coherent cohomology of $X^{\dR}$ (resp. $\widehat{X}$) computes the de Rham cohomology (resp. $\widehat{\O}$-cohomology) of $X$,
  which explains the notations.
\end{remark}

\begin{remark}
  Let $X$ be a qfd Gelfand stack.
  Then by adjoitness,
  there is a morphism $X\rightarrow (X^{\diamond})^{\dR}$,
  and we sometimes simplify the notation as $X\rightarrow X^{\dR}$.
  Let $X$ be a qfd Gelfand stack over $\operatorname{GSpec}K$,
  we can also define its relative analytic de Rham stack $X^{\dR/K}$,
  cf. \cite[Definition]{ABLBRCS25}.
\end{remark}

As promised,
a major advantage of restricting to qfd Gelfand rings is to ensure that the analytic de Rham stack functor $(-)^{\dR}$ has strong descent property,
i.e. it preserves colimits.

\begin{proposition}[{\cite[Theorem 5.6.6]{ABLBRCS25}}]\label{dRcolimit}
  The functor $(-)^{\dR}$ in the qfd case commutes with colimits.
  In particular,
  if $X_{\bullet}\rightarrow X$ is an arc-hypercover of qfd arc-stacks,
  then $|X^{\dR}_{\bullet}|\rightarrow X^{\dR}$ is an equivalence.
\end{proposition}

\begin{remark}
  Combining with Proposition \ref{ABLBRCS25Section4.5},
  we know that the analytic de Rham stack functor $(-)^{\dR}$ preserves both colimits and limits.
\end{remark}

As an application of analytic de Rham stacks,
let us discuss how we can define the overconvergent neighborhood of tubes in the setting of Gelfand stacks.

\begin{construction}[Tubes]\label{tubes}
  Let $(X,\overline{X},P)$ be a smooth frame in Section \ref{sectionnotation},
  and view the generic fiber $P_{\eta}$ of $P$ as a Gelfand stack.

  Define the adic tube $]X[_P$ of $X$ in $P_{\eta}$,
  as the adic subspace of $P_{\eta}$ whose underlying topological space is the preimage $(\operatorname{sp})^{-1}(X)$ of $X$ under the specialization map $\operatorname{sp}\colon P_{\eta}\rightarrow P_k$.

  Define the Berkovich tube $]X[_P^{\operatorname{Berk}}$ as the image of the adic tube $]X[_P$ in the underlying Berkovich spectrum,
  and define the arc-tube $]X[_P^{\operatorname{arc}}$ as the fiber product in the following diagram:
  \[\xymatrix{
    ]X[_P^{\operatorname{arc}}\ar[r]\ar[d] & (P_{\eta})^{\diamond}\ar[d]\\
    \underline{]X[_P^{\operatorname{Berk}}}\ar[r] & \underline{\M(P_{\eta})}
  }\]
  where the vertical morphisms come from \cite[Example 4.1.9]{ABLBRCS25}.
  Define the analytic de Rham stack of the tube as $]X[^{\dR}_P:=(\underline{]X[_P^{\operatorname{arc}}})^{\dR}$.
  Finally,
  we define the overconvergent tube $]X[_P^{\dagger}$ as the fiber product in the following diagram:
  \[
    \xymatrix{
      ]X[_P^{\dagger}\ar[r]\ar[d] & P_{\eta}\ar[d]\\
      ]X[^{\dR}_P\ar[r] & P_{\eta}^{\dR}
    }
  \]
  Note that the analytic de Rham stack of $]X[_P^{\dagger}$ agrees with $]X[_P^{\dagger,\dR}=]X[^{\dR}_P$.
\end{construction}

\begin{example}
  A standard example of computing tubes is the open embedding $(\aff^1_k\subset\proj^1_k,\proj^1_{K^{\circ}})$ and the closed embedding $(\{0\}\subset \proj^1_k,\proj^1_{K^{\circ}})$
  (where they are complements to each other after the change of coordinates $t\mapsto t^{-1}$).
  In this case,
  $]\aff^1_k[_{\proj^1_{K^{\circ}}}^{\dagger}=\disk^{\dagger}_K$,
  and $]\{0\}[_{\proj^1_{K^{\circ}}}^{\dagger}=\disk^{\circ}_K$ are closed unit disk and open unit disk respectively in Example \ref{disk}.
  Their analytic de Rham stacks of the tubes, arc-tubes, Berkovich tubes and adic tubes are,
  respectively,
  their analytic de Rham stacks, underlying arc-stacks,
  underlying Berkovich spaces and underlying adic spaces respectively.
\end{example}

Finally, in this section,
we briefly discuss the $6$-functor formalism related to Gelfand stacks.
Applying \cite[Theorem 1.2.7]{HM24} to qfd Gelfand rings,
we get a $6$-functor formalism on $\cat{GelfStk}^{\qfd}$ as follows.
In particular,
we obtain the notions of suave and prim maps between Gelfand stacks.

\begin{proposition}
  There exists a $6$-functor formalism $D\colon (\cat{GelfStk}^{\qfd},\widetilde{!\textnormal{-able maps}})\rightarrow \cat{Pr}^{L,\st}$
  extending $D\colon \operatorname{GSpec}A\mapsto D_{\blacksquare}(A)$.
\end{proposition}

\begin{remark}\label{categoricalKunneth}
  The \emph{categorical Künneth formula} holds for Gelfand stacks\footnote{Equivalently saying,
  the $6$-functor formalism of Gelfand stacks is \emph{$1$-affine} in the sense of \cite{Sch25}.}.
  Indeed,
  it is shown for analytic stacks in \cite[Corollary 1.5.1]{Kes25},
  and for the case of Gelfand stacks,
  we note that the functor $(-)_{\operatorname{An}}$ preserves finite limits in \cite[Lemma 4.2.7]{ABLBRCS25},
  and the $6$-functor formalism for Gelfand stacks is defined from the $6$-functor formalism for analytic stacks via the composition with $(-)_{\operatorname{An}}$.
\end{remark}

Our main goal is to use Gelfand stacks to give a stacky approach to rigid cohomology.
As already discussed in the introductory section,
the stack approach usually comes from transmutation,
and some properties of the $6$-functor formalism of stacks coming from transmutation are easier to check.
The following theorem is essentially in \cite{Aok26}.

\begin{theorem}\label{Aoki}
  Let $k$ be a ring and let $R$ be a $k$-algebra stack in $\cat{GelfStk}^{\qfd}$.
  Assume the following conditions:
  \begin{enumerate}
    \item the map $f\colon R\rightarrow\operatorname{GSpec}\rt_p$ is $!$-able and cohomologically smooth,
    and $f_!f^!1\rightarrow 1$ is an equivalence;
    \item the maps $i\colon \operatorname{GSpec}\rt_p\overset{0}{\hookrightarrow}R \hookleftarrow R^{\times}\colon j$ yield a closed-open decomposition.
  \end{enumerate} 
  Define the transmutation functor
  \[
  X\mapsto X_R
  \]
  by sending a separated scheme $X$ of finite type over $k$ to the qfd Gelfand stack $X_R$ which is the sheafification (sheafification in the sense of analytic stacks) of the presheaf taking $A$ to $X(R(A))$.
  Then the transmutation functor commutes with finite limits and Zariski glueing,
  and sends étale (resp. proper; smooth) morphisms of schemes to cohomologically étale (resp. cohomologically proper; cohomologically smooth) morphisms of qfd Gelfand stacks.
\end{theorem}
\begin{proof}
  This is the same as the proof in \cite[Theorem 10.6]{Sch23}.
\end{proof}

\begin{remark}
  The association $X\mapsto D(X_R)$ yields a motivic $6$-functor formalism.
  On the other hand,
  every motivic $6$-functor formalism comes from transmutation by \textit{Ringgestalten} in a higher categorical sense, cf. \cite[Lecture IX]{Sch25} and \cite{Aok26}.
\end{remark}

\subsection{Hyodo-Kato stacks of arc-stacks}\label{recollectiondR}

The first step towards the definition of Hyodo-Kato stacks is 
to study the relationship between arc-stacks over $\ff_p$ and $\rt_p$.
It recovers the tilting functor and the Fargues--Fontaine curve in \cite{FS21}.

\begin{construction}
  Consider the tilting functor $\flat\colon \cat{Perfd}_{\omega_1,\rt_p}^{\qfd}\rightarrow\cat{Perfd}_{\omega_1,\ff_p}^{\qfd}$.
  It sends arc-hypercovers to arc-hypercovers.
  Hence we have an adjoint pair
  \[
    \begin{tikzcd}
    \cat{ArcStk}_{\rt_p}^{\qfd}\arrow[r, shift left=-1.5ex, "(-)^{\diamondsuit}:=\flat^*"', "\top"]& \arrow[l, shift left=-1.5ex, "\Y_{(-)}^{\diamond}:=\flat_*"']\cat{ArcStk}_{\ff_p}^{\qfd}.
    \end{tikzcd}
    \]
  The functor $\flat^*$ is the diamond functor: $\flat^*(X)=X^{\diamondsuit}$,
  and the functor $\flat_*$ generalizes the Fargues--Fontaine curve (in fact punctured Fargues--Fontaine disk) seen as an arc-stack:
  if $X=\M_{\operatorname{arc}}(S)$ for a perfectoid ring $S$,
  then $\flat_*(X)=\Y_{S,\FF}^{\diamond}=(\Spa W(S^{\circ,\flat})\backslash V(p[\varpi^{\flat}]))^{\diamond}$ is the arc-stack of the punctured Fargues--Fontaine disk of $S$,
  whence the notations come.
\end{construction}

\begin{remark}
  Let $\Spd\rt_p:=\flat^*\M_{\operatorname{arc}}(\rt_p)$.
  Via the functor $\flat^*$,
  we can identify $\cat{ArcStk}_{\rt_p}^{\qfd}$ as a slice category $\cat{ArcStk}_{\ff_p,/\Spd\rt_p}^{\qfd}$ over $\cat{ArcStk}_{\ff_p}^{\qfd}$,
  as a map $X\rightarrow \Spd\rt_p$ just collects the information of $S$-points of $X$ for those perfectoid rings $S$ over $\rt_p$ with the same tiltings.
  Under this identification,
  $\flat^*$ becomes the forgetful functor and $\flat_*$ becomes the functor sending $X$ to $X\times_{\M_{\arc}(\ff_p)}\Spd\rt_p$.
\end{remark}

\begin{remark}\label{FFcurvecolimit}
  As a right adjoint,
  $\Y_{(-)}^{\diamond}$ also preserves colimits.
  Indeed,
  it suffices to prove it after restricting to every slice category $\cat{ArcStk}_{\M_{\arc}(A)}^{\qfd}$ for some perfectoid $A$.
  However,
  on this slice category,
  $\flat$ induces an equivalence between $\cat{ArcStk}_{\M_{\arc}(A)}^{\qfd}$ and $\cat{ArcStk}_{\M_{\arc}(A^{\flat})}^{\qfd}$,
  hence $\Y_{(-)}^{\diamond}=\flat_*$ preserves colimits.
\end{remark}

\begin{example}\label{computingflat}
  Let us compute some basic examples of $\Y_{(-)}^{\diamond}$:
  \begin{itemize}
    \item if $X=\M_{\arc}(S)$ for some perfectoid ring $S$,
    then $\Y_{X}^{\diamond}=\Y_{S,\FF}$ as arc-stacks;
    \item $\Y_{\M_{\arc}(\ff_p[T])}^{\diamond}=\varprojlim_{x\mapsto x^p}\aff^{1,\an,\diamond}_{\rt_p}$,
    as both sides have $R$-points given by $\varprojlim_{x\mapsto x^p}R\cong R^{\flat}$;
    \item $\Y_{\M_{\arc}(\ff_p[T^{\pm1}])}^{\diamond}=\varprojlim_{x\mapsto x^p}\mathbb{G}^{\an,\diamond}_{m,\rt_p}$ for the same reason as above.
  \end{itemize}
\end{example}

We now consider arc-stacks associated with schemes.

\begin{definition}\label{diamondofscheme}
  Let $X$ be a scheme of characteristic $p$.
  We define the diamonds\footnote{Despite their names, they are not actually diamonds in the sense of \cite{Sch17}.} associated with $X$.
  \begin{enumerate}
    \item The small diamond associated to $X$,
    denoted by $X^{\circ}_{\operatorname{arc}}$,
    is a qfd arc-stack over $\ff_p$ with functor of points given as the sheafification of $X^{\circ,\operatorname{pre}}_{\operatorname{arc}}(R):=X(R^{\circ})$.
    \item The big diamond $X_{\arc}$ associated to $X$ is a qfd arc-stack over $\ff_p$ defined as the sheafification of $X_{\operatorname{arc}}^{\operatorname{pre}}(R):=X(R)$.
  \end{enumerate}
\end{definition}

\begin{remark}
  Both associated diamonds satisfy Zariski descent for $X$,
  i.e. if $U\rightarrow X$ is a Zariski cover,
  then $U_{\arc}^{(\circ)}\rightarrow X_{\arc}^{(\circ)}$ is an arc-cover.
  Indeed,
  taking any $A$-point of $X_{\arc}^{(\circ),\operatorname{pre}}$ which is an element in $X(A^{(\circ)})$,
  then we can find an arc-cover $S\rightarrow \M_{\arc}(A)$ such that the map lifts to $S\rightarrow U^{(\circ)}_{\arc}$,
  by using the rational localization of a Zariski cover of $\Spec A^{(\circ)}$.
\end{remark}

\begin{remark}\label{preadicspace}
  Let us mention the relation with another construction in \cite[Lecture 18]{SW20}.
  There they associate to a pre-adic space $Y$ a $v$-sheaf $Y^{\diamondsuit}$.
  Let $X$ be a scheme of characteristic $p$.
  Then we can associate to it two pre-adic spaces,
  namely the big one $X^{\ad/\ff_p}$ by the glueing of $\Spec A\mapsto \Spa(A,\widetilde{\ff_p})$ and the small one $X^{\ad}$ by the glueing of $\Spec A\mapsto \Spa(A,A)$,
  cf. \cite[Discussion after Definition 9.5]{Sch19}.
  Then we can associate to them $v$-sheaves $X^{\ad(/\ff_p),\diamondsuit}$.
  The relation is that under the construction\footnote{Strictly speaking,
  we use the modified version $a'^*$ mentioned later in \cite[Section 12]{Sch24b}.} in \cite[Proposition 12.1]{Sch24b},
  their associated arc-stacks recover the aforementioned definition:
  \[
    a^*X^{\ad,\diamondsuit}=X^{\circ}_{\arc},\,
    a^*X^{\ad/\ff_p,\diamondsuit}=X_{\arc}.
  \]
  Indeed,
  both agree for affine schemes $X=\Spec A$ and both satisfy Zariski descent.
\end{remark}

The following proposition describes the punctured Fargues--Fontaine disk of small diamonds associated to a scheme.

\begin{proposition}\label{FFcurvesmalldiamond}
  The following properties hold.
  \begin{enumerate}
    \item Let $X$ be a scheme over $\ff_p$.
    Then $\Y_{X^{\circ}_{\arc}}^{\diamond}=(X^{\perf}_{\Ainf})^{\diamond}_{\eta}$,
    where $X^{\perf}=\varprojlim_{\varphi}X$ is the perfection of $X$;
    $X^{\perf}_{\Ainf}$ is a $p$-adic formal scheme,
    obtained by taking the Witt vectors on $X_{\perf}$ Zariski locally and then glueing;
    and $(X^{\perf}_{\Ainf})^{\diamond}_{\eta}$ is the arc-sheaf associated to the generic fiber of $X^{\Ainf}_{\perf}$.
    \item If $X$ further admits a $\delta$-lift to $\itg_p$,
    i.e. a lift to a formal scheme $\Xf$ with Frobenius lift,
    and if we denote its generic fiber by $\X$ over $\rt_p$ with Frobenius lift $\varphi$.
    Then $X^{\perf}_{\Ainf}=\varprojlim_{\varphi}\Xf$ and hence $\Y_{X^{\circ}_{\operatorname{arc}}}^{\diamond}\cong\varprojlim_{\varphi}\X^{\diamond}$.
  \end{enumerate}
\end{proposition}
\begin{proof}
  We first prove that $\Y_{X^{\circ}_{\arc}}^{\diamond}=(X_{\Ainf}^{\perf})^{\diamond}_{\eta}$.
  By Zariski localization,
  it suffices to prove the claim for an affine $X=\Spec R$.
  Then by definition,
  $\Y_{X^{\circ}_{\arc}}^{\diamond}(A)$ is the sheafification of $A\mapsto \Hom(R,A^{\circ,\flat})$.
  Let $R_{\perf}=\varinjlim_{\varphi} R$ be the perfection of $R$,
  then we have $\Hom(R,A^{\circ,\flat})=\Hom(R_{\perf},A^{\circ,\flat})$ by the universal property of perfections,
  which by Remark \ref{Kedlaya336} is identified with $\Hom_{\delta}(W(R_{\perf}),W(A^{\circ,\flat}))$.
  By composition with the $\theta$-map $\theta\colon W(A^{\circ,\flat})\rightarrow A^{\circ}$,
  we get a map $\Hom_{\delta}(W(R_{\perf}),W(A^{\circ,\flat}))\rightarrow \Hom_{p\text{-adic}}(W(R_{\perf}),A^{\circ})$,
  hence we obtain a map $\Hom(R_{\perf},A^{\circ,\flat})\rightarrow \Hom_{p\text{-adic}}(W(R_{\perf}),A^{\circ})$.
  On the other hand,
  if we are given an element of $\Hom_{p\text{-adic}}(W(R_{\perf}),A^{\circ})$,
  then it induces an element of $\Hom(W(R_{\perf})/p,A^{\circ}/p)$,
  which will induce an element of $\Hom(R_{\perf},A^{\circ,\flat})$ by the universal property of co-perfections.
  Moreover,
  these two maps are inverse to each other,
  hence we get a natural identification 
  \[
    \Hom(R^{\perf},A^{\circ,\flat})\cong \Hom_{p\text{-adic}}(W(R^{\perf}),A^{\circ}).
  \]
  This identification then gives 
  \[
    \Hom(R,A^{\circ,\flat})\cong \Hom(R^{\perf},A^{\circ,\flat})\cong \Hom_{p\text{-adic}}(W(R^{\perf}),A^{\circ})\cong \Hom_{p\text{-adic}}(W(R^{\perf})[p^{-1}],A).
  \]
  This proves the identification between $\Y_{X^{\circ}_{\arc}}^{\diamond}$ and the arc-sheaf associated to $\Spa W(R^{\perf})[p^{-1}]$ by passing to the sheafification,
  which is $(X^{\perf}_{\Ainf})^{\diamond}_{\eta}$ in our terminology.
  Then we prove that $X^{\perf}_{\Ainf}=\varprojlim_{\varphi}\Xf$ whenever $\Xf$ is a $\delta$-lift of $X$ to $\itg_p$.
  By Zariski localization,
  we can assume $\Xf=\Spf \mathfrak{R}$ with special fiber $X=\Spec R$ (hence $R=\mathfrak{R}/p$).
  Then $X^{\perf}_{\Ainf}=\Spf W(R^{\perf})$ and $\varprojlim_{\varphi}\Xf=\Spf (\mathfrak{R}^{\perf})_{p}^{\wedge}$,
  where $(\mathfrak{R}^{\perf})_{p}^{\wedge}$ is the $p$-adic completion of the perfection of $\mathfrak{R}$.
  However,
  both $W(R^{\perf})$ and $(\mathfrak{R}^{\perf})_{p}^{\wedge}$ are $p$-adically complete perfect $\delta$-rings with the same residue perfect ring $R^{\perf}$,
  hence by Remark \ref{Kedlaya336},
  they are naturally identified.
  Therefore,
  we have $X^{\perf}_{\Ainf}=\varprojlim_{\varphi}\Xf$ and hence $\Y_{X^{\circ}_{\operatorname{arc}}}^{\diamond}\cong\varprojlim_{\varphi}\X^{\diamond}$.
\end{proof}

\begin{remark}\label{Kedlaya336}
  Here we state the property about $\delta$-rings and Witt vectors that we use in Proposition \ref{FFcurvesmalldiamond} as follows.
  There exist equivalences of categories:
  \begin{enumerate}
    \item The category of $p$-adically complete perfect $\delta$-rings.
    \item The category of $p$-torsion-free,
    $p$-adically complete rings with perfect reduction mod $p$.
    \item The category of perfect rings of characteristic $p$.
  \end{enumerate}
  We refer to \cite[Proposition 3.3.6]{Ked21} for a proof.
\end{remark}

Finally, in this section,
we define the analytic de Rham stack of the punctured Fargues--Fontaine disk and the Hyodo--Kato stack.

\begin{definition}\cite[Section 6]{ABLBRCS25}\label{dRFFstk}
  Let $X$ be an object in $\cat{ArcStk}_{\ff_p}^{\qfd}$.
  Define the analytic de Rham stack of the punctured Fargues--Fontaine disk of $X$ as the qfd Gelfand stack $(\Y^{\diamond}_X)^{\dR}$.
  Note that $\Y^{\diamond}_X$ is equipped with a Frobenius $\varphi_X$ and we can form the Fargues--Fontaine curve $\FF_X^{\diamond}:=\Y^{\diamond}_X/\varphi_X^{\itg}$ of $X$.
  Define the Hyodo--Kato stack (or analytic de Rham stack of the Fargues--Fontaine curve) of $X$ as the following qfd Gelfand stack
  \[
    X^{\HK}:=(\Y^{\diamond}_X)^{\dR}/\varphi^{\itg}_X=(\FF_X^{\diamond})^{\dR}.
  \]
\end{definition}

The key descent property of these stacks is as follows.

\begin{proposition}\cite[Section 6]{ABLBRCS25}\label{dRFFpreservecolimit}
  The functors of taking the analytic de Rham stack of punctured Fargues--Fontaine disk and the Hyodo--Kato stack $\Y^{\diamond,\dR}_{(-)},\,(-)^{\HK}\colon \cat{ArcStk}_{\ff_p}^{\qfd}\rightarrow \cat{GelfStk}^{\qfd}$ preserve colimits.
\end{proposition}

\section{The arithmetic de Rham stack}\label{SectionarithdR}

We study the absolute arithmetic de Rham stack $X^{\arith}$ and the relative arithmetic de Rham stack $X^{\arith/K}$ of a scheme $X$ over $k$ in this section.
The motivation is that $X\mapsto X^{\arith/K}$ is a ``stacky approach'' to rigid cohomology and its coefficients,
playing a role similar to that of algebraic de Rham stacks for algebraic $D$-modules.
In this section,
we study its geometry and $6$-functor formalism.
We postpone the connection with rigid cohomology,
overconvergent isocrystals and arithmetic $D$-modules to the next section.

\subsection{Definition of the arithmetic de Rham stack and first properties}\label{arithmeticstackfirstproperties}
The stack is defined via transmutation,
so we begin with the basic case of $X=\aff^1_{\ff_p}$ over $\ff_p$.

\begin{definition}\label{arithmeticstackA1}
  Define the arithmetic de Rham stack of $\aff^1_{\ff_p}$ as the following qfd Gelfand stack over $\rt_p$:
  \[
  \aff^{1,\arith}_{\ff_p}:=\disk^{\dagger}_{\rt_p}/\disk^{\circ}_{\rt_p},
  \]
  where $\disk^{\dagger}_{\rt_p}$ is the overconvergent closed unit disk and $\disk^{\circ}_{\rt_p}$ is the open unit disk.
  It is equipped with a Frobenius morphism $\varphi$,
  which is induced from the Frobenii on $\disk^{\dagger}_{\rt_p}$ and $\disk^{\circ}_{\rt_p}$,
  both given by $T\mapsto T^p$ on the coordinate.
\end{definition}

\begin{remark}
  This stack is already mentioned in \cite[Lecture XII]{Sch23}.
\end{remark}

\begin{proposition}
  The stack $\aff_{\ff_p}^{1,\arith}$ is a perfect (i.e. $\varphi$ is an isomorphism) $\ff_p$-ring stack.
\end{proposition}
\begin{proof}
  To prove it is a perfect $\ff_p$-ring stack,
  it suffices to prove all the properties before sheafification.
  As a prestack,
  the functor-of-points description is given by $\disk^{\dagger}_{\rt_p}/\disk^{\circ}_{\rt_p}(A)=A^{\le 1}/A^{<1}$ where $A\in\cat{uPerfd}^{\qfd}_{\omega_1}$.
  Since $A^{<1}$ is an ideal in $A^{\le 1}$ with $pA^{\le 1}\subset A^{<1}$,
  the quotient $A^{\le 1}/A^{<1}$ is a ring over $\ff_p$.
  It is furthermore a perfect ring,
  as $A^{\le 1}/A^{<1}=\varprojlim_{r\rightarrow 0}(A^{\le 1}/A^{\le r})/(A^{<1}/A^{\le r})=A^{u,\le 1}/A^{u,<1}$,
  and $A^{u,\le 1}/A^{u,<1}$ is perfect since $A^u$ is perfectoid.
\end{proof}

\begin{definition}\label{arithdeRhamdefinition}
  Let $X\in\cat{Psh}(\cat{Ring}_{\ff_p}^{\op})$.
  We define its pre-arithmetic de Rham stack as the following qfd Gelfand pre-stack via transmutation:
  \[
  X^{\arith,\operatorname{pre}}(A):=X(\aff^{1,\arith}_{\ff_p}(A)),\,A\in\cat{uPerfd}^{\qfd}_{\omega_1}.
  \]
  Then we define its (absolute) arithmetic de Rham stack $X^{\arith}$ as the sheafification of $X^{\arith,\operatorname{pre}}$.
\end{definition}

\begin{remark}
  Let $X$ be a scheme.
  There is no need for sheafification if we only test on strictly totally disconnected qfd uniformly perfectoid rings.
  Indeed,
  if $A\rightarrow B$ is a $!$-cover between such rings,
  then $A^u\rightarrow B^u$ is an arc-cover of qfd strictly totally disconnected perfectoid rings from Proposition \ref{adjointnessui}.
  Then $A^{u,\circ}/A^{u,\circ\circ}\rightarrow B^{u,\circ}/B^{u,\circ\circ}$ is a flat cover,
  cf. \cite[Proposition 7.23]{Sch17},
  and therefore,
  the maps to $X$ satisfy descent.
  Moreover,
  if $A\rightarrow B$ is a $!$-equivalence,
  then $A^u\rightarrow B^u$ is also an arc-hypercover,
  hence $A^{u,\circ}/A^{u,\circ\circ}\rightarrow B^{u,\circ}/B^{u,\circ\circ}$ is a flat hypercover and $X$ satisfies flat hyperdescent.
\end{remark}

\begin{remark}\label{Zariskidescentarith}
  The stack $X^{\arith}$ satisfies Zariski descent.
  This can either be proved directly by showing a Zariski cover induces a surjection on the associated arithmetic de Rham stacks,
  or deduced from the stronger $h$-descent result proved below (Theorem \ref{fpqcdescentarith}).
\end{remark}

\begin{remark}\label{arithKunneth}
  The arithmetic de Rham stack preserves finite limits by construction.
  In particular,
  it preserves fiber products.
\end{remark}

Now we prove a very strong descent property of the arithmetic de Rham stack.

\begin{theorem}\label{fpqcdescentarith}
  Let $Y_{\bullet}\rightarrow X$ be an $h$-hypercover of affine schemes of finite type over $\ff_p$.
  Then $Y^{\arith}_{\bullet}\rightarrow X^{\arith}$ induces a $!$-equivalence.
\end{theorem}

\begin{remark}
  Using the same argument in the proof of \cite[Theorem 5.6.6]{ABLBRCS25},
  one can reformulate Theorem \ref{fpqcdescentarith} in several equivalent ways.
  Let $\cat{Sh}^{\operatorname{hyper}}_h(\cat{Ring}_{\ff_p}^{\operatorname{f.t.},\op})$ be the category of $h$-stacks,
  i.e. $h$-hypersheaves over the category $\cat{Ring}_{\ff_p}^{\operatorname{f.t.},\op}$ of affine schemes of finite type over $\ff_p$.
  Then the induced functor $(-)^{\arith}\colon \cat{Sh}^{\operatorname{hyper}}_h(\cat{Ring}_{\ff_p}^{\operatorname{f.t.},\op})\rightarrow \cat{GelfStk}^{\qfd}$ preserves colimits.
  Equivalently,
  the functor $(-)^{\arith}$ factors through a diagram:
  \[
  \xymatrix{
    \cat{Psh}(\cat{Ring}_{\ff_p}^{\operatorname{f.t.},\op})\ar[r]^{(-)^{\arith}}\ar[d] & \cat{GelfStk}^{\qfd}\\
    \cat{Sh}_h^{\operatorname{hyper}}(\cat{Ring}_{\ff_p}^{\operatorname{f.t.},\op})\ar@{-->}[ru]
  }\]
  In particular,
  an $h$-cover of $h$-stacks of finite type induces a $!$-cover of their associated arithmetic de Rham stacks.
\end{remark}

Before proving the theorem,
we set up some lemmas.

\begin{lemma}\label{dRFFarith}
  Let $X$ be an affine scheme of finite type over $\ff_p$.
  Let $\overline{X}$ be the qfd arc-stack over $\ff_p$ defined by sheafifying $\overline{X}^{\operatorname{pre}}(A):=X(A^{\circ}/A^{\circ\circ})$.
  Then we have $X^{\arith}=(\Y_{\overline{X}}^{\diamond})^{\dR}$ as the analytic de Rham stack of punctured Fargues--Fontaine disk of $\overline{X}$.
\end{lemma}
\begin{proof}
  It suffices to verify on $A$-points,
  where before sheafification the left side is $X(A^{\le 1}/A^{<1})$,
  and the right side is $X(A^{u,\flat,\circ}/A^{u,\flat,\circ\circ})$.
  Then the lemma follows because $A^{\le 1}/A^{<1}\cong A^{u,\circ}/A^{u,\circ\circ}\cong A^{u,\flat,\circ}/A^{u,\flat,\circ\circ}$.
\end{proof}

\begin{lemma}\label{smallrigidepimorphismarith}
  Let $X$ be an affine scheme of finite type over $\ff_p$.
  The natural morphism of qfd arc-stacks $X^{\circ}_{\arc}\rightarrow \overline{X}$ is an epimorphism.
\end{lemma}
\begin{proof}
  We can prove it before sheafification as sheafification preserves epimorphisms.
  It suffices to show that for any qfd perfectoid ring $R$ with $\M_{\arc}(R)\rightarrow \overline{X}^{\operatorname{pre}}$,
  there exists a qfd perfectoid $R$-algebra $S$ such that $\M_{\arc}(S)\rightarrow \M_{\arc}(R)$ is an arc-cover and the following diagram commutes
  \[\xymatrix{
    \M_{\arc}(S)\ar[r]\ar[d] & X^{\circ,\operatorname{pre}}_{\arc}\ar[d]\\
    \M_{\arc}(R)\ar[r] & \overline{X}^{\operatorname{pre}}
  }\]
  For the case $X=\aff^n$,
  $X^{\circ}_{\arc}\rightarrow \overline{X}$ is an epimorphism by the description of functor of points.
  For general $X$,
  choose a closed embedding $X\hookrightarrow\aff^n$ which exists by the assumption of being of finite type,
  then $X$ is an intersection of hypersurfaces.
  Because both functors of taking small diamonds $X\mapsto X^{\circ,\operatorname{pre}}_{\arc}$ and taking $X\mapsto \overline{X}^{\operatorname{pre}}$ preserve finite products,
  we reduce to the case of a hypersurface,
  which is the zero locus of a single polynomial equation
  denoted by $f(x_1,\ldots,x_n)\in \ff_p[x_1,\ldots,x_n]$.
  In this case,
  a map $\M_{\arc}(R)\rightarrow \overline{X}^{\operatorname{pre}}$ describes $n$ elements in $R^{\circ}/R^{\circ\circ}$ satisfying the polynomial equation given by $f$.

  For every $R$,
  there exists an arc-cover of $\M_{\arc}(R)$ by the arc-stack of a strictly totally disconnected perfectoid ring $S$ such that:
  \[
  S^{\circ}=\prod_{i\in I} K_i^{\circ}.
  \]
  It suffices to prove the lifting property of this $S$ (which may not be qfd),
  as if that holds,
  then we can take the perfectoid $R$-subalgebra in $S$ generated by the images of $x_1,\ldots, x_n$ which is qfd and an arc-cover of $R$.
  To prove the lifting property for this $S$,
  it reduces to the lifting property for each component,
  hence we can further assume that $S=K$ is an algebraically closed perfectoid field.
  In this case,
  it suffices to find lifts $a_1,\ldots,a_n\in K^{\circ}$ of $\overline{a_1},\ldots,\overline{a_n}\in K^{\circ}/K^{\circ\circ}$ such that $f(\overline{a_1},\ldots,\overline{a_n})=0$ lifts to $f(a_1,\ldots,a_n)=0$.
  We can find an invertible matrix $A\in \GL_{n}(\overline{\ff}_p)$ such that $f(x_1,\ldots,x_n)=g(y_1,\ldots,y_n)$ with $(y_1,\ldots,y_n)=(x_1,\ldots,x_n)\cdot A$ such that $g$ is of the form 
  \[
  g(y_1,\ldots,y_n)=y_1^n+\text{lower-degree terms in }y_1\in \overline{\ff}_p[y_1,\ldots,y_n].
  \]
  Being algebraically closed,
  both $K^{\circ}$ and $K^{\circ}/K^{\circ\circ}$ are $\overline{\ff}_p$-algebras,
  so it suffices to find lifts of $(\overline{a_1},\ldots,\overline{a_n})\cdot A$ to $K^{\circ}$ as solutions of $g$.
  First choose an arbitrary lift of $(\overline{a_1},\ldots,\overline{a_n})\cdot A$ to $K^{\circ}$,
  which we denote by $(b_1,\ldots,b_n)$.
  Then they satisfy $g(b_1,\ldots,b_n)=y\in K^{\circ\circ}$.
  Consider the polynomial $h(x)=g(b_1+x,\ldots,b_n)$ in $x$,
  since $K$ is algebraically closed,
  we can find a factorization $h(x)=\prod_{i=1}^n(x-r_i)$ with $\prod_{i=1}^nr_i=y\in K^{\circ\circ}$.
  Since $K$ is a field,
  at least one $r_i$ lies in $K^{\circ\circ}$.
  Then $(b_1+r_i,\ldots,b_n)$ is a desired lift of the solution to $K^{\circ}$.
\end{proof}

\begin{lemma}\label{hcoverarccover}
  Let $Y_{\bullet}\rightarrow X$ be an $h$-hypercover of affine schemes of finite type over $\ff_p$.
  Then the associated small diamond $Y_{\arc,\bullet}^{\circ}\rightarrow X_{\arc}^{\circ}$ is an arc-hypercover,
  and $\overline{Y}_{\bullet}\rightarrow\overline{X}$ is an arc-hypercover too.
\end{lemma}
\begin{proof}
  The association $X\mapsto X^{\ad,\diamondsuit}$ sends $h$-covers to $v$-covers,
  which is basically a reinterpretation of \cite[Proposition 3.7]{Gle24}.
  Under Remark \ref{preadicspace},
  we know that the associated small diamond $X^{\circ}_{\arc}$ is the arc-stack $a^*X^{\ad,\diamondsuit}$ corresponding to $X^{\ad,\diamondsuit}$.
  This proves that $X\mapsto X^{\circ}_{\arc}$ sends $h$-covers to arc-covers because $a^*$ commutes with colimits.
  Then $X\mapsto X^{\circ}_{\arc}$ also sends $h$-hypercovers to arc-hypercovers,
  since it commutes with finite limits.

  To prove that $\overline{Y}_{\bullet}\rightarrow\overline{X}$ is also an arc-hypercover,
  first we prove $X\mapsto\overline{X}$ sends $h$-covers to arc-covers.
  Let $Y\rightarrow X$ be an $h$-cover.
  We have a commutative diagram 
  \[\xymatrix{
    X'^{\circ}_{\arc}\ar[r]\ar[d] & X^{\circ}_{\arc}\ar[d]\\
    \overline{X'}\ar[r] & \overline{X}
  }\]
  with vertical maps being epimorphisms by Lemma \ref{smallrigidepimorphismarith}.
  Since $X'^{\circ}_{\arc}\rightarrow X^{\circ}_{\arc}$ is an arc-cover,
  we also deduce that $\overline{X'}\rightarrow \overline{X}$ is an arc-cover.
  Then the construction $X\mapsto\overline{X}$ also sends $h$-hypercovers to arc-hypercovers,
  since it commutes with finite limits.
\end{proof}

We can now prove the descent theorem.

\begin{proof}[Proof of Theorem \ref{fpqcdescentarith}]
  Assume $Y_{\bullet}\rightarrow X$ is an $h$-hypercover.
  Then by Lemma \ref{hcoverarccover},
  $\overline{Y}_{\bullet}\rightarrow\overline{X}$ is an arc-hypercover.
  Then by Proposition \ref{dRFFpreservecolimit},
  $(\Y_{\overline{Y}_{\bullet}}^{\diamond})^{\dR}\rightarrow(\Y_{\overline{X}}^{\diamond})^{\dR}$ induces a $!$-equivalence.
  By Lemma \ref{dRFFarith},
  this gives that $Y^{\arith}_{\bullet}\rightarrow X^{\arith}$ induces a $!$-equivalence,
  proving the descent property.
\end{proof}

\subsection{Geometry of the relative arithmetic de Rham stack}\label{sectionrelativearith}
We are interested in the case of $X$ being a separated scheme of finite type over $k$.
First note that we have a morphism
\[
\operatorname{GSpec}K\longrightarrow (\Spec k)^{\arith}.
\]
Indeed,
testing on points,
we need to give a morphism from $\Hom_{\rt_p}(K,A)$ to $\Hom_{\ff_p}(k,A^{\le 1}/A^{<1})$ for any $A\in\cat{uPerfd}^{\qfd}_{\omega_1}$.
An element in $\Hom_{\rt_p}(K,A)$ induces an element in $\Hom_{\rt_p}(K^{\circ},A^{\le 1})$,
hence induces an element in $\Hom_{\ff_p}(k,A^{\le 1}/A^{<1})$ for any $A\in\cat{uPerfd}^{\qfd}_{\omega_1}$ after reduction modulo the pseudo-uniformizer of $V$,
which gives the expected morphism.

\begin{definition}
  Let $X$ be a separated scheme of finite type over $k$.
  The relative arithmetic de Rham stack $X^{\arith/K}$ is defined by the Cartesian diagram 
  \[\xymatrix{
    X^{\arith/K}\ar[r]\ar[d] & \operatorname{GSpec}K\ar[d]\\
    X^{\arith}\ar[r] & (\Spec k)^{\arith}
  }\]
  If $K$ admits a Frobenius lift $\varphi_K$ and is perfect,
  then we define the relative arithmetic de Rham stack modulo Frobenius $X^{\arith/K}/\varphi_X^{\itg}$ by the following Cartesian diagram:
  \[\xymatrix{
    X^{\arith/K}/\varphi_X^{\itg}\ar[r]\ar[d] & \operatorname{GSpec}K/\varphi_K^{\itg}\ar[d]\\
    X^{\arith}/\varphi_X^{\itg}\ar[r] & (\Spec k)^{\arith}/\varphi_k^{\itg}
  }\]
\end{definition}

\begin{remark}
  The relative arithmetic de Rham stack can be alternatively defined via transmutation from $\aff^{1,\arith/K}_k=\disk^{\dagger}_K/\disk^{\circ}_K$,
  which is also a perfect ring stack under the Frobenius.
\end{remark}

\begin{remark}\label{hdescentrelativearith}
  The $h$-descent property (Theorem \ref{fpqcdescentarith}) also holds for relative arithmetic de Rham stacks.
\end{remark}

First we observe that,
if $k$ is a finite field and $K=W(k)[p^{-1}]$,
then the relative arithmetic de Rham stack $X^{\arith/K}$ agrees with the absolute one $X^{\arith}$,
by the following proposition.

\begin{proposition}[Finite field]\label{arithfinitefield}
  Let $k$ be a finite field of characteristic $p$ and $X=\Spec k$.
  Then we have $X^{\arith}=\operatorname{GSpec} K$ where $K=W(k)[p^{-1}]$.
\end{proposition}
\begin{proof}
  It suffices to identify $\Hom_{\ff_p}(k,A^{\le 1}/A^{<1})$ with $\Hom_{\rt_p}(K,A)$ which implies the result after sheafifying.
  Moreover,
  we can replace $A$ by $A^u$ as $A^{u,\le 1}/A^{u,<1}\cong A^{\le 1}/A^{<1}$ and $\Hom_{\rt_p}(K,A)\cong \Hom_{\rt_p}(K,A^u)$.
  Let $R$ be a perfectoid ring.
  \begin{enumerate}
    \item Let $f\in\Hom_{\ff_p}(k,R^{\le 1}/R^{<1})$.
    We claim that we can lift it to $g\in \Hom_{\itg_p}(W(k),R^{\le 1})$.
    Because $k$ is étale over $\ff_p$,
    it suffices to lift a root $\alpha$ of some polynomial equation $f(x)=0$ with coefficients in $R^{\le 1}/R^{<1}$.
    Choose an arbitrary lift $F(x)$ of $f(x)$,
    then the equation is lifted to $F(x)=y$ for some $y\in R^{<1}$.
    Then we apply Hensel's lemma to the ring $R^{\le 1}$ with $I=(y)$,
    and we can find a lift of the root.
    After inverting $p$,
    the morphism $g$ induces a morphism $h\in \Hom_{\rt_p}(K,R)$.
    \item Conversely,
    a morphism $h\in\Hom_{\rt_p}(K,R)$ must induce a morphism in $\Hom_{\itg_p}(W(k),R^{\le 1})$,
    and we get a morphism $f\in\Hom_{\ff_p}(k,R^{\le 1}/R^{<1})$ after reduction modulo $p$.
  \end{enumerate}
  The two constructions are mutually inverse because of the uniqueness of the lift in Hensel's lemma.
\end{proof}

\begin{remark}
  For general $k$ and $K$,
  $\operatorname{GSpec}K$ may be different from $(\Spec k)^{\arith}$.
  Indeed,
  the natural morphism $\operatorname{GSpec}K\rightarrow (\Spec k)^{\arith}$ factors through 
  \[
  \pi\colon (\operatorname{GSpec}K)^{\dR}\longrightarrow (\Spec k)^{\arith}.
  \]
  In fact,
  we have $(\Spec k)^{\arith}=((\Spec k)^{\arith})^{\dR}$ as $A^{\le1}/A^{<1}\cong A^{u,\le1}/A^{u,<1}$.
  In general,
  the morphism $\operatorname{GSpec} K\rightarrow (\operatorname{GSpec}K)^{\dR}$ will not be an isomorphism unless $K$ is étale over $\rt_p$.
\end{remark}

Let us give an example which will reappear later.
\begin{example}[Laurent series]\label{Laurentarith}
  Let $k=\ff_p(\!(t)\!)$ be the field of Laurent series over $\ff_p$.
  Then we can take $K$ to be the Amice ring of $\rt_p$ equipped with the $p$-adic topology:
  \[K=\E_{\rt_p}:=\left\{
    \sum_i a_i t^i\in \rt_p[\![t^{\pm 1}]\!]\bigg| \sup|a_i|<\infty,\,\lim_{i\rightarrow-\infty}|a_i|=0.
  \right\}\]
  There is a natural morphism
  \[
    \pi\colon (\operatorname{GSpec} \E_{\rt_p})^{\dR}\longrightarrow(\Spec \ff_p(\!(t)\!))^{\arith}.
  \]
  Similar descriptions also hold after replacing $\ff_p$ by $\ff_q$ and $\rt_p$ by $W(\ff_q)[p^{-1}]$.
\end{example}

We focus on the geometry of the relative de Rham stack which admits a cleaner description.

\begin{proposition}[Additive group]\label{arithadditivegroup}
  Let $X$ be $\aff^1_k=\Spec k[T]$.
  We have that $X^{\arith/K}\cong\disk^{\dagger}_K/\disk^{\circ}_K$.
\end{proposition}
\begin{proof}
  If $k=\ff_p$ and $K=\rt_p$,
  then the description of $(\aff^1_{\ff_p})^{\arith}$ is already given in Definition \ref{arithmeticstackA1}.
  The case of arbitrary $k$ and $K$ comes from Remark \ref{arithKunneth}.
\end{proof}

\begin{corollary}[Affine space]
  We have $\aff_k^{n,\arith/K}\cong \disk^{\dagger,n}_K/\disk^{\circ,n}_K$.
\end{corollary}
\begin{proof}
  By Proposition \ref{arithadditivegroup} and Remark \ref{arithKunneth}.
\end{proof}

\begin{proposition}[Multiplicative group]\label{arithgff}
  Let $X$ be $\mathbb{G}_{m,k}=\Spec k[T^{\pm1}]$ over $k$.
  We have that $X^{\arith/K}\cong\torus^{\dagger}_K/(1+\disk^{\circ}_K)$,
  where the action of $1+\disk^{\circ}_K$ on $\torus^{\dagger}_K$ is via multiplication.
\end{proposition}
\begin{proof}
  By Remark \ref{arithKunneth},
  we can reduce to the case $k=\ff_p$ and $K=\rt_p$.
  Then before sheafification,
  the functor-of-points description is $X^{\arith,\operatorname{pre}}(A)=(A^{\le 1}/A^{<1})^{\times}$.
  However,
  we can identify $(A^{\le 1}/A^{<1})^{\times}$ with $(A^{u,\le 1}/A^{u,<1})^{\times}$,
  which is identified with $A^{u,=1,\times}/(1+A^{u,<1})$ by selecting an invertible lift of an element $u\in (A^{u,\le 1}/A^{u,<1})^{\times}$ to $A^{\le 1,\times}$;
  hence also with $A^{=1,\times}/(1+A^{<1})$.
  Then we get this identification after sheafification.
\end{proof}

\begin{example}[Projective line]
  Consider the projective line $\proj^1_k$ over $k$.
  We know there are natural morphisms $\disk^{\dagger}_K\rightarrow(\aff^1_k)^{\arith/K}$ and $\torus^{\dagger}_K\rightarrow(\mathbb{G}_{m,k})^{\arith/K}$,
  which fit into the following Cartesian diagram:
  \[\xymatrix{
    \torus^{\dagger}_K\ar[r]\ar[d] & \disk^{\dagger}_K\ar[d]\\
    (\mathbb{G}_{m,k})^{\arith/K}\ar[r] & (\aff^1_k)^{\arith/K}
  }\]
  By glueing along the Zariski open cover $\aff^1_k\sqcup \aff^1_k$ of $\proj^1_k$,
  we also have a natural morphism $\proj^1_K\rightarrow(\proj^1_k)^{\arith/K}$.
  We will see in the upcoming Proposition \ref{propersmootharith}
  that this morphism is an epimorphism and the equivalence relation is given by the tubular neighborhood of $\proj^1_k$ in the diagonal.

  Also by glueing,
  the geometry of $\proj^{1,\arith/K}_k$ can be viewed as a ``raviolo'',
  with open part given by the disjoint union of two points $\operatorname{GSpec}\rt_p\coprod\operatorname{GSpec}\rt_p$,
  and a closed part given by $\torus^{\dagger}_K/1+\disk^{\circ}_K$ which one can imagine is the ``boundary'' of both $\operatorname{GSpec}\rt_p$.
\end{example}

We can give the main geometric properties of the relative arithmetic de Rham stacks.

\begin{proposition}\label{propersmootharith}\label{Xarithdescription}
  Let $X$ be a realizable scheme,
  with the proper smooth lift over $K^{\circ}$ denoted by $\mathfrak{Y}$ and its generic fiber denoted by $\Y$.
  Let $]X[_{\mathfrak{Y}}^{\dagger}$ be the overconvergent tube of $X$ in $\mathfrak{Y}$,
  and $]X[_{\mathfrak{Y}^{n+1}}^{\dagger}$ be the overconvergent tube of the diagonal embedding $X$ in $\Y^{n+1}$ from Construction \ref{tubes}.
  Then there is a natural surjection $\pi\colon ]X[_{\mathfrak{Y}}^{\dagger}\rightarrow X^{\arith/K}$,
  with Cech nerve of $\pi$ as $(]X[_{\mathfrak{Y}^{n+1}}^{\dagger})_{[n]\in\Delta}$.
  Moreover,
  $\pi$ is a suave cover.
\end{proposition}
\begin{proof}
  We first prove that
  if $X$ is proper smooth with proper smooth lift $\Xf$ with generic fiber $\X$ over $K$,
  then there is a natural surjection $\pi\colon \X\rightarrow X^{\arith/K}$.
  It suffices to construct it and show that it is surjective on points before sheafification,
  i.e. consider the map $\X(A)\rightarrow X^{\arith/K}(A)=X(A^{\le 1}/A^{<1})\times_{(\Spec k)(A^{\le 1}/A^{<1})}\operatorname{GSpec}K(A)$
  for any qfd separable nilperfectoid $A\in\cat{NilPerfd}^{\qfd}_{\omega_1}$.
  By the universal property of the generic fiber,
  cf. \cite[Proposition 2.2.2]{SW13},
  we have $\X(A^u)\cong \Xf(A^{u,\circ})$,
  and then we can construct the following morphism
  \[\X(A)\rightarrow \X(A^u)\cong \Xf(A^{u,\circ})\rightarrow X(A^{u,\circ}/p)\rightarrow X(A^{u,\circ}/A^{u,\circ\circ})\cong X(A^{\le 1}/A^{<1}).\]
  Combined with the structure morphism 
  $\X(A)\longrightarrow\operatorname{GSpec}K(A)$,
  we get a morphism 
  \[
  \X(A)\longrightarrow X(A^{\le 1}/A^{<1})\times_{(\Spec k)(A^{\le 1}/A^{<1})}\operatorname{GSpec}K(A).
  \]
  Moreover,
  the map $\X(A)\rightarrow \X(A^u)\times_{\operatorname{GSpec}K(A^u)}\operatorname{GSpec}K(A)$ is surjective by $\dagger$-formal smoothness of $\X$ over $\operatorname{GSpec}K$,
  cf. \cite[Definition 4.7.2]{ABLBRCS25};
  here we essentially use that $\X$ is proper and smooth.
  The map $\Xf(A^{u,\circ})\rightarrow X(A^{u,\circ}/p)\times_{\Spec k(A^{u,\circ}/p)}\Spf K^{\circ}(A^{u,\circ})$ is surjective by formal smoothness of $\Xf$.
  The map $X(A^{u,\circ}/p)\rightarrow X(A^{u,\circ}/A^{u,\circ\circ})\times_{\Spec k(A^{u,\circ}/A^{u,\circ\circ})}\Spec k(A^{u,\circ}/p)$ is surjective,
  because for any $f\in X(A^{u,\circ}/A^{u,\circ\circ})\times_{\Spec k(A^{u,\circ}/A^{u,\circ\circ})}\Spec k(A^{u,\circ}/p)$
  we can find a pseudo-uniformizer $\xi$ such that it is lifted to $X(A^{u,\circ}/\xi)\times_{\Spec k(A^{u,\circ}/\xi)}\Spec k(A^{u,\circ}/p)$ under the assumption that $X$ is of finite type over $k$,
  and then we can use formal smoothness of $X$ over $\Spec k$.
  Combining all surjections above,
  this proves the surjectivity of the morphism $\pi\colon \X\rightarrow X^{\arith/K}$.

  Then let $X$ be a general realizable scheme with proper smooth lift $\mathfrak{Y}$ and generic fiber $\Y$.
  Let $Y$ be the special fiber of $Y$.
  From the case of proper smooth schemes with a lift,
  it suffices to identify the fiber product of the following diagram with $]X[_{\mathfrak{Y}}^{\dagger}$:
  \[
  \xymatrix{
     & \Y\ar[d]\\
    X^{\arith/K}\ar[r] & Y^{\arith/K}
  }
  \]
  By the functor-of-points description,
  let $A\in\cat{NilPerfd}^{\qfd}_{\omega_1}$,
  then the fiber product has $A$-points consisting of $f\in\Y(A)\rightarrow\Y(A^u)\cong\mathfrak{Y}(A^{\le 1})$ such that its reduction $\overline{f}$ modulo $p$ satisfies $\overline{f}\in X(A^{u,\le 1}/p)$.
  However,
  this is also the functor-of-points description of $]X[_{\mathfrak{Y}}^{\dagger}$,
  because $]X[_{\mathfrak{Y}}^{\dagger}$ is obtained exactly as the Betti localization along the subspace $\M(]X[_{\mathfrak{Y}})\subset\M(\Y)$.
  This identifies the fiber product with $]X[_{\mathfrak{Y}}^{\dagger}$ and yields a natural surjection $\pi\colon ]X[_{\mathfrak{Y}}^{\dagger}\rightarrow X^{\arith/K}$.

  Then we compute the Cech nerve of the map $\pi$.
  The $n$-th Cech nerve of the map $\pi$ satisfies the following Cartesian diagram 
  \[\xymatrix{
    ]X[_{\mathfrak{Y}}^{\dagger,\times (n+1)/X^{\arith/K}}\ar[r]\ar[d] & ]X[_{\mathfrak{Y}}^{\dagger,\times (n+1)}\ar[d]\\
    X^{\arith/K}\ar[r]^{\Delta} & X^{\arith/K,\times (n+1)}
  }\]
  Since we have a Cartesian diagram 
  \[\xymatrix{
    ]X[_{\mathfrak{Y}}^{\dagger,\times (n+1)}\ar[d]\ar[r] & \Y^{\times(n+1)}\ar[d] \\
    X^{\arith/K,\times (n+1)}\ar[r] & Y^{\arith/K,\times(n+1)}
  }\]
  we also know that the following is Cartesian:
  \[\xymatrix{
    ]X[_{\mathfrak{Y}}^{\dagger,\times (n+1)/X^{\arith/K}}\ar[r]\ar[d] & \Y^{\times(n+1)}\ar[d]\\
    X^{\arith/K}\ar[r] & Y^{\arith/K,\times (n+1)}
  }\]
  This identifies the $n$-th Cech nerve $]X[_{\mathfrak{Y}}^{\dagger,\times (n+1)/X^{\arith/K}}$ with $]X[_{\mathfrak{Y}^{n+1}}^{\dagger}$.

  Finally,
  to prove that it is a suave cover,
  it suffices to deal with the proper smooth case as the realizable case is deduced by base change.
  It also suffices to prove $\pi$ is suave because we already know that it is a surjection.
  Suaveness is local on the source,
  cf. \cite[Lemma 4.5.7]{HM24},
  hence it suffices to check the suaveness after pullback along $\pi$.
  From the description of the Cech nerve,
  we know the following is a Cartesian diagram:
  \[\xymatrix{
  ]X[_{\Xf^2}^{\dagger}\ar[r]\ar[d] & \X\ar[d]\\
    \X\ar[r] & X^{\arith/K}
  }\]
  By the strong fibration theorem \cite[Théorème 1.3.7]{Ber96a},
  the map $]X[_{\Xf^2}^{\dagger}\rightarrow \X$ is a locally $\disk^{\circ,\times n}$-fibration for the analytic topology on $\X$
  (Berthelot proved this essentially for the underlying adic space but it can be generalized to the overconvergent tube directly by definition).
  Since $\disk^{\circ,\times n}$ is suave,
  the map $]X[_{\Xf^2}^{\dagger}\rightarrow \X$ is also suave.
  This proves the suaveness of $]X[_{\Xf^2}^{\dagger}\rightarrow \X$ hence of $\pi\colon\X\rightarrow X^{\arith}$ too.
\end{proof}

\begin{remark}
  In particular,
  the description does not depend on the choice of the frame in the sense that 
  if there is another choice of frame $(X,\mathfrak{Y}')$,
  then there is an identification $]X[^{\dagger}_{\mathfrak{Y}}/]X[^{\dagger}_{\mathfrak{Y}^2}$
  with $]X[^{\dagger}_{\mathfrak{Y}'}/]X[^{\dagger}_{\mathfrak{Y}'^2}$.
\end{remark}

\subsection{$6$-functor formalism}\label{Sectionsixfunctorarith}
We would like to study the $6$-functor formalism $X\mapsto D(X^{\arith})$,
and show that its cohomological properties are opposite to the axioms of being motivic,
cf. \cite[Lecture XI]{Sch23}.
Nevertheless,
this $6$-functor formalism also recovers classical results as we will see.
Precisely,
in this section,
we prove the following main theorem:

\begin{theorem}\label{arithsixfunctor}
  The $6$-functor formalism of absolute arithmetic de Rham stacks $X\mapsto D(X^{\arith})$ on $X\in\{\text{separated schemes over $\ff_p$ of finite type}\}$ has the following properties.
  \begin{enumerate}
    \item Let $f\colon X\rightarrow Y$ be an étale (resp. proper; smooth) morphism.
    Then $f^{\arith}\colon X^{\arith}\rightarrow Y^{\arith}$ is a cohomologically proper (resp. cohomologically étale; prim with invertible prim dual) morphism.
    \item Strong $\aff^1$-invariant property holds,
    i.e. $\F\rightarrow f^{\arith}_*f^{\arith,*}\F$ is an isomorphism for all $X$ and $\F\in D(X^{\arith})$,
    where $f\colon \aff^1_X\rightarrow X$ is the projection.
    \item If we have a closed-open decomposition $i\colon Z\rightarrow X\leftarrow U\colon j$,
    then $U^{\arith}$ is a closed subspace of $X^{\arith}$ with complement open $Z^{\arith}$ in the sense of Gelfand stacks.
    \item Assume that $f\colon X\rightarrow Y$ is smooth of relative dimension $n$.
    Then the prim dual of $f^{\arith}\colon X^{\arith}\rightarrow Y^{\arith}$ is identified as $\mathbb{P}_f(1_X)=1_X[-2n]$.
  \end{enumerate}
\end{theorem}

Moreover,
the same proof immediately yields the following parallel theorem in the relative setting.
\begin{theorem}\label{relativearithsixfunctor}
  The $6$-functor formalism of relative arithmetic de Rham stacks $X\mapsto D(X^{\arith/K})$ on $X\in\{\text{separated schemes over $k$ of finite type}\}$ has the following properties.
  \begin{enumerate}
    \item Let $f\colon X\rightarrow Y$ be an étale (resp. proper; smooth) morphism.
    Then $f^{\arith/K}\colon X^{\arith/K}\rightarrow Y^{\arith/K}$ is a cohomologically proper (resp. cohomologically étale; prim with invertible prim dual) morphism.
    \item Strong $\aff^1$-invariant property holds,
    i.e. $\F\rightarrow f^{\arith/K}_*f^{\arith/K,*}\F$ is an isomorphism for all $X$ and $\F\in D(X^{\arith/K})$,
    where $f\colon \aff^1_X\rightarrow X$ is the projection.
    \item If we have a closed-open decomposition $i\colon Z\rightarrow X\leftarrow U\colon j$,
    then $U^{\arith/K}$ is a closed subspace of $X^{\arith/K}$ with complement open $Z^{\arith/K}$ in the sense of Gelfand stacks.
    \item Assume that $f\colon X\rightarrow Y$ is smooth of relative dimension $n$.
    Then the prim dual of the map $f^{\arith/K}\colon X^{\arith/K}\rightarrow Y^{\arith/K}$ is identified as $\mathbb{P}_f(1_X)=1_X[-2n]$.
  \end{enumerate}
\end{theorem}

\begin{remark}\label{relativearithsixfunctormoduloFrobenius}
  If $K$ admits a Frobenius lift $\varphi_K$ and is perfect,
  then the $6$-functor formalism $X\mapsto D(X^{\arith/K}/\varphi^{\itg}_X)$ in the relative setting modulo Frobenius also behaves similarly to Theorem \ref{relativearithsixfunctor}.
\end{remark}

\begin{remark}\label{oppositemotivic}
  The phenomenon of being ``opposite'' to the motivic axioms also occurs for algebraic de Rham stacks,
  cf. \cite[Appendix to Lecture VIII]{Sch23}.
  In the forthcoming paper \cite{Aok} of Aoki,
  the main theorems in \cite{Aok26} (hence Theorem \ref{Aoki}) have opposite-motivic versions,
  and hence it applies to the case of algebraic de Rham stacks and also our case.
  Then the cohomological properties above (except $(4)$) reduce to the case of $X=\aff^1$,
  which is obtained almost for free.
\end{remark}

Now we start to prove Theorem \ref{arithsixfunctor}.
First we establish some preliminaries.
Under the notations in Example \ref{disk},
let us recall a fundamental computation from \cite{RC24a}.

\begin{proposition}[{\cite[Proposition 4.3.7]{RC24a}}]\label{analyticCartierduality}
  Let $X$ be a qfd Gelfand stack and $V^{\dagger}$ be a unitary overconvergent bundle over $X$ of rank $n$
  (which is equivalent to a map $X\rightarrow B\GL_n^{\dagger}$,
  where $B\GL_n^{\dagger}$ is $\operatorname{Vect}^{\operatorname{rk}=n,\le 1}$ in the notation of \cite[Theorem 3.2.37]{RC25b}).
  Let $g\colon X/V^{\vee,\circ}\rightarrow X$.
  Then $g$ is both cohomologically smooth and prim with invertible prim dual.
  Moreover,
  the suave dual and prim dual of $g$ are computed as $\mathbb{D}_g(1)=M [-n]$ and $\mathbb{P}_g(1)=1[-2n]$.
  Let $f\colon X\rightarrow X/V^{\vee,\circ}$,
  then $f^*M=\bigwedge^d \F(V)$ where $\F(V)$ is the sheaf over $X$ corresponding to the vector bundle $V$.
\end{proposition}

\begin{proposition}\label{A1cosmooth}
  Consider $f\colon \aff_{\ff_p}^{1,\arith}=\disk^{\dagger}_{\rt_p}/\disk^{\circ}_{\rt_p}\rightarrow\operatorname{GSpec}\rt_p$.
  Then:
  \begin{enumerate}
    \item The map $f$ is prim with prim dual $1[-2]$.
    \item The map $f$ satisfies that $\F\rightarrow f_{*}f^*\F$ is an equivalence for all $\F\in D_{\blacksquare}(\rt_p)$.
  \end{enumerate}
\end{proposition}
\begin{proof}
  We first prove that $g\colon \disk^{\dagger}_{\rt_p}/\disk^{\circ}_{\rt_p}\rightarrow B\disk^{\circ}_{\rt_p}$ is prim.
  Notice that $*\rightarrow B\disk^{\circ}_{\rt_p}$ is a suave (hence $D^*$-universal) cover,
  and by \cite[Lemma 4.5.7]{HM24},
  it suffices to prove the primness of the base change to this cover.
  The base change is computed as 
  \[\xymatrix{
    \disk^{\dagger}_{\rt_p}\ar[r]\ar[d] & {*}\ar[d]\\
    \disk^{\dagger}_{\rt_p}/\disk^{\circ}_{\rt_p}\ar[r] & B\disk^{\circ}_{\rt_p}
  }\]
  Then $\disk^{\dagger}_{\rt_p}\rightarrow\operatorname{GSpec}\rt_p$ is prim hence $g$ is prim.
  As we already know that $h\colon B\disk^{\circ}_{\rt_p}\rightarrow\operatorname{GSpec}\rt_p$ is prim from Proposition \ref{analyticCartierduality},
  hence the composition $f=h\circ g\colon \disk^{\dagger}_{\rt_p}/\disk^{\circ}_{\rt_p}\overset{g}{\rightarrow}B\disk^{\circ}_{\rt_p}\overset{h}{\rightarrow}\operatorname{GSpec}\rt_p$ is also prim.

  Next we compute the prim dual of $f$.
  Then $f$ is the composition of two morphisms $h\circ g\colon \disk^{\dagger}_{\rt_p}/\disk^{\circ}_{\rt_p}\overset{g}{\rightarrow}B\disk^{\circ}_{\rt_p}\overset{h}{\rightarrow}\operatorname{GSpec}\rt_p$,
  and hence $\mathbb{P}_f(1)\cong \mathbb{P}_g(1)\otimes g^*\mathbb{P}_h(1)$.
  In fact,
  $g\colon \disk^{\dagger}_{\rt_p}/\disk^{\circ}_{\rt_p}\rightarrow B\disk_{\rt_p}^{\circ}$ is even weakly cohomologically proper,
  since $\disk^{\dagger}_{\rt_p}\rightarrow\operatorname{GSpec}\rt_p$ is weakly cohomologically proper and they we use \cite[Lemma 4.6.3.(ii')]{HM24}.
  We also know that $h$ is prim with prim dual $1[-2]$ from Proposition \ref{analyticCartierduality},
  hence $\mathbb{P}_f(1)\cong \mathbb{P}_g(1)\otimes g^*\mathbb{P}_h(1)\cong 1[-2]$.

  Finally,
  to prove that $\F\rightarrow f_{*}f^*\F$ is an equivalence for all $\F\in D_{\blacksquare}(\rt_p)$,
  by projection formula it suffices to set $\F=1$.
  Then since $\disk^{\dagger}_{\rt_p}/\disk^{\circ}_{\rt_p}=\disk^{\dagger,\dR}_{\rt_p}/\disk^{\circ,\dR}_{\rt_p}$,
  and both $\disk^{\dagger,\dR}_{\rt_p}$ and $\disk^{\circ,\dR}_{\rt_p}$ compute the de Rham cohomology which is trivial,
  cf. \cite[Proposition 5.2.1]{ABLBRCS25},
  we obtain that $f_*f^*1\cong 1$.
\end{proof}

\begin{corollary}\label{Ancosmooth}
  Consider $f\colon\aff^{n,\arith}_{\ff_p}\rightarrow\operatorname{GSpec}\rt_p$.
  Then $f$ is prim with prim dual $1[-2n]$.
\end{corollary}
\begin{proof}
  This is a direct corollary from Proposition \ref{A1cosmooth},
  as $\aff^{n,\arith}_{\ff_p}\cong (\aff^{1,\arith}_{\ff_p})^{ n}$ by Remark \ref{arithKunneth},
  and the prim dual of fiber product can be computed termwise,
  cf. \cite[Proposition 6.5]{Sch23}.
\end{proof}

\begin{remark}\label{vectorbundlecosmooth}
  The Corollary \ref{Ancosmooth} above has a relative version as follows.
  Let $f\colon V\rightarrow X$ be the projection from a rank $n$ vector bundle $V$ over $X$ to $X$.
  Then $f^{\arith}\colon V^{\arith}\rightarrow X^{\arith}$ is prim with prim dual $1[-2n]$.
  Indeed,
  by writing $V$ Zariski locally as $\aff^n_X$ over $X$ and use the fact that the arithmetic de Rham stack satisfies Zariski descent (Remark \ref{Zariskidescentarith}),
  one obtains that $f^{\arith}\colon V^{\arith}\rightarrow X^{\arith}$ factors through as a composition 
  \[
  f^{\arith}\colon V^{\arith}\overset{g}{\longrightarrow} X^{\arith}/\mathcal{V}^{\circ}\overset{h}{\longrightarrow} X^{\arith}
  \]
  with $g$ weakly cohomologically proper,
  for some unitary overconvergent vector bundle $\mathcal{V}^{\le 1}$ over $X^{\arith}$.
  Then by Proposition \ref{analyticCartierduality},
  we know that $h$ is prim with prim dual $1[-2n]$,
  and hence $f^{\arith}$ is also prim with prim dual $1[-2n]$.
\end{remark}

Now let us establish excision properties.

\begin{proposition}\label{basicexcision}
  Consider the geometrically closed-open decomposition $i\colon\{0\}=\Spec\ff_p\hookrightarrow\aff^1_{\ff_p}\hookleftarrow\mathbb{G}_{m,\ff_p}\colon j$.
  It induces an open-closed decomposition on the $6$-functor formalism $X\mapsto D(X^{\arith})$.
\end{proposition}
\begin{proof}
  It suffices to show that the induced morphism 
  \[
  i\colon \{0\}^{\arith}=\operatorname{GSpec}\rt_p\longrightarrow \aff^{1,\arith}_{\ff_p}=\disk^{\dagger}_{\rt_p}/\disk^{\circ}_{\rt_p}
  \longleftarrow \mathbb{G}_{m,\ff_p}^{\arith}=\torus^{\dagger}_{\rt_p}/(1+\disk^{\circ}_{\rt_p})\colon j
  \]
  is an open-closed decomposition for the $6$-functor formalism $X\mapsto D(X)$.
  We already know that $\disk^{\dagger}_{\rt_p}\rightarrow\disk^{\dagger}_{\rt_p}/\disk^{\circ}_{\rt_p}$ is a suave cover from the proof of Proposition \ref{A1cosmooth},
  and along this morphism we get a Cartesian diagram 
  \[\xymatrix{
    \disk^{\circ}_{\rt_p}\ar[r]\ar[d] & \disk^{\dagger}_{\rt_p}\ar[d] & \torus^{\dagger}_{\rt_p}\ar[d]\ar[l]\\
    \operatorname{GSpec}\rt_p\ar[r] & \disk^{\dagger}_{\rt_p}/\disk^{\circ}_{\rt_p} & \torus^{\dagger}_{\rt_p}/(1+\disk^{\circ}_{\rt_p})\ar[l]
  }\]
  We already know that $\disk^{\circ}_{\rt_p}\longrightarrow \disk^{\dagger}_{\rt_p}\longleftarrow \torus^{\dagger}_{\rt_p}$ is an open-closed decomposition,
  and being an open-closed decomposition is local along a suave cover by the following Lemma \ref{openclosedsuavecover},
  hence we conclude that $\operatorname{GSpec}\rt_p\longrightarrow \disk^{\dagger}_{\rt_p}/\disk^{\circ}_{\rt_p}\longleftarrow \torus^{\dagger}_{\rt_p}/(1+\disk^{\circ}_{\rt_p})$ is also an open-closed decomposition.
\end{proof}

\begin{lemma}\label{openclosedsuavecover}
  Consider a Cartesian diagram
  \[\xymatrix{
    X'\ar[r]^{j'}\ar[d]_g & Y'\ar[d]^f & Z'\ar[d]\ar[l]\\
    X\ar[r]^j & Y & Z\ar[l]
  }\]
  Let $f$ be a suave cover,
  and $X'\longrightarrow Y'\longleftarrow Z'$ be an open-closed decomposition.
  Then $X\longrightarrow Y\longleftarrow Z$ is also an open-closed decomposition.
\end{lemma}
\begin{proof}
  Being cohomologically étale or cohomologically proper are both local on the target,
  cf. \cite[Lemma 4.6.3]{HM24}.
  Moreover,
  to check $j$ is an open immersion,
  it suffices to check $j_!$ is fully faithful,
  or equivalently $j^*j_!\cong \id$.
  However,
  it suffices to show $g^*j^*j_!\cong g^*$ by conservativity,
  which is ensured by base change formula and $j'$ being an immersion $g^*j^*j_!\cong j'^*f^*j_!\cong j'^*j'_!g^*\cong g^*$.
  The same argument works for proving $i$ is a closed immersion.
  Finally,
  we need to show 
  \[
  j_!j^*\longrightarrow \id\longrightarrow i_*i^*
  \]
  is a triangle.
  After applying $f^*$,
  by base change formula,
  it becomes 
  \[
  j'_!j'^*f^*\longrightarrow f^*\longrightarrow i'_*i'^* f^*
  \]
  which is a triangle,
  and the conservativity of $f^*$ gives that 
  \[
  j_!j^*\longrightarrow \id\longrightarrow i_*i^*
  \]
  is also a triangle.
\end{proof}

\begin{corollary}\label{contraryexcision}
  If we have a closed-open decomposition $i\colon Z\hookrightarrow X\hookleftarrow U\colon j$ of schemes of finite type over $\ff_p$,
  then it induces an open-closed decomposition on the $6$-functor formalism $X\mapsto D(X^{\arith})$.
\end{corollary}
\begin{proof}
  The property of being an open (or closed) immersion is Zariski local,
  because we have Zariski descent of arithmetic de Rham stack and being an open (or closed) immersion is preserved by a universal $D^*$-cover.
  Hence we can assume $X=\Spec A$ is affine of finite type and $Z=\Spec A/I$,
  and further by induction on the number of generators of $I$ one can assume $Z=\Spec A/f$ is determined by a single equation $f=0$
  (so $Y=\Spec A[f^{-1}]$).
  Then we have a Cartesian diagram 
  \[\xymatrix{
    Z\ar[r]\ar[d] & X\ar[d]^f & Y\ar[l]\ar[d] \\
    \{0\}\ar[r] & \aff^1_{\ff_p} & \mathbb{G}_{m,\ff_p}\ar[l]
  }\]
  In this case $Z^{\arith}\rightarrow X^{\arith}\leftarrow Y^{\arith}$ being an open-closed decomposition follows simply by the base change of the case 
  $\{0\}^{\arith}\rightarrow\aff^{1,\arith}_{\ff_p}\leftarrow\mathbb{G}_{m,\ff_p}^{\arith}$ which is proved in Proposition \ref{basicexcision}.
\end{proof}

Now we establish some preparations for proving cohomological properness and étaleness.

\begin{proposition}\label{P1etale}
  Consider $f\colon\proj^{n,\arith}_{\ff_p}\rightarrow\operatorname{GSpec}\rt_p$.
  Then $f$ is cohomologically étale.
\end{proposition}
\begin{proof}
  From Proposition \ref{propersmootharith},
  we have a natural suave cover $\proj^n_{\rt_p}\rightarrow\proj^{n,\arith}_{\ff_p}$.
  Since $\proj^n_{\rt_p}$ is suave over $\operatorname{GSpec}\rt_p$,
  this implies $\proj^{n,\arith}_{\ff_p}$ is also suave over $\operatorname{GSpec}\rt_p$ as suaveness is suave local on the source,
  cf. \cite[Lemma 4.5.8.(i)]{HM24}.
  To prove $f$ is cohomologically étale,
  since $f$ is already suave,
  by definition it suffices to show that $\Delta_f$ is cohomologically étale.
  However,
  by Corollary \ref{contraryexcision},
  the diagonal $\Delta_f$ is an open immersion hence is cohomologically étale.
\end{proof}

\begin{lemma}\label{etaleCartesian}
  Let $g\colon Y\rightarrow X$ be an étale morphism of schemes of finite type over $\ff_p$.
  Then we have a Cartesian diagram of qfd arc stacks over $\ff_p$:
  \[\xymatrix{
    Y^{\circ}_{\arc}\ar[d]\ar[r] & X^{\circ}_{\arc}\ar[d]\\
    \overline{Y}\ar[r] & \overline{X}
  }\]
\end{lemma}
\begin{proof}
  Let $A$ be a perfectoid ring over $\ff_p$.
  It suffices to prove that 
  \[
  Y(A^{\circ})=X(A^{\circ})\times_{X(A^{\circ}/A^{\circ\circ})} Y(A^{\circ}/A^{\circ\circ}).
  \]
  Let $\xi$ be a pseudo-uniformizer of $A^{\circ}$.
  Since $X$ is of finite type,
  we have that $X(A^{\circ}/A^{\circ\circ})=\varinjlim_{n\rightarrow \infty}X(A^{\circ}/(\xi^{1/p^n}))$ and $X(A^{\circ})=\varprojlim_{n\rightarrow\infty}X(A^{\circ}/(\xi^{p^n}))$.
  By the assumption of $Y\rightarrow X$ being étale (hence formally étale),
  we have that for all $n\ge m\in\itg$,
  the following diagram is Cartesian:
  \[\xymatrix{
    Y(A/(\xi^{p^n}))\ar[r]\ar[d] & X(A/(\xi^{p^n}))\ar[d]\\
    Y(A/(\xi^{p^m}))\ar[r] & X(A/(\xi^{p^m}))
  }\]
  By taking the limit along $n\rightarrow\infty$ and the colimit along $m\rightarrow-\infty$,
  we obtain a Cartesian diagram 
  \[\xymatrix{
    \varprojlim_{n\rightarrow\infty}Y(A/(\xi^{p^n}))\ar[r]\ar[d] & \varprojlim_{n\rightarrow\infty}X(A/(\xi^{p^n}))\ar[d]\\
    \varinjlim_{m\rightarrow-\infty}Y(A/(\xi^{p^m}))\ar[r] & \varinjlim_{m\rightarrow-\infty}X(A/(\xi^{p^m}))
  }\]
  which gives $Y(A^{\circ})=X(A^{\circ})\times_{X(A^{\circ}/A^{\circ\circ})} Y(A^{\circ}/A^{\circ\circ})$.
\end{proof}

\begin{lemma}\label{propersmalldiamond}
  Let $Y\rightarrow X$ be a morphism of separated schemes of finite type over $\ff_p$.
  Then it induces a proper map of qfd arc stacks $Y^{\circ}_{\arc}\rightarrow X^{\circ}_{\arc}$.
\end{lemma}
\begin{proof}
  Zariski locally,
  every $Y$ admits a closed embedding into $\aff^n_X$ for some $n$.
  Since a closed immersion on schemes induces a closed immersion on their associated small diamonds,
  which is proper,
  it suffices to prove the claim for the projection $\aff^n_X\rightarrow X$.
  This becomes $\aff^{n,\circ}_{\arc}\times_{\Spd\ff_p} X^{\circ}_{\arc}\rightarrow X^{\circ}_{\arc}$,
  and it is proper as $\aff^{n,\circ}_{\arc}$ is qcqs over $\Spd\ff_p$,
  because the base change $\aff^{1,\circ}_{\arc}\times_{\Spd\ff_p}\M_{\arc}(\ff_p(\!(t^{1/p^{\infty}})\!))$ to a perfectoid ring
  is represented by the overconvergent closed unit disk which is qcqs over $\M_{\arc}(\ff_p(\!(t^{1/p^{\infty}})\!))$.
\end{proof}

Having established the necessary preliminaries,
we can now prove Theorem \ref{arithsixfunctor}.
We actually mimic the proof of \cite[Theorem 10.6]{Sch23},
but with every notion changed to its opposite;
one might be able to apply this theorem directly to the opposite $6$-functor formalism $X\mapsto D(X^{\arith})^{\op}$,
but we do not know how to avoid problems about presentability.

\begin{proof}[Proof of Theorem \ref{arithsixfunctor}]
  The claim about excision is Corollary \ref{contraryexcision}.
  The claim about strong $\aff^1$-invariance comes directly from Proposition \ref{A1cosmooth}.

  Next step is to show that étale morphisms go to cohomologically proper morphisms.
  Let $X\rightarrow Y$ be an étale morphism.
  Then by Lemma \ref{etaleCartesian},
  we have a Cartesian diagram 
  \[\xymatrix{
    Y^{\circ}_{\arc}\ar[d]\ar[r] & X^{\circ}_{\arc}\ar[d]\\
    \overline{Y}\ar[r] & \overline{X}
  }\]
  By Lemma \ref{propersmalldiamond} and Lemma \ref{smallrigidepimorphismarith},
  the upper horizontal map is proper and the vertical maps are epimorphisms.
  Applying the analytic de Rham stack of punctured Fargues--Fontaine disks,
  we get a Cartesian diagram 
  \[\xymatrix{
    (\Y_{Y_{\arc}^{\circ}}^{\diamond})^{\dR}\ar[r]\ar[d] & (\Y_{X_{\arc}^{\circ}}^{\diamond})^{\dR}\ar[d]\\
    Y^{\arith}\ar[r] & X^{\arith}
  }\]
  such that the upper horizontal map is cohomologically proper (by \cite[Theorem 6.3.1]{ABLBRCS25}) and the vertical maps are epimorphisms (by Proposition \ref{dRFFpreservecolimit}).
  By descent of cohomological properness,
  cf. \cite[Lemma 4.6.3]{HM24},
  we obtain that the lower horizontal map is also cohomologically proper,
  which proves the claim.

  Now we show that smooth morphisms go to prim morphisms and compute the prim dual.
  A smooth morphism $f\colon X\rightarrow Y$ of relative dimension $n$ can,
  locally in Zariski topology,
  be written as an étale morphism over $\aff^n_Y$.
  After knowing étale morphisms are cohomologically proper and affine spaces are prim with prim dual $1[-2n]$ in Corollary \ref{Ancosmooth},
  we at least know that smooth morphisms are prim and if such a morphism is of relative dimension $n$,
  then the prim dual is locally isomorphic to $1_X[-2n]$ hence invertible.
  To identify the prim dual globally with $1_X[-2n]$,
  we use a deformation to the normal cone method.
  Consider the deformation to the normal cone of $\Delta\colon X\rightarrow X\times_Y X$,
  which is a variety $D(\Delta)$ over $\aff^1\times X\times_Y X$ such that
  at $\mathbb{G}_{m}\times X\times_Y X$ the fiber is $\mathbb{G}_{m}\times X\times_Y X$ while at $0$ the fiber is the tangent bundle $T_{X/Y}$.
  It satisfies a Cartesian diagram 
    \[\xymatrix{
      \gff\times X\ar[d]^{\Delta\times{\gff}}\ar[r] & \aff^1\times X\ar[d] & X\ar[l]^{0_X}\ar[d]^{s_0}\\
      \gff\times X\times_Y X\ar[r] & D(\Delta) & T_{X/Y}\ar[l]
    }\]
  Moreover,
  $p\colon X\times_Y X\rightarrow X$ is a smooth map with $\Delta$ being a section,
  i.e. we can complete the diagram as 
    \[\xymatrix{
      \gff\times X\ar[d]^{\Delta\times{\gff}}\ar[r] & \aff^1\times X\ar[d]^{\widetilde{\Delta}}\ar@/^/[r]^{f_X} & X\ar[l]^{0_X}\ar[d]^{s_0}\\
      \gff\times X\times_Y X\ar[r]\ar@/^/[u]^{p\times\gff} & D(\Delta)\ar@/^/[u]^{\widetilde{p}} & T_{X/Y}\ar[l]\ar@/^/[u]^{p_0}
    }\]
  Then the prim dual of $f$ is computed as $\mathbb{P}_f(1)=p_*\Delta_!1$.
  We claim that $\widetilde{p}_*\widetilde{\Delta}_!1=f^*_X\L$ for some invertible $\L\in D(X^{\arith})$.
  Since we already know that étale morphisms are cohomologically proper,
  and the claim satisfies descent along prim covers,
  we can work étale locally,
  and hence we reduce to the case of $X=\aff^n_Y$.
  Then in the case of $X=\aff^n_Y$,
  the diagonal is already isomorphic to the normal cone,
  hence we can just take $\L=1[-2n]$ which will satisfy the claim from Corollary \ref{Ancosmooth}.
  Knowing the claim,
  the fiber of $\widetilde{p}_*\widetilde{\Delta}_!1$ at $0_X$ and $1_X$ are identified with $\L$,
  i.e.
  \[
    \widetilde{p}_*\widetilde{\Delta}_!1|_{0_X}\cong  \widetilde{p}_*\widetilde{\Delta}_!1|_{1_X} \cong \L,
  \]
  and $\widetilde{p}_*\widetilde{\Delta}_!1|_{1_X}\cong p_*\Delta_!1$ by base change theorem.
  Hence the prim dual $\mathbb{P}_f(1)$ is isomorphic to $p_{0,*}s_{0,!}1$.
  By Remark \ref{vectorbundlecosmooth},
  we know $p_0$ is prim with prim dual $1[-2n]$,
  hence $p_{0,*}s_{0,!}1\cong p_{0,!}s_{0,!}1[-2n]\cong 1[-2n]$.
  In conclusion,
  we compute $\mathbb{P}_f(1)\cong 1[-2n]$.

  Finally we prove that proper morphisms are sent to cohomologically étale morphisms.
  First let $f\colon X\rightarrow Y$ be a projective map.
  Then it factors as $f\colon X\rightarrow\proj^n_Y\rightarrow Y$ where $X\rightarrow\proj^n_Y$ is a closed embedding.
  By excision and the fact that $\proj^n\rightarrow *$ is cohomologically étale from Proposition \ref{P1etale},
  we deduce that $f$ is cohomologically étale too.
  Now let $X\rightarrow Y$ be a general proper map.
  By Chow's lemma,
  we can find a projective map $X'\rightarrow Y$,
  and a proper surjective $\pi\colon X'\rightarrow X$ over $Y$ such that there is a dense open $U\subset X$ such that $\pi^{-1}(U)\cong U$.
  Furthermore,
  we can assume $\pi$ itself is also projective; see \cite[0201]{Stacks}.
  However,
  a projective surjection $\pi\colon X'\rightarrow X$ also induces a suave cover $X'^{\arith}\rightarrow X^{\arith}$ which is also cohomologically étale,
  and being cohomologically étale is local along such a cover.
  Hence $X\rightarrow Y$ is cohomologically étale because $X'\rightarrow Y$ is projective and therefore cohomologically étale.
\end{proof}

\begin{corollary}[Poincaré duality]\label{Poincaréarith}
  Let $f\colon X\rightarrow Y$ be a proper smooth morphism of $k$-schemes of relative dimension $n$.
  Denote $f^{\arith/K}\colon X^{\arith/K}\rightarrow Y^{\arith/K}$.
  Let $M$ be a perfect complex over $X^{\arith/K}$,
  then $f_*^{\arith/K}M$ is a perfect complex and there is a perfect pairing $f_*^{\arith/K}M\otimes f_*^{\arith/K}M^{\vee}[2n]\rightarrow 1_{Y^{\arith/K}}$ in $D(Y^{\arith/K})$.
\end{corollary}
\begin{proof}
  Since $f$ is both proper and smooth,
  $f^{\arith/K}$ is both prim with prim dual $1_X[-2n]$ and cohomologically étale by Theorem \ref{relativearithsixfunctor}.
  Therefore,
  $f_*^{\arith/K}$ sends dualizable objects to nuclear and compact objects,
  hence dualizable objects,
  cf. \cite[Proposition 9.3]{CS22};
  meanwhile,
  we can identify perfect complexes with dualizable objects by the Fredholm property \ref{Fredholmproperty}.
  Moreover,
  we can compute $(f_*^{\arith/K}M)^{\vee}\cong (f_!^{\arith/K}M[-2n])^{\vee}\cong (f_!^{\arith/K}M)^{\vee}[2n]\cong f_*^{\arith/K}f^{\arith/K,!}M^{\vee}[2n]\cong f_*^{\arith/K}M^{\vee}[2n]$.
\end{proof}

\subsection{Cartier duality}\label{Cartierduality}

Let $L$ be the (ramified) local field over $\rt_p$ obtained by adding a $(p-1)$-th root $\pi$ (where we choose and fix one) of $-p$.
In this section,
we prove that $\aff^{1,\arith/L}_{\ff_p}$ is self dual in the sense of $1$-categorical Cartier duality in \cite[Definition 2.4.7]{RC25b}.
We also establish a version for general vector bundles. 

We start with some preparations.

\begin{discussion}\label{choosegauge}
  Let $\disk^{\dagger,\sharp}_{L}:=\disk^{\dagger,\sharp}_{\rt_p}\times_{\operatorname{GSpec}\rt_p}\operatorname{GSpec}L$ be the closed unit PD-disk with $L$-coefficients,
  and $\disk^{\circ,\sharp}_{L}:=\disk^{\circ,\sharp}_{\rt_p}\times_{\operatorname{GSpec}\rt_p}\operatorname{GSpec}L$ be the open unit PD-disk with $L$-coefficients,
  which appear naturally in the analytic Cartier duality of Proposition \ref{analyticCartierduality}.
  Then in fact we have that 
  \[
    \disk^{\dagger,\sharp}_{L}\cong \disk^{\dagger,\le p^{-\frac{1}{p-1}}}_{L},\,
    \disk^{\circ,\sharp}_{L}\cong \disk^{\circ,\le p^{-\frac{1}{p-1}}}_{L}
  \]
  are isomorphic to a closed disk and an open disk with a different radius,
  cf. \cite[Remark 3.2.27]{RC25b}.
  Then rescaling by $\pi^{-1}=(-p)^{-\frac{1}{p-1}}$,
  we also obtain isomorphisms $\disk^{\dagger,\sharp}_{L}\cong\disk^{\dagger}_{L}$,
  and $\disk^{\circ,\sharp}_{L}\cong\disk^{\circ}_{L}$.
  Notice that all of these isomorphisms preserve group structures.
  We will freely use these isomorphisms in this section.
\end{discussion}

From now on,
we keep implicitly the subscripts $L$ and simplify the notations as $\disk^{\dagger}$ and $\disk^{\circ}$.

\begin{remark}
  The aforementioned discussion \ref{choosegauge} no longer holds in the Archimedean case.
  The reason is that the Archimedean norm of $n!$ grows faster than any exponential growth.
\end{remark}

\begin{remark}
  The choice of a $(p-1)$-th root of $-p$ also appears in the construction of Dwork's isocrystal.
  In fact,
  Dwork's isocrystal plays the role of exponential sheaf in the Fourier theory of arithmetic $D$-modules;
  see Section \ref{FourierHuyghe}.
\end{remark}

Fix the notations $[-]_!$,
$[-]^*$ and $\cat{bAlg}(K_{D,X})$ in \cite[Construction 2.4.1]{RC25b},
where we let $D$ denote the $6$-functor formalism of Gelfand stacks.
The main theorem in this section is stated below.

\begin{theorem}\label{CartierarithmeticdR}
  The stack $\aff^{1,\arith/L}_{\ff_p}$ is self-dual in the sense of Cartier duality over the category of kernels.
  Precisely,
  there is a natural isomorphism of the Cartier duality of $\aff^{1,\arith/L}_{\ff_p}$ in the sense of Cartier duality over the category of kernels,
  cf. \cite[Proposition 2.4.3]{RC25b},
  with itself:
  \[
  [\aff^{1,\arith/L}_{\ff_p}]_!\cong [\aff^{1,\arith/L}_{\ff_p}]^*\in \cat{bCAlg}(K_{D,L}).
  \]
\end{theorem}

\begin{remark}
  We have $L\cong\rt_p(\zeta_p)$,
  where $\zeta_p$ is a $p$-th root of unity.
  Indeed,
  $\pi$ and $(\zeta_p-1)^{p-1}$ differ by a unit $u$ such that $u\in 1+p\itg_p$,
  but such $u$ admits a $(p-1)$-th root in $\rt_p$.
  The theorem above is then a direct consequence from the motivic Deligne--Fourier equivalence in \cite[Lecture X]{Sch25} under the expectation of an opposite-motivic version in Remark \ref{oppositemotivic}.
  The Cartier duality in \cite[Lecture X]{Sch25} also explains the necessity of working over $L$ instead of $\rt_p$,
  since $[\aff^1]_{\operatorname{SH}}$ identifies with its Cartier dual only \emph{locally} in Gestalten.
\end{remark}

We start the proof by two lemmas.
Using \ref{choosegauge},
we can rewrite the analytic Cartier duality \cite[Proposition 3.2.21]{RC25b} as follows.

\begin{lemma}[Analytic Cartier duality, {\cite[Proposition 3.2.21]{RC25b}}]\label{1catCartierdual}
  The exponential pairing
  \begin{align*}
    \exp_{\pi}\colon \disk^{\dagger}\otimes \disk^{\circ}&\longrightarrow\gff\\
    (x,y)&\longmapsto \exp(\pi xy)=\sum_{n\ge 0}\frac{(\pi xy)^n}{n!}
  \end{align*}
  induces $1$-categorical Cartier dualities 
  \[
  \operatorname{FM}_1\colon \Spec [\disk^{\dagger}]_!\cong \Spec[B\disk^{\circ}]^*,\,
  \operatorname{FM}_2\colon \Spec[\disk^{\circ}]_!\cong \Spec[B\disk^{\dagger}]^* \in \cat{CMon}(\cat{Aff}_K).
  \]
\end{lemma}

Given the Fourier--Mukai transforms above,
we can establish their compatibility in the following sense.

\begin{lemma}\label{compatibilityFM}
  There is a commutative diagram in $\cat{bCAlg}(K_{D,L})$:
  \[\xymatrix{
    [\disk^{\circ}]_!\ar[r]^{\operatorname{FM}_1}\ar[d]_{f_!} & [B\disk^{\dagger}]^*\ar[d]^{g^*}\\
    [\disk^{\dagger}]_!\ar[r]^{\operatorname{FM}_2} & [B\disk^{\circ}]^*
  }\]
  where $g\colon B\disk^{\circ}\rightarrow B\disk^{\dagger}$ and $f\colon \disk^{\circ}\rightarrow\disk^{\dagger}$,
  under the Cartier dualities in Proposition \ref{1catCartierdual}.
\end{lemma}
\begin{proof}
  Note that the exponential pairing induces a pairing on 
  \[
    \exp_{\pi}\colon B\disk^{\dagger}\otimes \disk^{\circ}\cong B\disk^{\circ}\otimes \disk^{\dagger}\rightarrow B\gff.
  \]
  Let $\L=\exp_{\pi}^*(\O(1))$,
  where $\O(1)$ is the tautological bundle on $B\gff$.
  Then the $1$-categorical Cartier dualities are defined as Fourier--Mukai transforms for the pullback $\L_1$ (resp. $\L_2$) of $\L$ to $\disk^{\circ}\times B\disk^{\dagger}$ (resp. $\disk^{\dagger}\times B\disk^{\circ}$)
  Consider the following diagram:
  \[\xymatrix{
    \disk^{\dagger}\times B\disk^{\circ}\ar[r]^p\ar[d]_q & B\disk^{\circ}\ar[dr]^g & \\
    \disk^{\dagger}& \disk^{\circ}\times B\disk^{\circ} \ar[ul]^F\ar[d]_Q\ar[rd]^G\ar[u]_S & B\disk^{\dagger}\\
    & \disk^{\circ}\ar[ul]^f & \disk^{\circ}\times B\disk^{\dagger}\ar[l]^r\ar[u]_s
  }\]
  where the middle two squares are further Cartesian.
  Then starting from $M\in D(\disk^{\circ})$,
  we can compute 
  \[
    g^*\circ\operatorname{FM}_1(M)=g^*s_!(r^*M\otimes\L_1)
    =S_!G^*(r^*M\otimes\L_1)=S_!(Q^*M\otimes G^*\L_1),
  \]
  and 
  \[
    \operatorname{FM}_2\circ f_!(M)=p_!(q^*f_!M\otimes\L_2)=p_!(F_!Q^*M\otimes \L_2)=r_!F_!(Q^*M\otimes F^*\L_2)=S_!(Q^*M\otimes F^*\L_2).
  \]
  Hence it suffices to show that $G^*\L_1\cong F^*\L_2$.
  However,
  this comes from the fact that both $\L_1$ and $\L_2$ come from pullback of $\L$,
  as the following commutative diagram
  \[\xymatrix{
    \disk^{\circ}\times B\disk^{\circ}\ar[r]\ar[d]\ar[dr]^H & \disk^{\dagger}\times B\disk^{\circ}\ar[d]\\
    \disk^{\circ}\times B\disk^{\dagger}\ar[r] & B\disk^{\dagger}\otimes \disk^{\circ}\cong B\disk^{\circ}\otimes \disk^{\dagger}
  }\]
  gives that $G^*\L_1\cong H^*\L \cong F^*\L_2$.
\end{proof}

The way of proving the main theorem is to use the Cartier duality between $\disk^{\dagger}$ (or $\disk^{\circ}$) and $B\disk^{\circ}$ (or $B\disk^{\dagger}$) and their compatibility as above.

\begin{proof}[Proof of Theorem \ref{CartierarithmeticdR}]
  Note that the map $\Spec[\disk^{\dagger}]^*\longrightarrow \Spec[\aff^{1,\arith/L}_{\ff_p}]^*=\Spec[\disk^{\dagger}/\disk^{\circ}]^*$ is a $1$-étale cover.
  This produces a fiber/cofiber sequence in $\cat{CMod}(\cat{Aff}_K)$:
  \[
    \Spec[\disk^{\circ}]^*\longrightarrow \Spec[\disk^{\dagger}]^*\longrightarrow \Spec[\aff^{1,\arith/L}_{\ff_p}]^*.
  \]
  The Cartier duality switches fiber sequences and cofiber sequences,
  hence induces a fiber/cofiber sequence in $\cat{CMod}(\cat{Aff}_K)$:
  \[
    \Spec[\aff^{1,\arith/L}_{\ff_p}]_! \overset{f_!}{\longrightarrow} \Spec[\disk^{\dagger}]_!\longrightarrow \Spec[\disk^{\circ}]_!,
  \]
  which is identified with the following sequence by Lemma \ref{1catCartierdual} and Lemma \ref{compatibilityFM}:
  \[
    \Spec[\aff^{1,\arith/L}_{\ff_p}]_! \longrightarrow \Spec[B\disk^{\circ}]^*\overset{g^*}{\longrightarrow} \Spec[B\disk^{\dagger}]^*.
  \]
  However,
  the following diagram is Cartesian
  \[
  \xymatrix{
    \disk^{\dagger}/\disk^{\circ}\ar[r]\ar[d] & {*}\ar[d]\\
    B\disk^{\circ}\ar[r]^g & B\disk^{\dagger}
  }
  \]
  which identifies the fiber of $\Spec[B\disk^{\circ}]^*\overset{g^*}{\longrightarrow} \Spec[B\disk^{\dagger}]^*$ as 
  $\Spec[\disk^{\dagger}/\disk^{\circ}]^*=\Spec[\aff^{1,\arith/L}_{\ff_p}]^*$.
  Since all of the stacks are $1$-affine,
  this gives an isomorphism 
  \[
  [\aff^{1,\arith/L}_{\ff_p}]_!\cong [\aff^{1,\arith/L}_{\ff_p}]^*\in \cat{bCAlg}(K_{D,L}).\qedhere
  \]
\end{proof}

\begin{remark}[The Fourier--Mukai kernel $\L_{\pi}$]\label{FMkernel}
  Tracing through the proof,
  the equivalence in Theorem \ref{CartierarithmeticdR} also comes from a Fourier--Mukai transform as follows.
  Consider the following pairing:
  \[
    \aff^{1,\arith/L}_{\ff_p}\times \aff^{1,\arith/L}_{\ff_p}\longrightarrow 
    \aff^{1,\arith/L}_{\ff_p}\cong \disk^{\dagger}/\disk^{\circ}\overset{p}{\longrightarrow} B\disk^{\circ}\overset{\exp_{\pi}}{\longrightarrow} B\gff,
  \]
  where the first arrow is given by multiplication on $\aff^{1,\arith/L}_{\ff_p}$,
  and the last arrow is given by the exponential morphism $\exp_{\pi}\colon \disk^{\circ}\rightarrow\gff$,
  cf. \cite[Construction 3.2.19]{RC25b} and \ref{choosegauge}.
  This defines an exponential sheaf $\Exp_{\pi}\in D(\aff^{1,\arith/L}_{\ff_p})$ as 
  \[
  \Exp_{\pi}=p^*\exp_{\pi}^*\O(1),
  \]
  and a Fourier--Mukai kernel $\L_{\pi}$ as the pullback of $\Exp_{\pi}$ under this pairing,
  and the equivalence in Theorem \ref{CartierarithmeticdR} comes from Fourier--Mukai transform under this kernel:
  \[
    \operatorname{FM}_{\pi}(-)=p_{1,!}^{\arith/L}(p_2^{\arith/L,*}(-)\otimes\L)\colon [\aff^{1,\arith/L}_{\ff_p}]_!\cong [\aff^{1,\arith/L}_{\ff_p}]^*
  \]
  Here $p_{1},p_2\colon \aff^1\times\aff^1\rightarrow\aff^1$ are two projections.

  The Fourier--Mukai kernel also satisfies two key properties:
  \begin{enumerate}
    \item We have $p_{12}^{\arith/L,*}\L\otimes p_{23}^{\arith/L,*}\L\cong\Sigma_{13}^{\arith/L,*}\L$,
    where $p_{12}\colon \aff^1\times\aff^1\times\aff^1\rightarrow \aff^1\times\aff^1$ is the projection to the first and second coordinates,
    $p_{23}\colon \aff^1\times\aff^1\times\aff^1\rightarrow \aff^1\times\aff^1$ is the projection to the second and third coordinates,
    and $\Sigma_{13}\colon \aff^1\times\aff^1\times\aff^1\rightarrow \aff^1\times\aff^1$ is the addition on the first and third coordinates.
    \item We have $p_{2,!}^{\arith/K}\L\cong i_!^{\arith/L}1[-2]$,
    where $i\colon\{0\}\rightarrow\aff^1$.
  \end{enumerate}
  Both follow from formal properties and explicit computations\footnote{A key ingredient of the proof is the additivity of the exponential sheaf:
  let $\Sigma\colon \aff^1\times\aff^1\rightarrow\aff^1$ be the addition,
  then we have $\Sigma^{\arith/L,*}\Exp_{\pi}\cong\Exp_{\pi}\boxtimes\Exp_{\pi}$.}.
\end{remark}

\begin{remark}[Inverse Fourier transform]
  Since $\L_{\pi}$ satisfies the two properties in Remark \ref{FMkernel},
  by applying the same procedure in \cite[Théorème 1.2.1.1]{Lau87},
  we formally deduce that the inverse of the Fourier--Mukai transform is identified with itself up to a twist,
  i.e. let $\iota\colon\aff^1\rightarrow\aff^1$ denote the involution map $\iota\colon x\mapsto -x$.
  Then we have
  \[\operatorname{FM}^{-1}_{\pi}=\iota^{\arith/L,*}\circ\operatorname{FM}_{\pi}[2]\colon D(\aff^{1,\arith/L}_{\ff_p})\cong D(\aff^{1,\arith/L}_{\ff_p}).\]

  Theorem \ref{CartierarithmeticdR} implies the following ambidexterity result:
  \[
    \iota^{\arith/L,*}\circ\operatorname{FM}_{\pi}[2]\cong \operatorname{FM}^{-1}_{\pi}\cong \operatorname{FM}^{L}_{\pi}\cong \operatorname{FM}^{R}_{\pi}.
  \]
  Since we have $\operatorname{FM}_{\pi}^R(-)\cong p_{1,*}^{\arith/L}(p_2^{\arith/L,!}(-)\otimes\L^{-1}_{\pi})$,
  and $p_1$ is prim with prim dual $1[-2]$ from Proposition \ref{A1cosmooth},
  we deduce the following property of Fourier--Mukai transform (in a version opposite to that of \cite[Théorème 1.3.1.1]{Lau87},
  as the $6$-functor formalism of arithmetic de Rham stacks behaves opposite to motivic axioms):
  \[
    \operatorname{FM}_{\pi}(-)\cong \iota^{\arith/L,*} p_{1,!}^{\arith/L}(p_2^{\arith/L,!}(-)\otimes\L^{-1}_{\pi})\cong p_{1,!}^{\arith/L}(p_2^{\arith/L,!}(-)\otimes\L_{\pi}).
  \]
\end{remark}

An advantage of using the language in \cite{RC25b} is that we can easily obtain the Cartier duality for general vector bundles almost for free,
as shown below.

\begin{proposition}
  Let $V\rightarrow X$ be a vector bundle over a $\ff_p$-scheme $X$,
  and let $V^{\vee}$ be its dual vector bundle.
  It induces a pairing 
  \[
  m\colon V\times_{X}V^{\vee}\longrightarrow\aff^{1}_X.
  \]
  Let $p_{1}\colon V\times_X V^{\vee}\rightarrow V^{\vee}$ and $p_{2}\colon V\times_X V^{\vee}\rightarrow V$ be the two projections,
  then the Fourier--Mukai transform along the pairing induces an isomorphism:
  \[
  \operatorname{FM}_{V,\pi}(-)=p_{1,!}^{\arith/L}(p_2^{\arith/L,*}(-)\otimes m^{\arith/L,*}\Exp_{\pi})\colon [V^{\arith/L}]_!\overset{\cong}{\longrightarrow}[V^{\arith/L,\vee}]^*\in\cat{bCAlg}(K_{D,X^{\arith/L}}).
  \]
\end{proposition}
\begin{proof}
  By \cite[Remark 2.4.6]{RC25b},
  to prove the equivalence,
  it suffices to prove it locally in the $!$-topology of $X^{\arith}$.
  A Zariski cover induces a $!$-cover on arithmetic de Rham stacks (Remark \ref{Zariskidescentarith}),
  hence we may assume that $V$ is trivial,
  which is true Zariski locally on $X$.
  Then the assertion comes directly from the base change of the case of $\aff^n_{\ff_p}$ over $\Spec\ff_p$.
  In this case the pairing is a direct product of the component-wise pairings,
  which reduces to Theorem \ref{CartierarithmeticdR}.
\end{proof}

\begin{remark}
  There exists a version of Cartier duality for relative arithmetic de Rham stacks.
  Let $V\rightarrow X$ be a vector bundle over a $k$-scheme $X$ with dual $V^{\vee}$,
  and let $K(\pi)$ be the extension of $K$ obtained by adding a $(p-1)$-th root $\pi$ of $-p$.
  Then the Fourier--Mukai transform $\operatorname{FM}_{V,\pi}(-):=p_{1,!}^{\arith/K(\pi)}(p_2^{\arith/K(\pi),*}(-)\otimes m^{\arith/K(\pi),*}\Exp_{\pi})$ induces an isomorphism 
  $\operatorname{FM}_{V,\pi}(-)\colon [V^{\arith/K(\pi)}]_!\overset{\cong}{\longrightarrow}[V^{\arith/K(\pi),\vee}]^*$.
\end{remark}

\section{Arithmetic crystals and arithmetic $D$-modules}

Given the construction of the (relative) arithmetic de Rham stacks,
one can always consider the category of solid sheaves on these stacks.
In the same spirit as \cite{GR14},
one has two descriptions of the category,
in terms of crystals and the $D$-modules.
We will discuss both of them,
and show that $X^{\arith/K}$ is indeed a stacky approach to the theory of rigid cohomology and its coefficients.

\subsection{Arithmetic crystals}\label{arithmeticcrystals}
In this section,
we give a site-theoretic description of the (relative) arithmetic de Rham stacks $X^{\arith}$.
Then we reconstruct the theory of rigid cohomology and overconvergent isocrystals using this perspective.
First we define the (relative) analytic overconvergent site as follows.

\begin{definition}[Analytic overconvergent site and topos]\label{Analyticoverconvergentsite}
  The analytic overconvergent site $\cat{Pair}^{\dagger}$ is the category whose objects are pairs $(A,I)$,
  where $A\in\cat{GelfRing}_{\rt_p}^{\qfd}$ is a qfd Gelfand ring,
  $I$ is an ideal of $A^{\le 1}$ such that $pA^{\le 1}\subset I\subset A^{<1}$,
  and whose morphisms are morphisms from $A$ to $B$ sending $I$ to $J$.
  The topology we put on this site is the $!$-topology on $A$.

  We define the corresponding topos $\cat{Sh}(\cat{Pair}^{\dagger,\op})$ as the category of analytic stacks over $\cat{Pair}^{\dagger}$ in the sense of \cite[Definition 4.2.1]{ABLBRCS25}.
  Let $(A,I)$ be an object in $\cat{Pair}^{\dagger}$,
  then we denote by $\Spec^{\operatorname{pair}}(A,I)\in \cat{Sh}(\cat{Pair}^{\dagger,\op})$ the functor corepresented by $(A,I)$.
\end{definition}

\begin{discussion}
  Notice that we have a diagram 
  \[\xymatrix{
  \cat{Pair}^{\dagger}\ar[r]^F\ar[d]_G & \cat{Ring}_{\ff_p}\\
  \cat{GelfRing}_{\rt_p}^{\qfd}
  }\]
  where the horizontal map is given by $F\colon (A,I)\mapsto A^{\le 1}/I$,
  and the vertical map is given by $G\colon (A,I)\mapsto A$.
  Pursuing the definition,
  we know that for any $X\in\cat{PSh}(\cat{Ring}_{\ff_p}^{\op})$,
  $F_*X$ has functor of points as sending $(A,I)$ to $X(A^{\le 1}/I)$,
  and $G^*F_*X$ has functor of points as sending $A$ to $X(A^{\le 1}/A^{<1})$ by the following Lemma \ref{leftadjointofG},
  hence we have $X^{\arith,\operatorname{pre}}=G^{*}F_*X$ at presheaf level.
  After sheafification,
  we also know that $X^{\arith}=G^{\operatorname{sh},*}F_*^{\operatorname{sh}}X$ at sheaf level,
  where 
  \[\xymatrix{
  \cat{Sh}(\cat{Pair}^{\dagger,\op})\ar[d]_{G^{\operatorname{sh},*}} & \cat{PSh}(\cat{Ring}_{\ff_p}^{\op})\ar[l]_{F_*^{\operatorname{sh}}}\\
  \cat{GelfStk}_{\rt_p}^{\qfd}
  }\]
  Moreover,
  because of the choice of the topology on $\cat{Pair}^{\dagger}$,
  we obtain that $G^{\operatorname{sh},*}=G^*$,
  i.e. there is no need to sheafify $G$.
  In summary,
  we have a reinterpretation that $X^{\arith}=G^{*}F_*^{\operatorname{sh}}X$.
\end{discussion}

\begin{lemma}\label{leftadjointofG}
  The functor $G$ has a right adjoint given by sending $A$ to the pair $(A,A^{<1})$.
  As a corollary,
  $G^*\colon\cat{Psh}(\cat{Pair}^{\dagger,\op})\rightarrow\cat{Psh}(\cat{GelfStk}_{\rt_p}^{\qfd})$
  has functor-of-points description $G^*X(A)=X((A^{\le 1},A^{<1}))$.
\end{lemma}
\begin{proof}
  It suffices to prove that for any pair $(B,J)$,
  we have $\Hom_{\cat{Pair}^{\dagger}}((B,J),(A,A^{<1}))\cong \Hom_{\cat{GelfRing}_{\rt_p}^{\qfd}}(B,A)$.
  However,
  by definition,
  the left side only depends on the morphism from $B$ to $A$,
  since any such morphism automatically sends $J$ inside $A^{<1}$.
\end{proof}

\begin{definition}[Structure sheaf and arithmetic crystals]
  The structure sheaf on $\cat{Pair}^{\dagger,\op}$ is defined as $G_*\O$,
  i.e. the sheaf with $(A,I)$-value $A$.
  Let $X$ be an object in $\cat{Psh}(\cat{Ring}_{\ff_p}^{\op})$,
  then it can be viewed as an object in $\cat{Sh}(\cat{Pair}^{\dagger,\op})$,
  namely, the sheafification of $(A,I)\mapsto X(A^{\le 1}/I)$,
  i.e. $F^{\operatorname{sh}}_*X$.
  The category of (absolute) arithmetic crystals over $X$ is defined as
  \[
  \cat{Crys}_{\arith}(X):=\varprojlim_{\Spec^{\operatorname{pair}}(A,I)\rightarrow F_*^{\operatorname{sh}}X}D(G_*\O((A,I)))=\varprojlim_{(A,I)\rightarrow F_*^{\operatorname{sh}}X}D(A).
  \]
  It contains a full subcategory of dualizable objects:
  \[
    \cat{Crys}_{\arith}^{\perf}(X):=\varprojlim_{\Spec^{\operatorname{pair}}(A,I)\rightarrow F_*^{\operatorname{sh}}X}\cat{Perf}(G_*\O((A,I)))=\varprojlim_{(A,I)\rightarrow F_*^{\operatorname{sh}}X}\cat{Perf}(A).
  \]
\end{definition}

Directly from the discussion above,
one can prove the following proposition saying that the site will recover the theory of coefficients for $X^{\arith}$.

\begin{proposition}\label{arithoverconvergentsite}
  We have 
  \[
    R\Gamma(X^{\arith},\O_{X^{\arith}})\cong R\Hom_{\cat{Sh}(\cat{Pair}^{\dagger,\op})}(F_*^{\operatorname{sh}}X,G_*\O),
  \]
  and we have equivalences of categories 
  \[
    \cat{Crys}_{\arith}(X)\cong D(X^{\arith}),\,\cat{Crys}_{\arith}^{\perf}(X)\cong\cat{Perf}(X^{\arith}).
  \]
\end{proposition}
\begin{proof}
  We have a reinterpretation $R\Gamma(X^{\arith},\O_{X^{\arith}})=R\Hom_{\cat{GelfStk}_{\rt_p}^{\qfd}}(X^{\arith},\O)$.
  Then the identification of cohomologies follows from the formal adjointness:
  \[
   R\Hom_{\cat{GelfStk}_{\rt_p}^{\qfd}}(X^{\arith},\O)\cong
    R\Hom_{\cat{GelfStk}_{\rt_p}^{\qfd}}(G^*F_*^{\operatorname{sh}}X,\O)\cong R\Hom_{\cat{Sh}(\cat{Pair}^{\dagger,\op})}(F_*^{\operatorname{sh}}X,G_*\O).
  \]
  Denote the right adjoint of $G$ by $H$ in Lemma \ref{leftadjointofG}.
  Then $H^*(\operatorname{GSpec}A)=\Spec^{\operatorname{pair}}((A,I))$.
  We also have 
  \[
    D(X^{\arith})=\varprojlim_{\operatorname{GSpec}A\rightarrow G^*F^{\operatorname{sh}_*}X}D(A)
    =\varprojlim_{H^*(\operatorname{GSpec}A)\rightarrow F^{\operatorname{sh}_*}X}D(A)
    =\varprojlim_{(A,I)\rightarrow F_*^{\operatorname{sh}}X}D(A)=\cat{Crys}_{\arith}(X).
  \]
  Then the identification of the perfect complexes gives $\cat{Crys}_{\arith}^{\perf}(X)\cong\cat{Perf}(X^{\arith})$.
\end{proof}

We also have a relative version of the discussion above.
\begin{discussion}[Relative analytic overconvergent site]\label{relativeoverconvergent}
  The relative analytic overconvergent site $\cat{Pair}^{\dagger}_{(k,K)}$ is the category whose objects are pairs $(A,I)$,
  where $A\in\cat{GelfRing}_{K}^{\qfd}$ is a qfd Gelfand ring over $\operatorname{GSpec}K$,
  $I$ is an ideal of $A^{\le 1}$ such that $pA^{\le 1}\subset I\subset A^{<1}$ and $A^{\le 1}/I$ is a $k$-algebra,
  with morphisms and topology defined in the same way as in Definition \ref{Analyticoverconvergentsite}.
  We can form the corresponding topos $\cat{Sh}(\cat{Pair}^{\dagger,\op}_{(k,K)})$.

  We have a similar picture as follows:
  \[\xymatrix{
  \cat{Sh}(\cat{Pair}^{\dagger,\op}_{(k,K)})\ar[d]_{G^{\operatorname{sh},*}_K=G^*_K} & \cat{PSh}(\cat{Ring}_{k}^{\op})\ar[l]_{F_{K,*}^{\operatorname{sh}}}\\
  \cat{GelfStk}_{K}^{\qfd}
  }\]
  such that $X^{\arith/K}=G^*_KF_{K,*}^{\operatorname{sh}}X$.
  This follows by direct computation and Lemma \ref{leftadjointofG}.
  Then we can similarly define the structure sheaf as $G_{K,*}\O$,
  and the category of arithmetic crystals over $X/K$
  \[
    \cat{Crys}_{\arith}(X/K):=\varprojlim_{\Spec^{\operatorname{pair}}(A,I)\rightarrow F_{K,*}^{\operatorname{sh}}X}D(G_{K,*}\O((A,I)))=\varprojlim_{(A,I)\rightarrow F_*^{\operatorname{sh}}X}D(A),
  \]
  with $\cat{Crys}_{\arith}^{\perf}(X/K)$ the full subcategory of dualizable objects.
  We have a similar statement to Proposition \ref{arithoverconvergentsite},
  that for any $X\in \cat{PSh}(\cat{Ring}_{k}^{\op})$,
  we have
  \[
    R\Gamma(X^{\arith/K},\O_{X^{\arith/K}})\cong R\Hom_{\cat{Sh}(\cat{Pair}^{\dagger,\op}_{(k,K)})}(F_{K,*}^{\operatorname{sh}}X,G_{K,*}\O),
  \]
  and 
  \[
    \cat{Crys}_{\arith}(X/K)\cong D(X^{\arith/K}),\,\cat{Crys}_{\arith}^{\perf}(X/K)\cong\cat{Perf}(X^{\arith/K}).
  \]
\end{discussion}

The aforementioned perspectives are more or less a tautological reinterpretation of the arithmetic de Rham stack.
The main point of introducing it is just to describe the objects of $D(X^{\arith})$ as a category of crystals,
as was done previously for the infinitesimal site and the crystalline site,
for example in \cite[Appendix to Lecture VIII]{Sch23}.
However,
if $X$ is a scheme over $k$,
then we can also use the overconvergent site to reconstruct the relative arithmetic de Rham stack of $X$.
We briefly recall the notion of overconvergent site first introduced by Le Stum in \cite{LS11}.

\begin{definition}[Overconvergent site]
  The category $\cat{Var}^{\dagger}_K$ of overconvergent varieties over $(k,K)$ has objects given by pairs made of a formal embedding $X\subset P$ of a $k$-variety $X$ into a formal scheme $P$ over $K^{\circ}$
  and a morphism of analytic varieties $\lambda\colon V\rightarrow P_{\eta}$ over $K$,
  represented by a diagram $X\subset P\leftarrow V$,
  and a morphism from $X'\subset P'\leftarrow V'$ to $X\subset P\leftarrow V$ is a pair of morphisms $(f,u)$,
  where $f\colon X'\rightarrow X$ and $u$ is a morphism from a strict neighborhood of $]X'[_{V'}$ to $V$ such that they are compatible with specializations,
  cf. \cite[Definition 2.3.8]{LS11}.
  The overconvergent site is the category of overconvergent varieties equipped with the analytic topology on the adic tube,
  i.e. a family of morphisms $\{(X\subset P_i\leftarrow V_i)\rightarrow(X\subset P\leftarrow V)\}_i$ such that $\{V_i\}_i$ is an open covering of a neighborhood of $X$ in $V$ and $]X[_V=\bigcup_i]X[_{V_i}$.
\end{definition}

\begin{discussion}
  Denote $\cat{BerkSp}_{K}^{\qfd}$ the category of quasi-finite dimensional Berkovich spaces in \cite[Definition 4.3.1]{ABLBRCS25}.
  We have a similar diagram as before:
  \[\xymatrix{
    \cat{Var}^{\dagger}_K\ar[r]^F\ar[d]_G & \cat{Sch}_{k}\\
  \cat{BerkSp}_{K}^{\qfd}
  }\]
  where $F$ sends $X\subset P\leftarrow V$ to $X$,
  and $G$ sends $X\subset P\leftarrow V$ to $]X[_V^{\dagger}$.
  It gives the following diagram 
  \[\xymatrix{
  \cat{Sh}(\cat{Var}^{\dagger}_K)\ar[d]_{G^{\operatorname{sh},*}} & \cat{PSh}(\cat{Sch}_{k})\ar[l]_{F_*^{\operatorname{sh}}}\\
  \cat{GelfStk}_{K}^{\qfd}
  }\]
\end{discussion}

\begin{proposition}
  Let $X$ be a scheme over $k$.
  We have a natural isomorphism of qfd Gelfand stacks $G^{\operatorname{sh},*}F^{\operatorname{sh}}_*X\cong X^{\arith/K}$.
\end{proposition}
\begin{proof}
  First let $X$ be realizable with a frame $\mathfrak{Y}$ with its generic fiber $\mathfrak{Y}_{\eta}$.
  Then we have an epimorphism $(X\subset \mathfrak{Y}\leftarrow\mathfrak{Y}_{\eta})\rightarrow F_*^{\operatorname{sh}}X$,
  which is ensured by \cite[Theorem 2.5.11]{LS11}.
  We can compute its $n$-th Cech nerve as 
  \[
    (X\subset \mathfrak{Y}\leftarrow\mathfrak{Y}_{\eta})^{\times (n+1)_{/F_{*}^{\operatorname{sh}}X}}=(X\subset\mathfrak{Y}^{n+1}\leftarrow\mathfrak{Y}_{\eta}^{n+1} ).
  \]
  Then since $G^{\operatorname{sh},*}$ preserves colimits,
  we have that 
  \[
    G^{\operatorname{sh},*}F^{\operatorname{sh}}_*X\cong \varinjlim_{n\in \Delta}G^{\operatorname{sh},*}(X\subset \mathfrak{Y}\leftarrow\mathfrak{Y}_{\eta})^{\times (n+1)_{/F_{*}^{\operatorname{sh}}X}}
    \cong \varinjlim_{n\in \Delta}G^{\operatorname{sh},*}(X\subset\mathfrak{Y}^{n+1}\leftarrow\mathfrak{Y}_{\eta}^{n+1} )
    \cong \varinjlim_{n\in \Delta}]X[^{\dagger}_{\mathfrak{Y}^{n+1}}.
  \]
  By Proposition \ref{Xarithdescription},
  we know that $\varinjlim_{n\in \Delta}]X[^{\dagger}_{\mathfrak{Y}^{n+1}}\cong X^{\arith/K}$ and hence we get that $G^{\operatorname{sh},*}F^{\operatorname{sh}}_*X\cong X^{\arith}$ in the realizable case.
  Moreover,
  this isomorphism does not depend on the choice of the frame $\mathfrak{Y}$.
  It suffices to show the independence for an embedding of frames $\mathfrak{Y}\subset\mathfrak{Y}'$ as every pair of frames can be embedded into a common frame,
  and in this case,
  there is a map between two Cech nerves and all isomorphisms above are compatible with this map.

  Both sides satisfy Zariski descent by definition and Remark \ref{Zariskidescentarith}.
  Every scheme is Zariski locally realizable,
  hence it suffices to prove the compatibility in the following sense.
  Let $U\subset X$ be an open embedding of realizable schemes with a frame $\mathfrak{Y}$,
  then the following diagram commutes:
  \[\xymatrix{
    G^{\operatorname{sh},*}F^{\operatorname{sh}}_*U\ar[r]^{\cong}\ar[d] &  U^{\arith/K}\ar[d]\\
    G^{\operatorname{sh},*}F^{\operatorname{sh}}_*X\ar[r]^{\cong} & X^{\arith/K}
  }\]
  However,
  this diagram factors into two commutative diagrams below,
  hence commutes:
  \[\xymatrix{
    G^{\operatorname{sh},*}F^{\operatorname{sh}}_*U\ar[r]^{\cong}\ar[d] & \varinjlim_{n\in \Delta}]U[^{\dagger}_{\mathfrak{Y}^{n+1}}\ar[d] &  U^{\arith/K}\ar[d]\ar[l]_{\cong}\\
    G^{\operatorname{sh},*}F^{\operatorname{sh}}_*X\ar[r]^{\cong} & \varinjlim_{n\in \Delta}]X[^{\dagger}_{\mathfrak{Y}^{n+1}} & X^{\arith/K}\ar[l]_{\cong}
  }\]
  Therefore,
  we get a natural isomorphism $G^{\operatorname{sh},*}F^{\operatorname{sh}}_*X\cong X^{\arith/K}$ for any scheme $X$ over $k$.
\end{proof}

Then the following proposition makes sense.
\begin{proposition}\label{arithoverconvergentsite2}
  Let $X$ be a $k$-scheme.
  We have 
  \[
    R\Gamma(X^{\arith/K},\O_{X^{\arith/K}})\cong R\Hom_{\cat{Sh}(\cat{Var}^{\dagger}_K)}(F_*^{\operatorname{sh}}X,G_*^{\operatorname{sh}}\O).
  \]
  Moreover,
  we define the category of (dualizable) solid overconvergent crystals as
  \[
  \cat{Crys}^{\dagger}(X/K):=\varprojlim_{(S\subset P\leftarrow V)\rightarrow F_*^{\operatorname{sh}}X}D(]S[^{\dagger}_V),\,
  \cat{Crys}^{\dagger,\perf}(X/K):=\varprojlim_{(S\subset P\leftarrow V)\rightarrow F_*^{\operatorname{sh}}X}\cat{Perf}(]S[^{\dagger}_V).
  \]
  Then we have equivalences of categories 
  \[
  D(X^{\arith/K})\cong\cat{Crys}^{\dagger}(X/K),\,\cat{Perf}(X^{\arith/K})\cong \cat{Crys}^{\dagger,\perf}(X/K).
  \]
\end{proposition}
\begin{proof}
  The proof is the same as in the proof of Proposition \ref{arithoverconvergentsite},
  except that the only subtlety is the identification $D(X^{\arith/K})\cong\cat{Crys}^{\dagger}(X/K)$.
  By definition,
  the left side is $D(X^{\arith/K})=\varprojlim_{\operatorname{GSpec}A\rightarrow X^{\arith/K}}D(A)$.
  However,
  since $X^{\arith/K}$ is covered by those images of $G^{\operatorname{sh},*}$,
  which admits lifts to frames,
  then we can refine the terms in the limit as 
  \[
  D(X^{\arith/K})=\varprojlim_{G^{\operatorname{sh},*}(S\subset P\leftarrow V)\rightarrow G^{\operatorname{sh},*}F^{\operatorname{sh}}_*X}D(]S[^{\dagger}_V)
  =\varprojlim_{(S\subset P\leftarrow V)\rightarrow F^{\operatorname{sh}}_*X}D(]S[^{\dagger}_V).
  \]
  This clarifies the identification $D(X^{\arith/K})\cong\cat{Crys}^{\dagger}(X/K)$.
\end{proof}

We move on to discussing the relation between the overconvergent site and rigid cohomology and overconvergent isocrystals.
We first give a brief review of the history of rigid cohomology and overconvergent isocrystals,
following mainly \cite{Ber86} and \cite{Ber96a}.

\begin{definition}[Rigid cohomology]\label{rigidcohomologydef}
  Let $(X,\overline{X},P)$ be a proper smooth frame.
  We define the rigid cohomology of $X$ as follows.
  Define the overconvergent pullback $j^{\dagger}$ as $\varinjlim_{V}j_{V,*}j_V^{-1}$,
  where $V$ runs through all strict open neighborhoods of $]X[_P$ in $]\overline{X}[_P$,
  and $j_V\colon V\rightarrow ]\overline{X}[_P$.
  It is a functor on the derived category of sheaves on $]\overline{X}[_P$.
  Denote $\Omega_{]\overline{X}[_P/K}^{\bullet}$ the de Rham complex of $]\overline{X}[_P$ over $K$.
  The complex of rigid cohomology is defined as
  \[
    R\Gamma_{\rig,P}(X/K):=R\Gamma(]\overline{X}[_P,j^{\dagger}\Omega_{]\overline{X}[_P/K}^{\bullet}).
  \]
  The definition is independent of the choice of proper smooth frames and so we denote it simply as $R\Gamma_{\rig}(X/K)$,
  cf. \cite[2.3]{Ber86}.
  A $k$-scheme $X$ of finite type admits a proper smooth frame Zariski locally,
  whence we can choose such a cover $\mathcal{U}=U_1\coprod\cdots\coprod U_n$ with proper smooth frames $(U_i,\overline{U_i},P_i)$,
  and we define its rigid cohomology via Zariski glueing as follows:
  \[
  R\Gamma_{\rig}(X/K):=\varprojlim_{m\in\Delta}\bigoplus_{i_{1}<\cdots<i_m}R\Gamma_{\rig,P_{i_1}\times\cdots\times P_{i_m}}((U_{i_1}\cap\cdots\cap U_{i_m})/K)\in D^b(K\cat{-Mod}),
  \]
  where the connecting maps are given by pullback along frames.
\end{definition}

\begin{definition}[Pullback of rigid cohomology]\label{Pullbackofrigidcohomology}
  Let $f\colon X\rightarrow Y$ be a morphism between $k$-schemes,
  we can define a map $f^*\colon R\Gamma_{\rig}(Y/K)\rightarrow R\Gamma_{\rig}(X/K)$.
  First if $X$ admits a proper smooth frame $(X,\overline{X},P)$ and $Y$ admits a proper smooth frame $(Y,\overline{Y},Q)$,
  then $f$ can be modified to another morphism between frames $f'\colon (X,\overline{X}',P\times Q)\rightarrow (Y,\overline{Y},Q)$,
  where $X\hookrightarrow P_k\times Q_k$ is given by $X\overset{\Gamma_f}{\rightarrow} X\times Y\rightarrow P_k\times Q_k$
  and $\overline{X}'$ is the Zariski closure of this immersion.
  Then we define $f^*$ as 
  \[
  f^*\colon R\Gamma_{\rig}(Y/K)=R\Gamma_{\rig,Q}(Y/K)\overset{f'^*}{\longrightarrow}R\Gamma_{\rig,P\times Q}(X/K)=R\Gamma_{\rig}(X/K).
  \]
  This is independent of the choice of proper smooth frames,
  cf. \cite[2.3]{Ber86}.
  For general $f\colon X\rightarrow Y$,
  one can define the map by Zariski descent,
  cf. \cite[2.3]{Ber86}.
\end{definition}

\begin{remark}
  We can rewrite the rigid cohomology in the language of Gelfand stacks.
  Let $]X[_P^{\dagger}$ be the overconvergent tubular neighborhoods in Construction \ref{tubes}.
  Let $\Omega^{\bullet}_{]X[_P^{\dagger}/K}$ be the sheaf of the de Rham complex of $]X[_P^{\dagger}$ over $K$
  (which can be defined as the pullback of $\Omega^{\bullet}_{P_{\eta}/K}$ via $]X[_P^{\dagger}\rightarrow P_{\eta}$).
  Then we have
  \[
    H^i(]X[_P^{\dagger},\Omega_{]X[_P^{\dagger}/K}^{\bullet})=H^i(]\overline{X}[_P,j^{\dagger}\Omega_{]\overline{X}[_P/K}^{\bullet})=H^i_{\rig}(X/K),
  \]
  i.e. the de Rham cohomology of the overconvergent tubular neighborhood computes rigid cohomology.
\end{remark}

\begin{definition}[Overconvergent isocrystal, {\cite[Définition 2.3.2]{Ber96a}}]\label{overconvergentisocdef}
  Let $(X,\overline{X},P)$ be a proper smooth frame.
  An overconvergent isocrystal over $X$ is a finite locally free module\footnote{Equivalently, we can take coherent modules in the definition, cf. \cite[Remarque 2.3.3.(ii)]{Ber96a}.} $\E$ equipped with a flat connection defined on some strict neighborhood $V$ of $]X[_P$ in $]\overline{X}[_P$,
  with a condition on the connection about convergence described as follows.
  A flat connection on a vector bundle $\E$ is equivalent to a stratification of $\E$ in the sense of \cite[2.2.2]{Ber96a}.
  Then the condition we put on the connection is that the corresponding stratification extends to some strict neighborhood of $]X[_{P\times P}$,
  i.e. an isomorphism $\epsilon \colon p_1^{\dagger,*}\E\cong p_2^{\dagger,*}\E$ such that
  its reduction modulo $j'^{\dagger}\I^n$ is identified with $\epsilon_n\colon \mathcal{P}_n\otimes\E\cong \E\otimes\mathcal{P}_n$ coming from the stratification.
  Here $j'^{\dagger}\colon ]X[_{P\times P}^{\dagger}\rightarrow P_{\eta}\times P_{\eta}$,
  $\I=\ker(\O_{P_{\eta}\times P_{\eta}}\overset{\Delta^*}{\longrightarrow}\O_{P_{\eta}})$,
  $\mathcal{P}_n=\O_{P_{\eta}\times P_{\eta}}/\I^{n+1}$,
  and $p_1^{\dagger},p_2^{\dagger}\colon ]X[_{P\times P}^{\dagger}\rightarrow]X[_{P}^{\dagger}$ are the two projections.
  A morphism between overconvergent isocrystals is a morphism between the underlying vector bundles which is compatible with the connection.
  
  The definition is independent of the choice of the proper smooth frame $\overline{X}$ and $P$,
  cf. \cite[Théorème 2.3.1.(i)]{Ber96a} and \cite[Théorème 2.3.5]{Ber96a}.
  Thus,
  if we can find a proper smooth frame $(X,\overline{X},P)$ of $X$,
  then we can define $\cat{Isoc}^{\dagger}(X)$ the category of overconvergent isocrystals over $X$,
  which is independent of $\overline{X}$ and $P$.

  In general,
  every separated scheme $X$ of finite type over $k$ is Zariski locally realizable,
  hence admits a proper smooth frame.
  Denote the Zariski cover as $U_1,\ldots,U_n$ and the corresponding frames as $(U_i,\overline{U_i},P_i)$.
  We define an overconvergent isocrystal over $X$ as the following data:
  \begin{enumerate}
    \item An overconvergent isocrystal $\E_i$ over $U_i$ for every $i$,
    using the frame $P_i$.
    \item A glueing data $p_{ij}^*\E_i\cong q_{ij}^*\E_j$ satisfying the cocycle condition,
    where $p_{ij}^*$ is the pullback of overconvergent isocrystals along the map between frames $(U_i\cap U_j,\overline{U_i\cap U_j},P_i\times P_j)\rightarrow (U_i,\overline{U_i},P_i)$,
    and $q_{ij}^*$ is the pullback of overconvergent isocrystals along the map between frames $(U_i\cap U_j,\overline{U_i\cap U_j},P_i\times P_j)\rightarrow (U_j,\overline{U_j},P_j)$.
  \end{enumerate}
  By the independence of the choice of proper smooth frame,
  one can verify that the data satisfy Zariski descent hence is well-defined.
  Then we have the category $\cat{Isoc}^{\dagger}(X/K)$ of overconvergent isocrystals for general $X$.
\end{definition}

\begin{remark}\label{isocrystalabelian}
  The category $\cat{Isoc}^{\dagger}(X/K)$ of overconvergent isocrystals over $X$ is an abelian category,
  cf. \cite[Remarque 2.3.3.(iii)]{Ber96a},
  and we can form its bounded derived category $D^b(\cat{Isoc}^{\dagger}(X/K))$.
\end{remark}

\begin{definition}[Pullback of overconvergent isocrystals]\label{pullbackoverconvergentisocrystal}
  Let $f\colon X\rightarrow Y$ be a morphism between $k$-schemes,
  we can define a functor $f^*\colon \cat{Isoc}^{\dagger}(Y/K)\rightarrow \cat{Isoc}^{\dagger}(X/K)$.
  First if $X$ admits a proper smooth frame $(X,\overline{X},P)$ and $Y$ admits a proper smooth frame $(Y,\overline{Y},Q)$,
  then $f$ can be upgraded to $f'\colon (X,\overline{X}',P\times Q)\rightarrow (Y,\overline{Y},Q)$,
  where $X\hookrightarrow P_k\times Q_k$ is given by $X\overset{\Gamma_f}{\rightarrow} X\times Y\rightarrow P_k\times Q_k$
  and $\overline{X}'$ is the Zariski closure of this immersion.
  Then we define $f^*$ as $f^*=f'^*\colon \cat{Isoc}^{\dagger}(Y/K)\rightarrow \cat{Isoc}^{\dagger}(X/K)$,
  where $f'^*$ is defined a priori because it is a morphism between frames.
  The functor $f^*$ is independent of the choice of proper smooth frames,
  cf. \cite[Proposition 2.2.17]{Ber96a}.
  For general $f\colon X\rightarrow Y$,
  one can define the functor by Zariski descent.
\end{definition}

\begin{definition}[Tensor product of overconvergent isocrystals]
  There exists a symmetric monoidal structure on $\cat{Isoc}^{\dagger}(X/K)$ by tensor product.
  Let $(X,\overline{X},P)$ be a proper smooth frame.
  Given two overconvergent isocrystals $\E_1,\E_2$ over $X$,
  we can view them as vector bundles with flat connections on $]X[_P^{\dagger}$,
  then we define $\E_1\otimes\E_2$ as a vector bundles with flat connection by the Leibniz rule.
  It happens that the connection on $\E_1\otimes\E_2$ is overconvergent too,
  cf. \cite[Corollaire 2.2.10]{Ber96a},
  hence defines an overconvergent isocrystal over $X$.

  In general,
  every separated scheme $X$ of finite type over $k$ is Zariski locally realizable,
  and we can define the tensor product by Zariski descent.
\end{definition}

\begin{definition}[Overconvergent $F$-isocrystal]
  Let $K$ admit a Frobenius lift $\varphi_K$.
  By Definition \ref{pullbackoverconvergentisocrystal},
  the relative Frobenius $\varphi_{X/k}$ induces a functor $\varphi_{X/k}^*\colon \cat{Isoc}^{\dagger}(X^{(1)}/K)\rightarrow \cat{Isoc}^{\dagger}(X/K)$,
  where $X^{(1)}=X\times_{k,\varphi_k}k$.
  On the other hand,
  there is a $\varphi_K$-semilinear functor $\varphi_K^*\colon \cat{Isoc}^{\dagger}(X/K)\rightarrow \cat{Isoc}^{\dagger}(X^{(1)}/K)$.
  Then an overconvergent $F$-isocrystal is an overconvergent isocrystal $\E$ equipped with an isomorphism $\varphi_{X/k}^*\varphi_K^*\E\cong\E$.
  A morphism between overconvergent $F$-isocrystals is a morphism between overconvergent isocrystals compatible with the Frobenius-equivariant structure.
  We then obtain the category $\cat{F-Isoc}^{\dagger}(X/K)$ of overconvergent $F$-isocrystals over $X$.
\end{definition}

\begin{definition}[Rigid cohomology of overconvergent isocrystals]\label{rigidcohisocrys}
  We can define rigid cohomology of overconvergent isocrystals,
  generalizing Definition \ref{rigidcohomologydef}.
  First let $(X,\overline{X},P)$ be a proper smooth frame,
  and $\E$ be an overconvergent isocrystal over $X$.
  We define its rigid cohomology to be 
  \[
  R\Gamma_{\rig,P}(X/K,\E):=R\Gamma(]\overline{X}[_P,j^{\dagger}(\E\otimes\Omega_{V}^{\bullet})).
  \]
  Here $V$ is a strict neighborhood where $\E$ is defined,
  $\E\otimes \Omega_{V}^{\bullet}$ its de Rham complex,
  and $j^{\dagger}$ is the overconvergent pullback in Definition \ref{rigidcohomologydef}.
  This definition does not depend on the proper smooth frame $P$,
  cf. \cite[Théorème 2.3.1.(ii)]{Ber96a},
  and we simply denote it as $R\Gamma_{\rig}(X/K,\E)$.
  
  In general,
  let $\E$ be an overconvergent isocrystal over $X$,
  then we can choose a Zariski cover $\mathcal{U}=U_1\coprod\cdots\coprod U_n$ of $X$ such that $U_i$ admits a proper smooth frame $(U_i,\overline{U_i},P_i)$,
  and such an $\E$ gives an overconvergent isocrystal $\E_{i_1,\ldots,i_m}$ over $(U_{i_1}\cap\cdots\cap U_{i_m},\overline{U_{i_1}\cap\cdots\cap U_{i_m}},P_{i_1}\times\cdots\times P_{i_m})$.
  Then we define 
  \[
  R\Gamma_{\rig}(X/K,\E):=\varprojlim_{m\in\Delta}\bigoplus_{i_{1}<\cdots<i_m}R\Gamma_{\rig,P_{i_1}\times\cdots\times P_{i_m}}((U_{i_1}\cap\cdots\cap U_{i_m})/K,\E_{i_1,\ldots,i_m})\in D^b(K\cat{-Mod}),
  \]
  where the connecting maps are given by pullback along frames.
\end{definition}

\begin{definition}[Pushforward of overconvergent isocrystals]\label{pushforwardisocrystal}
  Let $f\colon X\rightarrow Y$ be a morphism of finite type.
  We can define a pushforward functor along this morphism (generalizing rigid cohomology to the relative case).
  Assume first that $f$ extends to a smooth morphism of smooth frames $f\colon (X,\overline{X},P)\rightarrow (Y,\overline{Y},Q)$.
  Then we define for $\E\in\cat{Isoc}^{\dagger}(X/K)$ the following pushforward,
  which a priori lands in the bounded derived category of $\O_{]Y[_Q^{\dagger}}$-modules:
  \[
  Rf_{\rig,*}\E:= Rf_*(j^{\dagger}(\E\otimes \Omega^{\bullet}_{]\overline{X}[_P/]\overline{Y}[_Q}))\in D^b(\O_{]Y[_Q^{\dagger}}\cat{-Mod}).
  \]
  This definition does not depend on $P$,
  cf. \cite[Théorème 2.3.1.(ii)]{Ber96a}.
  We can define $Rf_{\rig,*}\E$ for general $X$ when fixing $(Y,\overline{Y},Q)$,
  as any such $X$ admits a Zariski cover $\{U_i\}_{i=1}^n$ such that each $U_i$ admits a smooth frame which is smooth over $(Y,\overline{Y},Q)$,
  and then we define 
  \[
    Rf_{\rig,*}\E:=\varprojlim_{m\in\Delta}\bigoplus_{i_{1}<\cdots<i_m}Rf_{\rig,*,P_{i_1}\times\cdots\times P_{i_m}} (\E\mid_{U_{i_1}\cap\cdots\cap U_{i_m}})\in D^b(\O_{]Y[_Q^{\dagger}}\cat{-Mod}),
  \]
  which is independent of the choice,
  cf. \cite[Proposition 10.1.4]{CT03}.
  Then for a general $f\colon X\rightarrow Y$,
  we can define a collection of objects $\{Rf_{\rig,*}\E|_{(U,\overline{U},P)}\}_{(U,\overline{U},P)}$ where $(U,\overline{U},P)$ ranges over proper smooth frames such that $U$ is an open subscheme in $Y$
  (or more generally $U\rightarrow Y$ lies over $Y$),
  by taking $Rf_{\rig,*}\E|_{(U,\overline{U},P)}$ as the pushforward of $\E|_{(U,\overline{U},P)}$ from $X\times_YU$ to $U$ with fixed frame $(U,\overline{U},P)$.
\end{definition}

\begin{remark}\label{padicLiouville}
  It always happens in the theory of overconvergent isocrystals that the construction is a priori frame-dependent but ends up not being.
  This phenomenon leads to some fundamental questions being hard to prove.
  Let us mention two examples:
  \begin{enumerate}
    \item Berthelot's conjecture (\cite[4.3]{Ber86}) states that for a proper smooth morphism $f\colon X\rightarrow Y$,
    and an overconvergent isocrystal $\E$ over $X$,
    the collection of pushforwards $\{R^if_{\rig,*}\E|_{(U,\overline{U},P)}\}_{(U,\overline{U},P)}$ comes from an overconvergent isocrystal over $Y$.
    Several special cases are proved before,
    for example for those overconvergent isocrystals coming from motives,
    cf. \cite{EV25}.
    \item Let $X$ be a smooth scheme over $k$.
    Let $\E$ be an overconvergent isocrystal,
    then $H^i_{\rig}(X/K,\E)$ may not be finite-dimensional.
    This pathology is related to $p$-adic Liouville numbers; see \cite[Remarque after 4.4.12]{Ber96b}.
    Let $\alpha$ be a $p$-adic Liouville number with $|\alpha|<p^{-1/p-1}$,
    and consider the overconvergent isocrystal over $\gff$ defined by the differential equation $\nabla f=\alpha T^{-1}f$.
    Then the $H^1$ of this overconvergent isocrystal is identified with 
    \[
    K\langle T^{\pm1}\rangle^{\dagger}/\{ \alpha f+Tf' \mid f\in K\langle T^{\pm1}\rangle^{\dagger} \}.
    \]
    If we write $f=\sum_{n\in\itg} a_n T^n\in K\langle T^{\pm1}\rangle^{\dagger}$,
    i.e. $|a_n|<C\eta^{|n|}$ for some $\eta<1$,
    then $\alpha f+Tf'=\sum_{n\in\itg} a_n (\alpha+n) T^n$ has coefficients satisfying $\lim\inf|a_n (\alpha+n)|^{\frac{1}{n}}=0$,
    while there are infinitely many linearly independent functions $f\in K\langle T^{\pm1}\rangle^{\dagger}$ that do not satisfy this condition,
    which proves the infinite dimensionality of $H^1$.
    However,
    if $\E$ is an overconvergent $F$-isocrystal,
    then $H^i_{\rig}(X/K,\E)$ are finite-dimensional,
    which is the main theorem in \cite{Ked06}.
  \end{enumerate}
  We will review both of them in Section \ref{Applications}.
\end{remark}

To resolve the problem that the construction is frame-dependent,
one can use the overconvergent site introduced in \cite{LS11} and reviewed in this section previously.
The following proposition essentially comes from \cite{LS11}:
\begin{proposition}\label{arithrigidcoh1}
  Let $X$ be a scheme over $k$.
  Then there exist a $t$-structure on $\cat{Crys}^{\dagger,\perf}(X/K)$ and an identification of the heart with the category of overconvergent isocrystals:
  \[
  \cat{Isoc}^{\dagger}(X/K)\cong \cat{Crys}^{\dagger,\perf,\heartsuit}(X/K).
  \]
  Moreover,
  for a morphism $f\colon Y\rightarrow X$ between schemes over $k$,
  a proper smooth frame $(U,\overline{U},P)$ over $X$,
  and an overconvergent isocrystal $\E\in \cat{Isoc}^{\dagger}(X/K)$,
  we have 
  \[
  f_*\E|_{U\subset P\leftarrow P_{\eta}}(*)=Rf_{\rig,*}\E|_{(U,\overline{U},P)}\in D^b(\O_{]U[_P^{\dagger}}\cat{-Mod}).
  \]
  Here the construction of $f_*\E|_{U\subset P\leftarrow P_{\eta}}(*)$ is that we first consider $\E$ as an object in $\cat{Crys}^{\dagger,\perf}(X/K)$,
  then we do pushforward to an object $f_*\E\in \cat{Crys}^{\dagger,\perf}(Y/K)$ and then do pullback to an object $f_*\E|_{U\subset P\leftarrow P_{\eta}}\in D(]U[_P^{\dagger})$,
  and finally we take its underlying set $f_*\E|_{U\subset P\leftarrow P_{\eta}}(*)\in D^b(\O_{]U[_P^{\dagger}}\cat{-Mod})$.
\end{proposition}
\begin{proof}
  The $t$-structure on $\cat{Crys}^{\dagger,\perf}(X/K)$ is given by descent
  from the $t$-structures on $\cat{Perf}(]S[^{\dagger}_V)$,
  which is ensured by \cite[Proposition 5.4.10]{ABLBRCS25}.
  The identification of the hearts
  \[
    \cat{Isoc}^{\dagger}(X/K)\cong \cat{Crys}^{\dagger,\perf,\heartsuit}(X/K)
  \]
  comes from \cite[Proposition 3.5.11 and Last remark]{LS11}.
  From \cite[Theorem 4.6.7]{LS11},
  one can identify Le Stum's pushforward in the language of $j^{\dagger}\O\cat{-Mod}$ with the rigid cohomology\footnote{Over \emph{good} overconvergent varieties in \cite[Theorem 4.6.7]{LS11},
  which is satisfied by the overconvergent variety $(U\subset P\leftarrow P_{\eta})$ coming from a proper smooth frame.}.
  Then combining the fact that the following diagram commutes 
  \[\xymatrix{
    D(]U[_P^{\dagger})\ar[d]_{({*})}\ar[r]^{f_*} & D(]V[_Q^{\dagger})\ar[d]_{({*})}\\
    D^b(\O_{]U[_P^{\dagger}}\cat{-Mod})\ar[r]^{f_*} & D^b(\O_{]V[_Q^{\dagger}}\cat{-Mod})
  }\]
  and the fact that taking underlying set $S\mapsto S(*)$ commutes with limits,
  one obtains that Le Stum's pushforward also agrees with our pushforward in the solid setting,
  which proves the identification of $f_*\E|_{U\subset P\leftarrow P_{\eta}}(*)$ and $Rf_{\rig,*}\E|_{(U,\overline{U},P)}$.
\end{proof}

\begin{remark}
  One may ask if the equivalence on the hearts in Proposition \ref{arithrigidcoh1} upgrades to an equivalence $D^b(\cat{Isoc}^{\dagger}(X/K))\cong \cat{Crys}^{\dagger,\perf}(X/K)$ on the derived categories.
  Unfortunately,
  it is not true in general.
  Take $X=\proj^1_{\ff_p}$ over $\ff_p$,
  then the category $\cat{Isoc}^{\dagger}(X/\rt_p)\cong \cat{Vect}_{\rt_p}^{\operatorname{fd}}$.
  In fact,
  by \cite[Proposition 4.4]{Ogu84},
  for a proper smooth variety $X$ admitting a proper smooth lifting with geometrically connected and geometrically simply connected generic fiber,
  every convergent isocrystal on $X$ is trivial.
  On the other hand,
  by Proposition \ref{arithrigidcoh1},
  $\Ext^2_{\cat{Crys}^{\dagger,\perf}(\proj^1_{\ff_p}/\rt_p)}(1,1)$ is computed as the $H^2_{\rig}(\proj^1_{\ff_p}/\rt_p)$,
  hence $\Ext^2_{\cat{Crys}^{\dagger,\perf}(\proj^1_{\ff_p}/\rt_p)}(1,1)\cong\rt_p$ which is nontrivial.
\end{remark}

We can summarize the main theorem of this section.

\begin{theorem}\label{arithrigidcoh2}
  Let $X$ be a scheme over $k$.
  Then we have an equivalence of symmetric monoidal categories:
  \[
    D(X^{\arith/K})\cong\cat{Crys}^{\dagger}(X/K).
  \]
  There exists a $t$-structure on $\cat{Perf}(X^{\arith/K})$ and we have
  \[
    \cat{Perf}^{\heartsuit}(X^{\arith/K})\cong \cat{Isoc}^{\dagger}(X/K).
  \]
  Moreover,
  for a morphism $f\colon Y\rightarrow X$ between schemes over $k$,
  a proper smooth frame $(U,\overline{U},P)$ over $X$
  (which induces a map $\pi\colon ]U[_P^{\dagger}\rightarrow X^{\arith/K}$),
  and an overconvergent isocrystal $\E\in \cat{Isoc}^{\dagger}(X/K)$,
  we have 
  \[
  \pi^*f_*^{\arith/K}\E(*)=Rf_{\rig,*}\E|_{(U,\overline{U},P)}\in \O_{]U[_P^{\dagger}}\cat{-Mod}.
  \]
  In particular,
  we have 
  \[
    R\Gamma(X^{\arith/K},\O_{X^{\arith/K}})\cong R\Gamma_{\rig}(X/K).
  \]
\end{theorem}
\begin{proof}
  This is directly from combining Proposition \ref{arithrigidcoh1} and Proposition \ref{arithoverconvergentsite2}.
\end{proof}

\begin{remark}\label{Frobequivarithcrystal}
  Let $K$ admit a Frobenius lift $\varphi_K$ and be perfect.
  Then the $t$-structure in Theorem \ref{arithrigidcoh2} descends Frobenius equivariantly and we also have an equivalence of categories 
  \[
  \cat{Perf}^{\heartsuit}(X^{\arith/K}/\varphi_{X}^{\itg})\cong \cat{F-Isoc}^{\dagger}(X/K).
  \]
  This is directly from Theorem \ref{arithrigidcoh2} and the definition of $X^{\arith/K}/\varphi_{X}^{\itg}$.
\end{remark}

\begin{remark}
  Let $X$ be a smooth projective curve over $\overline{\ff}_p$.
  In \cite{AGKRRV20},
  only a \emph{restricted} version of geometric Langlands conjecture was established for $X$ in $\ell$-adic étale setting.
  We expect the moduli stack of $\check{G}$-bundles on $X^{\arith}$ to be of interest in the context of geometric Langlands conjecture,
  in particular we expect it to give a non-restricted version of geometric Langlands conjecture in $p$-adic coefficients.
  This will be studied in future work.
\end{remark}

A direct corollary is to recover the Poincaré duality of rigid cohomology of overconvergent isocrystals in the proper smooth case:

\begin{corollary}
  Let $X$ be a proper smooth variety over $k$ of dimension $n$.
  Let $\E\in\cat{Isoc}^{\dagger}(X/K)$.
  Then $R\Gamma_{\rig}(X/K,\E)$ is a perfect $K$-complex,
  and we have a perfect pairing $R\Gamma_{\rig}(X/K,\E)\otimes R\Gamma_{\rig}(X/K,\E)[2n]\rightarrow K$.
\end{corollary}
\begin{proof}
  Directly from Theorem \ref{arithrigidcoh2} and Corollary \ref{Poincaréarith}.
\end{proof}

Another direct corollary is to recover cohomological descent of rigid cohomology and overconvergent isocrystals,
recovering \cite[Theorem 4.5.1]{Tsu03} as a part:
\begin{corollary}\label{properhyperdescentTsuzuki}
  Let $X_{\bullet}\rightarrow X$ be a proper hypercovering.
  Then $\cat{Isoc}^{\dagger}(X/K)\cong\varprojlim_{n\in\Delta} \cat{Isoc}^{\dagger}(X_n/K)$ and for any $\E\in \cat{Isoc}^{\dagger}(X/K)$,
  we have $R\Gamma_{\rig}(X/K,\E)\cong\varprojlim_{n\in\Delta}R\Gamma_{\rig}(X_n/K,\E|_{X_n})$.
  In particular,
  we have $R\Gamma_{\rig}(X/K)\cong\varprojlim_{n\in\Delta}R\Gamma_{\rig}(X_n/K)$.
\end{corollary}
\begin{proof}
  Because relative arithmetic de Rham stacks satisfy $h$-hyperdescent (Remark \ref{hdescentrelativearith}),
  we have a $!$-equivalence $X_{\bullet}^{\arith/K}\rightarrow X^{\arith/K}$.
  Then applying Theorem \ref{arithrigidcoh2} and by Fredholm property \ref{Fredholmproperty},
  we get the descent.
\end{proof}

\subsection{Arithmetic $D$-modules}\label{reviewpadiccoefficient}

We study $D(X^{\arith})$ in this section in a different style,
which will show the connection with arithmetic $D$-modules.
To motivate the theory,
it is not enough to consider only the class of overconvergent isocrystals,
because six operations may not preserve them,
and the correct generalization is the theory of arithmetic $D$-modules.
As an $\ell$-adic analogue,
overconvergent isocrystals play the role of local systems,
i.e. ``smooth'' perverse sheaves,
while an abstract $6$-functor formalism only works for general constructible sheaves,
of which arithmetic $D$-modules are an analogue.
We first give a brief review of arithmetic $D$-modules,
following mainly \cite{Ber96b}.

\begin{definition}[($m$-)PD structure, {\cite[Définition 1.3.1]{Ber96b}}]
  Let $A$ be a ring living over $\itg_{(p)}$.
  A PD-ideal $I\subset A$ is an ideal $I$ of $A$ such that we have a collection $\gamma_n\colon I\rightarrow A$ which formally codes the divided power $x\mapsto \frac{x^n}{n!}$,
  in other words they satisfy:
  \begin{enumerate}
    \item $\gamma_0(x)=1$ and $\gamma_1(x)=x$;
    \item $\gamma_n(x+y)=\sum_{i=0}^{n}\gamma_i(x)\gamma_{n-i}(y)$;
    \item $\gamma_n(x)\gamma_m(x)=\frac{(m+n)!}{m!n!}\gamma_{n+m}(x)$;
    \item $\gamma_n(\lambda x)=\lambda^n\gamma_n(x)$;
    \item $\gamma_n(\gamma_m(x))=\frac{(mn)!}{(m!)^n n!}\gamma_{nm}(x)$.
  \end{enumerate}
  An $m$-PD structure over an ideal $I\subset A$ is a PD-ideal $J\subset I$ with $I^{(p^m)}+pI\subset J$.
\end{definition}

\begin{definition}[$m$-PD neighborhood, {\cite[Définition 2.1.1]{Ber96b}}]
  Let $X\rightarrow Y$ be a closed embedding of formal schemes over $\itg_{(p)}$,
  we define its $m$-PD neighborhood as follows.
  Denote $I$ the ideal sheaf.
  Then we formally add $\frac{x^{p^{m+k}}}{(p^k)!}$ for all $x\in I$ to $\O_{Y}$ and quotient by all of their formal relations.
  It defines a new ring which makes $I$ an $m$-PD ideal in this ring.
  Its spectrum is called the $m$-PD neighborhood of $X$ in $Y$,
  and is denoted by $X_{Y}^{m,\sharp}$.
\end{definition}

\begin{definition}[$\widehat{\D}_{\Xf}^{(m)}$ and $\D_{\Xf}^{\dagger}$, {\cite[2.2.1, 2.4.1]{Ber96b}}]
  Let $\Xf$ be a formal scheme of finite type over $K^{\circ}$.
  Let $\I$ be the ideal sheaf of the diagonal embedding $\Delta\colon\Xf\rightarrow\Xf\times_{K^{\circ}}\Xf$ and let $\Delta_n(\Xf)$ be the formal scheme corresponding to $\O_{\Xf\times_{W(k)}\Xf/\I^{n+1}}$.
  View them as schemes over $\itg_{(p)}$ and then denote $\Delta(\Xf)^{m,\sharp}_n$ as the $m$-PD neighborhood of the closed embedding $\Delta_n\colon\Xf\rightarrow\Delta_n(\Xf)$.
  There are two projections $p_1,p_2\colon\Delta(\Xf)^{m,\sharp}_n\rightarrow \Xf$,
  and we view $\O_{\Delta(\Xf)_n^{m,\sharp}}$ as an $\O_{\Xf}$-module via $p_1$.
  We define the sheaf (over $\Xf$) of (completed) arithmetic differential operators of order $n$ and level $m$ as 
  \[
  \D_{\Xf,n}^{(m)}:=\underline{\Hom}_{\O_{\Xf}}(\O_{\Delta(\Xf)_n^{m,\sharp}},\O_{\Xf}),\,
  \widehat{\D}_{\Xf,n}^{(m)}:=\varprojlim_i \D_{\Xf,n}^{(m)}/p^i \D_{\Xf,n}^{(m)}.
  \]
  The sheaf of (completed) arithmetic differential operators of level $m$ is then defined as $\D_{\Xf}^{(m)}:=\bigcup_n \D_{\Xf,n}^{(m)}$ and $\widehat{\D}_{\Xf}^{(m)}:=\varprojlim_i \D_{\Xf}^{(m)}/p^i \D_{\Xf}^{(m)}$.
  The sheaf of arithmetic differential operators is then defined as 
  \[
  \D_{\Xf}^{\dagger}:=\varinjlim_m \widehat{\D}_{\Xf}^{(m)}.
  \]
  The $\D_{\Xf}^{(m)}$, $\widehat{\D}_{\Xf}^{(m)}$ and $\D_{\Xf}^{\dagger}$ are sheaves of rings,
  cf. \cite[2.2.1]{Ber96b}.
  We also denote $\widehat{\D}_{\Xf,\rt}^{(m)}:=\widehat{\D}_{\Xf}^{(m)}\otimes_{K^{\circ}}K$ and $\D_{\Xf,\rt}^{\dagger}:=\D_{\Xf}^{\dagger}\otimes_{K^{\circ}}K$.
\end{definition}

\begin{remark}\label{localcoordinate}
  If $\Xf$ is smooth,
  with a choice of local coordinates $t_1,\ldots,t_n$ and corresponding derivatives $\partial_1,\ldots,\partial_n$,
  then the local sections of $\widehat{\D}_{\Xf,\rt}^{(m)}$ have the following description in \cite[2.4.1]{Ber96b}
  \[
  \left\{\sum a_{k_1,\ldots,k_n}^{l_1,\ldots,l_n}t_1^{l_1}\cdots t_n^{l_n}\frac{(\lfloor\frac{k_1}{p^m}\rfloor)!\partial_1^{k_1}}{(k_1)!}\cdots
  \frac{(\lfloor\frac{k_n}{p^m}\rfloor)!\partial_n^{k_n}}{(k_n)!}\mid\,a_{k_1,\ldots,k_n}^{l_1,\ldots,l_n}\rightarrow 0\right\},
  \]
  and the local sections of $\D_{\Xf,\rt}^{\dagger}$ have the description (see \cite[Proposition 2.4.4]{Ber96b})
  \[
  \left\{\sum a_{k_1,\ldots,k_n}^{l_1,\ldots,l_n}t_1^{l_1}\cdots t_n^{l_n}\frac{\partial_1^{k_1}}{(k_1)!}\cdots \frac{\partial_n^{k_n}}{(k_n)!}\mid
  \exists\eta<1,|a_{k_1,\ldots,k_n}^{l_1,\ldots,l_n}|<C\eta^{k_1+\cdots+k_n},\,\lim_{l_1,\ldots,l_n\rightarrow \infty}a_{k_1,\ldots,k_n}^{l_1,\ldots,l_n}= 0\right\}.
  \]
  Therefore,
  if $\Xf$ is smooth with a choice of local coordinates $t_1,\ldots,t_n$,
  then we can define a sheaf on $\Xf_{\eta}$ as follows:
  \[
  \D^{\dagger}_{\Xf_{\eta}}=\varinjlim_{r<1}\Hom_{\O_{\Xf_{\eta}}}(\O_{\Delta^{r}(\Xf_{\eta})},\O_{\Xf_{\eta}}),
  \]
  where $\Delta^{r}(\Xf_{\eta})$ is the subspace of $]X[_{\Xf^2}$ given by localizing at the locus where the generator $\{t_i\otimes 1-1\otimes t_i\}_{i=1,\ldots, n}$ of the ideal sheaf $\I=\ker(\O_{\Xf_{\eta}^2}\overset{\Delta}{\rightarrow} \O_{\Xf_{\eta}})$ determined by $t_1,\ldots,t_n$ has norm less than or equal to $r$.
  Then $\varinjlim_{r<1}\Delta^{r}(\Xf_{\eta})=]X[_{\Xf^2}$,
  and the definition of $\D^{\dagger}_{\Xf_{\eta}}$ is independent of the choice of the generators of $\I$.
  The local description above shows that $\D_{\Xf,\rt}^{\dagger}|_{\Xf_{\eta}}=\D^{\dagger}_{\Xf_{\eta}}$.
\end{remark}

\begin{remark}
  This $\D_{\Xf,\rt}^{\dagger}$ serves as the ring of differential operators defining \textit{convergent} isocrystals instead of overconvergent isocrystals.
  Precisely,
  let $X$ be the special fiber of $\Xf$,
  then the category of convergent isocrystals over $X$ embeds fully faithfully into the category of $\D_{\Xf,\rt}^{\dagger}$-modules,
  with image identified with those $\O_{\Xf,\rt}$-coherent $\D_{\Xf,\rt}^{\dagger}$-modules; see \cite[4.1.4]{Ber96b}.
\end{remark}

The sheaf of rings $\D_{\Xf,\rt}^{\dagger}$ is equipped with a topology,
which makes the category of quasi-coherent modules over it hard to define
(for example, the independence of the choice of frames and the descent).
In \cite[4.2.1]{Ber02},
Berthelot introduced a category $\underset{\rightarrow}{LD}{}_{\rt,\operatorname{qc}}(\widehat{\D}^{(\cdot)}_{\Xf})$,
and defined cohomological operations of this category.
Later,
Caro \cite{Car04}, \cite{Car09} studied subcategories (of overcoherent or overholonomic objects) of this category,
and formed the category of (overcoherent or overholonomic) arithmetic $D$-modules.
However,
in our setting,
a natural way of defining quasi-coherent sheaves over topological rings is via the theory of condensed mathematics,
precisely solid sheaves,
introduced in \cite{CS19}.
We make the following definition and claim it to be a good replacement of $\underset{\rightarrow}{LD}{}_{\rt,\operatorname{qc}}(\widehat{\D}^{(\cdot)}_{\Xf})$ in our situation.

\begin{definition}
  Let $\D^{\dagger}_{\Xf,\rt}$ be the condensed sheaf of $K$-algebras over $\Xf$ equipped with the ind-topology
  coming from the $p$-adic topology on $\widehat{\D}^{(m)}_{\Xf}$ which is a sheaf of Banach algebras over $\Xf$.
  It is a sheaf of LB $\rt_p$-solid algebras of compact type (hence dual nuclear Fréchet, cf. \cite[Theorem 1.3]{ST02} or \cite[Remark 1.2]{RCRJ22}) over $\Xf$.
  Let $D_{\blacksquare}(\D^{\dagger}_{\Xf,\rt})$ be the category of (left) $\rt_p$-solid $\D^{\dagger}_{\Xf,\rt}$-sheaves on $\Xf$.
\end{definition}

\begin{remark}\label{anotherdescriptionarithmeticD}
  If $\Xf$ is smooth,
  then by Remark \ref{localcoordinate},
  we have $\D_{\Xf,\rt}^{\dagger}|_{\Xf_{\eta}}=\D^{\dagger}_{\Xf_{\eta}}$,
  and hence we also have an equivalence of categories $D_{\blacksquare}(\D^{\dagger}_{\Xf,\rt})\cong \cat{LMod}_{\D^{\dagger}_{\Xf_{\eta}}}(D(\X))$ as $\D_{\Xf,\rt}^{\dagger}$ is already rational.
\end{remark}

There exist six operations for solid $\D^{\dagger}_{\Xf,\rt}$-modules,
which are inspired by the constructions in \cite{Be} for algebraic $D$-modules and \cite{Ber02} for $\underset{\rightarrow}{LD}{}_{\rt,\operatorname{qc}}(\widehat{\D}^{(\cdot)}_{\Xf})$.
We will discuss them below.

\begin{definition}[Naïve pullback functor, solid case]
  Let $f\colon \Xf\rightarrow \mathfrak{Y}$ be a morphism between smooth formal schemes over $K^{\circ}$.
  We define the naïve pullback functor $f^{\Delta}\colon D_{\blacksquare}(\D^{\dagger}_{\mathfrak{Y},\rt})\rightarrow D_{\blacksquare}(\D^{\dagger}_{\mathfrak{X},\rt})$ as
  \[
  f^{\triangle}(-):=f^*\D_{\mathfrak{Y},\rt}^{\dagger}\otimes_{f^{-1}\D_{\mathfrak{Y},\rt}^{\dagger},\blacksquare}f^{-1}(-),
  \]
  where $f^*\D_{\mathfrak{Y},\rt}^{\dagger}:=\O_{\mathfrak{X},\rt}\otimes_{f^{-1}\O_{\mathfrak{Y},\rt},\blacksquare}f^{-1}\D_{\mathfrak{Y},\rt}^{\dagger}$ is a $(\D_{\Xf,\rt}^{\dagger},f^{-1}\D_{\mathfrak{Y},\rt}^{\dagger})$-bimodule,
  with the $\D_{\Xf,\rt}^{\dagger}$-action induced from the $\D_{\Xf,\rt}^{\dagger}$-action on $\O_{\Xf,\rt}$.
  We also denote $f^*\D_{\mathfrak{Y},\rt}^{\dagger}$ as $\D^{\dagger}_{\Xf\rightarrow\mathfrak{Y},\rt}$,
  and so $f^{\triangle}(-)=\D^{\dagger}_{\Xf\rightarrow\mathfrak{Y},\rt}\otimes_{f^{-1}\D_{\mathfrak{Y},\rt}^{\dagger},\blacksquare}f^{-1}(-)$.
\end{definition}

\begin{definition}[Extraordinary pullback functor, solid case]
  Let $f\colon \Xf\rightarrow \mathfrak{Y}$ be a morphism between smooth formal schemes.
  We define the extraordinary inverse image $f^{!,\Be}\colon D_{\blacksquare}(\D^{\dagger}_{\mathfrak{Y},\rt})\rightarrow D_{\blacksquare}(\D^{\dagger}_{\mathfrak{X},\rt})$
  (``Be'' for both Bernstein and Berthelot,
  that Bernstein firstly introduced this notion in the theory of algebraic $D$-module and then Berthelot put it in the situation of arithmetic $D$-module)
  as \[f^{!,\Be}:=f^{\triangle}[\dim X-\dim Y].\]
\end{definition}

\begin{definition}[Extraordinary direct image functor, solid case]
  Let $f\colon \Xf\rightarrow \mathfrak{Y}$ be a morphism between smooth formal schemes.
  We define the extraordinary direct image functor $f_+\colon D_{\blacksquare}(\D^{\dagger}_{\mathfrak{X},\rt})\rightarrow D_{\blacksquare}(\D^{\dagger}_{\mathfrak{Y},\rt})$ as 
  \[
  f_+(-):=f_*(\D_{\mathfrak{Y}\leftarrow \Xf,\rt}^{\dagger}\otimes_{\D_{\Xf,\rt}^{\dagger},\blacksquare}(-)),
  \]
  where $\D_{\mathfrak{Y}\leftarrow \Xf,\rt}^{\dagger}:=f^{-1}\D_{\mathfrak{Y},\rt}^{\dagger}\otimes_{f^{-1}\O_{\mathfrak{Y},\rt},\blacksquare}(\Omega_{\Xf,\rt}^{\dim\Xf}\otimes_{\O_{\Xf,\rt}} f^{*}(\Omega_{\mathfrak{Y},\rt}^{\dim\mathfrak{Y}})^{-1})$
  is a $(f^{-1}\D_{\mathfrak{Y},\rt}^{\dagger},\D_{\mathfrak{X},\rt}^{\dagger})$-bimodule,
  which comes from the $(\D_{\Xf,\rt}^{\dagger},f^{-1}\D_{\mathfrak{Y},\rt}^{\dagger})$-bimodule on $\D^{\dagger}_{\Xf\rightarrow\mathfrak{Y},\rt}$ and the identification
  \[
    \D_{\mathfrak{Y}\leftarrow \Xf,\rt}^{\dagger}=\Omega_{\Xf,\rt}^{\dim\Xf}\otimes_{\O_{\Xf,\rt}}\D^{\dagger}_{\Xf\rightarrow\mathfrak{Y},\rt}\otimes_{f^{-1}\O_{\mathfrak{Y},\rt}}f^{-1}(\Omega_{\mathfrak{Y},\rt}^{\dim\mathfrak{Y}})^{-1}.
  \]
\end{definition}

\begin{remark}\label{extraordinaryrightDmod}
  The extraordinary pushforward behaves more simply on right $D$-modules,
  which is just 
  \[
  f_+(-):=f_*((-)\otimes_{\D_{\Xf,\rt}^{\dagger},\blacksquare}\D_{\Xf\rightarrow\mathfrak{Y},\rt}^{\dagger})[\dim Y-\dim X],
  \]
\end{remark}

\begin{remark}[De Rham resolution]\label{deRhamresolution}
  We note that if $f\colon \Xf\rightarrow \mathfrak{Y}$ is a \textit{smooth} morphism between smooth formal schemes,
  then there is a de Rham resolution as follows\footnote{One can find an analogue in the theory of algebraic $D$-modules in \cite{Be},
  and do the same strategy of the proof.
  Indeed,
  one can use a trick of deforming differential operators to principle symbols,
  and then for principle symbols the isomorphism comes from an overconvergent version of Koszul resolutions,
  cf. \cite[Lemma 4.3.6]{RC24a}.}:
  \[
  \Omega^{\bullet}_{\Xf/\mathfrak{Y},\rt}[\dim\Xf-\dim\mathfrak{Y}]\otimes_{\O_{\Xf,\rt}}\D_{\Xf,\rt}^{\dagger} \cong\D_{\mathfrak{Y}\leftarrow\Xf,\rt}^{\dagger},
  \]
  Then
  \begin{align*}
    f_+M=f_*(\D_{\mathfrak{Y}\leftarrow \Xf,\rt}^{\dagger}\otimes_{\D_{\Xf,\rt}^{\dagger},\blacksquare}M)=f_*(\Omega^{\bullet}_{\Xf/\mathfrak{Y},\rt}\otimes_{\O_{\Xf,\rt},\blacksquare}M)[\dim\Xf-\dim\mathfrak{Y}]
  \end{align*}
  computes the relative de Rham cohomology of $M$ up to a shift.
\end{remark}

\begin{definition}[Tensor product, solid case]\label{tensorproductarithD}
  Let $\Xf$ be a smooth formal scheme over $K^{\circ}$,
  and $\Delta_{\Xf}\colon \Xf\rightarrow \Xf\times \Xf$ be the diagonal embedding.
  The tensor product structure on $D_{\blacksquare}(\D^{\dagger}_{\Xf,\rt})$ is defined as 
  \[
  M\otimes N:= \Delta^{\triangle}_X(M\boxtimes N).
  \]
\end{definition}

Now we state and prove the main theorem of this section,
identifying sheaves on $X^{\arith/K}$ as a category of arithmetic $D$-modules over $X$.
If $X$ is realizable with proper smooth frame $(X,\overline{X},P)$,
then we can make the following definition.
By Remark \ref{localcoordinate},
we have a sheaf of rings $\D^{\dagger}_{P_{\eta}}$ over $P$,
and we denote $\D^{\dagger}_{]X[^{\dagger}_P}:=\D^{\dagger}_{P_{\eta}}|_{]X[^{\dagger}_P}$.
Another description of $\D^{\dagger}_{]X[^{\dagger}_P}$ is given as (in the same terminology as Remark \ref{localcoordinate})
\[\D^{\dagger}_{]X[^{\dagger}_P}:=\varinjlim_{r<1}\underline{\Hom}_{\O_{]X[^{\dagger}_P}}(\O_{\Delta^r(]X[^{\dagger}_P)},\O_{]X[^{\dagger}_P}).\] 

\begin{theorem}\label{Xarithcoefficients}
  The description of $D(X^{\arith/K})$ as arithmetic $D$-modules is as follows.
  \begin{enumerate}
    \item Let $X$ be a realizable scheme over $k$ with proper smooth frame $(X,\overline{X},P)$.
    There is an equivalence of categories 
    \[
      D(X^{\arith/K})\cong \cat{LMod}_{\D^{\dagger}_{]X[^{\dagger}_P}}(D(]X[^{\dagger}_P)).
    \]
    \item We specialize to the case where $X$ is proper and smooth, and admits a proper smooth lift $\Xf$.
    Equipped with tensorial symmetric monoidal structure on $D(X^{\arith/K})$
    and the solid tensor product in Definition \ref{tensorproductarithD},
    then the equivalence is upgraded to an equivalence of symmetric monoidal categories under Remark \ref{anotherdescriptionarithmeticD}:
    \[
      D(X^{\arith/K})\cong \cat{LMod}_{\D^{\dagger}_{\X}}(D(\X))\cong D_{\blacksquare}(\D^{\dagger}_{\Xf,\rt}).
    \]
    \item Let $f\colon \Xf\rightarrow \mathfrak{Y}$ be a morphism between proper smooth formal schemes,
    and $X\rightarrow Y$ the induced morphism between their special fibers.
    Then the induced morphism $f^{\arith/K,*}\colon X^{\arith/K}\rightarrow Y^{\arith/K}$ identifies $f^{\arith/K,*}$ with the naïve pullback functor $f^{\triangle}$,
    and $f_!^{\arith/K}$ with a shift of the extraordinary pushforward functor $f_+[\dim X-\dim Y]$ under the equivalence above.
  \end{enumerate}
\end{theorem}

The logic of the proof of Theorem \ref{Xarithcoefficients} is to first prove $(2)$ and $(3)$ of the theorem,
and then deduce $(1)$ from $(2)$.
Let us start with some prerequisites.

\begin{discussion}[Algebraic de Rham stack]\label{algdR}
  Let $X$ be a Gelfand stack over $\rt_p$.
  We have its algebraic de Rham stack $X^{\operatorname{alg},\dR}$.
  The construction for good $X$ is intuitively given by $X^{\operatorname{alg},\dR}=X/\widehat{\Delta(X)}$,
  where $\Delta\colon X\rightarrow X\times X$ is the diagonal morphism and $\widehat{\Delta(X)}$ is the formal completion.
  A rigorous definition of $X^{\operatorname{alg},\dR}$ is given by functor-of-points description,
  cf. \cite{RC24a}.

  Given an affine Gelfand stack $X=\operatorname{GSpec}A$ coming from a (Berkovich) smooth dagger affinoid algebra $A$ in the sense of \cite{GK00} (or \cite[Definition 4.3.4]{ABLBRCS25}),
  we can study solid sheaves on $X^{\operatorname{alg},\dR}$.
  Let $D_A$ be the associative $\rt_{p,\blacksquare}$-algebra of differential operators on $A$,
  then we have an equivalence of symmetric monoidal $\infty$-categories 
  \[
  D(X^{\operatorname{alg},\dR})\cong \cat{LMod}_{D_A}(D_{\blacksquare}(\rt_p)).
  \]
  This is a reformulation of the result from the case of schemes,
  cf. \cite[10.2]{RC25a},
  to the case of dagger varieties,
  and we briefly recall the machinery.
  Consider the commutative diagram 
  \[\xymatrix{
    X\ar[r]^{\pi}\ar[dr]_p & X^{\operatorname{alg},\dR}\ar[d]^{p^{\dR}}\\
     & \text{*}
  }\]
  The morphism $p$ is proper and $p_*$ is conservative,
  so we deduce that the functor $p_*$ induces an equivalence $p_*\colon D(X)\cong \cat{Mod}_{A}(D_{\blacksquare}(\rt_p))$ by the Barr--Beck--Lurie theorem;
  in this case,
  $A=p_*p^*1$ is equipped with the algebra structure coming from adjointness.
  The morphism $\pi$ is suave and $\pi^*$ is conservative.
  First,
  by the smoothness,
  the morphism $\pi$ is surjective,
  hence it suffices to check the suaveness after pullback along $\pi$.
  Then for a Berkovich smooth $\dagger$-space $X$,
  analytic locally it is standard Berkovich smooth,
  and by Taylor expansions,
  the morphism $\pi$ becomes the form $X\times\widehat{\mathbb{G}}_a^d\rightarrow X$,
  which is the base change of the map in \cite[Lemma 8.22]{RC25a},
  hence is suave.
  The conservativity of $\pi^*$ comes directly from that $\pi$ is a suave cover.

  By a similar strategy to \cite[10.2]{RC25a},
  we identify $D(X^{\operatorname{alg},\dR})$ using Barr--Beck--Lurie theorem.
  We consider the functor $p_*\pi^*$,
    which is conservative and preserves geometric realizations,
    hence we obtain an equivalence $p_*\pi^*\colon D(X^{\operatorname{alg},\dR})\cong \cat{Mod}_{p_*\pi^*\pi_{\natural}p^*}(D_{\blacksquare}(\rt_p))$.
    Here $\pi_{\natural}$ is the left adjoint of $\pi^*$ whose existence is due to the suaveness of $\pi$.
    To compute the monad $p_*\pi^*\pi_{\natural}p^*$,
    notice that for any $M\in D(X^{\operatorname{alg},\dR})$,
    we have 
    \[p_*\pi^*M=p_*\underline{\Hom}_{D(X^{\operatorname{alg},\dR})}(\pi_{\natural}1,M)=\Hom_{D(X^{\operatorname{alg},\dR})}(\pi_{\natural}1,M).\]
    Since $p_*\pi^*$ is conservative,
    we deduce that the object $\pi_{\natural}1$ is compact and is a generator for $D(X^{\operatorname{alg},\dR})$.
    Then by the recognition principle (cf. \cite[Theorem 7.1.2.1]{Lur17}),
    we obtain that
    \[
      \cat{Mod}_{p_*\pi^*\pi_{\natural}p^*}(D_{\blacksquare}(\rt_p))\cong \cat{LMod}_{\End_{D(X^{\operatorname{alg},\dR})}(\pi_{\natural}1)^{\op}}(D_{\blacksquare}(\rt_p)).
    \]
    This means that the counit map
    \[
    p_*\pi^*\pi_{\natural}p^*p_*\pi^*M\longrightarrow p_*\pi^*M
    \]
    is identified with 
    \[
    \End_{D(X^{\operatorname{alg},\dR})}(\pi_{\natural}1)^{\op}\otimes \Hom_{D(X^{\operatorname{alg},\dR})}(\pi_{\natural}1,M)\longrightarrow \Hom_{D(X^{\operatorname{alg},\dR})}(\pi_{\natural}1,M).
    \]
    In summary,
    we have equivalences of categories:
    \[
    p_*\pi^*\colon D(X^{\operatorname{alg},\dR})\cong \cat{Mod}_{p_*\pi^*\pi_{\natural}p^*}(D_{\blacksquare}(\rt_p))\cong \cat{LMod}_{\End_{D(X^{\operatorname{alg},\dR})}(\pi_{\natural}1)^{\op}}(D_{\blacksquare}(\rt_p)).
    \]
    However,
    $\End_{D(X^{\operatorname{alg},\dR})}(\pi_{\natural}1)^{\op}\cong p_*\pi^*\pi_{\natural}1$ is the algebra $D_A$ of differential operators over $A$,
    essentially from the computation in \cite[Theorem 10.17]{RC25a};
    it is also an identification of associative algebras,
    as the algebra structure on both sides is given by taking $M=\pi_{\natural}1$ in the counit map.
    Therefore,
    we have an equivalence of symmetric monoidal $\infty$-categories 
    \[
    p_*\pi^*\colon D(X^{\operatorname{alg},\dR})\cong \cat{LMod}_{D_A}(D_{\blacksquare}(\rt_p)).
    \]
\end{discussion}

\begin{remark}
  Another viewpoint of the functor $p_*\pi^*$ in \ref{algdR} is from the following computation:
  \[
  p_*\pi^*M\cong p^{\dR}_!\pi_!\pi^*M\cong p^{\dR}_!(\pi_!1\otimes M).
  \]
  This has a canonical left $\End_{D(X^{\operatorname{alg},\dR})}(\pi_!1)$-module structure.
  Moreover,
  we have an isomorphism 
  \[
    p_*\pi^*\pi_{\natural}1\cong p_!^{\dR}\pi_!\pi^*\pi_{\natural}1\cong p_!^{\dR}(\pi_!\pi^*1\otimes \pi_{\natural} 1)
    \cong p_!^{\dR}(\pi_!\pi^*1\otimes \pi_{!}\pi^! 1)\cong p_!^{\dR}\pi_{!}\pi^!\pi_! 1\cong p_*\pi^!\pi_!1\cong \End_{D(X^{\operatorname{alg},\dR})}(\pi_!1).
  \]
  The counit map 
  \[
    p_*\pi^*\pi_{\natural}p^*p_*\pi^*M\cong p_*\pi^*\pi_{\natural}1\otimes p_*\pi^*M\longrightarrow p_*\pi^*M
  \]
  is also identified with 
  \[
    \End_{D(X^{\operatorname{alg},\dR})}(\pi_!1)\otimes p^{\dR}_!(\pi_!1\otimes M)\longrightarrow p^{\dR}_!(\pi_!1\otimes M)
  \]
  by the recognition principle.
  Hence this gives another equivalence of categories:
  \[
  p_*\pi^*\colon D(X^{\operatorname{alg},\dR})\cong \cat{LMod}_{\End_{D(X^{\operatorname{alg},\dR})}(\pi_!1)}(D_{\blacksquare}(\rt_p)).
  \]
  Moreover,
  $\End_{D(X^{\operatorname{alg},\dR})}(\pi_!1)\cong p_*\pi^!\pi_!1$ is nothing but the algebra $D_{\omega_A}^{\op}$,
  where $D_{\omega_A}$ is the algebra of differential operators over the canonical sheaf $\omega_A:=\wedge^{\dim A}\Omega_A$,
  which is identified with $D_A^{\op}$ in classical theory of algebraic $D$-modules.
  Therefore,
  via this different description of $p_*\pi^*$,
  we obtain the same equivalence of symmetric monoidal $\infty$-categories 
  \[
    p_*\pi^*\colon D(X^{\operatorname{alg},\dR})\cong \cat{LMod}_{D^{\op}_{\omega_A}}(D_{\blacksquare}(\rt_p))\cong \cat{LMod}_{D_A}(D_{\blacksquare}(\rt_p)).
  \]
\end{remark}

\begin{remark}[Right $D$-modules]\label{rightDmod}
  We can take another realization $p_*\pi^!$ of the category $D(X^{\operatorname{alg},\dR})$.
  Indeed,
  since $\pi$ is suave with invertible suave dual,
  this is a twisted form of the realization $p_*\pi^*$,
  and we can compute 
  \[
    p_*\pi^!M=\Hom_{D(X^{\operatorname{alg},\dR})}(\pi_!1,M)=p^{\dR}_!(\pi_{\natural}1\otimes M).
  \]
  Then similarly by the recognition principle,
  this realization gives equivalences of categories 
  \[
    D(X^{\operatorname{alg},\dR})=\cat{Mod}_{p_*\pi^!\pi_!p^*}(D_{\blacksquare}(\rt_p))\cong 
    \cat{LMod}_{\End_{D(X^{\operatorname{alg},\dR})}(\pi_!1)^{\op}}(D_{\blacksquare}(\rt_p))\cong 
    \cat{LMod}_{\End_{D(X^{\operatorname{alg},\dR})}(\pi_{\natural}1)}(D_{\blacksquare}(\rt_p)).
  \]
  Since $\End_{D(X^{\operatorname{alg},\dR})}(\pi_!1)^{\op}=D_{\omega_A}\cong D_A^{\op}=\End_{D(X^{\operatorname{alg},\dR})}(\pi_{\natural}1)$,
  we obtain an equivalence of categories 
  \[
  p_*\pi^!\colon D(X^{\operatorname{alg},\dR})\cong \cat{RMod}_{D_A}(D_{\blacksquare}(\rt_p)).
  \]
\end{remark}

\begin{remark}
  The discussion above about the relation between algebraic de Rham stacks and algebraic $D$-modules is not specific to smooth dagger varieties.
  For example,
  we also have a similar statement for the type of formally smooth geometric objects,
  e.g. smooth schemes \cite{RC25a} and smooth solid Tate adic spaces \cite{RC24a}.
\end{remark}

Now we start the proof of Theorem \ref{Xarithcoefficients},
identifying $D(X^{\arith/K})$ with a category of some sheaves over differential operators.
We first establish the identification in the case of being an affine open subscheme of the special fiber of a proper smooth formal scheme.

\begin{proposition}\label{arithcoefficientaffinecase}
  Let $X$ be an affine realizable $k$-scheme and assume that we can choose a proper smooth frame $P$ such that $X$ embeds openly in $P_k$.
  Then $]X[_{P}^{\dagger}=\operatorname{GSpec}A$ is an affinoid smooth dagger space,
  and we have an equivalence of categories 
  \[
  D(X^{\arith/K})\cong \cat{LMod}_{D_{A}^{\dagger}}(D_{\blacksquare}(K)).
  \]
  Here we denote $A_r:=(A\otimes A)\langle \{\frac{x_i}{r}\}_{x_i\in I}\rangle$ viewed as an $A$-module via left action where $I=\ker(A\otimes A\overset{\Delta}{\rightarrow} A)$ with a choice of generators $\{x_i\}_{x_i\in I}$,
  and $D_A^{\dagger}:=\bigcup_{r<1}\Hom_{A}(A_r,A)$,
  which is the completion of the algebra of algebraic differential operators $D_A=\bigcup_n \Hom_A((A\otimes A)/I^n,A)$,
  consisting of those ``overconvergent'' differential operators.
\end{proposition}

\begin{proof}
  In the first step of the proof,
  we provide all the geometric setup.
  By Proposition \ref{Xarithdescription},
  we have a suave cover $\pi^{\arith/K}\colon]X[_{P}^{\dagger}\rightarrow X^{\arith/K}$ and a description of its Cech nerve.
  Note that there is also a natural surjection $\mu\colon ]X[_{P}^{\dagger,\operatorname{alg},\dR/K}\rightarrow X^{\arith/K}$ such that $\pi^{\arith/K}=\mu\circ\pi$,
  and filling the following diagram:
  \[\xymatrix{
    ]X[_{P}^{\dagger}\ar[r]^{\pi}\ar[dr]_p & ]X[_{P}^{\dagger,\operatorname{alg},\dR/K}\ar[d]^{p^{\dR}}\ar[r]^{\mu} & X^{\arith/K}\ar[dl]^{p^{\arith/K}}\\
     & \text{*} & 
  }\]
  Both $\pi$ and $\mu$ are suave covers:
  we have discussed $\pi$ in \ref{algdR};
  and that about $\mu$ is from the fact that $\pi^{\arith/K}$ is suave and $\pi$ is a suave cover.
  Moreover,
  the map $p$ is proper.

  Let us consider the functor $p_*\pi^*\mu^*=p_*\pi^{\arith/K,*}$.
  It is conservative and $p_*\pi^{\arith/K,*}$ preserves geometric realizations,
  so we can use Barr--Beck--Lurie theorem.
  By an argument similar to that in \ref{algdR},
  we have equivalences of categories:
  \[
    p_*\pi^*\mu^*\colon D(X^{\arith/K})\cong \cat{Mod}_{p_*\pi^{\arith/K,*}\pi^{\arith/K}_{\natural}p^*}(D_{\blacksquare}(K))\cong \cat{LMod}_{\End_{D(X^{\arith/K})}(\pi_{\natural}^{\arith/K}1)^{\op}}(D_{\blacksquare}(K)).
  \]
  Denote $\End_{D(X^{\arith/K})}(\pi_{\natural}^{\arith/K}1)^{\op}\cong p_*\pi^{\arith/K,*}\pi^{\arith/K}_{\natural}p^*1$ as $D_A^?$.
  It suffices to identify $D_A^?$ with $D_A^{\dagger}$ as an associative solid $K$-algebra.

  First we identify $D_A^?$ with $D_A^{\dagger}$ as solid $K$-modules.
  We have a Cartesian diagram 
  \[\xymatrix{
    \Delta^{\circ}(]X[_{P}^{\dagger})\ar[r]^q \ar[d]_r & ]X[_{P}^{\dagger}\ar[d]^{\pi^{\arith/K}}\\
    ]X[_{P}^{\dagger}\ar[r]^{\pi^{\arith/K}} & X^{\arith/K}
  }\]
  Hence $D_A^?\cong p_*\pi^{\arith/K,*}\pi^{\arith/K}_{\natural}p^*1\cong p_{*}q_{\natural}1$.
  Notice that $q_{\natural}1$ is a basic nuclear object in $D(]X[^{\dagger}_P)$,
  cf. \cite[Lemma 5.5.4, Lemma 5.5.5]{ABLBRCS25}.
  As $]X[^{\dagger}_P=\operatorname{GSpec}A$ with $A$ being dual nuclear Fréchet over $K$,
  we obtain that $p_*q_{\natural}1$ is also basic nuclear hence dual nuclear Fréchet in $D_{\blacksquare}(K)$,
  cf. \cite[Lemma A.0.12]{ABLBRCS25}.
  The dual $(p_*q_{\natural}1)^{\vee}\cong p_*q_*q^*p^!1$ is a nuclear Fréchet space,
  and by writing $\Delta^{\circ}(]X[_{P}^{\dagger})$ as $\Delta^{\circ}(]X[_{P}^{\dagger})=\varinjlim_{r<1}\Delta^r(]X[_{P}^{\dagger})$,
  where $\Delta^r(]X[_{P}^{\dagger}):=\operatorname{GSpec}A_r$ and $q_r\colon \Delta^r(]X[_{P}^{\dagger})\rightarrow ]X[_{P}^{\dagger}$,
  we can write
  \[p_*q_*q^*p^!1\cong \varprojlim_{r<1} p_*q_{r,*}q_{r}^*p^!1=\varprojlim_{r<1}p_*(A_r\otimes p^!1).\]
  By the duality between dual nuclear Fréchet spaces and nuclear Fréchet spaces in \cite[Section 3.6]{RCRJ22},
  we know that
  \[p_*q_{\natural}1\cong(p_*q_*q^*p^!1)^{\vee}\cong \varinjlim_{r<1}\Hom(A_r,\underline{\Hom}(p^!1,p^!1))\cong \varinjlim_{r<1}p_*\underline{\Hom}(A_r,A)\cong D_A^{\dagger}.\]
  Therefore,
  $p_*q_{\natural}1 \cong D_A^{\dagger}$ as solid $K$-modules.

  Then note that the functor $\mu_{\natural}$ induces a homomorphism of algebras
  \[ f\colon D_A\cong \End_{D(]X[_P^{\dagger,\operatorname{alg},\dR/K})}(\pi_{\natural}1)^{\op}\longrightarrow  D_A^?:=\End_{D(X^{\arith/K})}(\pi_{\natural}^{\arith/K}1)^{\op}.\]
  We claim this functor is compatible with the identification of $D^{?}_A$ and $D^{\dagger}_A$ as solid $K$-modules before,
  i.e.
  the following diagram holds as solid $K$-modules:
  \[\xymatrix{
    D_A\ar[r]^{f}\ar[rd]^g & D^{?}_A\ar[d]^{\cong,\text{ as objects in } D_{\blacksquare}(K)}\\
     & D^{\dagger}_A
  }\]
  Completing the Cartesian diagram as a transfer of two Cartesian diagrams:
  \[\xymatrix{
    \widehat{\Delta}(]X[_{P}^{\dagger})\ar[rd]^{i}\ar[rdd]\ar[rrd]^{r} & & &\\
    & \Delta^{\circ}(]X[_{P}^{\dagger})\ar[r]^q \ar[d] & ]X[_{P}^{\dagger}\ar[ddr]^{\pi}\ar[d]^{\pi^{\arith/K}} & \\
    & ]X[_{P}^{\dagger}\ar[rrd]_{\pi}\ar[r]^{\pi^{\arith/K}} & X^{\arith/K} & \\
    & & & ]X[_{P}^{\dagger,\operatorname{alg},\dR/K}\ar[ul]_{\mu}
  }\]
  then the map $f$ is identified with the unit map for $\mu_{\natural}$,
  which satisfies the following commutative diagram 
  \[
  \xymatrix{
    D_A=p_*\pi^*\pi_{\natural}1\ar[d]_{\cong}\ar[r]^{\operatorname{unit}_{\mu_{\natural}}} & D_A^?=p_*\pi^{\arith/K,*}\pi^{\arith/K}_{\natural}1\ar[d]^{\cong}\\
    D_A=p_*r_{\natural}1\ar[r]^{\operatorname{counit}_{i_{\natural}}} & D_A^{\dagger}=p_* q_{\natural}1
  }\]
  However,
  the counit map for $i_{\natural}$ is exactly the map $g$ by definition,
  which proves the commutativity of the diagram concerning $f$ and $g$ above.
  Hence $f$ shares the same properties with $g$ that it is injective with dense image.
  Moreover,
  we have the observation that $D_A^?$ is static (because $D_A^{\dagger}$ is),
  hence the algebra structure on $D_A^?$ is uniquely determined by that on $D_A$ by the following reason.
  Since $\im(f)$ is dense in $D_A^?$,
  then any element $x\in\underline{D_A^?(*)}$ is a sequential limit $x_n\rightarrow x$ for $x_n\in \underline{D_A(*)}$.
  Then by continuity of multiplication (coming from the fact that they are \textit{condensed} algebras),
  we have $x\cdot_{D_A^?} y=(\lim x_n)\cdot_{D_A^?}(\lim y_m)=\lim x_n\cdot_{D_A} y_m$ is uniquely determined by the multiplication coming from $D_A$.
  However,
  this is also the multiplication on $D_A^{\dagger}$,
  because that is how the algebra structure on $D^{\dagger}_A$ is defined.
  In summary,
  this proves $D_A^?\cong D_A^{\dagger}$ as solid $K$-algebras.
\end{proof}

\begin{remark}[Realization to algebraic $D$-modules]\label{algdRfullfaithful}
  The morphism $\mu$ admits the following strong property that $\mu^*\colon D(X^{\arith/K})\rightarrow D(]X[_{P}^{\dagger,\operatorname{alg},\dR/K})$ is fully faithful.
  Indeed,
  let $\mu_{\natural}$ be the left adjoint of $\mu^*$,
  then it suffices to prove that the counit map $\mu_{\natural}1\rightarrow 1$ is an isomorphism:
  if that holds,
  then by projection formula for $\mu_{\natural}$ (in fact $\mu_{\natural}\mu^*(-)=\mu_!\mu^!(-)=\mu_!\mu^!1\otimes(-)$)
  \[
  \Hom(\mu^*M,\mu^*N)\cong \Hom(\mu_{\natural}\mu^*M,N)\cong \Hom(\mu_{\natural}1\otimes M,N)\cong \Hom(M,N).
  \]
  To prove that $\mu_{\natural}1\rightarrow 1$ is an isomorphism,
  we can first reduce to the case of $X=\aff^1$,
  by étale descent of both sides.
  For $X=\aff^1$,
  the morphism $\mu$ is given by $\mu\colon \disk^{\dagger}/\widehat{\Ga}\rightarrow\disk^{\dagger}/\disk^{\circ}$.
  We have the following Cartesian diagram,
  by using a trivialization of the diagonal embedding $\Delta\colon \disk^{\dagger}\rightarrow \disk^{\dagger}\times \disk^{\dagger}$:
  \[\xymatrix{
    \disk^{\dagger}\times\disk^{\circ}/\widehat{\Ga}\ar[r]^{h}\ar[d] & \disk^{\dagger}\ar[d]^{\pi}\\
    \disk^{\dagger}/\widehat{\Ga}\ar[r]^{\mu} & \disk^{\dagger}/\disk^{\circ}
  }\]
  To prove $\mu_{\natural}1\cong1$,
  it suffices to prove $\pi^*\mu_{\natural}1\cong 1$ as $\pi$ is a suave cover,
  and by base change formula we have $\pi^*\mu_{\natural}1\cong h_{\natural}1$,
  so it suffices to prove $h_{\natural}1\cong 1$,
  equivalently,
  to prove $g_{\natural}1\cong 1$ for $g\colon \disk^{\circ}/\widehat{\Ga}\rightarrow *$.
  Since $\disk^{\circ}/\widehat{\Ga}$ is the algebraic de Rham stack $\disk^{\circ,\operatorname{alg},\dR/K}$ of $\disk^{\circ}$,
  we have $1\cong g_*1$ as $g_*1$ computes the (algebraic) de Rham cohomology of $\disk^{\circ}$ which is isomorphic to $\rt_p$ by the Poincaré lemma.
  Moreover,
  the dual of $g_{\natural}1$ is computed as $(g_{\natural}1)^{\vee}\cong \Hom(g_{\natural}1,1)\cong g_*g^*1\cong g_*1\cong 1$.
  Since $g_{\natural}1$ itself is dual nuclear Fréchet,
  we also have $g_{\natural}1\cong 1^{\vee}\cong 1$ by the duality between dual nuclear Fréchet spaces and nuclear Fréchet spaces in \cite[Section 3.6]{RCRJ22}.
  This proves the full-faithfulness of $\mu^*$.
  Moreover,
  the adjoint pair of functors 
  \[
  \mu_{\natural}\dashv \mu^*\colon \cat{LMod}_{D_A}(D_{\blacksquare}(K))\leftrightarrows \cat{LMod}_{D_A^{\dagger}}(D_{\blacksquare}(K)).
  \]
  has the following description:
  \begin{enumerate}
    \item For a left $D_A$-module $N$,
    $\mu_{\natural}N=D_A^{\dagger}\otimes_{D_A}N$ is the left $D_A^{\dagger}$-module by tensoring $N$ with $D_A^{\dagger}$ where $D_A^{\dagger}$ is viewed as a right $D_A$-module via $f$.
    \item For a left $D_A^{\dagger}$-module $M$,
    $\mu^*M=M$ is the left $D_A$-module with the $D_A$-module structure induced via $f\colon D_A\rightarrow D_A^{\dagger}$. 
  \end{enumerate}
  In summary,
  the slogan is that we have a fully faithful realization from arithmetic $D$-modules over $X$ to algebraic $D$-modules of the overconvergent generic fiber of its formal lift.
\end{remark}

\begin{remark}[Realization to analytic $D$-modules]\label{fullfaithfulanalyticdeRham}
  Similar to Remark \ref{algdRfullfaithful},
  the category of arithmetic $D$-modules also admits a fully faithful realization to the category of \emph{analytic $D$-modules} in the sense of \cite{RCRJ}.
  Let $\nu\colon ]X[_{P}^{\dagger,\dR/K}\rightarrow X^{\arith/K}$ be the natural map,
  then the pullback functor $\nu^*\colon D(X^{\arith/K})\rightarrow D(]X[_{P}^{\dagger,\operatorname{alg},\dR/K})$ is fully faithful.
  The proof is an analogue to the proof in Remark \ref{algdRfullfaithful}:
  \begin{enumerate}
    \item First the map $\nu$ is suave,
    as $\nu$ is a cover and the map from the first Cech nerve of $\nu$ is given by $]X[_{P\times P}^{\dagger,\dR}\rightarrow ]X[_{P}^{\dagger,\dR}$,
    which is suave as $]X[_{P\times P}^{\dagger}\rightarrow ]X[_{P}^{\dagger}$ is suave.
    \item Then it suffices to prove that $\nu_{\natural}1\rightarrow 1$ is an isomorphism.
    Then this comes from a similar computation in Remark \ref{algdRfullfaithful} identifying both as the de Rham cohomology of $]X[_{P}^{\dagger}$,
    cf. \cite[Proposition 5.2.1]{ABLBRCS25}.
  \end{enumerate}
  Moreover,
  this is compatible with the fully faithful embedding of analytic $D$-modules to algebraic $D$-modules,
  cf. \cite[Proposition 5.2.11]{RC24a}.
\end{remark}

\begin{remark}\label{rightarithDmod}
  Similar to Remark \ref{rightDmod},
  a different realization functor $p_*\pi^{\arith/K,!}$ gives another description via right $D$-modules:
  \[
    p_*\pi^{\arith/K,!}\colon D(X^{\arith/K})\cong \cat{RMod}_{D_A^{\dagger}}(D_{\blacksquare}(K)).
  \]
\end{remark}

With the case of being affine open in the special fiber of a proper smooth formal scheme established,
the case of being the special fiber of a proper smooth formal scheme is by descent from this special case.

\begin{proposition}\label{propersmootharithcoefficients}
  Let $\Xf$ be a proper smooth formal scheme and $X$ be its special fiber.
  Then we have a natural equivalence of categories 
  \[
  D(X^{\arith/K})\cong \cat{LMod}_{\D^{\dagger}_{\X}}(D(\X))\cong D_{\blacksquare}(\D^{\dagger}_{\Xf,\rt}).
  \]
\end{proposition}
\begin{proof}
  The category $\cat{LMod}_{\D^{\dagger}_{\X}}(D(\X))$ localizes over the descendable topology on $\Xf_{\eta}$.
  This means that the association taking any affine open $U\subset X$ to the category $\cat{LMod}_{\D^{\dagger}_{]U[_{\Xf}^{\dagger}}}(D(]U[_{\Xf}^{\dagger}))$
  defines a sheaf of $\infty$-categories on $X$ and
  \[
    \cat{LMod}_{\D^{\dagger}_{\X}}(D(\X))=\varprojlim_{\text{affine open }U\subset X}\cat{LMod}_{\D^{\dagger}_{]U[_{\Xf}^{\dagger}}}(D(]U[_{\Xf}^{\dagger})).
  \]
  Let $D^{\dagger}_{\O(]U[^{\dagger}_{\Xf})}$ be the algebra $D_A^{\dagger}$ in Proposition \ref{arithcoefficientaffinecase},
  then we have an equivalence of categories 
  \[
    \cat{LMod}_{\D^{\dagger}_{]U[_{\Xf}^{\dagger}}}(D(]U[_{\Xf}^{\dagger}))\cong \cat{LMod}_{D^{\dagger}_{\O(]U[^{\dagger}_{\Xf})}}(D_{\blacksquare}(K))
  \]
  by affineness.
  Moreover,
  the localization functors also translate to the functor $\O(]U[^{\dagger}_{\Xf})\otimes_{\O(]V[^{\dagger}_{\Xf})}(-)$,
  i.e. if $U\subset V$ are open affine subschemes of $X$,
  then the following diagram commutes:
  \[\xymatrix{
    \cat{LMod}_{\D^{\dagger}_{]V[_{\Xf}^{\dagger}}}(D(]V[_{\Xf}^{\dagger}))\ar[r]^{\cong}\ar[d]_{L_{U\subset V}} & \cat{LMod}_{D^{\dagger}_{\O(]V[^{\dagger}_{\Xf})}}(D_{\blacksquare}(K))
    \ar[d]^{\O(]U[^{\dagger}_{\Xf})\otimes_{\O(]V[^{\dagger}_{\Xf})}(-)}\\
    \cat{LMod}_{\D^{\dagger}_{]U[_{\Xf}^{\dagger}}}(D(]U[_{\Xf}^{\dagger}))\ar[r]^{\cong} & \cat{LMod}_{D^{\dagger}_{\O(]U[^{\dagger}_{\Xf})}}(D_{\blacksquare}(K))
  }\]
  To prove this claim,
  it suffices to notice the following key observation:
  \begin{itemize}
    \item If $U\subset V$ are open affine subschemes of $X$,
    then $D^{\dagger}_{\O(]U[^{\dagger}_{\Xf})}\cong \O(]U[^{\dagger}_{\Xf})\otimes_{\O(]V[^{\dagger}_{\Xf})}D^{\dagger}_{\O(]V[^{\dagger}_{\Xf})}$.
  \end{itemize}
  To explain,
  we consider the Cartesian diagram from Proposition \ref{Xarithdescription}:
  \[\xymatrix{
    ]U[^{\dagger}_{\Xf}\ar[r]^{\iota}\ar[d]_{\pi_U} & ]V[^{\dagger}_{\Xf}\ar[d]^{\pi_V}\\
    U^{\arith/K}\ar[r]^i & V^{\arith/K}
  }\]
  Then $i$ and $\iota$ are closed embedding by excision property in Theorem \ref{arithsixfunctor} $(3)$.
  By base change formula,
  we have $\iota^*\pi_V^*\pi_{V,\natural}1\cong \pi_U^*i^*\pi_{V,\natural}1\cong \pi_U^*\pi_{U,\natural}\iota^*1$.
  From the proof of Proposition \ref{arithcoefficientaffinecase},
  we know that $\pi_U^*\pi_{U,\natural}1\cong D^{\dagger}_{\O(]U[^{\dagger}_{\Xf})}$
  and $\pi_V^*\pi_{V,\natural}1\cong D^{\dagger}_{\O(]V[^{\dagger}_{\Xf})}$,
  hence the isomorphism $\iota^*\pi_V^*\pi_{V,\natural}1\cong \pi_U^*\pi_{U,\natural}\iota^*1$ translates to 
  $D^{\dagger}_{\O(]U[^{\dagger}_{\Xf})}\cong \O(]U[^{\dagger}_{\Xf})\otimes_{\O(]V[^{\dagger}_{\Xf})}D^{\dagger}_{\O(]V[^{\dagger}_{\Xf})}$.

  Moreover,
  under the equivalences in Proposition \ref{arithcoefficientaffinecase},
  the following diagram also commutes 
  \[\xymatrix{
     \cat{LMod}_{D^{\dagger}_{\O(]V[^{\dagger}_{\Xf})}}(D_{\blacksquare}(K))\ar[r]^{\cong}\ar[d]_{\O(]U[^{\dagger}_{\Xf})\otimes_{\O(]V[^{\dagger}_{\Xf})}(-)}
     & D(V^{\arith/K})\ar[d]^{i^*}\\
    \cat{LMod}_{D^{\dagger}_{\O(]U[^{\dagger}_{\Xf})}}(D_{\blacksquare}(K))\ar[r]^{\cong} & D(U^{\arith/K})
  }\]
  In fact,
  that is how the equivalences in Proposition \ref{arithcoefficientaffinecase} are defined:
  The morphisms $\pi_U,\pi_V$ give conservative functors of categories
  $\pi^*_V\colon D(V^{\arith/K})\rightarrow D(]V[^{\dagger}_{\Xf})$,
  and $\pi^*_U\colon D(U^{\arith/K})\rightarrow D(]U[^{\dagger}_{\Xf})$,
  which translate to forgetful functors $\cat{LMod}_{D^{\dagger}_{\O(]V[^{\dagger}_{\Xf})}}(D_{\blacksquare}(K))\rightarrow \cat{Mod}_{\O(]V[^{\dagger}_{\Xf})}(D_{\blacksquare}(K))$
  and $\cat{LMod}_{D^{\dagger}_{\O(]U[^{\dagger}_{\Xf})}}(D_{\blacksquare}(K))\rightarrow \cat{Mod}_{\O(]U[^{\dagger}_{\Xf})}(D_{\blacksquare}(K))$.
  Moreover,
  $i^*$ translates to $\iota^*$ under these conservative functors,
  which is given as $\iota^*(-)=\O(]U[^{\dagger}_{\Xf})\otimes_{\O(]V[^{\dagger}_{\Xf})}(-)$.

  By $h$-descent of the arithmetic de Rham stack (Theorem \ref{fpqcdescentarith}),
  the category $D(X^{\arith/K})$ also localizes on the Zariski topology over $X$ along $*$-pullbacks.
  By the compatibilities of transition maps in the descent data above,
  we then obtain that 
  \[
    \cat{LMod}_{\D^{\dagger}_{\X}}(D(\X))=\varprojlim_{\text{affine open }U\subset X}\cat{LMod}_{\D^{\dagger}_{]U[_{\Xf}^{\dagger}}}(D(]U[_{\Xf}^{\dagger}))
    \cong \varprojlim_{\text{affine open }U\subset X}D(U^{\arith/K})=D(X^{\arith/K}).\qedhere
  \]
\end{proof}

Now we establish $(3)$ of Theorem \ref{Xarithcoefficients}.

\begin{proposition}\label{Xarithsixfunctorcalssical}
  Let $f\colon \Xf\rightarrow \mathfrak{Y}$ be a morphism between proper smooth formal schemes,
  and $X\rightarrow Y$ the induced morphism between their special fibers.
  Then the induced morphism $f^{\arith/K}\colon X^{\arith/K}\rightarrow Y^{\arith/K}$ identifies $f^{\arith/K,*}$ with $f^{\triangle}$,
  and $f_!^{\arith/K}$ with $f_+[\dim X-\dim Y]$,
  under the equivalence in Proposition \ref{propersmootharithcoefficients}.
\end{proposition}
\begin{proof}
  Firstly we prove the identification of $f^{\arith/K,*}$ with $f^{\triangle}$.
  Both functors localize over the Zariski topology on the target and the source,
  hence it suffices to deal with the following special case:
  $U$ is an affine open subscheme of $X$ mapping to $V$,
  which is an affine open subscheme of $Y$.
  In this case,
  the overconvergent tubular neighborhoods are $\dagger$-affinoid.
  Denote $g\colon ]U[_{\Xf}^{\dagger}\rightarrow ]V[_{\mathfrak{Y}}^{\dagger}$.

  By the description from Proposition \ref{Xarithdescription},
  we have the following commutative diagram 
  \[\xymatrix{
    ]U[_{\Xf}^{\dagger}\ar[r]^g\ar[d]_{\pi_U} & ]V[_{\mathfrak{Y}}^{\dagger}\ar[d]^{\pi_V}\\
  U^{\arith/K}\ar[r]^{f^{\arith/K}} & V^{\arith/K}
  }\]
  with $\pi_U$ and $\pi_V$ being suave covers.
  Moreover,
  by Proposition \ref{arithcoefficientaffinecase},
  the functor $p_*\pi^*_U$ (resp. $q_*\pi^*_V$) is identified with the forgetful functor 
  $\cat{LMod}_{D^{\dagger}_{\O(]U[^{\dagger}_{\Xf})}}(D_{\blacksquare}(\rt_p))\rightarrow D_{\blacksquare}(\rt_p)$
  (resp. $\cat{LMod}_{D^{\dagger}_{\O(]V[^{\dagger}_{\mathfrak{Y}})}}(D_{\blacksquare}(\rt_p))\rightarrow D_{\blacksquare}(\rt_p)$).
  Let $M\in D(V^{\arith/K})\cong \cat{LMod}_{D^{\dagger}_{\O(]V[^{\dagger}_{\mathfrak{Y}})}}(D_{\blacksquare}(\rt_p))$.
  Then by projection formula
  \[p_*\pi_U^*f^{\arith/K,*}N=p_*g^*\pi^*_VN=q_*g_*g^*\pi^*_VN=q_*(g_*1\otimes \pi^*_VN).\]
  This identifies with $f^{\Delta}(q_*\pi_V^*N)$ as a solid $\rt_p$-module (even as a solid $\O(]U[^{\dagger}_{\Xf})$-module),
  and it is enough to identify the structure of $D^{\dagger}_{\O(]U[^{\dagger}_{\Xf})}$-module.
  The structure of $D^{\dagger}_{\O(]U[^{\dagger}_{\Xf})}$-module on $p_*\pi_U^*f^{\arith/K,*}N$ is identified with 
  the action of $\End(\pi_{U,\natural}1)^{\op}$ on $p_*\pi_U^*f^{\arith/K,*}N=\Hom(\pi_{U,\natural}1,f^{\arith/K,*}N)$.
  Notice that the identification as solid $\rt_p$-modules above is reinterpreted to the following natural isomorphism as solid $\rt_p$-modules:
  \[
    \Hom(\pi_{U,\natural}1,f^{\arith/K,*}\pi_{V,\natural}1)\otimes_{\End(\pi_{V,\natural}1)}\Hom(\pi_{V,\natural}1,N)\overset{h,\cong}{\longrightarrow}
    \Hom(\pi_{U,\natural}1,f^{\arith/K,*}N).
  \]
  Moreover,
  the $\End(\pi_{U,\natural}1)^{\op}$-action on $\Hom(\pi_{U,\natural}1,f^{\arith/K,*}N)$ is identified with the $\End(\pi_{U,\natural}1)^{\op}$-action on 
  $\Hom(\pi_{U,\natural}1,f^{\arith/K,*}\pi_{V,\natural}1)$ via $h$,
  i.e. the morphism $h$ also induces an isomorphism as $D^{\dagger}_{\O(]U[^{\dagger}_{\Xf})}$-modules.
  However,
  notice that 
  $\Hom(\pi_{U,\natural}1,f^{\arith/K,*}\pi_{V,\natural}1)\cong \O(]V[^{\dagger}_{\mathfrak{Y}})\otimes_{\O(]U[^{\dagger}_{\Xf})} D^{\dagger}_{\O(]U[^{\dagger}_{\Xf})}$ as $D_{\O(]V[^{\dagger}_{\mathfrak{Y}})}^{\dagger}$-modules,
  therefore,
  by the definition of naïve pullback functor,
  $\Hom(\pi_{U,\natural}1,f^{\arith/K,*}\pi_{V,\natural}1)\otimes_{\End(\pi_{V,\natural}1)}\Hom(\pi_{V,\natural}1,N)$ is identified with $f^{\Delta}(q_*\pi_V^*N)$ as $D_{\O(]V[^{\dagger}_{\mathfrak{Y}})}^{\dagger}$-modules.
  This identifies $f^{\arith/K,*}$ with $f^{\triangle}$.

  Next we prove the identification of $f_!^{\arith/K}$ with $f_+[\dim X-\dim Y]$.
  Again,
  both functors localize over the Zariski topology on the target and the source,
  hence it suffices to deal with the following special case:
  $U$ is an affine open subscheme of $X$ mapping to $V$,
  which is an affine open subscheme of $Y$,
  hence with $\dagger$-affinoid overconvergent tubular neighborhoods.
  Similarly,
  we have the following commutative diagram 
  \[\xymatrix{
    ]U[_{\Xf}^{\dagger}\ar[r]^g\ar[d]_{\pi_U} & ]V[_{\mathfrak{Y}}^{\dagger}\ar[d]^{\pi_V} \\
  U^{\arith/K}\ar[r]^{f^{\arith/K}} & V^{\arith/K}
  }\]
  Let $M\in D(]U[^{\dagger}_{\Xf})$,
  then we have 
  \[
  f_{!}^{\arith/K}\pi_{U,!}M\cong \pi_{V,!}g_!M\cong \pi_{V,!}g_*M.
  \]
  The $\pi_{U,!}M$ corresponds to the object $p_*\pi_U^!\pi_{U,!}M$ in $\cat{RMod}_{D^{\dagger}_{\O(]U[^{\dagger}_{\Xf})}}(D_{\blacksquare}(\rt_p))$ following Remark \ref{rightarithDmod}.
  One has
  \[
  p_*\pi_U^!\pi_{U,!}M\cong p_!^{\arith/K}\pi_{U,!}\pi_U^!\pi_{U,!}M\cong p_!^{\arith/K}(\pi_{U,!}M\otimes \pi_{U,!}\pi_U^!1)\cong p_*(\pi_U^*\pi_{U,\natural}1\otimes M),
  \]
  hence it corresponds to the right $D^{\dagger}_{\O(]U[^{\dagger}_{\Xf})}$-module $M\otimes_{\O(]U[^{\dagger}_{\Xf})} D^{\dagger}_{\O(]U[^{\dagger}_{\Xf})}$.
  Similarly,
  the $\pi_{V,!}g_*M$ corresponds to the right $D^{\dagger}_{\O(]V[^{\dagger}_{\mathfrak{Y}})}$-module $M\otimes_{\O(]V[^{\dagger}_{\mathfrak{Y}})} D^{\dagger}_{\O(]V[^{\dagger}_{\mathfrak{Y}})}$.
  By the description of the extraordinary pushforward on right $D$-modules in Remark \ref{extraordinaryrightDmod},
  one gets that 
  \[
  f_+(M\otimes_{\O(]U[^{\dagger}_{\Xf})} D^{\dagger}_{\O(]U[^{\dagger}_{\Xf})})=M\otimes_{\O(]V[^{\dagger}_{\mathfrak{Y}})} D^{\dagger}_{\O(]V[^{\dagger}_{\mathfrak{Y}})}[\dim Y-\dim X].
  \]
  Combining with the former property that $f_{!}^{\arith/K}\pi_{U,!}M\cong \pi_{V,!}g_*M$,
  we obtain the identification between $f_!^{\arith/K}$ and $f_+[\dim X-\dim Y]$ for those objects of the form $\pi_{U,!}M$ for some $M\in D(]U[^{\dagger}_{\Xf})$.
  In general,
  by constructions,
  both functors preserve colimits,
  and $D(U^{\arith/K})$ is generated under colimits by objects $\pi_{U,!}M$ where $M\in D(]U[^{\dagger}_{\Xf})$
  which proves the identification between $f_!^{\arith/K}$ and $f_+[\dim X-\dim Y]$ for all objects.
\end{proof}

Now we can state the proof of Theorem \ref{Xarithcoefficients}.

\begin{proof}[Proof of Theorem \ref{Xarithcoefficients}]
  The part $(2)$ is already in Proposition \ref{propersmootharithcoefficients} execpt the claim about symmetric monoidal structures.
  To prove $(1)$,
  we notice the following Cartesian diagram from Proposition \ref{Xarithdescription}:
  \[\xymatrix{
  ]X[^{\dagger}_P\ar[r]\ar[d] & P_{\eta}\ar[d]\\
  X^{\arith/K}\ar[r] & P_k^{\arith/K}
  }\]
  By Proposition \ref{propersmootharithcoefficients},
  we have $D(P_k^{\arith/K})\cong \cat{LMod}_{\D^{\dagger}_{P_{\eta}}}(D(P_{\eta}))$,
  where the isomorphism is induced via pullback along $P_{\eta}\rightarrow P_k^{\arith/K}$.
  Then we can conclude that $D(X^{\arith/K})$ classifies objects in $\cat{LMod}_{\D^{\dagger}_{P_{\eta}}}(D(P_{\eta}))$ which localize over $]X[^{\dagger}_P$.
  Since $\D^{\dagger}_{P_{\eta}}|_{]X[^{\dagger}_P}=\D_{]X[^{\dagger}_P}^{\dagger}$,
  we obtain $D(X^{\arith/K})\cong \cat{LMod}_{\D^{\dagger}_{]X[^{\dagger}_P}}(D(]X[^{\dagger}_P))$.
  This proves part $(1)$.
  Part $(3)$ is already established in Proposition \ref{Xarithsixfunctorcalssical},
  and it only remains to identify the symmetric monoidal structures in part $(2)$.
  Then this comes from applying $(3)$ to the diagonal embedding $\Delta\colon \Xf\rightarrow \Xf\times\Xf$,
  which identifies $\Delta^{\arith/K,*}$ with $\Delta^{\triangle}$,
  as the tensor products on both sides are given by $\Delta^{\arith/K,*}(-\boxtimes -)$ or $\Delta^{\triangle}(-\boxtimes -)$.
\end{proof}

Finally,
in the section,
we study arithmetic $D$-modules for non-realizable schemes.

\begin{definition}[Solid arithmetic $D$-modules]
  Let $X$ be a $k$-scheme.
  We define \[D_{\blacksquare}(\D^{\dagger}_{X/K}):=D(X^{\arith/K})\]
  as the category of solid arithmetic $D$-modules over $X$.
  We also denote the three operations in the $6$-functor formalism $X\mapsto D_{\blacksquare}(\D^{\dagger}_{X/K})$ as
  $(f^{\triangle},f_{\triangle},\otimes)$,
  which reinterpret $(f^{\arith/K,*},f^{\arith/K}_!,\otimes)$ on $D(X^{\arith/K})$.
  If both $X$ and $Y$ are smooth,
  then we also denote $f_+:=f_{\triangle}[\dim Y-\dim X]$ and $f^{!,\Be}:=f^{\triangle}[\dim X-\dim Y]$.
  The functors $f^{\triangle}$, $f_+$ and $f^{!,\Be}$ are identified with those for $D_{\blacksquare}(\D^{\dagger}_{\Xf,\rt})$ in the liftable case,
  so there is no abuse of terminology.
\end{definition}

\begin{remark}
  Since every $X$ is realizable Zariski locally,
  where we choose such a cover $(U_i,\overline{U_i},P_i)$,
  from Theorem \ref{Xarithcoefficients},
  the category $D_{\blacksquare}(\D^{\dagger}_{X/K})$ is the glueing of the category $\cat{LMod}_{\D^{\dagger}_{]U_i[^{\dagger}_{P_i}}}(D(]U_i[^{\dagger}_{P_i}))$ along this Zariski cover,
  where the terminology \emph{solid arithmetic $D$-modules over $X$} comes from.
\end{remark}

\begin{corollary}\label{f_*smooth}
  Let $f\colon X\rightarrow Y$ be a smooth morphism between smooth schemes over $k$.
  Then the induced morphism $f\colon X^{\arith/K}\rightarrow Y^{\arith/K}$ identifies $f_*^{\arith/K}$ with $f_+[\dim Y-\dim X]$.

  In particular,
  $(f^{!,\Be},f_+)$ becomes an adjoint pair in $D_{\blacksquare}(\D^{\dagger}_{X/K})$.
\end{corollary}
\begin{proof}
  Since $f$ is smooth,
  by Theorem \ref{arithsixfunctor},
  $f^{\arith/K}$ is prim with priml dual $1[-2(\dim X-\dim Y)]$,
  i.e. $f_*^{\arith/K}\cong f_!^{\arith/K}[-2(\dim X-\dim Y)]$.
  Since $f_!^{\arith/K}$ is identified with $f_+[\dim X-\dim Y]$,
  we have an identification between $f_*^{\arith/K}$ with $f_+[\dim Y-\dim X]$.
  In particular,
  $(f^{!,\Be},f_+)$ is identified with $(f^{\arith/K,*}[\dim X-\dim Y],f_*^{\arith/K}[\dim Y-\dim X],)$ hence becomes an adjoint pair.
\end{proof}

\begin{remark}
  Proposition \ref{Xarithsixfunctorcalssical} and Corollary \ref{f_*smooth} recover adjunction formulae\footnote{In the classical setting,
  people use a de Rham resolution to prove smooth adjunction formula,
  and use dual functor to prove proper adjunction formula (only for holonomic objects).
  Our approach is more abstract and general.} in solid arithmetic $D$-modules:
  \begin{enumerate}
    \item Let $f\colon X\rightarrow Y$ be a proper morphism between realizable schemes.
    By Theorem \ref{arithsixfunctor},
    $f$ is cohomologically étale on the arithmetic de Rham stacks.
    Therefore $(f_!^{\arith/K},f^{\arith/K,*})$ is an adjoint pair.
    Equivalently,
    $(f_+[\dim X-\dim Y],f^{\Delta})$ is an adjoint pair,
    which recovers the proper adjunction formula.
    \item Let $f\colon X\rightarrow Y$ be a smooth morphism between realizable schemes.
    Then $(f^{!,\Be}=f^+[\dim X-\dim Y],f_+)$ is an adjoint pair,
    which recovers smooth adjunction formula.
  \end{enumerate}
\end{remark}

\begin{remark}[How do people study $D$-modules before stacky approach?]
  When $X$ is a singular variety over a field of characteristic $0$ (resp. over $k$),
  then the category of algebraic (resp. arithmetic) $D$-modules over $X$ and the six operations are usually very hard to establish.
  Historically,
  people use an embedding to smooth varieties to resolve this obstruction,
  and this is Kashiwara's equivalence.
  Let us review this procedure for (solid) arithmetic $D$-modules in this remark.

  \begin{enumerate}
    \item (Local cohomology)
    Let $P$ be a proper smooth formal scheme and $Z$ be a closed subscheme of $P_k$.
    We can define the functors $R\underline{\Gamma}_{Z}^{\dagger}(-),\,({}^{\dagger}Z)\colon D_{\blacksquare}(\D^{\dagger}_{P,\rt})\rightarrow D_{\blacksquare}(\D^{\dagger}_{P,\rt})$ sitting in a triangle
    \[
    R\underline{\Gamma}_{Z}^{\dagger}(-)\longrightarrow\id\longrightarrow ({}^{\dagger}Z).
    \]
    First let $Z$ be a divisor,
    then it suffices to define $({}^{\dagger}Z)$ as $R\underline{\Gamma}_{Z}^{\dagger}(-)$ can be defined via the triangle.
    We define \[({}^{\dagger}Z)(-):=\D^{\dagger}_{P,\rt}({}^{\dagger}Z)\otimes_{\D^{\dagger}_{P,\rt},\blacksquare}(-)\colon D_{\blacksquare}(\D^{\dagger}_{P,\rt})\rightarrow D_{\blacksquare}(\D^{\dagger}_{P,\rt}),\]
    where $\D^{\dagger}_{P,\rt}({}^{\dagger}Z)=\O_{P_k}({}^{\dagger}Z)\otimes_{\O_P,\blacksquare}\D_{P,\rt}^{\dagger}$ and 
    $\O_{P_k}({}^{\dagger}Z)$ is the sheaf of $\rt_p$-solid algebras of overconvergent functions on $]P_k\backslash Z[_{P}$,
    i.e. $\O_{P_k}({}^{\dagger}Z)=\O_{]P_k\backslash Z[_{P}^{\dagger}}$.
    In general,
    we can choose divisors $T_1,\ldots,T_n$ such that $Z=T_1\cap\cdots\cap T_n$,
    then we define 
    \[
    R\underline{\Gamma}_{Z}^{\dagger}(-):=R\underline{\Gamma}_{T_1}^{\dagger}(-)\circ \cdots\circ R\underline{\Gamma}_{T_n}^{\dagger}(-),
    \]
    and define $({}^{\dagger}Z)$ by the triangle.
    This definition generalizes to any locally closed subscheme $X$ by setting
    $R\underline{\Gamma}_X^{\dagger}:=({}^{\dagger}(\overline{X}\backslash X))\circ R\underline{\Gamma}_{\overline{X}}^{\dagger}$.
    \item (Arithmetic $D$-modules via frame)
    Let $X$ be a realizable scheme and $(X,\overline{X},P)$ be the associated proper smooth frame.
    Let $D_{\blacksquare,P}(\D^{\dagger}_{X/K})$ be the full subcategory of $D_{\blacksquare}(\D^{\dagger}_{P,\rt})$
    consisting of objects $M$ such that there exists an isomorphism $M\cong R\underline{\Gamma}_{X}^{\dagger}(M)$.
    This is the category of solid arithmetic $D$-modules (framed by $P$) over $X$.
    \item (Kashiwara's equivalence)
    Both the local cohomology and the category $D_{\blacksquare,P}(\D^{\dagger}_{X/K})$ have some choices in the construction.
    However,
    consider $i\colon X\rightarrow P_k$,
    then one can identify the functor $i^{\arith/K}_!i^{\arith/K,*}\colon D(P_k^{\arith/K})\rightarrow D(P_k^{\arith/K})$
    with $R\underline{\Gamma}_X^{\dagger}\colon D_{\blacksquare}(\D^{\dagger}_{P,\rt})\rightarrow D_{\blacksquare}(\D^{\dagger}_{P,\rt})$.
    Indeed,
    one can reduce to the case of being affine open with divisor-complement,
    then the identification is directly from the excision property in Theorem \ref{arithsixfunctor}
    and the proof in Proposition \ref{propersmootharithcoefficients}.
    With the identification of functors established,
    then $D(X^{\arith/K})$ is the full subcategory of $D(P_k^{\arith/K})$ consisting of $M$ such that $M\cong i_!i^*M$
    (actually we don't have a natural transformation between $i_!i^*$ and $\id$ so this really means that
    the correspondence $M\leftarrow \overline{i}_!\overline{i}^*M\rightarrow i_!i^*M$ becomes an isomorphism,
    where $\overline{i}\colon\overline{X}\rightarrow P_k$ is the compactification of $X$),
    which is equivalent to the full subcategory of $D_{\blacksquare}(\D^{\dagger}_{P,\rt})$ consisting of $M$ such that $M\cong R\underline{\Gamma}_X^{\dagger}(M)$,
    which is $D_{\blacksquare,P}(\D^{\dagger}_{X/K})$.
    \item (Definition of six operations)
    Let $f\colon X\rightarrow Y$ be a morphism between realizable schemes.
    Then we define the naïve pullback functor $f^{\triangle}\colon D_{\blacksquare}(\D^{\dagger}_{Y/K})\rightarrow D_{\blacksquare}(\D^{\dagger}_{X/K})$
    by using Kashiwara's equivalence as follows.
    Let $(X,\overline{X},P)$ and $(Y,\overline{Y},Q)$ be the proper smooth frames respectively.
    By replacing $P$ by $P\times Q$ if necessary,
    we can always assume that $f$ can be lifted to a morphism $F\colon P\rightarrow Q$ between frames,
    which is the trick in Definition \ref{Pullbackofrigidcohomology} or Definition \ref{pullbackoverconvergentisocrystal}.
    Then we define 
    \[
    f^{\triangle}:=R\underline{\Gamma}_{X}^{\dagger}\circ F^{\triangle}.
    \]
    Then this exactly corresponds to $f^{\arith/K}\colon D(Y^{\arith/K})\rightarrow D(X^{\arith/K})$ hence does not depend on the choice of frames.
    Similar definitions can be made for extraordinary pullback and pushforward.
  \end{enumerate}
  This approach via Kashiwara's equivalence (Kashiwara's approach) is equivalent to the approach using relative arithmetic de Rham stack (stacky approach),
  as shown above.
  However,
  there is always a problem of choosing a frame when defining everything in Kashiwara's approach!
  This is invisible in the stacky approach,
  and meanwhile the way we prove the frame-independence in Kashiwara's approach is exactly by using the stacky approach.
\end{remark}

One can define and study a Frobenius-equivariant category of solid arithmetic $D$-modules.

\begin{definition}[Solid arithmetic $D$-modules with Frobenius structure]
  Let $K$ admit a Frobenius lift $\varphi_K$ and be perfect.
  Let $X$ be a scheme over $k$ and $\varphi_{X/k}\colon X\rightarrow X^{(1)}$ be the relative Frobenius of $X$ over $k$.
  It induces a Frobenius pullback endofunctor $\varphi^{\triangle}_{X/k}$ from the category of arithmetic $D$-modules over $X^{(1)}$ to the category of arithmetic $D$-modules over $X$.
  Then an arithmetic $D$-module with Frobenius structure is an object $M\in D_{\blacksquare}(\D_{X/K}^{\dagger})$ equipped with an isomorphism $\varphi_{X/k}^{\triangle}\varphi_K^*M\cong M$.
  A morphism between such objects is a morphism between arithmetic $D$-modules compatible with the Frobenius-equivariant structure.
  Let $D_{\blacksquare}(\D_{X/K}^{\dagger})^{F\text{-equiv}}$ be the resulting category.
\end{definition}

\begin{remark}\label{FarithmeticDmodule}
  We have a direct reinterpretation of this category using arithmetic de Rham stacks.
  Let $K$ admit a Frobenius lift $\varphi_K$ and be perfect.
  We have an equivalence of categories 
  \[
    D_{\blacksquare}(\D_{X/K}^{\dagger})^{F\textnormal{-equiv}}\cong D(X^{\arith/K}/\varphi_{X}^{\itg}).
  \]
\end{remark}

\begin{remark}[Drinfeld's lemma for arithmetic $D$-modules]
  Let $q=p^s$,
  and we work over $k=\ff_q$ equipped with \emph{$s$-th Frobenius} and $K$ equipped with the Frobenius lift $\varphi_K=\id_K$.
  The stack $\operatorname{GSpec} K/\varphi_K^{\itg}=\operatorname{GSpec} K/\id_K^{\itg}$ lives over $\operatorname{GSpec}K$,
  hence so does $X^{\arith/K}/\varphi_X^{\itg}$.
  Let $X_1,\ldots,X_n$ be $k$-schemes.
  Let $F_1,\ldots,F_n$ be the partial Frobenii on $X_1\times_k\cdots\times_k X_n$,
  i.e. $F_i=\id\times\cdots\times \varphi_{X_i}\times\cdots\times \id$.
  For a category $\C$ with all $F_i^*$-actions,
  denote by $\C^{\Phi\textnormal{-equiv}}$ the category of $F_i$-equivariant objects $M$ in $\C$ such that
  all isomorphisms $\alpha_i\colon F_i^*M\cong M$ commute pairwise,
  i.e. $F_i(\alpha_j)\circ \alpha_j\cong F_j(\alpha_i)\circ \alpha_j$ as morphisms from $F_i^*F_j^*M\cong F_j^*F_i^*M$ to $M$.
  Then we have an equivalence of symmetric monoidal categories 
    \[
    D_{\blacksquare}(\D_{X_1\times_k\cdots\times_k X_n/K}^{\dagger})^{\Phi\textnormal{-equiv}}\cong 
    D_{\blacksquare}(\D_{X_1/K}^{\dagger})^{F\textnormal{-equiv}}\otimes_{D_{\blacksquare}(K)}\cdots\otimes_{D_{\blacksquare}(K)} D_{\blacksquare}(\D_{X_n/K}^{\dagger})^{F\textnormal{-equiv}}.
    \]
  Indeed,
  this is directly deduced by applying categorical Künneth formula (Remark \ref{categoricalKunneth}) to the product of Gelfand stacks:
    \[
      (X_1\times_k \cdots\times_k X_n)^{\arith/K}/(F_1^{\itg}\times\cdots\times F_n^{\itg})=X_1^{\arith/K}/\varphi_{X_1}^{\itg}\times_{K}\cdots\times_{K} X_n^{\arith/K}/\varphi_{X_n}^{\itg}.
    \]
  We also expect to establish Drinfeld's lemma in overconvergent isocrystals using the stacky approach,
  recovering results in \cite{Ked24} and \cite{KX23}.
  This will be explored in future work.
\end{remark}

\section{Hyodo--Kato stacks of schemes in characteristic $p$}

Let $X$ be a separated scheme of finite type over $k$ of characteristic $p$.
Recall in Definition \ref{diamondofscheme},
we can associate to $X$ two qfd arc-stacks over $\ff_p$:
the small diamond $X^{\circ}_{\arc}$ and the big diamond $X_{\arc}$.
Then we can take the Hyodo--Kato stacks of these qfd arc-stacks from Definition \ref{dRFFstk},
and denote them as follows
\[
X^{\HK}:=X_{\arc}^{\HK},\quad X^{\HK}_{\le 1}:=X_{\arc}^{\circ,\HK}.
\]
In this section,
we analyze the $6$-functor formalism and the connections with the arithmetic de Rham stack we studied before.
They serve as special cases of Hyodo--Kato stacks in \cite{ABLBRCS25},
and in the proof of many properties,
we follow or directly reproduce the proofs in \cite{ABLBRCS25}.

\subsection{Definition and $6$-functor formalism}\label{rigidexplicitproperty}
Let $X$ be a separated scheme of finite type over $k$.
Let $X_{\arc}$ (resp. $X_{\arc}^{\circ}$) be the big (resp. small) diamond associated to $X$ in Definition \ref{diamondofscheme}.

\begin{definition}
  We follow the notations in Definition \ref{dRFFstk}.
  The stacks $X^{\RIG}$ and $X^{\HK}$ are defined as:
  \[
  X^{\RIG}:=(\Y_{X_{\arc}}^{\diamond})^{\dR},\,X^{\HK}:= X_{\arc}^{\HK}.
  \]
  The stacks $X^{\RIG}_{\le 1}$ and $X^{\HK}_{\le 1}$ are defined as:
  \[
  X^{\RIG}_{\le 1}:=(\Y_{X^{\circ}_{\arc}}^{\diamond})^{\dR},\,X^{\HK}:= (X_{\arc}^{\circ})^{\HK}.
  \]
  In particular,
  let $\varphi_X$ be the absolute Frobenius morphism of (the scheme; the big diamond of; the small diamond of) $X$,
  and also denote by $\varphi_X$ the one induced on the punctured Fargues--Fontaine disk,
  then $X^{\HK}=X^{\RIG}/\varphi^{\itg}_X$ and $X^{\HK}=X^{\RIG}_{\le 1}/\varphi^{\itg}_X$.
\end{definition}

\begin{remark}\label{bigrigiddefinition}\label{smallrigiddefinition}
  We have the functor-of-points description of these stacks as follows (or equivalently as definitions).
  The $X^{\RIG}$ is (the sheafification) of the following qfd Gelfand stack
  \[
  X^{\RIG}(A):=X(A^{u,\flat}),\,A\in\cat{uPerfd}^{\qfd}_{\omega_1},
  \]
  and the $X^{\RIG}_{\le 1}$ is (the sheafification) of the following qfd Gelfand stack
  \[
  X^{\RIG}_{\le 1}(A)=X(A^{u,\circ,\flat}),\,A\in\cat{uPerfd}^{\qfd}_{\omega_1}.
  \]
  There is no actual need to sheafify,
  by \cite[Remark 6.1.4]{ABLBRCS25} and the functor-of-points description of $X_{\arc}$ and $X_{\arc}^{\circ}$ in Definition \ref{diamondofscheme}.
  Moreover,
  the Frobenius morphisms induced there are given by $x\mapsto x^p$ on $A^{u,\flat}$ or $A^{u,\circ,\flat}$ via the functor-of-points description.
\end{remark}

The stacks $X^{\RIG}$ and $X^{\RIG}_{\le 1}$ also arise from transmutation,
as we explain in the following remark.

\begin{remark}\label{transumationrigid}
  Let $\aff^1_{\ff_p}$ be the affine line over $\ff_p$. 
  From the functor-of-points description,
  we have $\aff^{1,\RIG}_{\ff_p}(A)=A^{u,\flat}=\varprojlim_{x\mapsto x^p} A^u$,
  and $\aff^{1,\RIG}_{\ff_p,\le 1}(A)=A^{u,\circ,\flat}=\varprojlim_{x\mapsto x^p} A^{u,\circ}$ hence by comparing functor of points we have
  \[
  \aff^{1,\RIG}_{\ff_p}=\varprojlim_{\varphi}\aff^{1,\an,\dR}_{\rt_p},\,
  \aff^{1,\RIG}_{\ff_p,\le 1}=\varprojlim_{\varphi}\disk^{\dagger,\dR}_{\rt_p}
  \]
  as the perfection of the analytic de Rham stack of (the perfectoidization of) the analytic affine line and the overconvergent disk
  where $\varphi$ is the Frobenius morphism sending $T$ to $T^p$ on the coordinate.

  As qfd Gelfand stacks,
  $\aff^{1,\RIG}_{\ff_p}$ and $\aff^{1,\RIG}_{\ff_p,\le 1}$ are naturally ring objects;
  the ring structure is given as on $A$-points $\aff^{1,\RIG}_{\ff_p}(A)=A^{u,\flat}$ and $\aff^{1,\RIG}_{\ff_p,\le 1}(A)=A^{u,\circ,\flat}$ by the ring structure on tilts:
  \[
  (x+y)_n:= \lim_{m\rightarrow 0}(x_{n+m}+y_{n+m})^{p^m},\, (x\cdot y)_n:=x_n\cdot y_n.
  \]
  
  Then $X^{\RIG}$ and $X^{\RIG}_{\le 1}$ are defined via transmutation,
  i.e. 
  \[
  X^{\RIG}(A)=X(\aff^{1,\RIG}_{\ff_p}(A)),\,X^{\RIG}_{\le 1}(A)=X(\aff^{1,\RIG}_{\ff_p,\le 1}(A)),\,A\in\cat{uPerfd}^{\qfd}_{\omega_1}.
  \]
\end{remark}

\begin{remark}\label{hdescentrigid}
  The stacks $X^{\RIG}$ and $X^{\RIG}_{\le 1}$ (hence also $X^{\HK}$ and $X^{\HK}_{\le 1}$) satisfy $h$-hyperdescent.
  Indeed,
  by Proposition \ref{dRFFpreservecolimit},
  it suffices to show that if $X\rightarrow Y$ is an $h$-hypercover of affine schemes of finite type,
  then $X_{\arc}\rightarrow Y_{\arc}$ and $X^{\circ}_{\arc}\rightarrow Y^{\circ}_{\arc}$ are arc-hypercovers.
  The case of small diamonds directly follows from Lemma \ref{hcoverarccover}.
  To deal with the case of big diamonds,
  it suffices to show that the functor of taking big diamonds sends $h$-covers of affine schemes of finite type to arc-covers.
  Let $X=\Spec A\rightarrow Y=\Spec B$ be an $h$-cover and write $A=B[T_1,\ldots,T_n]/(f_1,\ldots,f_m)$.
  It suffices to find an arc-cover $S\rightarrow\M_{\arc}(R)$ of qfd perfectoid spaces for any map $\M_{\arc}(R)\rightarrow Y_{\arc}$ such that it lifts to a map $S\rightarrow X_{\arc}$:
  \[\xymatrix{
    S\ar@{-->}[r]\ar[d] & X_{\arc}\ar[d]\\
    \M_{\arc}(R)\ar[r] & Y_{\arc}
  }\]
  We can simply take $S$ as the fiber product $\M_{\arc}(R)\times_{Y_{\arc}}X_{\arc}$,
  which is identified with the perfectoidization of
  the vanishing locus $V(f_1^{1/p^{\infty}},\ldots,f_m^{1/p^{\infty}})\subset \aff^{n,\perf,\an}_R$ which is qfd perfectoid,
  and it suffices to show that $S\rightarrow\M_{\arc}(R)$ is an arc-cover.
  By \cite[Theorem 11.26]{BS17},
  being an $h$-cover implies that $B\rightarrow A$ is descendable,
  so is $R\rightarrow R[T_1,\ldots,T_n]/(f_1,\ldots,f_m)$,
  and hence $R\rightarrow R[T_1^{1/p^{\infty}},\ldots,T_n^{1/p^{\infty}}]/(f_1^{1/p^{\infty}},\ldots,f_m^{1/p^{\infty}})$ is also descendable.
  Since $R$ is already Gelfand,
  the induced map on their Gelfandifications is also descendable,
  which is given by $V(f_1^{1/p^{\infty}},\ldots,f_m^{1/p^{\infty}})\rightarrow \operatorname{GSpec}R$.
  Then it also induces an arc-cover by \cite[Lemma 4.5.1]{ABLBRCS25},
  which proves the claim.
\end{remark}

\begin{remark}\label{rigidKunneth}
  The stacks $X^{\RIG}_{(\le 1)}$ preserve finite limits by construction via functor of points.
  In particular,
  they preserve fiber products. 
\end{remark}

Note that there is a natural morphism $X_{\arc}^{\circ}\rightarrow X_{\arc}$ from the small diamond to the big diamond associated to $X$,
which induces morphisms $X^{\RIG}_{\le 1}\rightarrow X^{\RIG}$ and $X^{\HK}_{\le 1}\rightarrow X^{\HK}$.

\begin{lemma}[Valuation criterion]\label{valuationcriterion}
  Let $X\rightarrow Y$ be a proper morphism.
  Then the following diagram is Cartesian:
  \[
  \xymatrix{
    X^{\circ}_{\arc}\ar[r]\ar[d] & Y^{\circ}_{\arc}\ar[d]\\
    X_{\arc}\ar[r] & Y_{\arc}
  }
  \]
\end{lemma}
\begin{proof}
  By Remark \ref{preadicspace},
  the diagram is the following diagram after applying the functor $a^*$:
  \[
  \xymatrix{
    X^{\ad,\diamondsuit}\ar[r]\ar[d] & Y^{\ad,\diamondsuit}\ar[d]\\
    X^{\ad/\ff_p,\diamondsuit}\ar[r] & Y^{\ad/\ff_p,\diamondsuit}
  }
  \]
  Since $a^*$ preserves finite limits,
  it suffices to prove this diagram is Cartesian.
  However,
  as discrete adic spaces,
  the following diagram is Cartesian:
  \[
  \xymatrix{
    X^{\ad}\ar[r]\ar[d] & Y^{\ad}\ar[d]\\
    X^{\ad/\ff_p}\ar[r] & Y^{\ad/\ff_p}
  }
  \]
  If $Y=\Spec \ff_p$,
  then this is already discussed in \cite[Proposition 9.6]{CS19}.
  For general $Y$,
  we can mimic the same proof.
  The map $X^{\ad}\rightarrow X^{\ad/\ff_p}\times_{Y^{\ad/\ff_p}} Y^{\ad}$ is an open immersion as both $X^{\ad}\rightarrow X^{\ad/\ff_p}$ and $X^{\ad/\ff_p}\times_{Y^{\ad/\ff_p}} Y^{\ad}\rightarrow X^{\ad/\ff_p}$ is.
  Therefore,
  to check it is an isomorphism,
  it suffices to check the bijectivity of the underlying point sets.
  However,
  points of $X^{\ad}$ are equivalent to maps from spectra of valuation rings,
  and points of $X^{\ad/\ff_p}$ are equivalent to maps from spectra of valuation fields,
  cf. \cite[Discussion after Definition 9.5]{CS19},
  hence the bijectivity directly comes from the valuation criterion for properness.
\end{proof}
\begin{proof}
  We freely consider arc-stacks as partially proper $v$-stacks in \cite[Section 12]{Sch24b}.
  By \cite[Proposition 27.4]{Sch17} and its proof,
  we know that the induced map $X_{\arc}\rightarrow Y_{\arc}$ is representable by locally spatial diamonds and is proper hence qcqs.
  Since $X^{\circ}_{\arc}\rightarrow Y^{\circ}_{\arc}$ is already qcqs in Lemma \ref{propersmalldiamond},
  we know that the morphism $X^{\circ}_{\arc}\rightarrow X_{\arc}\times_{Y_{\arc}}Y^{\circ}_{\arc}$ is also a qcqs morphism between $v$-stacks.
  By \cite[Lemma 12.5]{Sch17},
  to check $X^{\circ}_{\arc}\rightarrow X_{\arc}\times_{Y_{\arc}}Y^{\circ}_{\arc}$ is an isomorphism,
  it suffices to check it on geometric points.
  By the valuation criterion of properness of morphisms between schemes,
  we know that if $K$ is a perfectoid field,
  then the following diagram is Cartesian 
  \[
  \xymatrix{
    X^{\circ}_{\arc}(K)\ar[r]\ar[d] & Y^{\circ}_{\arc}(K)\ar[d]\\
    X_{\arc}(K)\ar[r] & Y_{\arc}(K)
  }
  \]
  which shows that $X^{\circ}_{\arc}(K)\cong (X_{\arc}\times_{Y_{\arc}}Y^{\circ}_{\arc})(K)$ on geometric points,
  hence $X^{\circ}_{\arc}\cong X_{\arc}\times_{Y_{\arc}}Y^{\circ}_{\arc}$ is an isomorphism and the diagram is Cartesian.
\end{proof}

\begin{lemma}\label{propermorphismcriterion}
  Let $X\rightarrow Y$ be a proper morphism.
  Then the following diagram is Cartesian:
  \[
  \xymatrix{
    X^{\RIG}_{\le 1}\ar[r]\ar[d] & Y^{\RIG}_{\le 1}\ar[d]\\
    X^{\RIG}\ar[r] & Y^{\RIG}
  }
  \]
  A similar statement holds for the Hyodo--Kato stacks of $X$ and $Y$.
\end{lemma}
\begin{proof}
  By valuation criterion of properness in Lemma \ref{valuationcriterion},
  this is a Cartesian diagram by directly applying the analytic de Rham stack of punctured open disk to the diagram in Lemma \ref{valuationcriterion}.
  For Hyodo--Kato stacks,
  it is obtained by applying the Hyodo--Kato stack to the diagram.
\end{proof}

Let us compute some examples and study their geometry.

\begin{example}[Field of characteristic $p$ and finite field]\label{finiteHK}
  Take $X=\Spec k$ where $k$ is a field of characteristic $p$,
  then we have 
  \[
    (\Spec k)^{\RIG}_{\le 1}=(\operatorname{GSpec} W(k^{\perf})[p^{-1}])^{\dR}.
  \]
  Indeed,
  we have from Proposition \ref{FFcurvesmalldiamond} that:
  \[
  \Y_{(\Spec k)^{\circ}_{\arc}}^{\diamond}=((\Spec k)_{\Ainf}^{\perf})^{\diamond}_{\eta}=\M_{\arc}(W(k^{\perf})[p^{-1}]),
  \]
  Then we have 
  \[
    (\Spec k)^{\RIG}_{\le 1}=\Y_{(\Spec k)^{\circ}_{\arc}}^{\diamond,\dR}=(\operatorname{GSpec} W(k^{\perf})[p^{-1}])^{\dR}.
  \]
  In general,
  we don't expect a good description of $(\Spec k)^{\RIG}$
  but for the case of finite field we have $(\Spec k)^{\RIG}=(\Spec k)^{\RIG}_{\le 1}$.
    Let $k$ be a finite field.
    Since $k$ is étale over $\ff_p$,
    every element in $(\Spec k)^{\RIG}(A)=\Hom_{\ff_p}(k,A^{u,\flat})$ will lie in $(\Spec k)^{\RIG}_{\le 1}(A)=\Hom_{\ff_p}(k,A^{u,\circ,\flat})$.
    Moreover,
    since $W(k)[p^{-1}]$ is étale over $\rt_p$,
    we also have $(\operatorname{GSpec}W(k)[p^{-1}])^{\dR}=\operatorname{GSpec}W(k)[p^{-1}]$ hence 
    \[
    (\Spec k)^{\RIG}=(\Spec k)^{\RIG}_{\le 1}=\operatorname{GSpec} W(k)[p^{-1}],\,
    (\Spec k)^{\HK}=(\Spec k)^{\HK}_{\le 1}=\operatorname{GSpec}W(k)[p^{-1}]/\varphi^{\itg}.
    \]
\end{example}

The Hyodo-Kato stacks of a scheme are the same as the perfection of the scheme,
hence they give the same theory.
This means that we can always work over a \emph{perfect} base field $k$.
The advantage of working over a perfect base is that we can define a relative version of Hyodo--Kato stacks. 
In the case that $k$ is perfect,
we have a sequence of morphisms
\[
\operatorname{GSpec}K\longrightarrow\operatorname{GSpec}W(k)[p^{-1}]
\cong\operatorname{GSpec}W(k^{\perf})[p^{-1}]\longrightarrow
(\Spec k)^{\RIG}_{\le 1}\longrightarrow (\Spec k)^{\RIG}.
\]
Hence we can make the following definition.
\begin{definition}\label{relativeHK}
  Let $k$ be a perfect field.
  The relative analytic de Rham stacks of punctured Fargues--Fontaine disks are defined by the following Cartesian diagram:
  \[\xymatrix{
    X^{\RIG/K}_{(\le 1)}\ar[r]\ar[d] & \operatorname{GSpec}K\ar[d]\\
    X^{\RIG}_{(\le 1)}\ar[r] & (\Spec k)^{\RIG}_{(\le 1)}
  }\]
  If $K$ admits a Frobenius lift $\varphi_K$ and is perfect,
  we can define the relative Hyodo--Kato stacks as the following Cartesian diagram
  \[\xymatrix{
    X^{\HK/K}_{(\le 1)}\ar[r]\ar[d] & \operatorname{GSpec}K/\varphi_{K}^{\itg}\ar[d]\\
    X^{\HK}_{(\le 1)}\ar[r] & (\Spec k)^{\HK}_{(\le 1)}
  }\]
\end{definition}

\begin{remark}
  There is also an alternative way of defining Hyodo--Kato stacks via transmutation from
  \[
  \aff^{1,\RIG/K}_k=\varprojlim_{\varphi}\aff^{1,\an,\dR}_K,\,
  \aff^{1,\RIG/K}_{k,\le 1}=\varprojlim_{\varphi}\disk^{\dagger,\dR}_K.
  \]
\end{remark}

\begin{definition}[Tate twist for $X^{\HK}$ or $X^{\HK/K}$]\label{Tatetwist}
  We know that $(\Spec \ff_p)^{\HK}=\operatorname{GSpec}\rt_p/\varphi_{\rt_p}^{\itg}$,
  hence the vector bundles on $(\Spec \ff_p)^{\HK}$ are equivalent to $F$-isocrystals over $\rt_p$.
  Let $1(n)\in D((\Spec\ff_p)^{\HK})$ be the rank-one $F$-isocrystal given as $\varphi\cdot e=p^{-n}e$ for a generator $e\in 1(n)$.
  For any $X$ over $k$,
  we also denote $1(n)\in D(X^{\HK})$ or $D(X^{\HK/K})$ as the pullback $f^*1(n)$ of $1(n)$ along the morphism $f$ to the base $(\Spec \ff_p)^{\HK}$.
  They are called Tate twists for $X^{\HK}$ or $X^{\HK/K}$.
\end{definition}

Let us study some examples and basic geometry for Hyodo--Kato stacks.

\begin{example}[Additive group]\label{rigidstackadditive}
  Let $X$ be the affine line $\aff^1_k=\Spec k[T]$ over $k$.
  We have that
  \[X^{\RIG/K}\cong\varprojlim_{\varphi}\aff^{1,\an,\dR/K}_{K},\,X^{\RIG/K}_{\le 1}\cong\varprojlim_{\varphi}\disk^{\dagger,\dR/K}_{K},\]
  and hence 
  \[X^{\HK/K}\cong\varprojlim_{\varphi}\aff^{1,\an,\dR/K}_{K}/\varphi^{\itg},\,X^{\HK/K}_{\le 1}\cong\varprojlim_{\varphi}\disk^{\dagger,\dR/K}_{K}/\varphi^{\itg},\]
  where $\varphi$ is the Frobenius of $\mathbb{G}_{a,K}$ (or $\disk^{\dagger}_K$),
  sending $T$ to $T^p$ on the coordinate and restricting to $\varphi_K$ on $K$ on the coordinate.
  This comes from Remark \ref{rigidKunneth} and the computation in Remark \ref{transumationrigid}.
\end{example}

\begin{example}[Multiplicative group]
  Take $X$ to be $\gff=\Spec k[T^{\pm1}]$ over $k$.
  We have that
  \[X^{\RIG/K}\cong\varprojlim_{\varphi}\mathbb{G}_{m,K}^{\dR/K},\,X^{\RIG/K}_{\le 1}\cong\varprojlim_{\varphi}\mathbb{T}_{K}^{\dagger,\dR/K},\]
  and hence 
  \[X^{\HK/K}\cong\varprojlim_{\varphi}\mathbb{G}_{m,K}^{\dR/K}/\varphi^{\itg},\,X^{\HK/K}_{\le 1}\cong\varprojlim_{\varphi}\mathbb{T}_{K}^{\dagger,\dR/K}/\varphi^{\itg},\]
  where $\varphi$ is the absolute Frobenius of $\mathbb{G}_{m,K}$ (or $\mathbb{T}_{K}^{\dagger}$),
  again,
  sending $T$ to $T^p$ and restricting to the $\varphi_K$ on $K$ on the coordinate.
  By Remark \ref{rigidKunneth},
  it suffices to deal with $k=\ff_p$.
  Then $\gff^{\RIG}(A)=A^{u,\flat,\times}$ and $\mathbb{G}^{\RIG}_{m,\le 1}(A)=A^{u,\circ,\flat,\times}$ whence the description comes.
\end{example}

\begin{example}[Laurent series]\label{LaurentHK}
    Let $k=\ff_p(\!(t)\!)$ be the field of Laurent series.
    Then,
    from Example \ref{finiteHK},
    we have that 
    \[
    \Y_{(\Spec k)^{\circ}_{\arc}}^{\diamond}=((\Spec k)^{\perf}_{\Ainf})^{\diamond}_{\eta}=\M_{\arc}(\E_{\rt_p,\perf}),
    \]
    where $\E_{\rt_p,\perf}$ is the perfection of the Amice ring $\E_{\rt_p}$ (Example \ref{Laurentarith}) under Kummer Frobenius $t\mapsto t^p$,
    equipped with the $p$-adic topology,
    and hence 
    \[
      (\Spec k)^{\RIG}_{\le 1}\cong (\Y_{(\Spec k)^{\circ}_{\arc}}^{\diamond})^{\dR}
      \cong (\M_{\arc}(\E_{\rt_p,\perf}))^{\dR}\cong \varprojlim_{\varphi}(\M_{\arc}(\E_{\rt_p}))^{\dR}.
    \]
    Now we compute the stack $(\Spec k)^{\RIG}$ of $k$.
    In fact,
    we have a Cartesian diagram 
    \[\xymatrix{
      \Hom(\ff_p(\!(t)\!),A^{u,\flat})\ar[r]\ar[d] & \Hom(\ff_p[t^{\pm1}],A^{u,\flat})\cap \Hom(\ff_p[t],A^{u,\circ,\flat})\ar[d] \\
      \Hom(\ff_p[\![t]\!],A^{u,\circ,\flat})\ar[r] &  \Hom(\ff_p[t],A^{u,\circ,\flat})\\
    }\]
    Indeed,
    first the element $t$ must map to an element of norm $\le 1$.
    If an element $f\in \ff_p(\!(t)\!)$ with $\deg f=0$ maps to an element $a$ of norm greater than $1$,
    without loss of generality we can assume $f$ has constant term $1$,
    then $(1-f)^{-1}=1+f+f^2+\cdots$ exists and it must map to $1+a+a^2+\cdots$,
    which means $a$ must be bounded.
    In general,
    every Laurent series can be written as $t^if$ for some $f$ with degree $0$,
    and then the value is determined by that on $t$ and $f$ with degree $0$.
    This interprets to a Cartesian diagram of Gelfand stacks
    \[\xymatrix{
      (\Spec\ff_p(\!(t)\!))^{\RIG}\ar[r]\ar[d] & \varprojlim_{\varphi}\disk^{\dagger,\times,\dR}_{\rt_p}\ar[d] \\
      (\Spec \ff_p[\![t]\!])^{\RIG}_{\le 1}\ar[r] &  \varprojlim_{\varphi}\disk^{\dagger,\dR}_{\rt_p}\\
    }\]
    where $\disk^{\dagger,\times}=\gff^{\an}\cap\disk^{\dagger}$ is the punctured overconvergent disk.
    From Proposition \ref{FFcurvesmalldiamond},
    we have 
    \[
      (\Spec \ff_p[\![t]\!])^{\HK}_{\le 1}=((\Spec \ff_p[\![t]\!])^{\perf}_{\Ainf})^{\diamond,\dR}_{\eta}=(\operatorname{GSpec}\itg_p[\![t^{1/p^{\infty}}]\!][p^{-1}])^{\dR},
    \]
    where $\itg_p[\![t^{1/p^{\infty}}]\!][p^{-1}]$ is equipped with the $p$-adic topology.
    Then we have
    \[
      (\Spec\ff_p(\!(t)\!))^{\RIG}\cong (\operatorname{GSpec}\itg_p[\![t^{1/p^{\infty}}]\!][p^{-1}])^{\dR}\times_{\varprojlim_{\varphi}\disk^{\dagger,\dR}_{\rt_p}}\varprojlim_{\varphi}\disk^{\dagger,\times,\dR}_{\rt_p}
      \cong (\varprojlim_{\varphi}\disk_{\rt_p}^{\circ\circ,\times})^{\dR},
    \]
    where $\disk_{\rt_p}^{\circ\circ,\times}=\operatorname{GSpec}\itg_p[\![t]\!][p^{-1}]\times_{\disk^{\dagger}_{\rt_p}}\disk^{\dagger,\times}_{\rt_p}=\disk^{\circ\circ}_{\rt_p}\times_{\disk^{\dagger}_{\rt_p}}\disk^{\dagger,\times}_{\rt_p}$ is a version of a punctured open disk.
    To give an intuition of this disk,
    we can compute the ring of functions on this disk as the following ring of functions (with $p$-adic topology!)
    \[
    \O(\disk_{\rt_p}^{\circ\circ,\times})=\left\{
    \sum_i a_i t^i\in \rt_p[\![t^{\pm 1}]\!]\bigg| \sup_{i\rightarrow\infty}|a_i|<\infty,\,\f \eta>1,\lim_{i\rightarrow-\infty}|a_i|\eta^i=0
    \right\}.
    \]
    In summary,
    we have 
    \[
      (\Spec \ff_p(\!(t)\!))^{\RIG}_{\le 1}\cong (\operatorname{GSpec}\E^{\perf}_{\rt_p})^{\dR},\,
      (\Spec\ff_p(\!(t)\!))^{\RIG}\cong (\varprojlim_{\varphi}\disk_{\rt_p}^{\circ\circ,\times})^{\dR}.
    \]
    Above the field $k(\!(t)\!)$ is equipped with the \emph{discrete} topology.
    If $k(\!(t)\!)$ is equipped with the \textit{$t$-adic} topology,
    and we consider its associated diamond $\Spd k(\!(t)\!)$,
    then we have that $(\Spd k(\!(t)\!))^{\HK}\cong \varprojlim_{\varphi}\disk^{\circ,\times,\dR}_K/\varphi^{\itg}$ by \cite[Lemma 6.2.1]{ABLBRCS25}.
\end{example}

Now we turn to the study of the geometry of these stacks.
To start,
if we can find a good lift of a smooth scheme $X$,
then we have the following description of $X^{\RIG/K}_{\le 1}$ and $X^{\HK/K}_{\le 1}$.

\begin{proposition}\label{smallrigidstackdescription}
  Let $X$ be a smooth scheme over a perfect field $k$ which embeds openly into a proper smooth scheme $Y$ with a proper smooth lift $\mathfrak{Y}$ and generic fiber $\Y$.
  Let $K$ admit a Frobenius lift $\varphi_K$ and be perfect,
  and let the Frobenius lift to $\mathfrak{Y}$.
  Then it also restricts to a Frobenius morphism $\varphi\colon ]X[^{\dagger}_{\mathfrak{Y}}\rightarrow]X[^{\dagger}_{\mathfrak{Y}}$ by the construction of tubes (Construction \ref{tubes}).
  Then $X^{\RIG/K}_{\le 1}\cong \varprojlim_{\varphi}]X[^{\dagger,\dR/K}_{\mathfrak{Y}}$.
  Equivalently,
  there is a natural surjection $\varprojlim_{\varphi}]X[_{\mathfrak{Y}}^{\dagger}\rightarrow X^{\RIG/K}_{\le 1}$,
  with Cech nerve as $\varprojlim_{\varphi}(\Delta^{\dagger}_n(]X[_{\mathfrak{Y}}^{\dagger}))_{[n]\in\Delta}$.

  As a result,
  we also have $X^{\HK/K}_{\le 1}\cong \varprojlim_{\varphi}]X[^{\dagger,\dR/K}_{\mathfrak{Y}}/\varphi^{\itg}$.
\end{proposition}
\begin{proof}
  First consider the case when $k=\ff_p$.
  The perfectoidization of $]X[^{\dagger}_{\mathfrak{Y}}$ is the same as $]X[_{\mathfrak{Y}}^{\diamond}$,
  which is the diamond of a formal lift of $X$.
  By Proposition \ref{FFcurvesmalldiamond},
  we know that $X^{\RIG}_{\le 1}=(\Y_{X_{\arc}^{\circ}}^{\diamond})^{\dR}\cong (\varprojlim_{\varphi}]X[_{\mathfrak{Y}}^{\diamond})^{\dR}\cong \varprojlim_{\varphi}]X[_{\mathfrak{Y}}^{\diamond,\dR}$,
  using the property that taking the analytic de Rham stack commutes with limits.
  Then we use the facts that $]X[_{\mathfrak{Y}}^{\diamond,\dR}=]X[_{\mathfrak{Y}}^{\dagger,\dR}$ and $]X[_{\mathfrak{Y}}^{\dagger}$ surjects to $]X[_{\mathfrak{Y}}^{\dagger,\dR}$,
  cf. \cite[Corollary 4.7.5]{ABLBRCS25},
  to deduce the surjectivity and the computation of Cech nerve.
  For general $k$,
  the result is obtained by base change from the case of $\ff_p$.
\end{proof}

In general,
for a proper smooth scheme $X$ admitting a proper smooth lifting,
we can find a chart of $X^{\RIG/K}$ (and $X^{\RIG/K}_{\le 1}$) and simplify its description as follows.

\begin{corollary}[Proper smooth scheme]\label{rigidificationsmoothvariety}
  Let $X$ be a proper smooth scheme over a perfect field $k$ such that it admits a proper smooth formal lift $\Xf$ with Frobenius $\varphi$ with generic fiber $\X$ over $K$.
  Let $K$ admit a Frobenius lift $\varphi_K$ and be perfect.
  Then we have
  \[
  X^{\RIG/K}\cong X^{\RIG/K}_{\le 1}\cong (\varprojlim_{\varphi}\X)^{\dR/K}=\varprojlim_{\varphi}\X^{\dR/K}.
  \]

  As a result,
  we have $X^{\HK/K}\cong X^{\HK/K}_{\le 1}\cong \varprojlim_{\varphi}\X^{\dR/K}/\varphi^{\itg}$.
\end{corollary}
\begin{proof}
  Applying Lemma \ref{propermorphismcriterion} to $Y=\Spec k$,
  we get that $X^{\RIG/K}_{\le 1}\cong X^{\RIG/K}$ and hence $X^{\HK/K}_{\le 1}\cong X^{\HK/K}$ for a proper $k$-scheme $X$.
  Combining with Proposition \ref{smallrigidstackdescription},
  we obtain that if $X$ is a proper smooth variety with a proper smooth lifting with Frobenius,
  then $X^{\RIG/K}_{\le 1}\cong \varprojlim_{\varphi} \X^{\dR/K}$,
  which gives the expected result.
\end{proof}

We move on to studying cohomological behavior and $6$-functor formalism for $X^{\RIG}$ and $X^{\HK}$ (not for $X^{\RIG}_{\le 1}$ and $X^{\HK}_{\le 1}$!).
The results in \cite{ABLBRCS25} play a crucial role and significantly simplify the proofs.

\begin{proposition}\label{rigidringstack}
  The stack $\aff^{1,\RIG}_{\ff_p}$ is a ring stack such that
  \[i\colon \{0\}\hookrightarrow \aff^{1,\RIG}_{\ff_p}\hookleftarrow \aff^{1,\RIG,\times}_{\ff_p}=\mathbb{G}_{m,\ff_p}^{\RIG}\colon j\] is a closed-open decomposition.
  It satisfies that the morphism $f\colon \aff^{1,\RIG}_{\ff_p}\rightarrow\operatorname{GSpec}\rt_p$ is $!$-able and cohomologically smooth,
  and the natural morphism $f_!f^!1\rightarrow 1$ is an isomorphism.
  
  As a consequence,
  the stack $\aff^{1,\HK}_{\ff_p}\rightarrow (\Spec\ff_p)^{\HK}$ is also $!$-able and cohomologically smooth,
  with the natural morphism $f_!f^!1\rightarrow 1$ being an isomorphism,
  and $(\Spec\ff_p)^{\HK}\overset{0}{\hookrightarrow} \aff^{1,\HK}_{\ff_p}\hookleftarrow \mathbb{G}_{m,\ff_p}^{\HK}$ is a closed-open decomposition.
\end{proposition}
\begin{remark}
  It has a realization of motives in the sense of the main theorem in \cite{Aok26}.
\end{remark}
\begin{proof}[Proof of Proposition \ref{rigidringstack}]
  First,
  note that all properties for $\aff^{1,\HK}_{\ff_p}$ follow from those for $\aff^{1,\RIG}_{\ff_p}$,
  as the latter is a cohomologically étale cover of the former and all the properties satisfy descent along such a cover,
  cf. \cite[Lemma 4.5.7]{HM24}.
  The closed-open decomposition comes from \cite[Remark 6.1.11]{ABLBRCS25},
  as we have a closed-open decomposition of qfd arc stacks:
  \[
  \Spd\ff_p\hookrightarrow\aff^1_{\ff_p,\arc}\hookleftarrow\mathbb{G}_{m,\ff_p,\arc}.
  \]

  The claim that $f$ is $!$-able and cohomologically smooth is \cite[Corollary 5.7.3]{ABLBRCS25}.
  The claim that $f_!f^!1\rightarrow 1$ is an isomorphism comes from \cite[Lemma 6.2.2]{ABLBRCS25}.
  In fact,
  the morphism $f_!f^!1=f_{\natural}1\rightarrow 1$ being an isomorphism is equivalent to the full faithfulness of $f^*$ which is \cite[Lemma 6.2.2]{ABLBRCS25}.
\end{proof}

\begin{proposition}\label{ishriek}
  The morphism $g\colon\gff^{\HK}\rightarrow \operatorname{GSpec}\rt_p$ satisfies $g_*1\cong 1\oplus 1(-1)[-1]$.
  As a corollary,
  let $i\colon (\Spec\ff_p)^{\HK}\overset{0}{\hookrightarrow}\aff^{1,\HK}_{\ff_p}$,
  then we have $i^!1\cong 1(-1)[-2]$,
\end{proposition}
\begin{proof}
  The computation about $g$ follows from \cite[Lemma 6.2.5]{ABLBRCS25}.
  By the excision property in Proposition \ref{rigidringstack},
  we know that 
  \[
  i^!1=f_*i_*i^!1=\operatorname{fib}(f_*1\rightarrow f_*j_*j^*1)=\operatorname{fib}(1\rightarrow 1\oplus 1(-1)[-1])=1(-1)[-2].\qedhere
  \]
\end{proof}

We can also compute the Hyodo--Kato cohomology of the classifying stack of $\mathbb{G}_{m,\ff_p}$.
\begin{proposition}
  Let $f\colon (B\gff)^{\HK}\rightarrow\operatorname{GSpec}\rt_p$.
  Then $f_*1\cong\bigoplus_{n\ge 0} 1(-n)[2n]$.
\end{proposition}
\begin{proof}
  See \cite[Remark 6.2.11]{ABLBRCS25}.
\end{proof}

It turns out that the computation of $(B\gff)^{\HK}$ is related to the theory of first Chern classes.

\begin{proposition}\label{strongchernclass}
  The $6$-functor formalism of $X^{\HK}$ on $\ff_p$-schemes $X$ has a strong theory of first Chern classes $c_1$,
  in the sense of \cite[Definition 5.2.8]{Zav23}.
  In particular,
  for any $\ff_p$-scheme $X$,
  denote $f\colon (\proj^n_X)^{\HK}\rightarrow X^{\HK}$ the morphism coming from the natural projection and $\O(1)$ the tautological line bundle on $\proj^n_X$,
  then we have an isomorphism 
  \[
  \sum_{k=0}^n c_1(\O(1))^k(n-k)[2n-2k]\colon \bigoplus_{k=0}^n1_X(n-k)[2n-2k]\longrightarrow f_*1_{\proj^n_X}(n)[2n].
  \]
\end{proposition}
\begin{proof}
  See \cite[Lemma 6.2.9]{ABLBRCS25}.
\end{proof}

With all the setup,
we can establish the $6$-functor formalism for $X^{\HK}$ on $\ff_p$-schemes.

\begin{theorem}\label{PoincaredualityrigidFrobenius}
  The $6$-functor formalism $X\rightarrow D(X^{\HK})$ on the category of separated $\ff_p$-schemes of finite type satisfies the following properties.
  \begin{enumerate}
    \item Let $f\colon X\rightarrow Y$ be an étale (resp. proper; smooth) morphism,
    then $f^{\HK}\colon X^{\HK}\rightarrow Y^{\HK}$ is cohomologically étale (resp. cohomologically proper; cohomologically smooth).
    \item Strong $\aff^1$-invariant property holds,
    i.e. $\F\rightarrow f_*^{\HK}f^{\HK,*}\F$ is an isomorphism for all $X$ and $\F\in D(X^{\HK})$,
    where $f\colon \aff^1_X\rightarrow X$ is the projection;
    \item If we have a closed-open decomposition $i\colon Z\rightarrow X\leftarrow U\colon j$,
    then $D(Z^{\HK})$ is a closed subspace of $D(X^{\HK})$ with complement open $D(Y^{\HK})$.
    \item For a smooth map $f\colon X\rightarrow Y$ of relative dimension $n$,
    we have $f^{\HK,!}1\cong 1(n)[2n]$.
  \end{enumerate}
\end{theorem}
\begin{proof}
  Part $(1)$,
  $(2)$ and $(3)$ come from applying Theorem \ref{Aoki} and Proposition \ref{rigidringstack} to the stack $X^{\HK}$,
  which is defined via transmutation (Remark \ref{transumationrigid}).
  Part $(4)$ follows from the proof of \cite[Theorem 5.7.7]{Zav23},
  under the properties that the $6$-functor formalism for $X^{\HK}$ admits a strong theory of Chern classes;
  see Proposition \ref{strongchernclass}.
\end{proof}

One can obtain a similar statement in the relative setting by using the same strategy.

\begin{theorem}\label{PoincaredualityrigidFrobeniusrelative}
  Let $k$ be a perfect field.
  The $6$-functor formalism $X\rightarrow D(X^{\HK/K})$ on the category of separated $k$-schemes of finite type satisfies the following properties.
  \begin{enumerate}
    \item Let $f\colon X\rightarrow Y$ be an étale (resp. proper; smooth) morphism,
    then $f^{\HK/K}\colon X^{\HK/K}\rightarrow Y^{\HK/K}$ is cohomologically étale (resp. cohomologically proper; cohomologically smooth).
    \item Strong $\aff^1$-invariant property holds,
    i.e. $\F\rightarrow f_*^{\HK/K}f^{\HK/K,*}\F$ is an isomorphism for all $X$ and $\F\in D(X^{\HK/K})$,
    where $f\colon \aff^1_X\rightarrow X$ is the projection;
    \item If we have a closed-open decomposition $i\colon Z\rightarrow X\leftarrow U\colon j$,
    then $D(Z^{\HK/K})$ is a closed subspace of $D(X^{\HK/K})$ with complement open $D(Y^{\HK/K})$.
    \item For a smooth map $f\colon X\rightarrow Y$ of relative dimension $n$,
    we have $f^{\HK/K,!}1\cong 1(n)[2n]$.
  \end{enumerate}
\end{theorem}

\begin{remark}
  One can also apply \cite[Theorem 6.3.1]{ABLBRCS25} to obtain all the properties of $6$-functor formalism in Theorem \ref{PoincaredualityrigidFrobenius}.
  In fact,
  directly from the definition,
  an étale (resp. proper; smooth) morphism of schemes $f\colon X\rightarrow Y$ gives
  an étale (resp. proper; smooth) morphism of arc-stacks $f\colon X_{\arc}\rightarrow Y_{\arc}$,
  cf. \cite[Definition 5.1.6]{ABLBRCS25}.
\end{remark}

\begin{remark}[$6$-functor formalisms for $X^{\RIG}_{\le 1}$ and $X^{\HK}_{\le 1}$]
  One may ask about how the $6$-functor formalism for $X^{\RIG}_{\le 1}$ and $X^{\HK}_{\le 1}$ behaves.
  If $X\rightarrow Y$ is proper,
  then $X^{\RIG}_{\le 1}\rightarrow Y^{\RIG}_{\le 1}$ behaves the same as $X^{\RIG}\rightarrow Y^{\RIG}$ by Proposition \ref{propermorphismcriterion},
  hence is cohomologically proper.
  On the other hand,
  if $X\rightarrow Y$ is étale or smooth,
  then $X^{\RIG}_{\le 1}\rightarrow Y^{\RIG}_{\le 1}$ is not cohomologically étale or smooth in general.
  For example,
  let $X\rightarrow Y$ be $\aff^1_{\ff_p}\rightarrow\Spec\ff_p$,
  then $\aff^{1,\RIG}_{\ff_p,\le 1}=\varprojlim_{\varphi}\disk^{\dagger,\dR}_{K}$ which is not cohomologically smooth.
  Moreover,
  since $X^{\circ}_{\arc}$ is proper for \textit{any} $k$-scheme $X$ (by Lemma \ref{propersmalldiamond}),
  one obtains that $X^{\RIG}_{\le 1}$ is cohomologically proper for any $X$ by \cite[Theorem 6.3.1]{ABLBRCS25}.
\end{remark}

\subsection{Relation to analytic prismatization}\label{Sectionanalyticprismatization}

In the forthcoming work by Anschütz--Le Bras--Rodríguez Camargo--Scholze \cite{ALBRCS},
they define and study the rational analytic prismatization $X^{\Prism}$ for any qfd Gelfand stack $X$ over $\rt_p$.
We refer to \cite[Section 3]{Hau26} for an existing reference about analytic prismatization.
In this section,
we will discuss its relation with $X^{\RIG}$ and $X^{\HK}$.
This is only for completeness,
and not used elsewhere in this paper.

\begin{discussion}[Gelfand stack over $\N_{|p|<1}/\rn_{> 0}$]
  To realize the rational analytic prismatization of a scheme of positive characteristic,
  one needs to work in the category of Gelfand stacks not necessarily over $\rt_p$.
  First,
  one should generalize the notion of Gelfand rings beyond Gelfand rings over $\rt_p$.
  Let $\itg(\!(\pi)\!)$ be the condensed ring with the $\pi$-adic topology.
  Intuitively,
  the element $\pi$ encodes a pseudo-uniformizer.
  For any solid $\itg(\!(\pi)\!)$-algebra $A$,
  one can define $A^{\le 1}$ consisting of elements of norm $\le 1$ and $A^b$ consisting of bounded elements similarly as in \cite[Definition 2.6.1]{RC24a} and \cite[Definition 2.2.8]{ABLBRCS25}.
  A solid $\itg(\!(\pi)\!)$-algebra $A$ is called Gelfand if $A=A^b$ and $A^{\le 1}/\pi$ is discrete.

  If we take $\pi=p$,
  then it recovers the notion of Gelfand rings over $\rt_p$.
  However,
  different choices of pseudo-uniformizers make no distinction of the theory,
  so working over $\itg(\!(\pi)\!)$ brings redundancy.
  Instead,
  let $\N$ be the stack of norms in \cite[Lecture 20]{CS23}.
  Then there is a descendable cover $\operatorname{AnSpec}\itg_{\blacksquare}(\!(\pi)\!)\rightarrow\N$,
  and an action of $\rn_{> 0}:=\rn_{>0,\Betti}$ on $\N$ given by exponentiating.
  One observes that the definitions of $A^{\le 1}$ and $A^b$ do not depend on the choice of the pseudo-uniformizers and are invariant under exponentiating,
  hence the notion of Gelfand rings over $\itg(\!(\pi)\!)$ descends to Gelfand rings over $\N/\rn_{> 0}$.
  This eliminates all the redundancy.

  We care about $p$-adic geometry,
  so we restrict to the locus that $p$ has norm less than $1$,
  i.e. we work over $\N_{|p|<1}/\rn_{> 0}$.
  We can form the category of Gelfand stacks over $\N_{|p|<1}/\rn_{> 0}$,
  in a way similar to that in \cite[Section 4.2]{ABLBRCS25}.
\end{discussion}

\begin{discussion}[Rational analytic prismatization]
  Let $X$ be a Gelfand stack over $\N_{|p|<1}/\rn_{> 0}$.
  Let $A$ be a qfd separable uniformly totally disconnected perfectoid ring over $\rt_p$.
  Then $A^u$ is perfectoid hence we can define a Berkovich space $\Y_{A^u,[0,\infty)}$ over $\N_{|p|<1}/\rn_{> 0}$ which is the Fargues--Fontaine disk,
  and it admits the $\theta$-map $\theta\colon \operatorname{GSpec}A^u\rightarrow\Y_{A^u,[0,\infty)}$ and Frobenius map $\varphi\colon \Y_{A^u,[0,\infty)}\rightarrow \Y_{A^u,[0,\infty)}$.
  Let the modified Fargues--Fontaine disk $\Y_{A,[0,\infty)}$ be the derived Berkovich space over $\N_{|p|<1}/\rn_{> 0}$ defined via the following pushout diagram\footnote{Pushout in the category of derived Berkovich spaces over $\N_{|p|<1}/\rn_{> 0}$,
  cf. \cite[Definition 4.3.1]{ABLBRCS25} for a definition over $\rt_p$.}:
  \[
  \xymatrix{
    \coprod_{n\ge 0} \operatorname{GSpec}A^u\ar[r]\ar[d]^{\coprod_{n\ge0} \varphi^n\circ\theta} & \coprod_{n\ge 0} \operatorname{GSpec}A\ar[d]\\
    \Y_{A^u,[0,\infty)}\ar[r] & \Y_{A,[0,\infty)}
  }
  \]
  A Cartier divisor $D\subset\Y_{A,[0,\infty)}$ is called of degree one if its pullback $D\times_{\Y_{A,[0,\infty)}}\Y_{A^u,[0,\infty)}$ is a degree one Cartier divisor in the sense of Fargues--Scholze \cite{FS21}.
  Define $X^{\Prism}$,
  the rational analytic prismatization of $X$,
  as the qfd Gelfand stack over $\rt_p$ by sheafifying the following pre-stack:
  \[
  X^{\Prism}(A)=\{(D,f)\mid D\text{ is a degree one Cartier divisor on }\Y_{A,[0,\infty)},\,f\colon D\rightarrow X\},\,A\in \cat{uPerfd}^{\qfd}_{\omega_1}.
  \]
  It admits a Frobenius map $\varphi_X\colon X^{\Prism}\rightarrow X^{\Prism}$ induced by the one on $\Y_{A,[0,\infty)}$.
  Define the perfect analytic prismatization as $X^{\Prism,\perf}:=\varprojlim_{\varphi_X}X^{\Prism}$ and $X^{\Div^1}:=X^{\Prism,\perf}/\varphi_X^{\itg}$.
\end{discussion}

Our main result in this section is the following comparison.
\begin{proposition}\label{analyticprismatization}
  Let $X$ be a scheme over $\ff_p$.
  We can view it as an analytic stack over $\AnSpec(\ff_{p,\blacksquare})$ by taking the induced analytic structure.
  Then we view it as a Gelfand stack over $\N_{|p|<1}/\rn_{> 0}$ by base change from $\AnSpec(\ff_{p,\blacksquare})$ to $\N_{|p|<1}/\rn_{> 0}$.
  Then we have that $X^{\RIG}\cong X^{\Prism}\cong X^{\Prism,\perf}$ and $X^{\HK}\cong X^{\Div^1}$.
\end{proposition}
\begin{proof}
  Since $X$ is of characteristic $p$,
  the degree one Cartier divisor must be of characteristic $p$ too,
  which means it factors through $D\rightarrow \Y_{\{p=0\}}\subset \Y_{A,[0,\infty)}$.
  Because both divisors $D$ and $\Y_{\{p=0\}}$ are of degree one,
  it follows that $D\cong \Y_{\{p=0\}}$.
  However,
  $\Y_{\{p=0\}}=\operatorname{GSpec}(A^{u,\flat})$ and then $X^{\Prism}(A)=X(A^{u,\flat})$ hence gives the same functor of points of $X^{\RIG}$.
  Therefore,
  we have $X^{\RIG}\cong X^{\Prism}$.
  However,
  $X^{\RIG}$ is already perfect as $A^{u,\flat}$ is,
  so $X^{\RIG}\cong X^{\Prism}\cong X^{\Prism,\perf}$ and $X^{\RIG}/\varphi^{\itg}\cong X^{\Div^1}$.
\end{proof}

\begin{remark}
  The proposition gives that the de Rham stack of the Fargues--Fontaine disk (resp. Hyodo--Kato stack) coincides with the rational analytic prismatization
  (resp. Frobenius-quotient of perfect rational analytic prismatization) of $X$.
  In general,
  the latter serves as the analytic de Rham stack of the former;
  and thus the proposition means that the distinction disappears when working in characteristic $p$.
\end{remark}

\begin{remark}\label{integralanalyticprismatization}
  The rational analytic prismatization inputs a Gelfand stack over $\N_{|p|<1}/\rn_{> 0}$,
  and outputs a Gelfand stack over $\rt_p$.
  There is an integral version of analytic prismatization taking a Gelfand stack over $\N_{|p|<1}/\rn_{> 0}$
  to a Gelfand stack over $\N_{|p|<1}/\rn_{> 0}$,
  with the $\rt_p$-part given by the rational analytic prismatization.
  The definition is more involved,
  so we omit it here.
  Let us only mention that if $X$ is a $k$-scheme,
  then the $\ff_p$-part of the integral analytic prismatization of $X$ can be thought of as an analytic version of the crystalline stack studied in \cite{Bha22}.
\end{remark}

\subsection{Identification of perfect complexes}\label{Sectionidentificationperfectcomplex}
Hyodo--Kato stacks can be viewed as another stacky approach to the theory of rigid cohomology (and partly of it coefficients).
In this section,
we study this stacky approach compared with the approach by arithmetic de Rham stacks.
A first step is to relate the Hyodo--Kato stacks with the arithmetic de Rham stack in the following way.

\begin{discussion}\label{smallrigidbridge}
  Let $X$ be a scheme over $\ff_p$.
  We have a canonical morphism $A^{\circ}/A^{\circ\circ}\leftarrow A^{\circ}\rightarrow A$,
  which induces a diagram of qfd arc-stacks over $\ff_p$:
  \[
  \xymatrix{
    & X_{\arc}^{\circ}\ar[dr]\ar[dl] & \\
    \overline{X} & & X_{\arc}
  }
  \]
  Let $\varphi_X\colon X^{\arith}\rightarrow X^{\arith}$ be the Frobenius morphism coming from the absolute Frobenius $\varphi_X$ on $X$.
  Applying Hyodo--Kato stacks to this diagram,
  we get a diagram of qfd Gelfand stacks:
  \[
  \xymatrix{
    & X^{\HK}_{\le 1}\ar[dr]\ar[dl] & \\
   X^{\arith}/\varphi_X^{\itg} & & X^{\HK}
  }
  \]
  Let $K$ admit a Frobenius lift and be perfect.
  Then the diagram above also induces a diagram on the relative versions of these stacks:
  \[
  \xymatrix{
    & X^{\HK/K}_{\le 1}\ar[dr]^{\alpha}\ar[dl]_{\beta} & \\
   X^{\arith/K}/\varphi_{X}^{\itg} & & X^{\HK/K}
  }
  \]
\end{discussion}

The main result in this section is the following theorem:
\begin{theorem}\label{equivalenceperfectcomplex}
  Let $X$ be a scheme of finite type over a perfect field $k$.
  Let $K$ admit a Frobenius lift $\varphi_K$ and be perfect.
  Then the diagram 
  \[
  \xymatrix{
    & X^{\HK/K}_{\le 1}\ar[dr]^{\alpha}\ar[dl]_{\beta} & \\
   X^{\arith/K}/\varphi_{X}^{\itg} & & X^{\HK/K}
  }
  \]
  induces a fully faithful embedding $\alpha_*\beta^*\colon D( X^{\arith/K}/\varphi_{X}^{\itg})\hookrightarrow D(X^{\HK/K})$.
  Moreover,
  we have an identification of perfect complexes between $\cat{Perf}(X^{\arith/K}/\varphi_{X}^{\itg})$ and $\cat{Perf}(X^{\HK/K})$,
  in the sense that $\alpha^*\colon\cat{Perf}(X^{\HK/K}_{\le 1})\cong \cat{Perf}(X^{\HK/K})$,
  and $\beta^*\colon\cat{Perf}(X^{\arith/K}/\varphi^{\itg}_{X})\cong \cat{Perf}(X^{\HK/K}_{\le 1})$.
\end{theorem}

\begin{proof}[Proof of the full faithfulness in Theorem \ref{equivalenceperfectcomplex}]
  We first prove the full faithfulness of $\alpha_*\beta^*$.
  The map $\alpha$ is a closed immersion,
  because $X^{\circ}_{\arc}\rightarrow X_{\arc}$ is a closed immersion of qfd arc stacks and by \cite[Remark 6.1.11]{ABLBRCS25}.
  This proves that $\alpha_*$ is fully faithful,
  and it suffices to show that $\beta^*$ is fully faithful.
  Consider $\beta\colon X^{\HK/K}_{\le 1}\rightarrow X^{\arith/K}/\varphi^{\itg}_X$.
  Both sides satisfy $h$-hyperdescent by Theorem \ref{fpqcdescentarith} and Remark \ref{hdescentrigid},
  then by de Jong's alteration \cite{dJ96},
  it suffices to deal with smooth $X$.
  By Zariski descent,
  we can furthermore assume that $X$ is étale over $\aff^n$,
  hence $X$ has the following same property as $\aff^n_k$ by étaleness:
  \begin{itemize}
    \item $X$ is a smooth scheme which embeds openly into a proper smooth scheme $Y$ with a proper smooth lift $\mathfrak{Y}$ equipped with Frobenius $\varphi$.
  \end{itemize}
  Let $\mathfrak{Y}_{\eta}$ be the generic fiber of $\mathfrak{Y}$.
  Then by Proposition \ref{Xarithdescription} and Proposition \ref{smallrigidstackdescription},
  and the fact that $X^{\arith/K}$ is a perfect stack (i.e. $X^{\arith/K}\cong\varprojlim_{\varphi_{X}} X^{\arith/K}$),
  the morphism $\beta$ becomes the following shape:
  \[
  \beta\colon \varprojlim_{\varphi}]X[^{\dagger,\dR/K}_{\mathfrak{Y}}/\varphi^{\itg}\longrightarrow \varprojlim_{\varphi}(]X[^{\dagger}_{\mathfrak{Y}}/]X[^{\dagger}_{\mathfrak{Y}^2})/\varphi^{\itg}.
  \]
  By \cite[Proposition 5.4.9]{ABLBRCS25},
  we have $D(\varprojlim_{\varphi}]X[^{\dagger,\dR/K}_{\mathfrak{Y}})\cong \varprojlim_{\varphi} D_!(]X[^{\dagger,\dR/K}_{\mathfrak{Y}})$,
  hence it reduces to proving that the pullback along $\nu\colon ]X[^{\dagger,\dR/K}_{\mathfrak{Y}}\rightarrow X^{\arith/K}$ gives a fully faithful embedding,
  which is just Remark \ref{fullfaithfulanalyticdeRham}.
\end{proof}

It remains to prove the identification of perfect complexes in Theorem \ref{equivalenceperfectcomplex}.
Before proving the equivalence of perfect complexes,
let us make some preparations.

\begin{lemma}\label{lemmaprototype1}
  Let $X_1\subset X_2\subset\cdots$ be a sequence of open immersions of Gelfand stacks.
  Then pullback induces an equivalence of categories $\cat{Perf}(\varinjlim_i X_i)\cong \varprojlim_i \cat{Perf}(X_i)$.
\end{lemma}
\begin{proof}
  The system $\{X_i\}$ is a cover of $\varinjlim_i X_i$,
  so $D(-)$ satisfies descent.
  By computing the Cech nerve,
  we get $D(\varinjlim_i X_i)\cong \varprojlim_i D(X_i)$.
  Now,
  by Fredholm property \ref{Fredholmproperty} of Gelfand stacks,
  the category of perfect complexes is exactly the category of dualizable objects,
  and checking whether an object is dualizable is local,
  cf. \cite[Lemma 2.2.7]{RC24a}.
\end{proof}

We also need the spreading out property of perfect complexes in \cite[Lemma 2.18]{Hau26} and a variation.
Recall the definition of an overconvergent normed divisor in \cite[Definition 2.17]{Hau26}.
\begin{definition}[Overconvergent normed divisor]
  Let $X$ be a Gelfand stack. An overconvergent normed divisor is a closed substack $Z$ of X arising from a Cartesian square
  \[\xymatrix{
    Z\ar[r]\ar[d] & X\ar[d]\\
    \text{*}/\mathbb{T}^{\dagger,\dR}\ar[r] & \aff^{1,\an,\dR}/\mathbb{T}^{\dagger,\dR}
  }\]
  Let $Z_{\epsilon}$ ($\epsilon >0$) be the stack defined by 
  \[\xymatrix{
    Z_{\epsilon}\ar[r]\ar[d] & X\ar[d]\\
    \disk^{\dagger,\le \epsilon,\dR}/\mathbb{T}^{\dagger,\dR}\ar[r] & \aff^{1,\an,\dR}/\mathbb{T}^{\dagger,\dR}
  }\]
\end{definition}

We can make a variation of the definition as follows.
\begin{definition}[Overconvergent normed tube]
  Let $X$ be a Gelfand stack. An overconvergent normed tube is a closed substack $Z$ of X arising from a Cartesian square
  \[\xymatrix{
    Z\ar[r]\ar[d] & X\ar[d]\\
    \disk^{\dagger,\dR}/\mathbb{T}^{\dagger,\dR}\ar[r] & \aff^{1,\an,\dR}/\mathbb{T}^{\dagger,\dR}
  }\]
  Let $Z_{r}$ ($r\ge 1$) be the stack defined by 
  \[\xymatrix{
    Z_{r}\ar[r]\ar[d] & X\ar[d]\\
    \disk^{\dagger,\le r,\dR}/\mathbb{T}^{\dagger,\dR}\ar[r] & \aff^{1,\an,\dR}/\mathbb{T}^{\dagger,\dR}
  }\]
\end{definition}

\begin{lemma}[Spreading out]\label{lemmaprototype2}
  Let $X$ be a nice coverable Gelfand stack,
  cf. \cite[Definition 2.11]{Hau26}.
  Let $Z$ be an overconvergent normed divisor or an overconvergent normed tube of $X$.
  Then we can spread out perfect complexes on $Z$ to a small neighborhood of it,
  i.e. the pullback induces an equivalence of categories:
  \begin{enumerate}
  \item $\cat{Perf}(Z)\cong \varinjlim_{\epsilon>0}\cat{Perf}(Z_{\epsilon})$ in the case of overconvergent normed divisors;
  \item $\cat{Perf}(Z)\cong \varinjlim_{r>1}\cat{Perf}(Z_{r})$ in the case of overconvergent normed tubes.
  \end{enumerate} 
\end{lemma}
\begin{proof}
  For the case of overconvergent normed divisors,
  we refer to \cite[Lemma 2.18]{Hau26}.
  For the case of overconvergent normed tubes,
  the proof is the same as in the proof of \cite[Lemma 2.18]{Hau26},
  after replacing $Z_{\epsilon}$ there by $Z_r$ here everywhere.
\end{proof}

The following lemma is a prototype of identifications of perfect complexes that we will use.

\begin{lemma}\label{prototype}
  Let $\cdots\subset X_{-1}\subset X_0\subset X_1\subset \cdots$ be a sequence of successive open immersions of Gelfand stacks.
  Denote $X_{-\infty}:=\varprojlim_{n\rightarrow-\infty} X_n$ and $X_{\infty}:=\varinjlim_{n\rightarrow\infty} X_n$,
  and assume that $X_{-\infty}$ is an overconvergent normed divisor or an overconvergent normed tube in $X_{\infty}$ with a basis of neighborhoods given as $X_n$, $n\in\itg$.
  Assume $F_i$ is an isomorphism $F_i\colon X_i\cong X_{i+1}$,
  and denote by $F_{-\infty}$ the isomorphism $F_{-\infty}\colon X_{-\infty}\cong X_{-\infty}$ induced by $\varprojlim_{n\rightarrow -\infty} F_n$,
  and by $F_{\infty}\colon X_{\infty}\cong X_{\infty}$ the isomorphism induced by $\varinjlim_{n\rightarrow\infty} F_n$.
  Then the pullback along the natural morphism $\iota\colon X_{-\infty}/F^{\itg}_{-\infty}\rightarrow X_{\infty}/F^{\itg}_{\infty}$ induces an
  equivalence of categories $\iota^*\colon \cat{Perf}(X_{\infty}/F^{\itg}_{\infty})\cong \cat{Perf}(X_{-\infty}/F^{\itg}_{-\infty})$.
  Moreover,
  if the stacks and morphisms are all over a base Gelfand stack $Y$,
  then the equivalence of categories is compatible with pushforwards to $D(Y)$ along maps to $Y$.
\end{lemma}
\begin{proof}
  By Lemma \ref{lemmaprototype1} and Lemma \ref{lemmaprototype2},
  we have $\cat{Perf}(X_{\infty})\cong \varprojlim_{n\rightarrow\infty}\cat{Perf}(X_n)$
  and also $\cat{Perf}(X_{-\infty})\cong \varinjlim_{n\rightarrow-\infty}\cat{Perf}(X_n)$.
  By descent of perfect complexes,
  we have $\cat{Perf}(X_{\infty}/F^{\itg}_{\infty})\cong (\varprojlim_{n\rightarrow\infty}\cat{Perf}(X_n))^{F_{\infty}\text{-equiv}}$ as the $F_{\infty}$-equivariant category,
  and similarly $\cat{Perf}(X_{-\infty}/F^{\itg}_{-\infty})\cong (\varinjlim_{n\rightarrow-\infty}\cat{Perf}(X_n))^{F_{-\infty}\text{-equiv}}$.
  However,
  using the fact that $F_n$ is an isomorphism,
  we obtain that $(\varprojlim_{n\rightarrow\infty}\cat{Perf}(X_n))^{F_{\infty}\text{-equiv}}$
  is naturally equivalent to the equalizer 
  \[
  \operatorname{eq}(\begin{tikzcd}
    \cat{Perf}(X_{n+1})\arrow[r, shift left=1ex, "F_n^*"]\arrow[r, shift left=-1ex, "
    i^*"']& \cat{Perf}(X_{n})
    \end{tikzcd})
  \]
  for some (and hence arbitrary) $n$,
  by simply pulling back to the $n$-th piece.
  Meanwhile $\cat{Perf}(X_{-\infty}/F^{\itg}_{-\infty})\cong (\varinjlim_{n\rightarrow-\infty}\cat{Perf}(X_n))^{F_{-\infty}\text{-equiv}}$
  is also naturally equivalent to the same equalizer for arbitrary $n$,
  by pullback from the $n$-th piece.
  Moreover,
  these equivalences are compatible with the pullback functor $\iota^*$.
  Therefore,
  $\iota^*$ induces a natural equivalence
  $\iota^*\colon \cat{Perf}(X_{\infty}/F^{\itg}_{\infty})\cong \cat{Perf}(X_{-\infty}/F^{\itg}_{-\infty})$.
  Moreover,
  the equivalence is compatible with pushforwards to $D(Y)$ as every aforementioned equivalence is relative over $Y$.
\end{proof}

\begin{proof}[Proof of Theorem \ref{equivalenceperfectcomplex} that $\alpha^*$ is an equivalence]
  Consider the map $\alpha\colon X^{\HK/K}_{\le 1}\rightarrow X^{\HK/K}$.
  Both sides satisfy $h$-hyperdescent by Remark \ref{hdescentrigid},
  then by de Jong's alteration \cite[Theorem 4.1]{dJ96},
  it suffices to deal with smooth $X$ over $k$.
  By Zariski descent,
  we can furthermore assume that $X$ is étale over $\aff^n_k$,
  hence $X$ has the following same property as $\aff^n_k$ by étaleness:
  \begin{itemize}
    \item $X$ is a smooth scheme which embeds openly into a proper smooth scheme $Y$ with proper smooth lift $\mathfrak{Y}$ equipped with a Frobenius lift $\varphi$.
  \end{itemize}
  Let $\mathfrak{Y}_{\eta}$ be the generic fiber of $\mathfrak{Y}$.
  By Proposition \ref{smallrigidstackdescription},
  we have $X^{\HK/K}_{\le 1}\cong\varprojlim_{\varphi}]X[^{\dagger,\dR/K}_{\mathfrak{Y}}/\varphi^{\itg}$.
  On the other hand,
  by Corollary \ref{rigidificationsmoothvariety},
  we have $Y^{\RIG/K}\cong Y^{\RIG/K}_{\le 1}\cong \varprojlim_{\varphi}\mathfrak{Y}_{\eta}^{\dR/K}$,
  and $X^{\RIG/K}$ is an open substack of $Y^{\RIG/K}$ by excision.
  We give a precise description of $X^{\RIG/K}$ as follows.

  We construct a subspace $]X[^{\an,\operatorname{Berk}}_{\mathfrak{Y}}$ of $\M(\mathfrak{Y}_{\eta})$ as follows.
  Locally on $\mathfrak{Y}$,
  we can find $f_1,\ldots,f_n$ such that the complement of $X$ in $Y$ is the vanishing locus of the reductions $\overline{f_i}$.
  Let $]X[^{\an,\operatorname{Berk}}_{\mathfrak{Y}}$ be locally as $\{x\in\M(\mathfrak{Y}_{\eta})\mid \e i,v_x(f_i)>0\}$
  (Recall that the Berkovich tube $]X[^{\operatorname{Berk}}_{\mathfrak{Y}}$ is locally given as $\{x\in\M(\mathfrak{Y}_{\eta})\mid \e i,v_x(f_i)\ge 1\}$).
  Let $]X[^{\an,\dR/K}_{\mathfrak{Y}}$ be the substack of $\mathfrak{Y}_{\eta}^{\dR/K}$ localizing at this topological subspace $]X[^{\an,\operatorname{Berk}}_{\mathfrak{Y}}$
  under the Betti realization to Berkovich spectra in Construction \ref{tubes}.
  Notice that the Frobenius $\varphi$ also acts on $]X[^{\operatorname{Berk}}_{\mathfrak{Y}}$ hence on $]X[^{\an,\dR/K}_{\mathfrak{Y}}$ by construction.

  We claim that we have an isomorphism $X^{\RIG/K}\cong \varprojlim_{\varphi}]X[^{\an,\dR/K}_{\mathfrak{Y}}$,
  and the morphism $X^{\RIG/K}\rightarrow Y^{\RIG/K}$ is given as $\varprojlim_{\varphi}]X[^{\an,\dR/K}_{\mathfrak{Y}}\rightarrow \varprojlim_{\varphi}\mathfrak{Y}_{\eta}^{\dR/K}$.
  Indeed,
  since the construction of $\varprojlim_{\varphi}]X[^{\an,\dR/K}_{\mathfrak{Y}}$ comes from Betti localization,
  it suffices to prove it locally,
  and hence we can assume them to be affine.
  In the affine case,
  the open immersion $X\rightarrow Y$ must be of the form $X=\Spec A[f_1^{-1},\ldots,f^{-1}_n]\rightarrow Y=\Spec A$,
  hence is a Cartesian pullback from the case $\aff^n-\{0\}\rightarrow \aff^n$ where $Y\rightarrow\aff^n$ is given by $(f_1,\ldots,f_n)$.
  By Remark \ref{rigidKunneth},
  we have a Cartesian diagram
  \[\xymatrix{
    X^{\RIG/K}\ar[r]\ar[d] & Y^{\RIG/K}\ar[d]\\
    (\aff^n-\{0\})^{\RIG/K}\ar[r] & (\aff^n)^{\RIG/K}
  }\]
  By excision and Example \ref{rigidstackadditive},
  we know that $(\aff^n-\{0\})^{\RIG/K}\rightarrow (\aff^n)^{\RIG/K}$ is given as $\varprojlim_{\varphi}(\aff^n-\{0\})^{\an,\dR/K}\rightarrow \varprojlim_{\varphi}\aff^{n,\an,\dR/K}$,
  which is exactly the Betti localization coming from $]\aff^n-\{0\}[^{\an,\operatorname{Berk}}_{\aff^{n,\an}}\rightarrow\M(\aff^{n,\an})$.
  Therefore,
  we conclude that $X^{\RIG/K}\cong\varprojlim_{\varphi}]X[^{\an,\dR/K}_{\mathfrak{Y}}\rightarrow Y^{\RIG/K}\cong \varprojlim_{\varphi}\mathfrak{Y}_{\eta}^{\dR/K}$.

  Then the morphism $\alpha$ is the closed embedding 
  $\alpha\colon \varprojlim_{\varphi}]X[^{\dagger,\dR/K}_{\mathfrak{Y}}/\varphi^{\itg}\rightarrow \varprojlim_{\varphi}]X[^{\an,\dR/K}_{\mathfrak{Y}}/\varphi^{\itg}$.
  We claim that $\alpha$ is of the form $\iota$ in Lemma \ref{prototype}.
  It satisfies a Cartesian diagram via Betti localization
  \[\xymatrix{
    \varprojlim_{\varphi}]X[^{\dagger,\dR/K}_{\mathfrak{Y}}/\varphi^{\itg}\ar[r]\ar[d] & \varprojlim_{\varphi}]X[^{\an,\dR/K}_{\mathfrak{Y}}/\varphi^{\itg}\ar[d]\\
    \varprojlim_{\varphi}\underline{]X[_{\mathfrak{Y}}^{\operatorname{Berk}}}/\varphi^{\itg}\ar[r] & \varprojlim_{\varphi}\underline{]X[^{\an,\operatorname{Berk}}_{\mathfrak{Y}}}/\varphi^{\itg}
  }\]
  We construct a subspace $]X[_{n,\mathfrak{Y}}^{\operatorname{Berk}}$ of $\M(\mathfrak{Y}_{\eta})$ defined as follows.
  Locally we choose $f_1,\ldots, f_n$ such that the complement of $X$ in $Y$ is the vanishing locus of the reductions $\overline{f_i}$.
  Let $]X[_{n,\mathfrak{Y}}^{\operatorname{Berk}}$ be locally given by 
  $\{x\in\M(\mathfrak{Y}_{\eta})\mid \e i, v_x(f_i)> p^{-p^n}\}$.
  Let $]X[_{n,\mathfrak{Y}}^{\dR/K}$ be the substack of $\mathfrak{Y}_{\eta}^{\dR/K}$ under the Betti localization.
  Notice that the Frobenius lifting $\varphi$ takes $]X[_{n-1,\mathfrak{Y}}^{\operatorname{Berk}}$ to $]X[_{n,\mathfrak{Y}}^{\operatorname{Berk}}$,
  and hence $]X[_{n-1,\mathfrak{Y}}^{\dR/K}$ to $]X[_{n,\mathfrak{Y}}^{\dR/K}$.
  Denote 
  \[
    X_n=\varprojlim_{\varphi}(]X[_{n,\mathfrak{Y}}^{\dR/K}\leftarrow ]X[_{n-1,\mathfrak{Y}}^{\dR/K}\leftarrow\cdots).
  \]
  Then $X_{-\infty}=\varprojlim_{\varphi}]X[_{\mathfrak{Y}}^{\dagger,\dR/K}$,
  and $X_{\infty}=\varprojlim_{\varphi}]X[_{\mathfrak{Y}}^{\an,\dR/K}$.
  Both can be verified under Betti localization,
  and on the hierarchy of topological spaces,
  $]X[_{\mathfrak{Y}}^{\operatorname{Berk}}$ is the limit of $]X[_{n,\mathfrak{Y}}^{\operatorname{Berk}}$,
  while $]X[_{\mathfrak{Y}}^{\an,\operatorname{Berk}}$ is the colimit of $]X[_{n,\mathfrak{Y}}^{\operatorname{Berk}}$ by the definitions.
  Moreover,
  as $n\rightarrow \infty$,
  $p^{-p^n}\rightarrow 0$,
  hence $\{X_n\}$ form the neighborhoods of $X_{-\infty}$ which is an overconvergent normed tube in $X_{\infty}$.
  This proves the claim that $\alpha$ is of the form $\iota$ in Lemma \ref{prototype}.
  Hence by Lemma \ref{prototype},
  $\alpha^*$ induces an equivalence of categories of perfect complexes.
\end{proof}

\begin{remark}
  Let $X$ be a smooth scheme which embeds openly into a proper smooth scheme $Y$ with a proper smooth lift $\mathfrak{Y}$ equipped with a Frobenius lift.
  In the proof above,
  we know that under a choice of coordinates,
  $X^{\RIG}_{\le 1}$ is an overconvergent normed tube of $X^{\RIG}$,
  whence the subscript $(-)_{\le 1}$ comes.
\end{remark}

\begin{proof}[Proof of Theorem \ref{equivalenceperfectcomplex} that $\beta^*$ is an equivalence]
  Consider $\beta\colon X^{\HK/K}_{\le 1}\rightarrow X^{\arith/K}/\varphi^{\itg}_X$.
  Both sides satisfy $h$-hyperdescent by Theorem \ref{fpqcdescentarith} and Remark \ref{hdescentrigid},
  then by de Jong's alteration \cite{dJ96},
  it suffices to deal with smooth $X$.
  By Zariski descent,
  we can furthermore assume that $X$ is étale over $\aff^n$,
  hence $X$ has the following same property as $\aff^n_k$ by étaleness:
  \begin{itemize}
    \item $X$ is a smooth scheme which embeds openly into a proper smooth scheme $Y$ with proper smooth lift $\mathfrak{Y}$ equipped with Frobenius $\varphi$.
  \end{itemize}
  Let $\mathfrak{Y}_{\eta}$ be the generic fiber of $\mathfrak{Y}$.
  Then by Proposition \ref{Xarithdescription} and Proposition \ref{smallrigidstackdescription},
  and the fact that $X^{\arith/K}$ is a perfect stack,
  the morphism $\beta$ fits into the following commutative diagram 
  \[
  \xymatrix{
    & \varprojlim_{\varphi}]X[_{\mathfrak{Y}}^{\dagger}/\varphi^{\itg}\ar^{\pi_1}[dr]\ar_{\pi_2}[dl] &\\
    X^{\HK/K}_{\le 1}\ar^{\beta}[rr] & & X^{\arith/K}/\varphi^{\itg}_X
  }
  \]
  with the $n$-th level of Cech nerves of $\pi_1$ computed as $\varprojlim_{\varphi}]X[_{\mathfrak{Y}^{n+1}}^{\dagger}$,
  and of $\pi_2$ computed as $\varprojlim_{\varphi}\Delta^{\dagger}_{n+1}(]X[_{\mathfrak{Y}}^{\dagger})$.
  We claim that at each level $n$,
  the induced morphism $\beta_n\colon\varprojlim_{\varphi}\Delta^{\dagger}_{n+1}(]X[_{\mathfrak{Y}}^{\dagger})/\varphi^{\itg}\rightarrow \varprojlim_{\varphi}]X[_{\mathfrak{Y}^{n+1}}^{\dagger}/\varphi^{\itg}$ 
  on Cech nerves is of the form $\iota$ in Lemma \ref{prototype}.

  Denote by $\mathfrak{J}$ the subsheaf consisting of integral sections of the ideal sheaf of the diagonal closed embedding $\Delta_n\colon ]X[_{\mathfrak{Y}}\rightarrow ]X[_{\mathfrak{Y}}^{\times n}$,
  i.e. writing $]X[_{\mathfrak{Y}}$ as the generic fiber of a formal scheme $\Xf$ then $\mathfrak{J}$ is the ideal sheaf of the diagonal embedding $\Xf\rightarrow\Xf^n$.
  Then $\Delta^{\dagger}_n(]X[_{\mathfrak{Y}}^{\dagger})$ is the substack of $]X[_{\mathfrak{Y}}^{\dagger,\times n}$ corresponding to the closed subspace $\{x\in ]X[_{\mathfrak{Y}}^{\operatorname{Berk},\times n}\mid v_x(\mathfrak{J})=0\}$ under Betti localization,
  and $]X[_{\mathfrak{Y}^{n}}^{\dagger}$ is the substack of $]X[_{\mathfrak{Y}}^{\dagger,\times n}$ corresponding to the subspace $\{x\in ]X[_{\mathfrak{Y}}^{\operatorname{Berk},\times n}\mid v_x(\mathfrak{J})< 1\}$ under Betti localization.
  Then we denote by $\Delta^{m}_n(]X[_{\mathfrak{Y}}^{\dagger})$ the substack of $]X[_{\mathfrak{Y}}^{\dagger,\times n}$ corresponding to the subspace $\{x\in ]X[_{\mathfrak{Y}}^{\operatorname{Berk},\times n}\mid v(\mathfrak{J})< p^{-p^{-m}}\}$ under Betti localization.
  Then the Frobenius map $\varphi$ takes $\Delta^{m-1}_n(]X[_{\mathfrak{Y}}^{\dagger})$ to $\Delta^{m}_n(]X[_{\mathfrak{Y}}^{\dagger})$ and we denote
  \[
  X_m=\varprojlim_{\varphi}(\Delta^{m}_n(]X[_{\mathfrak{Y}}^{\dagger})\leftarrow \Delta^{m-1}_n(]X[_{\mathfrak{Y}}^{\dagger})\leftarrow\cdots).
  \]
  Then $X_{-\infty}=\varprojlim_{\varphi}\Delta^{\dagger}_n(]X[_{\mathfrak{Y}}^{\dagger})$,
  and $X_{\infty}=\varprojlim_{\varphi}]X[_{\mathfrak{Y}^{n}}^{\dagger}$,
  via verifying on the underlying topological spaces under Betti localization.
  Moreover,
  as $p^{-p^{-m}}\rightarrow 0$ when $m\rightarrow-\infty$,
  we obtain that $\{X_m\}$ form the neighborhoods of $X_{-\infty}$ which is an overconvergent normed divisor in $X_{\infty}$.
  Hence $\varprojlim_{\varphi}\Delta^{\dagger}_n(]X[_{\mathfrak{Y}}^{\dagger})/\varphi^{\itg}\rightarrow \varprojlim_{\varphi}\Delta^{\circ}_n(]X[_{\mathfrak{Y}}^{\dagger})/\varphi^{\itg}$ is of the form $\iota$ in Lemma \ref{prototype}.

  By Lemma \ref{prototype},
  $\beta_n^*$ induces an equivalence of categories of perfect complexes.
  By descent of perfect complexes along Cech nerves,
  we obtain that $\beta^*$ also induces an equivalence of categories of perfect complexes.
\end{proof}

\begin{remark}\label{compatibilityperfectcomplex}
  Let $f\colon X\rightarrow Y$ be a morphism.
  Then the compatibility of pullbacks with Theorem \ref{equivalenceperfectcomplex} is given by the following commutative diagram:
  \[\xymatrix{
    \cat{Perf}(Y^{\arith/K}/\varphi_Y^{\itg})\ar[r]^{\cong}\ar[d]_{f^{\arith/K,*}} & \cat{Perf}(Y^{\HK/K})\ar[d]^{f^{\HK/K,*}}\\
    \cat{Perf}(X^{\arith/K}/\varphi_X^{\itg})\ar[r]^{\cong} & \cat{Perf}(X^{\HK/K})
  }
  \]
  Note that in general there is no compatibility of pushforwards with Theorem \ref{equivalenceperfectcomplex},
  except when $Y=\Spec k$,
  in which case we have a commutative diagram:
  \[\xymatrix{
    \cat{Perf}(X^{\arith/K}/\varphi_{X}^{\itg})\ar[r]^{\cong}\ar[d]_{f^{\arith/K}_*} & \cat{Perf}(X^{\HK/K})\ar[d]^{f^{\HK/K}_*}\\
    D_{\blacksquare}(K)^{\varphi_K\text{-equiv}}\ar[r]^{\cong} & D_{\blacksquare}(K)^{\varphi_K\text{-equiv}}
  }
  \]
  Indeed,
  both functors are identified with $f^{\HK/K}_{\le 1,*}$ under the equivalence of perfect complexes:
  \begin{enumerate}
    \item For $f^{\arith/K}_*$,
    we have $f^{\HK/K}_{\le 1,*}\beta^*M=f^{\arith/K}_*\beta_*\beta^*M=f^{\arith/K}_*M$,
    as $\beta_*\beta^*M\cong M$ by the full faithfulness of $\beta^*$.
    \item For $f^{\HK/K}_*$,
    by the descent method in the proof,
    we can assume that $\alpha\colon X^{\HK/K}_{\le 1}\rightarrow X^{\HK/K}$ is of the form $\iota$ in Lemma \ref{prototype}.
    Then $f^{\HK/K}_{\le 1,*}$ and $f^{\HK/K}_*$ compute the same pushforward by Lemma \ref{prototype},
    applying to $Y=\operatorname{GSpec}K/\varphi_K^{\itg}$.
  \end{enumerate}
\end{remark}

\subsection{The Hyodo--Kato approach to $p$-adic coefficients}\label{Sectionrigidstackrigidcoh}
As a direct application of Theorem \ref{equivalenceperfectcomplex},
we can indeed show that $X^{\HK/K}$ (in fact, also $X^{\HK/K}_{\le 1}$) is a stacky approach to rigid cohomology and overconvergent $F$-isocrystals.
Precisely,
we prove:

\begin{theorem}\label{rigidstackrigidcohomology}
  Let $X$ be a scheme of finite type over a perfect field $k$.
  Let $K$ admit a Frobenius lift $\varphi_K$ and be perfect.
  Then there exists a $t$-structure on $\cat{Perf}(X^{\HK/K})$ and we have an equivalence of categories
  \[
  \cat{Perf}^{\heartsuit}(X^{\HK/K})\cong \cat{F-Isoc}^{\dagger}(X/K).
  \]
  Moreover,
  let $f^{\HK/K}\colon X^{\HK/K}\rightarrow (\Spec k)^{\HK/K}=\operatorname{GSpec}K/\varphi_K^{\itg}$,
  then we have a natural isomorphism
  \[
  f_*^{\HK/K}1\cong R\Gamma_{\rig}(X/K)\in  D_{\blacksquare}(K)^{\varphi_K\textnormal{-equiv}}.
  \]
\end{theorem}
\begin{proof}
  By Theorem \ref{equivalenceperfectcomplex},
  we have an equivalence of categories which preserves $t$-structures:
  \[
  \cat{Perf}(X^{\HK/K})\cong\cat{Perf}(X^{\arith/K}/\varphi_{X}^{\itg}).
  \]
  By Remark \ref{compatibilityperfectcomplex},
  we know that $f^{\HK/K}_*1\cong f^{\arith/K}_*1\in D_{\blacksquare}(K)$.
  Then by Remark \ref{Frobequivarithcrystal},
  we know that $\cat{Perf}^{\heartsuit}(X^{\HK/K})\cong \cat{F-Isoc}^{\dagger}(X/K)$ and $f_*^{\HK/K}1$ computes the rigid cohomology $R\Gamma_{\rig}(X/K)$.
\end{proof}

\begin{remark}
  Combining with Theorem \ref{PoincaredualityrigidFrobenius},
  we recover the Poincaré duality of rigid cohomology again.
\end{remark}

\begin{remark}[Comparison between crystalline cohomology and rigid cohomology]
  As a continuation of Remark \ref{integralanalyticprismatization},
  if we consider the integral analytic prismatization of $X$,
  then the $\rt_p$-part of it gives the rational analytic prismatization of $X$,
  which is $X^{\RIG}$,
  and the $\ff_p$-part of it gives an analytic version of the crystalline stack.
  Intuitively,
  this picture gives a stacky interpretation of the comparison between rigid cohomology and crystalline cohomology for proper smooth $X$,
  where properness and smoothness ensure that the pushforward is a perfect complex whence we can do the comparison.
  We learnt this from Maximilian Hauck.
\end{remark}

\begin{remark}[Passing to non-perfect field]\label{passingtononperfect}
  Let $X$ be a scheme of finite type over a (non-perfect) field $k$,
  and let $K$ be equipped with a Frobenius lift.
  Then applying Theorem \ref{rigidstackrigidcohomology},
  the $X^{\HK/K_{\perf}}_{k_{\perf}}$ satisfies that $\cat{Perf}^{\heartsuit}(X^{\HK/K_{\perf}}_{k_{\perf}})\cong \cat{Perf}^{\heartsuit}(X^{\arith/K_{\perf}}_{k_{\perf}}/\varphi^{\itg}) \cong \cat{F-Isoc}^{\dagger}(X_{k_{\perf}}/K_{\perf})$,
  and $f_*^{\HK/K_{\perf}}1\cong f_*^{\arith/K_{\perf}}\cong R\Gamma_{\rig}(X_{k_{\perf}}/K_{\perf})$.
  We always have an embedding from $\cat{F-Isoc}^{\dagger}(X/K)$ to $\cat{F-Isoc}^{\dagger}(X_{k_{\perf}}/K_{\perf})$,
  sending $\E$ to $\E_{K_{\perf}}$,
  and an isomorphism $R\Gamma_{\rig}(X_{k_{\perf}}/K_{\perf},\E_{K_{\perf}})\cong  R\Gamma_{\rig}(X/K,\E)\otimes_K K_{\perf}$.
\end{remark}

\begin{remark}\label{regularDmodule}
  The category of perfect complexes over $X^{\HK/K}$ is related to convergent $F$-log-isocrystals introduced in \cite{Shi00} and \cite{Shi02},
  and hence the identification of perfect complexes in Theorem \ref{equivalenceperfectcomplex} can be viewed as a version of the semistable reduction theorem,
  cf. \cite[Theorem 5.0.1]{Ked11}.
  Let us consider the example of $X=\aff^1_{k}$.
  We have from Example \ref{rigidstackadditive} that
  \[
  \aff^{1,\HK/K}_k\cong\varprojlim_{\varphi}\aff^{1,\an,\dR/K}_K/\varphi^{\itg}.
  \]
  By spreading out of perfect complexes,
  cf. \cite[0BC7]{Stacks},
  we have an equivalence of categories between $\cat{Perf}(\aff^{1,\HK/K}_k)$ and $\cat{Perf}(\aff^{1,\an,\dR/K}_K)^{\varphi\text{-equiv}}$.
  From \cite[Remark 5.2.3]{ABLBRCS25},
  this is equivalent to the category of perfect complexes of $(\varphi,\nabla)$-modules over $\aff^{1,\an}_K$.
  Let $\E$ be a vector bundle of $(\varphi,\nabla)$-modules over $\aff^{1,\an}_K$.
  We claim that after the pullback along some generically finite surjection $f\colon Y\rightarrow\proj^{1}_K$ étale near $\infty$,
  the $f^*\E$ determines a regular connection\footnote{An analogue for complex manifolds also holds and becomes much simpler.
  In fact,
  vector bundles with flat connection on the analytification $X^{\an}$ of a $\cn$-variety $X$ are equivalent to vector bundles with \emph{regular} flat connection on $X$,
  which is a deep theorem à la Deligne \cite{Del70},
  hence it is not necessary to perform an alteration.
  In the language of \cite{Sch24a},
  they are also equivalent to vector bundles on the algebraic de Rham stack (or the analytic de Rham stack) of $X^{\an}$.},
  i.e. admits a logarithmic structure along the divisor $f^{-1}(\{\infty\})$.
  First,
  let $\disk^{\circ,\times}_{\infty,K}$ be the punctured open unit disk centered at $\infty$.
  Then the pullback of $\E$ along $\disk^{\circ,\times}_{\infty,K}\rightarrow\aff^{1,\an}_K$ gives an object in $\cat{VB}(\varprojlim_{\varphi}\disk^{\circ,\times,\dR}_{\infty,K}/\varphi^{\itg})$,
  which is equivalent to a $(\varphi,\nabla)$-module $\F$ over the Robba ring $\R_K$ by spreading out,
  cf. \cite[Corollary 7.2.10]{ABLBRCS25}.
  By the $p$-adic monodromy theorem in \cite[Theorem 1.1]{Ked04};
  revisited in \cite[Section 7]{ABLBRCS25},
  there exists a finite separable extension $f\colon k(\!(t)\!)\rightarrow k'(\!(u)\!)$ such that $\F\otimes_{\R_K}\R_{K'}$ is unipotent,
  which determines a generically finite surjection $\proj_{K'}^1\rightarrow \proj_{K}^{1}$ étale near $\infty$
  such that the pullback of the connection on $\E$ around $\infty\in \proj_{K}^{1}$ along this map is unipotent,
  which belongs to the category $\cat{F-Isoc}(k'[\![u]\!]_{\log})$ of logarithmic connections with Frobenius structure.
  
  Moreover,
  one can choose $f(t)$ to be a rational polynomial of $u$,
  and such an extension also determines an alteration $g\colon \proj^1_{k'}\rightarrow \proj^1_k$.
  This shows that one can make the generically finite surjection $\proj_{K'}^1\rightarrow \proj_{K}^{1}$ induced from an alteration of algebraic varieties.
\end{remark}

\begin{remark}
  The full category $D(X^{\HK/K})$ of quasi-coherent sheaves on $X^{\HK/K}$ can be described as analytic $D$-modules over the Fargues--Fontaine curve $\FF_X$,
  hence will not give the theory of arithmetic $D$-modules.
  Nevertheless,
  the category of solid arithmetic $D$-modules with Frobenius structure over $X$ embeds fully faithfully into this category by Theorem \ref{equivalenceperfectcomplex}.
\end{remark}

\section{Applications to arithmetic}\label{Applications}

We give some applications to the theory of $p$-adic coefficients by using the stacky approach method.

\subsection{Berthelot's conjecture}\label{SectionBerthelotconj}
In this section,
we give the proof of Berthelot's conjecture in \cite{Ber86}.
More precisely,
we prove the following theorem:
\begin{theorem}[Berthelot's conjecture]\label{Berthelotconj}
  Let $f\colon X\rightarrow Y$ be a proper smooth morphism between schemes of finite type over $k$.
  Then the rigid cohomology of an overconvergent isocrystal over $X$ along $f$ is also an overconvergent isocrystal over $Y$.
  More precisely,
  for an overconvergent isocrystal $\E$ over $X$,
  the collection of pushforwards $\{R^if_{\rig,*}\E|_{(U,\overline{U},P)}\}_{(U,\overline{U},P)}$ comes from an overconvergent isocrystal over $Y$.
\end{theorem}
\begin{proof}
  By Theorem \ref{arithrigidcoh2},
  we know that $\cat{Perf}^{\heartsuit}(X^{\arith/K})\cong \cat{Isoc}^{\dagger}(X/K)$,
  hence we can view $\E$ as an object in $\cat{Perf}(X^{\arith/K})$.
  By Corollary \ref{Poincaréarith},
  we have that the pushforward $f_*^{\arith/K}$ preserves perfect complexes,
  i.e. $f_*^{\arith/K}\colon\cat{Perf}(X^{\arith/K})\rightarrow\cat{Perf}(Y^{\arith/K})$.
  Hence $f^{\arith/K}_*\E$ is an object in $\cat{Perf}(Y^{\arith/K})$.
  The $i$-th cohomological piece of $f^{\arith/K}_*\E$ is an object in $\cat{Perf}^{\heartsuit}(Y^{\arith/K})$,
  hence is an overconvergent isocrystal by Theorem \ref{arithrigidcoh2}.
  Also by Theorem \ref{arithrigidcoh2},
  we know that for a proper smooth frame $(U,\overline{U},P)$ of a Zariski open subscheme $U$ of $X$,
  the realization of $f^{\arith/K}_*\E$ computes the rigid cohomology of $\E$ along $f$,
  i.e. let $\pi\colon]U[^{\dagger}_P\rightarrow X^{\arith/K}$,
  then
  \[
  \pi^*f_*^{\arith/K}\E(*)=Rf_{\rig,*}\E|_{(U,\overline{U},P)}\in D^b(\O_{]U[_P^{\dagger}}\cat{-Mod}).
  \]
  The functor $\pi^*$ is $t$-exact on perfect complexes,
  hence $H^i(\pi^*f_*^{\arith/K}\E)=\pi^*H^i(f_*^{\arith/K}\E)$,
  which is the pullback of an overconvergent isocrystal along $\pi$.
  This proves that $H^i(\pi^*f_*^{\arith/K}\E)=\pi^*H^i(f_*^{\arith/K}\E)$ is a finite locally free $\O_{]U[_P^{\dagger}}$-module,
  and therefore,
  the complex $\pi^*f_*^{\arith/K}\E$ splits.
  Then the functor $(-)(*)$ of taking underlying modules preserves the cohomological pieces of $\pi^*f_*^{\arith/K}\E$,
  i.e. $H^i(\pi^*f_*^{\arith/K}\E)(*)\cong H^i(\pi^*f_*^{\arith/K}\E(*))$.
  Therefore,
  we have 
  \[
    \pi^*H^i(f_*^{\arith/K}\E)(*)\cong H^i(\pi^*f_*^{\arith/K}\E)(*)\cong H^i(\pi^*f_*^{\arith/K}\E(*))\cong R^if_{\rig,*}\E|_{(U,\overline{U},P)}.
  \]
  This proves Berthelot's conjecture that the $i$-th cohomological piece $H^i(f^{\arith/K}_*\E)$ is an overconvergent isocrystal over $Y$,
  whose realization to $(U,\overline{U},P)$ computes $\{R^if_{\rig,*}\E|_{(U,\overline{U},P)}\}_{(U,\overline{U},P)}$.
\end{proof}

\subsection{Fourier transform of arithmetic $D$-modules}\label{FourierHuyghe}
We apply the Cartier duality in Section \ref{Cartierduality} to study the Fourier transform of arithmetic $D$-modules,
recovering some results of \cite{Huy04}.
Applying Theorem \ref{Xarithcoefficients} to $\aff^1_{\ff_p}$ and $L=\rt_p(\pi)$,
we have the following result.

\begin{discussion}\label{A1arithcoefficient}
  We have an equivalence of categories from Theorem \ref{arithcoefficientaffinecase}:
  \[
  D(\aff^{1,\arith/L}_{\ff_p})\cong \cat{LMod}_{\D^{\dagger}_{\proj^1_L}({}^{\dagger}\infty)}(D(\disk^{\dagger}_L))\cong \cat{LMod}_{D^{\dagger}(\infty)_L}(D_{\blacksquare}(\rt_p)).
  \]
  Here $D^{\dagger}(\infty)_L$ is the following solid $\rt_p$-algebra
  \[D^{\dagger}(\infty)_L=\left\{\sum_{m,n=0}^{\infty} a_{m,n}T^m\frac{\partial^n}{n!}\mid a_{m,n}\in L,\e c>0,\eta<1,\f n,m,|a_{m,n}|<\eta^{m+n}\right\}.\]
  Under this equivalence:
  \begin{enumerate}
    \item The tensorial symmetric monoidal structure on $D(\aff^{1,\arith/L}_{\ff_p})$ identifies with the tensor product of arithmetic $D$-modules.
    \item The morphism $f\colon \aff^{1,\arith/L}_{\ff_p}\rightarrow\operatorname{GSpec} L$ identifies $f^*$ with $f^{\Delta}$,
    and $f_*$ with $f_+[-1]$.
    \item Denote $\pi\colon \disk^{\dagger}_L\rightarrow \aff^{1,\arith/L}_{\ff_p}=\disk^{\dagger}_L/\disk^{\circ}_L$,
    then $\pi^*\colon D(\aff^{1,\arith/L}_{\ff_p})\rightarrow D(\disk^{\dagger}_L)$ identifies with viewing a $D^{\dagger}(\infty)_L$-module as a $L\langle T\rangle^{\dagger}$-module.
  \end{enumerate}
  Moreover,
  from Theorem \ref{arithrigidcoh2}
  one can identify the following full subcategories
  \[
  \cat{Perf}^{\heartsuit}(\aff^{1,\arith/L}_{\ff_p})\cong \cat{Isoc}^{\dagger}(\aff^1_{\ff_p}/L).
  \]
\end{discussion}

\begin{construction}[Dwork's isocrystal]
  The exponential sheaf $\Exp_{\pi}$ defining the Fourier--Mukai transform in Section \ref{Cartierduality} is a vector bundle on $\aff^{1,\arith/L}_{\ff_p}$.
  It corresponds to an overconvergent isocrystal under \ref{arithrigidcoh2}.
  Unraveling the definition,
  it is given by Dwork's isocrystal $\L_{\pi}$,
  where $\L_{\pi}=L\langle T\rangle^{\dagger}\cdot e$,
  and $\nabla e=\pi e\otimes dT$.
  We have a Cartesian diagram 
  \[\xymatrix{
    \disk^{\dagger}\ar[r]^{\pi}\ar[d]_p & \disk^{\dagger}/\disk^{\circ}\ar[d]^g\\
    \text{*}\ar[r]_{f} & B\disk^{\circ}
  },\]
  Then we have 
  \[p_!\pi^*\Exp_{\pi}=p_!\pi^*g^*\exp_{\pi}^*\O(1)=f^*g_!g^*\exp_{\pi}^*\O(1)=f^*(g_!g^*1\otimes \exp_{\pi}^*\O(1))=f^*g_!g^*1\otimes f^*\exp_{\pi}^*\O(1).\]
  The $\partial$-action on $\Exp_{\pi}$ is exactly the one on $f^*g_!g^*1\otimes f^*\exp_{\pi}^*\O(1)$ given by monadicity applying to $f^*$.
  However,
  $f^*g_!g^*1$ has trivial $\partial$-action and $f^*\exp_{\pi}^*\O(1)$ has the $\partial$-action $\partial \cdot e=\pi e$.
  Indeed,
  the $\O(1)$ gives rise to a one-dimensional $L$-vector space $L\cdot e$ such that the $\O(\gff)=L[t^{\pm 1}]$-comodule structure is given by $\rho\colon e\mapsto e\otimes t$.
  Hence the vector bundle $\exp^*_{\pi}\O(1)\in D(B\disk^{\circ})$ gives rise to a one-dimensional $\O(\disk^{\dagger})$-module,
  or equivalently a one-dimensional $\O(\disk^{\circ})$-comodule,
  which has the $\O(\disk^{\circ})$-comodule structure as $\rho\colon e\mapsto e\otimes \exp(\pi t)$,
  or equivalently the module structure as $\partial \cdot e=\pi e$.
\end{construction}

Now we establish the Fourier transform for solid arithmetic $D$-modules.
We start with some preparations.

\begin{lemma}\label{compatibilityFM2}
  There is a commutative diagram in $\cat{bCAlg}(K_{D,L})$:
  \[\xymatrix{
    [\disk^{\dagger}/\disk^{\circ}]_!\ar[r]^{\operatorname{FM}}\ar[d]_{g_!} & [\disk^{\dagger}/\disk^{\circ}]^*\ar[d]^{f^*}\\
    [B\disk^{\circ}]_!\ar[r]^{\operatorname{FM}'} & [\disk^{\dagger}]^*
  }\]
  where $g\colon \disk^{\dagger}/\disk^{\circ}\rightarrow B\disk^{\circ}$ and $f\colon \disk^{\dagger}\rightarrow\disk^{\dagger}/\disk^{\circ}$.
\end{lemma}
\begin{proof}
  The proof is quite similar to the proof of Lemma \ref{compatibilityFM}.
  We use the following diagram chasing:
  \[\xymatrix{
    \disk^{\dagger}/\disk^{\circ}\times \disk^{\dagger}/\disk^{\circ}\ar[r]^p\ar[d]_q & \disk^{\dagger}/\disk^{\circ}\ar[dr]^g & \\
    \disk^{\dagger}/\disk^{\circ}& \disk^{\dagger}\times \disk^{\dagger}/\disk^{\circ}\ar[ul]^F\ar[d]_Q\ar[rd]^G\ar[u]_S & B\disk^{\circ}\\
    & \disk^{\dagger}\ar[ul]^f & \disk^{\dagger}\times B\disk^{\circ}\ar[l]^r\ar[u]_s
  }\]
  where the middle two squares are further Cartesian.
  Denote the Fourier--Mukai kernels on $\disk^{\dagger}/\disk^{\circ}\times \disk^{\dagger}/\disk^{\circ}$
  and $\disk^{\dagger}\times B\disk^{\circ}$ respectively by $\L$ and $\K$.
  Then starting from $M\in D(\disk^{\dagger}/\disk^{\circ})$,
  we can compute 
  \[
    f^*\circ\operatorname{FM}(M)=f^*q_!(p^*M\otimes\L)
    =Q_!F^*(p^*M\otimes\L)=Q_!(S^*M\otimes F^*\L),
  \]
  and 
  \[
    \operatorname{FM}'\circ g_!(M)=r_!(s^*g_!M\otimes\K)=r_!(G_!S^*M\otimes \K)=r_!G_!(S^*M\otimes G^*\K)=Q_!(S^*M\otimes G^*\K).
  \]
  Hence the commutativity of the diagram is reduced to proving that $F^*\L\cong G^*\K$.
  Note that both Fourier--Mukai kernels descend to their tensor products,
  and that we have a commutative diagram 
  \[\xymatrix{
    \disk^{\dagger}\otimes \disk^{\dagger}/\disk^{\circ}\ar[r]\ar[d] & \disk^{\dagger}/\disk^{\circ}\otimes \disk^{\dagger}/\disk^{\circ}\ar[d]\\
    \disk^{\dagger}\otimes B\disk^{\circ} \ar[r] & B\gff
  }\]
  Computing in both directions gives the same morphism,
  namely the exponential map $\exp_{\pi}\colon \disk^{\dagger}\otimes\disk^{\circ}\rightarrow\gff$ on $(-1)$-th degree,
  cf. \cite[Construction 3.2.19]{RC25b},
  and the trivial map $\disk^{\dagger}\otimes\disk^{\dagger}\rightarrow *$ on $0$-th degree.
  This proves the identification of pullbacks of Fourier--Mukai kernels on both sides to $\disk^{\dagger}\otimes \disk^{\dagger}/\disk^{\circ}$,
  hence also to $\disk^{\dagger}\times \disk^{\dagger}/\disk^{\circ}$.
\end{proof}

We obtain the following solid version of Fourier transform of arithmetic $D$-modules appearing before in \cite{Huy04}.

\begin{theorem}\label{FouriertransformarithmeticDmodule}
  The Fourier--Mukai transform in Theorem \ref{CartierarithmeticdR}
  induces an equivalence: 
  \[
  \operatorname{Four}_{\pi}\colon \cat{LMod}_{\D^{\dagger}_{\proj^1_L}({}^{\dagger}\infty)}(D(\disk^{\dagger}_L))\overset{\cong}{\longrightarrow} \cat{LMod}_{\D^{\dagger}_{\proj^1_L}({}^{\dagger}\infty)}(D(\disk^{\dagger}_L)).
  \]
  Moreover,
  it has the following description under the equivalence of categories $\cat{LMod}_{\D^{\dagger}_{\proj^1_L}({}^{\dagger}\infty)}(D(\disk^{\dagger}_L))\cong \cat{LMod}_{D^{\dagger}(\infty)_L}(D_{\blacksquare}(\rt_p))$.
  Given a left $D^{\dagger}(\infty)_L$-module $M$,
  $\operatorname{Four}_{\pi}(M)=D^{\dagger}(\infty)_L\otimes_{D^{\dagger}(\infty)_L,S_{\pi}}M[1]$,
  where $S_{\pi}\colon D^{\dagger}(\infty)_L\cong D^{\dagger}(\infty)_L$ is the isomorphism sending $T$ to $\pi^{-1}\partial$ and $\partial$ to $-\pi T$.
\end{theorem}
\begin{proof}
  The $(p-1)$-th root $\pi$ of $-p$ defines an exponential sheaf and hence a Fourier--Mukai functor $\operatorname{FM}$ in Section \ref{Cartierduality},
  which is an equivalence by Theorem \ref{CartierarithmeticdR}.
  We define $\operatorname{Four}_{\pi}$ to be this Fourier--Mukai functor,
  via the equivalence in \ref{A1arithcoefficient}.
  To give the description as desired,
  first,
  we notice that the $L\langle T\rangle^{\dagger}$-module structure on $\operatorname{Four}_{\pi}(M)=p_{1,!}(p^*_2M\otimes \L_{\pi})$ is given by $f^*\circ \operatorname{FM}(M)$,
  which is $\operatorname{FM}'\circ g_!(M)$ in Lemma \ref{compatibilityFM2}.
  However,
  $\operatorname{FM}'\circ g_!(M)$ is exactly the $L\langle T\rangle^{\dagger}$-module given by $T$ acting as $\pi^{-1}\partial$.
  We consider again the Cartesian diagram 
  \[\xymatrix{
    \disk^{\dagger}\ar[r]^{\pi}\ar[d]_p & \disk^{\dagger}/\disk^{\circ}\ar[d]^g\\
    \text{*}\ar[r]_{f} & B\disk^{\circ}
  },\]
  and $\operatorname{FM}'\circ g_!(M)$ is given as $f^*g_!M$ equipped with the $L\langle T\rangle^{\dagger,\sharp}$-module structure coming from monadicity,
  up to a dimension shifting by $1$,
  as Cartier duality shifts the $t$-structure,
  cf. \cite[Section 4.3]{RC24a}.
  By base change formula,
  and because base change is compatible with monadicity,
  we have $f^*g_!M\cong p_!\pi^*M$ and the $L\langle T\rangle^{\dagger}$-module structure on $p_!\pi^*M$ comes from the action of $L\langle \partial\rangle^{\dagger,\sharp}$.
  Fixing the isomorphism in \ref{choosegauge} (which depends on the $(p-1)$-th root $\pi$ of $p$),
  we know that the $L\langle T\rangle^{\dagger}$-module structure on $\operatorname{FM}'\circ g_!(M)$ is given by $T$ acting as $\pi^{-1}\partial$.
  On the other hand,
  the Weyl identity $\partial T=T\partial +1$ restricts $\partial$ to act as $-\pi T$.
  This proves the desired description.
\end{proof}

\begin{remark}
  Our Fourier transform differs from the Fourier--Huyghe transform in \cite{Huy04} by a dimension shift,
  where the shift there is used to match the $t$-structures on both sides.
\end{remark}

\subsection{Kedlaya's finite dimensionality}\label{SectionKedfinite}
We already know that if an overconvergent isocrystal does not admit an $F$-structure,
then it may have bad behavior in the sense
that its cohomologies may not be finite-dimensional,
as we have seen in Remark \ref{padicLiouville}.
However,
with $F$-structure,
this issue disappears.
In this section,
we reprove the following result from \cite{Ked06}:
\begin{theorem}\label{Kedlayafinitedimensionality}
  Let $X$ be a smooth scheme over $k$,
  let $K$ admit a Frobenius lift,
  and let $\E$ be an overconvergent $F$-isocrystal on $X$.
  Then $H^i_{\rig}(X,\E)$ is finite-dimensional over $K$.
\end{theorem}
\begin{proof}
  Without loss of generality,
  we can assume $k$ to be perfect and $K=W(k)[p^{-1}]$ by Remark \ref{passingtononperfect}.
  Let $f\colon X\rightarrow\Spec k$.
  Theorem \ref{arithrigidcoh2} and Theorem \ref{rigidstackrigidcohomology} tell us that
  both of them are stacky approaches to the theory of overconvergent $F$-isocrystals.
  Precisely,
  we can view $\E$ as either a perfect complex $M$ on $X^{\HK/K}$,
  or a perfect complex $N$ on $X^{\arith/K}/\varphi^{\itg}_X$,
  such that
  \[
    R\Gamma_{\rig}(X,\E)\cong f_*^{\HK/K}M\cong f^{\arith/K}_*N\in D_{\blacksquare}(K).
  \]
  We have the following observations:
  \begin{enumerate}
    \item $f^{\arith/K}_*N$ is a basic nuclear object.
    By \cite[Proposition 5.5.12]{ABLBRCS25},
    the $f^{\arith/K}_!$ preserves locally nuclear objects and $\omega_1$-compact objects\footnote{Even though $X^{\arith}$ is not a qfd Berkovich space,
    still Zariski locally it admits a suave cover from a qfd Berkovich space with all terms in Cech nerves being qfd Berkovich spaces,
    so we can adapt this proposition to $X^{\arith}$.}.
    Then $f^{\arith/K}_*$ also preserves locally nuclear objects and $\omega_1$-compact objects,
    as $f^{\arith/K}$ is prim with invertible prim dual by Theorem \ref{arithsixfunctor}.
    Since $N$ is dualizable which is both locally nuclear and $\omega_1$-compact,
    $f^{\arith/K}_*N$ is also both locally nuclear and $\omega_1$-compact hence basic nuclear,
    cf. \cite[Lemma A.0.6]{ABLBRCS25}.
    \item $f^{\HK/K}_*M$ is the dual of a basic nuclear object.
    We have $f^{\HK/K}_*M\cong (f^{\HK/K}_!(M^{\vee}\otimes (f^!1)^{-1}))^{\vee}$ as $f^{\HK/K}$ is cohomologically smooth by Theorem \ref{PoincaredualityrigidFrobeniusrelative},
    and $f^{\HK/K}_!$ preserves locally nuclear objects and $\omega_1$-compact objects by \cite[Proposition 5.5.12]{ABLBRCS25},
    so $f^{\HK/K}_!(M^{\vee}\otimes (f^!1)^{-1})$ is both locally nuclear and $\omega_1$-compact hence basic nuclear by \cite[Lemma A.0.6]{ABLBRCS25},
    and then $f^{\HK/K}_*M\cong (f^{\HK/K}_!(M^{\vee}\otimes (f^!1)^{-1}))^{\vee}$ is the dual of a basic nuclear object.
  \end{enumerate} 
  Therefore,
  $R\Gamma_{\rig}(X,\E)$ is a perfect complex by Lemma \ref{basicnucleardualizable} below and the fact that it is bounded because it is locally a de Rham complex,
  hence each cohomology $H^i_{\rig}(X,\E)$ is finite-dimensional.
\end{proof}

\begin{lemma}\label{basicnucleardualizable}
  If a bounded complex of solid $K$-module $M\in D_{\blacksquare}(K)$ is both a basic nuclear object and the dual of a basic nuclear object,
  then it is a perfect complex.
\end{lemma}
\begin{proof}
  Write $M=N^{\vee}$ with $M$ and $N$ being basic nuclear.
  A basic nuclear module can be represented by a complex whose terms are dual nuclear Fréchet,
  cf. \cite[Proposition A.0.12]{ABLBRCS25},
  hence $H^i(M)$ and $H^i(N)$ are quotients of dual nuclear Fréchet spaces.
  Following the proof of \cite[Theorem 13.6]{CS22},
  we have a short exact sequence 
  \[
  0\longrightarrow \Ext^1(H^{-i+1}(N)^0,K)=H^i(M)^0\longrightarrow H^i(M)\longrightarrow \Hom(\overline{H^{-i}(N)},K)=\overline{H^i(M)}\longrightarrow 0,
  \]
  where $V^0$ (resp. $\overline{V}$) is the closure of $0$ in $V$ (resp. the maximal quasi-separated quotient).

  We first prove that $\overline{H^i(M)}$ is finite-dimensional.
  Since $\overline{H^i(M)}$ is quasi-separated and being a quotient of DNF spaces,
  we can deduce that it is nuclear Fréchet from the fact that $M$ is quasi-separated by \cite[Lemma A.0.11]{ABLBRCS25}.
  The same also holds for $\overline{H^{-i}(N)}$,
  hence $\overline{H^i(M)}$ is also the dual of a nuclear Fréchet space.
  From \cite[Theorem 3.40.(2)]{RCRJ22},
  we know that if $V$ is Fréchet and $W$ is LS,
  then $\underline{\Hom}_K(W,V)=W^{\vee}\otimes_{K,\blacksquare}V$.
  Applying it to $W=V=\overline{H^i(M)}$,
  we obtain that $\underline{\Hom}(\overline{H^i(M)},\overline{H^i(M)})\cong \overline{H^i(M)}^{\vee}\otimes \overline{H^i(M)}$ which implies that $\overline{H^i(M)}$ is dualizable hence finite dimensional,
  cf. \cite[Lemma B.1.16.(ii)]{HM24}.

  We then prove that $H^i(M)^0$ is also finite-dimensional.
  Since the category of static basic nuclear $K$-modules form an abelian subcategory,
  cf. \cite[Corollary 8.17]{CS22},
  we have that both $H^i(M)^0$ and $H^{-i+1}(N)^0$ are basic nuclear,
  and hence are quotients of dual nuclear Fréchet spaces,
  i.e. we have short exact sequences
  \[
    0\longrightarrow V_1\longrightarrow V_2\longrightarrow H^i(M)^0\longrightarrow 0,\quad
    0\longrightarrow V_1'\longrightarrow V_2'\longrightarrow H^{-i+1}(N)^0\longrightarrow 0
  \]
  with $V_1, V_2,V_1',V_2'$ being dual nuclear Fréchet.
  This also induces a short exact sequence 
  \[
  0\longrightarrow V_2'^{\vee}\longrightarrow V_1'^{\vee}\longrightarrow\Ext^1(H^{-i+1}(N)^0,K)\longrightarrow0.
  \]
  We claim that this extends to a diagram of short exact sequences
  \[\xymatrix{
    0\ar[r] & V_1\ar[r]\ar[d] & V_2\ar[r]\ar[d]^f & H^i(M)^0\ar[r]\ar@{=}[d] & 0\\
    0\ar[r] & V_2'^{\vee}\ar[r]^g & V_1'^{\vee}\ar[r] & \Ext^1(H^{-i+1}(N)^0,K)\ar[r] & 0
  }\]
  Indeed,
  the obstruction of such a lift lies in $\Ext^1(V_2,V_2'^{\vee})$,
  which vanishes as $\Ext^1(V_2,V_2'^{\vee})=\Ext^1(V_2\otimes_{K,\blacksquare}V_2',K)=0$ by topological Mittag-Leffler,
  cf. \cite[Lemma 3.27]{RCRJ22}.
  Let $f\colon V_2\longrightarrow V_1'^{\vee}$ be such a map.
  By \cite[Lemma 3.36.(1)]{RCRJ22},
  one can write $V_2=\varinjlim_n P_n$ as a filtered colimit of Banach spaces with injective trace-class transition maps.
  Let $W_n=g(V_2'^{\vee})+f(P_n)$,
  then we have $V_1'^{\vee}=\bigcup_n W_n$.
  Since $V_1'^{\vee}$ is a Baire space as it is nuclear Fréchet,
  we obtain that at least one $W_n$ is of the second Baire category\footnote{A subset is called of the second Baire category if it is not a countable union of nowhere dense subsets,
  cf. \cite[Definition 2.1]{Rud91}.}.
  By Lemma \ref{nonmeagre},
  there exists one $n$ such that $W_n=V_1'^{\vee}$.
  Then the map $(g,f|_{P_n})\colon V_2'^{\vee}\oplus P_n\longrightarrow V_1'^{\vee}$ is a surjection.
  However,
  the map $(0,f|_{P_n})\colon V_2'^{\vee}\oplus P_n\longrightarrow V_1'^{\vee}$ is a trace-class map,
  by the assumption that $P_n\rightarrow V_2$ is trace-class and \cite[Lemma 8.2]{CS22}.
  Then by Lemma \ref{nuclearperturbation},
  we have that $H^i(M)^0=\coker((g,0))=\coker((g,f|_{P_n})-(0,f|_{P_n}))$ is finite dimensional.

  Combining both,
  we obtain that $H^i(M)$ is discrete and finite-dimensional.
  With the boundedness,
  this proves that $M$ is a perfect complex.
\end{proof}

\begin{lemma}[Open embedding theorem]\label{nonmeagre}
  Let $f\colon W\rightarrow V$ be a continuous linear map between Fréchet spaces.
  If $\im(f)$ is of the second Baire category,
  then we have $\im(f)=V$.
\end{lemma}
\begin{proof}
  Over an Archimedean base,
  this is \cite[Theorem 2.11]{Rud91}.
  Over a non-Archimedean base,
  one can produce exactly the same proof.
\end{proof}

\begin{lemma}\label{nuclearperturbation}
  Let $f\colon W\rightarrow V$ be an epimorphism
  and $u\colon W\rightarrow V$ be a trace-class morphism,
  as morphisms between static solid $K$-modules.
  Then $\coker(f+u)$ is discrete and finite-dimensional.
\end{lemma}
\begin{proof}
  Note that $\overline{u}\colon W\rightarrow \coker(f+u)$ is an epimorphism as $f$ is.
  By \cite[Lemma 8.4]{CS22},
  we can find a factoring $W\rightarrow C\rightarrow V$ of $u$ with $C$ being compact and $W\rightarrow C$ being trace-class.
  Moreover,
  we can find a compact projective $P$ which admits an epimorphism $P\rightarrow C$.
  Then the composition gives an epimorphism $g\colon P\rightarrow \coker(f+u)$.
  Moreover,
  since $P$ is projective,
  the morphism $P\rightarrow V$ admits a lift $P\rightarrow W$ such that the following diagram commutes
  \[
  \xymatrix{P\ar[r]\ar[rd] & W\ar[d]^f\\
  & V}
  \]
  Therefore,
  $g\colon P\rightarrow \coker(f+u)$ is trace-class as it factors as $P\rightarrow W\overset{-\overline{u}}{\rightarrow} \coker(f+u)$
  with $\overline{u}$ being trace-class,
  by \cite[Lemma 8.2]{CS22}.
  Hence $g$ comes from an element in $P^{\vee}\otimes \coker(f+u)(*)$.
  However,
  there is a surjection 
  \[
    P^{\vee}\otimes P(*)\longrightarrow P^{\vee}\otimes \coker(f+u)(*),
  \]
  hence we can find a trace-class endomorphism $h\colon P\rightarrow P$ such that $gh=g$,
  i.e. $g(1-h)=0$,
  which induces an epimorphism $\overline{g}\colon \coker(1-h)\rightarrow \coker(f+u)$.
  By the Fredholm property in \ref{Fredholmproperty1},
  we know that $\coker(1-h)$ is discrete and finite-dimensional.
  Since $\overline{g}$ is an epimorphism,
  we deduce that $\coker(f+u)$ is also discrete and finite-dimensional.
\end{proof}

\begin{remark}
  If one admits \cite[Conjecture 3.41]{RCRJ22},
  then \cite[Corollary 3.42]{RCRJ22} holds,
  and the proof of Lemma \ref{basicnucleardualizable} can be simplified.
  We decide to use an ad hoc method above to avoid discussing this conjecture\footnote{This conjecture is related to the continuum hypothesis,
  and is probably independent of ZFC;
  see \cite{BLH25}.}.
\end{remark}

\begin{remark}
  The full result of \cite{Ked06} states that for a separated scheme $X$ of finite type \textit{not necessarily smooth} over $k$,
  and an overconvergent $F$-isocrystal $\E$ on $X$,
  we have that $H^i_{\rig}(X,\E)$ is finite dimensional over $K$.
  However,
  it is actually formal to deduce this full result from the case of smooth schemes which we reproved above,
  by using a cohomological descent method of Chiarellotto and Tsuzuki,
  where they proved proper hyperdescent of rigid cohomology.
  cf. \cite[Section 9.2]{Ked06}.
  This cohomological descent property can also be recovered from the stacky approach;
  see Corollary \ref{properhyperdescentTsuzuki}.
\end{remark}

\begin{remark}
  One may ask if Theorem \ref{Kedlayafinitedimensionality} has a relative version in the sense that pushforwards along smooth morphisms $X\rightarrow Y$ preserve overconvergent $F$-isocrystals.
  Unfortunately,
  it is false in general,
  where a counterexample is given by the open immersion $\mathbb{G}_{m,k}\rightarrow\aff^1_k$.
  Regarding our proof of the theorem,
  Lemma \ref{basicnucleardualizable} is very special for the base $K$ instead of arbitrary base.
\end{remark}

\subsection{Overconvergence of rigid cohomology over Laurent series fields}\label{sectionoverconvergent}
Let $k$ be a finite field with and let $K=W(k)[p^{-1}]$.
First in \cite{Ked00} for proper smooth schemes,
and later in \cite{LP14a} for a general scheme $X$ over $k(\!(t)\!)$,
it is shown that the rigid cohomology $H^i_{\rig}(X/\E_K)$ over the Amice ring $\E_K$ (which we defined in Example \ref{Laurentarith}) is overconvergent,
i.e. comes from a $(\varphi,\nabla)$-refinement $H^i_{\rig}(X/\E_K^{\dagger})$ over the bounded Robba ring $\E_K^{\dagger}$,
which means that $H^i_{\rig}(X/\E_K^{\dagger})$ is a $(\varphi,\nabla)$-module over $\E_K^{\dagger}$ and $H^i_{\rig}(X/\E_K^{\dagger})\otimes_{\E_K^{\dagger}}\E_K\cong H^i_{\rig}(X/\E_K)$.
We review this overconvergent nature in this section.

\begin{definition}[Bounded Robba ring]
  Let $\E_K^{\dagger}$ be the bounded Robba ring,
  equipped with the \textit{$p$-adic} topology:
  \[
    \E_K^{\dagger}:=
    \left\{
    \sum_i a_i t^i\in K[\![t^{\pm 1}]\!]\bigg| \sup|a_i|<\infty,\,\e \eta<1,\lim_{i\rightarrow-\infty}|a_i|\eta^i=0.
    \right\}
  \]
\end{definition}

\begin{lemma}
  The bounded Robba ring $\E_K^{\dagger}$ is Gelfand,
  and we have an isomorphism of Gelfand stacks $\operatorname{GSpec}\E_K^{\dagger}\cong \disk^{\circ\circ}_K\times_{\disk^{\dagger}_K}\mathbb{T}^{\dagger}_K$.
  Therefore,
  we have a closed-open decomposition $\operatorname{GSpec}\E_K^{\dagger}\hookrightarrow \disk_K^{\circ\circ}\hookleftarrow\disk_K^{\circ}$.
\end{lemma}
\begin{proof}
  Since we put the $p$-adic topology on the bounded Robba ring,
  it is automatically Gelfand.
  We also have an isomorphism of Gelfand rings 
  \[
    \E_K^{\dagger}\cong K^{\circ}[\![t]\!][p^{-1}]\otimes_{K\bra t\ket^{\dagger}}K\bra t^{\pm 1}\ket^{\dagger},
  \]
  which gives the isomorphism.
\end{proof}

Due to the lemma above,
we will also denote $\operatorname{GSpec}\E_K^{\dagger}$ by $\disk^{\circ\circ}_K\backslash \disk^{\circ}_K$.

\begin{construction}\label{boundedRobbaring}
  We study the arithmetic de Rham stack and the Hyodo--Kato stacks of $\Spec k(\!(t)\!)$ via the bounded Robba ring.
  In Example \ref{Laurentarith},
  we constructed a chart $\pi \colon (\operatorname{GSpec} \E_{K})^{\dR}\longrightarrow(\Spec k(\!(t)\!))^{\arith}$.
  Since the uniform completion of $\E_K^{\dagger}$ is $\E_K$,
  we have $(\disk^{\circ\circ}_K\backslash \disk^{\circ}_K)^{\dR}\cong (\operatorname{GSpec}\E_K^{\dagger})^{\dR}\cong (\operatorname{GSpec}\E_K)^{\dR}$.
  Hence we can rewrite $\pi$ as 
  $\pi \colon (\disk^{\circ\circ}_K\backslash \disk^{\circ}_K)^{\dR}\longrightarrow(\Spec k(\!(t)\!))^{\arith}$,
  which will induce a morphism on their perfections modulo Frobenius:
  \[
    \pi^{\perf} \colon \varprojlim_{\varphi}(\disk^{\circ\circ}_K\backslash \disk^{\circ}_K)^{\dR}/\varphi^{\itg}\longrightarrow(\Spec k(\!(t)\!))^{\arith}/\varphi^{\itg}.
  \]
  
  On the other hand,
  by Example \ref{LaurentHK},
  we also have
  \[
      (\Spec \ff_p(\!(t)\!))^{\RIG}_{\le 1}\cong \varprojlim_{\varphi}(\disk^{\circ\circ}_K\backslash \disk^{\circ}_K)^{\dR},\,
      (\Spec\ff_p(\!(t)\!))^{\RIG}\cong (\varprojlim_{\varphi}\disk_{\rt_p}^{\circ\circ,\times})^{\dR}.
  \]
  Then the diagram in \ref{smallrigidbridge} for $\Spec k(\!(t)\!)$ becomes
  \[
    \xymatrix{
       & \varprojlim_{\varphi}(\disk^{\circ\circ}_K\backslash \disk^{\circ}_K)^{\dR}/\varphi^{\itg}\ar[dr]^{i}\ar[dl]_{\pi^{\perf}} & \\
      (\Spec \ff_p(\!(t)\!))^{\arith}/\varphi^{\itg} & & \varprojlim_{\varphi}\disk_{K}^{\circ\circ,\times,\dR}/\varphi^{\itg}
    }
  \]
  
  One can extend this diagram as follows.
  Let $k(\!(t)\!)$ be the field of Laurent series equipped with the \textit{$t$-adic} topology,
  and consider its associated diamond $\Spd k(\!(t)\!)$.
  Then we have that $(\Spd k(\!(t)\!))^{\HK}\cong \varprojlim_{\varphi}\disk^{\circ,\times,\dR}_K/\varphi^{\itg}$,
  by \cite[Lemma 6.2.1]{ABLBRCS25}.
  The map $u\colon \Spd k(\!(t)\!)\rightarrow (\Spec k(\!(t)\!))_{\arc}$ induces a map on their Hyodo--Kato stacks,
  which is the embedding $j\colon \varprojlim_{\varphi}\disk^{\circ,\times,\dR}_K/\varphi^{\itg}\rightarrow \varprojlim_{\varphi}\disk^{\circ\circ,\times,\dR}_K/\varphi^{\itg}$,
  whence we obtain the following diagram:
  \[
    \xymatrix{
       & \varprojlim_{\varphi}(\disk^{\circ\circ}_K\backslash \disk^{\circ}_K)^{\dR}/\varphi^{\itg}\ar[dr]^{i}\ar[dl]_{\pi^{\perf}} & \varprojlim_{\varphi}\disk_{K}^{\circ,\times,\dR}/\varphi^{\itg}\ar[d]^{j}\\
      (\Spec \ff_p(\!(t)\!))^{\arith}/\varphi^{\itg} & & \varprojlim_{\varphi}\disk_{K}^{\circ\circ,\times,\dR}/\varphi^{\itg}
    }
  \]
  It is not hard to find that $i$ and $j$ form a closed-open decomposition of the common target.
\end{construction}

We can study the geometry of the construction above.

\begin{proposition}\label{geometryRobba}
  We have the following descriptions and equivalences of categories.
  \begin{enumerate}
    \item Let $\widehat{\mathbb{G}}_{a,K}$ be the Gelfand stack sending $A$ to the nilradical $\Nil(A)$.
    There exist natural actions of $\widehat{\mathbb{G}}_{a,K}$ on $\disk_{K}^{\circ\circ,\times}$ and $\disk_{K}^{\circ\circ}\backslash\disk_{K}^{\circ}$
    such that the natural maps $\disk_{K}^{\circ\circ,\times}\rightarrow \disk_{K}^{\circ\circ,\times,\dR}$ and $\disk_{K}^{\circ\circ}\backslash\disk_{K}^{\circ}\rightarrow (\disk_{K}^{\circ\circ}\backslash\disk_{K}^{\circ})^{\dR}$
    factor through 
    \[
      \disk_{K}^{\circ\circ,\times}/\widehat{\mathbb{G}}_{a,K}\longrightarrow\disk_{K}^{\circ\circ,\times,\dR},\quad
      (\disk_{K}^{\circ\circ}\backslash\disk_{K}^{\circ})/\widehat{\mathbb{G}}_{a,K}\longrightarrow (\disk_{K}^{\circ\circ,\times}\backslash\disk_{K}^{\circ,\times})^{\dR}.
    \]
    Moreover,
    there exists an equivalence of categories induced by the pullback functor
    \[
    \cat{Perf}(\varprojlim_{\varphi}(\disk_{K}^{\circ\circ,\times}/\widehat{\mathbb{G}}_{a,K})/\varphi^{\itg})\cong 
    \cat{Perf}(\varprojlim_{\varphi}((\disk_{K}^{\circ\circ}\backslash\disk_{K}^{\circ})/\widehat{\mathbb{G}}_{a,K})/\varphi^{\itg}).
    \]
    Moreover,
    their hearts are both equivalent to the category of $(\varphi,\nabla)$-modules over $\E_K^{\dagger}$.
    \item There exists an equivalence of categories
    \[
      \cat{Perf}(\disk_{K}^{\circ\circ}\backslash\disk_{K}^{\circ})^{\varphi\textnormal{-equiv}}\cong
      \cat{Perf}(\varprojlim_{\varphi}(\disk_{K}^{\circ\circ}\backslash\disk_{K}^{\circ})/\varphi^{\itg})
    \]
    and also a compatible equivalence with the pullback along $\operatorname{GSpec}\E_K\rightarrow\operatorname{GSpec}\E_K^{\dagger}=\disk_{K}^{\circ\circ}\backslash\disk_{K}^{\circ}$:
    \[
      \cat{Perf}(\operatorname{GSpec} \E_{K})^{\varphi\textnormal{-equiv}}\cong
      \cat{Perf}(\varprojlim_{\varphi}\operatorname{GSpec} \E_{K}/\varphi^{\itg}).
    \]
    \item \textnormal{(\cite[Corollary 7.2.10]{ABLBRCS25}.)} There is an equivalence of categories between $\cat{Perf}^{\heartsuit}(\varprojlim_{\varphi}\disk^{\circ,\times,\dR}_K/\varphi^{\itg})$ and
    the category of $(\varphi,\nabla)$-modules over the \textit{unbounded} Robba ring $\R_K$.
  \end{enumerate}
\end{proposition}
\begin{proof}
  First we prove $(1)$.
  The description of the $\widehat{\mathbb{G}}_{a,K}$-action on $\disk^{\circ\circ,\times}_K$ is given at the functor-of-points level,
  by sending an element $(\epsilon,a)$ in $\widehat{\mathbb{G}}_{a,K}\times \disk_{K}^{\circ\circ,\times}(A)=\Nil(A)\times \disk_{K}^{\circ\circ,\times}(A)$
  to the element in $\disk_{K}^{\circ\circ,\times}(A)$ that valuates on each $f(t)$ the element $\sum_{n\ge 0}a(\frac{f^{(n)}(t)}{n!})\epsilon^n$,
  which is a finite sum as $\epsilon\in\Nil(A)$ hence makes sense.
  There is also a $\mathbb{G}^{\dagger}_{a,K}$-action on $\torus^{\dagger}_{K}$,
  cf. \cite[Remark 5.1.9]{ABLBRCS25}, hence a compatible $\widehat{\mathbb{G}}_{a,K}$-action on $\torus^{\dagger}_{K}$,
  whence the $\widehat{\mathbb{G}}_{a,K}$-action on $\disk_{K}^{\circ\circ}\backslash\disk_{K}^{\circ}$ comes.
  The natural map $\disk_{K}^{\circ\circ,\times}\rightarrow \disk_{K}^{\circ\circ,\times,\dR}$ satisfies the following commutative diagram 
  \[\xymatrix{
    \widehat{\mathbb{G}}_{a,K}\times \disk_{K}^{\circ\circ,\times}\ar[r]^p\ar[d]_{\operatorname{act}} & \disk_{K}^{\circ\circ,\times}\ar[d]\\
    \disk_{K}^{\circ\circ,\times}\ar[r] & \disk_{K}^{\circ\circ,\times,\dR}
  }\]
  This can also be checked at the functor-of-points level,
  that the difference between the $A$-points $\disk_{K}^{\circ\circ,\times}$ obtained by projection and action is given by valuating on each $f(T)$ the element $\sum_{n\ge 1}a(\frac{f^{(n)}(T)}{n!})\epsilon^n$ which lies in $\Nil(A)$.
  Hence the map factors through $\disk_{K}^{\circ\circ,\times}\rightarrow \disk_{K}^{\circ\circ,\times}/\widehat{\mathbb{G}}_{a,K}\rightarrow \disk_{K}^{\circ\circ,\times,\dR}$.
  The case of $\disk_{K}^{\circ\circ}\backslash\disk_{K}^{\circ}$ also follows.
  To prove the equivalence of categories,
  we note that $\disk^{\circ\circ,\times}_K=\disk^{\circ\circ}_K\times_{\disk^{\dagger}_K}\disk^{\dagger,\times}_K$,
  and $\disk_{K}^{\circ\circ}\backslash\disk_{K}^{\circ}=\disk^{\circ\circ}_K\times_{\disk^{\dagger}_K}\torus^{\dagger}_K$.
  Then the morphism $\varprojlim_{\varphi}((\disk_{K}^{\circ\circ}\backslash\disk_{K}^{\circ})/\widehat{\mathbb{G}}_{a,K})/\varphi^{\itg}\rightarrow \varprojlim_{\varphi}(\disk_{K}^{\circ\circ,\times}/\widehat{\mathbb{G}}_{a,K})/\varphi^{\itg}$
  is of the form $\iota$ in Lemma \ref{prototype},
  hence induces an equivalence of perfect complexes.
  Finally,
  since $\disk_{K}^{\circ\circ}\backslash\disk_{K}^{\circ}=\operatorname{GSpec}\E_K^{\dagger}$,
  the heart of the category of perfect complexes over $\varprojlim_{\varphi}((\disk_{K}^{\circ\circ}\backslash\disk_{K}^{\circ})/\widehat{\mathbb{G}}_{a,K})/\varphi^{\itg}$
  is equivalent to the category of $(\varphi,\nabla)$-modules over $\E_K^{\dagger}$.

  The property $(2)$ is directly from spreading out of perfect complexes,
  cf. \cite[0BC7]{Stacks}.
  The property $(3)$ is from \cite[Corollary 7.2.10]{ABLBRCS25}.
\end{proof}

\begin{remark}
  The disk $\disk_{K}^{\circ\circ}$ and the punctured disk $\disk_{K}^{\circ\circ,\times}$ behave quite pathological.
  For example,
  contrary to the open unit disk $\disk_{K}^{\circ}$,
  the disk $\disk_{K}^{\circ\circ}$ is neither equipped with a group structure,
  nor with a $\mathbb{G}_{a,K}^{\dagger}$-action.
  This makes the description of its analytic de Rham stack unclear\footnote{One can also use the \emph{$(p,t)$-adic topology} on the bounded Robba ring or $\itg_p[\![t]\!][p^{-1}]$,
  then there should be $\dagger$-formal smoothness,
  a Hopf algebra structure and a $\mathbb{G}_{a,K}^{\dagger}$-action,
  hence a description of their analytic de Rham stacks in terms of quotient by the $\mathbb{G}_{a,K}^{\dagger}$-action.
  However,
  with the $(p,t)$-adic topology,
  they are \emph{not} Gelfand,
  which only makes sense in \cite{RC24a} and is beyond the setting of this paper.},
  and our best approximation is to use the $\widehat{\mathbb{G}}_{a,K}$-action mentioned in Proposition \ref{geometryRobba}.
\end{remark}

Following the discussion above,
we give the following definition.
\begin{definition}
  Let $X$ be a separated scheme of finite type over $\Spec k(\!(t)\!)$ and $f\colon X\rightarrow \Spec k(\!(t)\!)$.
  We define $R\Gamma_{\rig}(X/\E^{\dagger}_K)\in D((\Spec k(\!(t)\!))^{\HK})$ to be the following object:
  \[
    R\Gamma_{\rig}(X/\E^{\dagger}_K):=f^{\HK}_{*}1\in D((\Spec k(\!(t)\!))^{\HK})=D(\varprojlim_{\varphi}\disk_{K}^{\circ\circ,\times,\dR}/\varphi^{\itg}).
  \]
\end{definition}

The following is the main theorem of this section.

\begin{theorem}\label{LazdaPal}
  The object $R\Gamma_{\rig}(X/\E^{\dagger}_K)$ has realizations in the categories in Proposition \ref{geometryRobba} as follows.
  \begin{enumerate}
    \item The pullback of $R\Gamma_{\rig}(X/\E^{\dagger}_K)$ along
    \[
      \varprojlim_{\varphi}(\disk_{K}^{\circ\circ}\backslash\disk_{K}^{\circ})/\varphi^{\itg}\longrightarrow
      \varprojlim_{\varphi}(\disk_{K}^{\circ\circ}\backslash\disk_{K}^{\circ})^{\dR}/\varphi^{\itg}\overset{i}{\longrightarrow}
      \varprojlim_{\varphi}\disk_{K}^{\circ\circ,\times,\dR}/\varphi^{\itg}
    \]
    gives rise to a perfect complex in $\cat{Perf}(\varprojlim_{\varphi}(\disk_{K}^{\circ\circ}\backslash\disk_{K}^{\circ})/\varphi^{\itg})$.
    Moreover,
    the further pullback of this object to $\varprojlim_{\varphi}\operatorname{GSpec} \E_{K}/\varphi^{\itg}$
    computes the $\varphi$-equivariant rigid cohomology $R\Gamma_{\rig}(X/\E_K)$ of $X$ over $\Spec k(\!(t)\!)$.
    \item The pullback of $R\Gamma_{\rig}(X/\E^{\dagger}_K)$ along 
    \[
      \varprojlim_{\varphi}(\disk_{K}^{\circ\circ,\times}/\widehat{\mathbb{G}}_{a,K})/\varphi^{\itg}\longrightarrow\varprojlim_{\varphi}\disk_{K}^{\circ\circ,\times,\dR}/\varphi^{\itg}
    \]
    lies in 
    $\cat{Perf}(\varprojlim_{\varphi}(\disk_{K}^{\circ\circ,\times}/\widehat{\mathbb{G}}_{a,K})/\varphi^{\itg})$,
    hence gives a $(\varphi,\nabla)$-module over $\E_K^{\dagger}$ on each cohomological piece.
    \item 
    The pullback of $H^i_{\rig}(X/\E^{\dagger}_K)$ along
    \[
      j\colon \varprojlim_{\varphi}\disk^{\circ,\times,\dR}_K/\varphi^{\itg}\longrightarrow
      \varprojlim_{\varphi}\disk^{\circ\circ,\times,\dR}_K/\varphi^{\itg}
    \]
    is a $(\varphi,\nabla)$-module over $\R_K$,
    which identifies exactly with the base change of the $(\varphi,\nabla)$-modules over $\E_K^{\dagger}$ to $\R_K$.
  \end{enumerate}
\end{theorem}
\begin{proof}
  First we prove $(1)$.
  The pullback of $R\Gamma_{\rig}(X/\E^{\dagger}_K)$ to $\operatorname{GSpec} \E_{K,\perf}/\varphi^{\itg}$ is nothing but $f^{\HK/\E_{K,\perf}}_*1$,
  which is identified with $R\Gamma_{\rig}(X/\E_K)\otimes_{\E_K}\E_{K,\perf}$ by Remark \ref{passingtononperfect},
  which is a perfect complex by the finite-dimensionality of rigid cohomology in Theorem \ref{Kedlayafinitedimensionality}.
  However,
  the (bounded) Robba ring is a field,
  hence we have the decompletion of perfect complexes\footnote{Let $K$ be a Gelfand field and $\widehat{K}$ be its $p$-adic completion.
  The decompletion of perfect complexes means that if the pullback of an object along $\operatorname{GSpec}\widehat{K}\rightarrow\operatorname{GSpec}K$ is a perfect complex,
  then the object itself is also a perfect complex.} along $\varprojlim_{\varphi}\operatorname{GSpec} \E_{K}/\varphi^{\itg}\leftarrow\operatorname{GSpec} \E_{K,\perf}/\varphi^{\itg}$,
  which gives rise to a perfect complex on $\varprojlim_{\varphi}\operatorname{GSpec} \E_{K}/\varphi^{\itg}$.
  Via the equivalence in Proposition \ref{geometryRobba},
  it corresponds to $R\Gamma_{\rig}(X/\E_K)$ with $\varphi$-equivariant structure.
  A further decompletion from $\operatorname{GSpec}\E_K$ to $\operatorname{GSpec}\E_K^{\dagger}=\disk^{\circ\circ}_K\backslash\disk^{\circ}_K$ gives rise to a perfect complex in $\cat{Perf}(\varprojlim_{\varphi}(\disk_{K}^{\circ\circ}\backslash\disk_{K}^{\circ})/\varphi^{\itg})$.

  Then we prove $(2)$.
  By Proposition \ref{geometryRobba},
  it suffices to show that  the pullback of $R\Gamma_{\rig}(X/\E^{\dagger}_K)$ to
  $\varprojlim_{\varphi}((\disk_{K}^{\circ\circ}\backslash\disk_{K}^{\circ})/\widehat{\mathbb{G}}_{a,K})/\varphi^{\itg}$ is a perfect complex.
  By descent of perfect complexes,
  it suffices to show that the pullback to $\varprojlim_{\varphi}(\disk_{K}^{\circ\circ}\backslash\disk_{K}^{\circ})/\varphi^{\itg}$ is a perfect complex,
  which follows from $(1)$.

  Finally,
  for $(3)$,
  we notice that there is a commutative diagram 
  \[\xymatrix{
    \varprojlim_{\varphi}(\disk^{\circ,\times}_K/\widehat{\mathbb{G}}_{a,K})/\varphi^{\itg}\ar[r]\ar[d] &
      \varprojlim_{\varphi}(\disk^{\circ\circ,\times}_K/\widehat{\mathbb{G}}_{a,K})/\varphi^{\itg}\ar[d]\\
    \varprojlim_{\varphi}\disk^{\circ,\times,\dR}_K/\varphi^{\itg}\ar[r] &
      \varprojlim_{\varphi}\disk^{\circ\circ,\times,\dR}_K/\varphi^{\itg}
  }\]
  The pullback along the left vertical map induces an equivalence of vector bundles,
  cf. \cite[Remark 5.2.3]{ABLBRCS25}.
  The pullback along the upper horizontal map on vector bundles gives exactly the base change functor:
  \begin{align*}
    \cat{Mod}^{(\varphi,\nabla)}_{\E_K^{\dagger}}&\longrightarrow
    \cat{Mod}^{(\varphi,\nabla)}_{\R_K}\\
    M & \longmapsto M\otimes_{\E^{\dagger}_K}\R_K.
  \end{align*}
  Therefore,
  we obtain $(3)$ from $(2)$ by comparing two pullbacks.
\end{proof}

\begin{remark}
  The aforementioned theorem recovers part of the results in the series of papers \cite{LP14a}, \cite{LP14b}, \cite{LP15}
  for the unit coefficient.
  We hope that the stacky approaches will also give the theory of general coefficients,
  such as overconvergent $F$-isocrystals (which is already done in those papers) and arithmetic $D$-modules.
  We expect to review it in the future.
\end{remark}

\end{document}